\documentclass[10pt,reqno]{amsart}
\usepackage[margin=1in]{geometry}
\usepackage{bbm}
\usepackage{amsfonts}
\usepackage{latexsym, amssymb, amsmath, amscd, amsthm, amsxtra}
\usepackage{mathtools}
\usepackage{enumerate}
\usepackage[all]{xy}
\usepackage{mathrsfs}
\usepackage{fancyhdr}
\usepackage{listings}
\usepackage{hyperref}
\usepackage{cleveref}
\usepackage{soul}
\usepackage{xcolor}
\usepackage{dsfont}
\usepackage{stmaryrd}
\usepackage{bm}
\usepackage{comment}
\usepackage{diagbox}
\usepackage{verbatim}
\usepackage{caption}
\usepackage{subcaption}

\allowdisplaybreaks[4]
\hypersetup{colorlinks=true}
\usepackage{graphicx, color}
\usepackage{cite} 

\numberwithin{equation}{section}
\numberwithin{figure}{section}

\newcommand{\rmnum}[1]{\uppercase\expandafter{\romannumeral #1}}

\theoremstyle{plain} 
\newtheorem{theorem}{Theorem}[section]
\newtheorem*{theorem*}{Theorem}
\newtheorem{lemma}[theorem]{Lemma}
\newtheorem*{lemma*}{Lemma}

\newtheorem*{corollary*}{Corollary}

\newtheorem*{proposition*}{Proposition}
\newtheorem{definition}[theorem]{Definition}
\newtheorem*{definition*}{Definition}

\newtheorem*{conjecture*}{Conjecture}

\newtheorem{assumption}{Assumption}

\theoremstyle{definition} 
\newtheorem{example}[theorem]{Example}
\newtheorem*{example*}{Example}
\newtheorem{remark}[theorem]{Remark}
\newtheorem*{remark*}{Remark}

\newcommand{\VV}{H_{\Lambda}}
\renewcommand{\r}{\mathrm}  
\newcommand{\cal}{\mathcal} 
\newcommand{\scr}{\mathscr} 
 
\newcommand{\wh}{\widehat}
\newcommand{\wt}{\widetilde}
\renewcommand{\txt}[1]{\text{\rm{#1}}}

\definecolor{darkred}{rgb}{0.9,0,0.3}
\definecolor{darkblue}{rgb}{0,0.3,0.9}

\usepackage{ifthen}
\def\comment#1{\ifthenelse{\isodd{\value{page}}}{\marginpar{\raggedright\scriptsize{\textcolor{darkred}{#1}}}}{\marginpar{\raggedleft\scriptsize{\textcolor{darkred}{#1}}}}}

\renewcommand{\P}{\mathbb{P}}
\newcommand{\E}{\mathbb{E}}
\newcommand{\R}{\mathbb{R}}
\newcommand{\C}{\mathbb{C}}
\newcommand{\N}{\mathbb{N}}
\newcommand{\Z}{\mathbb{Z}}

\newcommand{\cG}{{\mathcal G}}

\newcommand{\cW}{{\mathcal W}}
\newcommand{\cT}{\mathcal T}
\newcommand{\cI}{\mathcal{I}}
\newcommand{\cC}{\mathcal{C}}
\newcommand{\cR}{{\mathcal R}}
\newcommand{\cO}{{\mathcal{O}}}
\newcommand{\cE}{{\mathcal{E}}}

\newcommand{\fc}{\mathfrak c}
\newcommand{\ft}{{\mathfrak t}}
\newcommand{\fr}{{\mathfrak{r}}}

\newcommand{\fR}{\mathfrak R}
\newcommand{\fC}{\mathfrak C}
\newcommand{\fP}{\mathfrak P}
\newcommand{\fM}{\mathfrak M}
\newcommand{\fS}{\mathfrak S}
\newcommand{\fI}{\mathfrak I}
\newcommand{\fE}{\mathfrak E}
\newcommand{\fO}{\mathfrak O}
\newcommand{\sG}{\mathsf G}
\newcommand{\sM}{\mathsf M}

\newcommand{\sR}{\mathsf R}
\newcommand{\sC}{\mathsf C}
\newcommand{\sP}{\mathsf P}
\newcommand{\smm}{\mathsf m}
\newcommand{\scG}{\mathscr{G}}
\newcommand{\scT}{{\mathscr{T}}}

\newcommand{\scR}{{\mathscr{R}}}

\newcommand{\ii}{\mathrm{i}}
\newcommand{\dd}{\mathrm{d}}

\newcommand*{\deq}{\mathrel{\vcenter{\baselineskip0.65ex \lineskiplimit0pt \hbox{.}\hbox{.}}}=}

\renewcommand{\leq}{\le}
\renewcommand{\geq}{\ge}
\renewcommand{\epsilon}{\varepsilon}

\renewcommand{\tilde}{\wt}
\renewcommand{\bar}{\overline}
\renewcommand{\Bar}{\overline}

\newcommand{\qq}[1]{[\![{#1}]\!]}

\newcommand{\p}[1]{({#1})}

\newcommand{\pa}[1]{\left({#1}\right)}

\newcommand{\q}[1]{[{#1}]}

\newcommand{\qa}[1]{\left[{#1}\right]}

\newcommand{\ha}[1]{\left\{{#1}\right\}}

\newcommand{\abs}[1]{\lvert #1 \rvert}

\newcommand{\absa}[1]{\left\lvert #1 \right\rvert}

\newcommand{\norm}[1]{\lVert #1 \rVert}

\newcommand{\norma}[1]{\left\lVert #1 \right\rVert}

\newcommand{\avg}[1]{\langle #1 \rangle}

\newcommand{\avga}[1]{\left\langle #1 \right\rangle}

\newcommand{\diag}{\mathrm{diag}}
\DeclareMathOperator{\tr}{\mathrm{Tr}}

\DeclareMathOperator{\re}{\mathrm{Re}}
\DeclareMathOperator{\im}{\mathrm{Im}}

\newcommand{\dist}{\mathrm{dist}}

\newcommand{\spec}{\mathrm{spec}}

\newcommand{\bu}{{\bf{u}}}
\newcommand{\bv}{{\bf{v}}}
\newcommand{\bw}{{\bf{w}}}

\newcommand{\tLambda}{{\wt{\Lambda}}}
\newcommand{\hLambda}{{\wh{\Lambda}}}

\newcommand{\isize}{\mathit{Size}}
\newcommand{\sreplace}{\mathsf{Replace}}
\newcommand{\scut}[1]{\mathsf{Cut}_{#1}}
\newcommand{\splug}[1]{\mathsf{Plug}_{#1}}

\newcommand{\freplace}{\mathfrak{Replace}}
\newcommand{\fcut}[1]{\mathfrak{Cut}_{#1}}
\newcommand{\fplug}[1]{\mathfrak{Plug}_{#1}}
\newcommand{\finsert}[1]{\mathfrak{Insert}_{#1}}
\newcommand{\fexchange}{\mathfrak{Exchange}}
\newcommand{\fmerge}{\mathfrak{Merge}}
\newcommand{\fslash}[1]{\mathfrak{Slash}_{#1}}

\newcommand{\be}{\begin{equation}}
\newcommand{\ee}{\end{equation}}

\newcommand{\e}{{\varepsilon}}

\newcommand{\rd}{\mathrm{d}}

\newcommand{\bre}{\txt{\boldsymbol{\mathrm{e}}}}

\newcommand{\oo}{\mathrm{o}}
\newcommand{\OO}{\mathrm{O}}

\newcommand{\HS}{\txt{HS}}
\newcommand{\tsc}{{\mathrm{sc}}}

\newcommand{\bsum}[3]{\sum_{#1=1}^D\sum_{#2,#3\in \cI_{#1}}}
\newcommand{\dsum}[1]{\sum_{#1=1}^D}

\newcommand{\iab}{\alpha\beta}
\newcommand{\iba}{\beta\alpha}

\newcommand{\tl}{\tLambda}

\newcommand{\mm}{\ell}
\newcommand{\eq}{\overset{\E}{=}}
\newcommand{\kk}{l'}

\newcommand{\opr}[1]{\mathrm{O}_\prec\left({#1}\right)}
\allowdisplaybreaks

\numberwithin{equation}{section}
\usepackage{environ}

\begin{document}


\title{Localization-delocalization transition for a random block matrix model at the edge}

\author{Jiaqi Fan$^\star$}
\thanks{$^\star$Qiuzhen College, Tsinghua University, Beijing, China, \href{mailto:fanjq24@mails.tsinghua.edu.cn}{fanjq24@mails.tsinghua.edu.cn}}


\author{Bertrand Stone$^\dagger$}
\thanks{$^\dagger$Department of Mathematics, University of California, Los Angeles, Los Angeles, CA, USA, \href{mailto:bertrand.stone@math.ucla.edu}{bertrand.stone@math.ucla.edu}}


\author{Fan Yang$^\ddagger$}
\thanks{$^\ddagger$Yau Mathematical Sciences Center, Tsinghua University, and Beijing Institute of Mathematical Sciences and Applications, Beijing, China,
\href{mailto:fyangmath@mail.tsinghua.edu.cn}{fyangmath@mail.tsinghua.edu.cn}}


\author{Jun Yin$^\S$}
\thanks{$^\S$Department of Mathematics, University of California, Los Angeles, Los Angeles, CA, USA,
\href{mailto:jyin@math.ucla.edu}{jyin@math.ucla.edu}}


\begin{abstract}

Consider a random block matrix model consisting of $D$ random systems arranged along a circle, where each system is modeled by an independent $N\times N$ complex Hermitian Wigner matrix. Neighboring systems interact via an arbitrary deterministic $N\times N$ matrix $A$. 
In this paper, we extend the localization-delocalization transition previously established in \cite{stone2024randommatrixmodelquantum} for the bulk eigenvalue spectrum to the entire spectrum, including the spectral edges. Let $[E^-,E^+]$ denote the support of the limiting spectral density, and define $\kappa_E:=|E-E^+|\wedge |E-E^-|$ as the distance from a given energy $E \in [E^-, E^+]$ to the spectral edges. 
We show that for eigenvalues near $E$, the corresponding eigenvectors undergo a localization–delocalization transition when $\|A\|_{\mathrm{HS}}$ crosses the critical threshold $(\kappa_E + N^{-2/3})^{-1/2}$. In the delocalized phase, the extreme eigenvalues asymptotically follow the Tracy-Widom distribution, while in the localized phase, the edge eigenvalue statistics asymptotically match those of $D$ independent GUE ensembles, up to a deterministic shift.
Our results recover those of \cite{stone2024randommatrixmodelquantum} in the bulk regime, where $\kappa_E \asymp 1$, and further reveal the presence of mobility edges near $E^\pm$ when $1 \ll \|A\|_{\mathrm{HS}} \ll N^{1/3}$. Specifically, bulk eigenvectors corresponding to energies $E$ with $\kappa_E \gg \|A\|_{\mathrm{HS}}^{-2}$ are delocalized, while those with $\kappa_E \ll \|A\|_{\mathrm{HS}}^{-2}$ are localized.

\end{abstract}

\maketitle

{
\hypersetup{linkcolor=black}
\tableofcontents
}

\section{Introduction}

Since the seminal work of Anderson \cite{PhysRev.109.1492}, the phenomenon of Anderson localization/delocalization has been a fundamental framework for understanding the transport properties of electrons in disordered media. The localized and delocalized phases correspond to two distinct physical regimes, distinguished by the spatial behavior of the electron wave function. 
In the localized phase, wave functions are confined to finite spatial regions, suppressing quantum diffusion and resulting in insulating behavior. In contrast, the delocalized phase is characterized by spatially extended wave functions that enable macroscopic quantum transport, leading to conductivity. Over time, this phenomenon has been recognized as a universal feature of a broad class of disordered systems and has become a cornerstone of condensed matter physics, as well as a central topic in mathematical physics and related fields \cite{Lagendijk09Fifty,Abrahams201050years,Sheng2006Intro,lee1985disordered,thouless1974electrons,Anderson1978Local}.

Mathematically, Anderson \cite{PhysRev.109.1492} proposed studying localization through the following random Schr{\"o}dinger operator defined on the $d$-dimensional lattice $\Z^d$ (with the case $d=3$ being of particular physical relevance). This operator, commonly known as the \emph{Anderson model}, is given by:
\be\label{eq;Anderson}
H_{\mathrm{Anderson}}=-\lambda\Delta+ V, 
\ee
where $\Delta$ is the discrete Laplacian on $\Z^d$, $V$ is a random potential with i.i.d.~random diagonal entries, and $\lambda>0$ is a coupling constant that represents the reciprocal of the disorder strength. It is predicted that the Anderson model undergoes a localization-delocalization transition, depending on the energy, dimension, and disorder strength. 
More precisely, in dimensions $d=1$ and $d=2$, the Anderson model exhibits localization at all energies for any nonzero disorder strength $\lambda>0$ \cite{PRL_Anderson,mott1961theory,Borland1963TheNature}. 
In higher dimensions ($d\ge 3$), the behavior is more intricate. In the strong disorder regime (i.e., small $\lambda$), all eigenvectors are expected to be exponentially localized.
In contrast, in the weak disorder regime (i.e., large $\lambda$), it is conjectured that a sharp transition occurs between localized and delocalized phases as the energy crosses a critical threshold, known as the \emph{mobility edge} (see, e.g., \cite{Aizenman_book,Kirsch2007}): near the spectral edges, eigenvectors remain localized, but upon crossing the mobility edge into the bulk of the spectrum, the eigenvectors become delocalized. 
 


In dimension 1, Anderson localization has been rigorously established for a long time (see, e.g., \cite{Carmona1982Exp,David2002Local,JIMSPL1977FAIA,cmp/1103908590,ishii1973localization}). 
In higher dimensions $d\geq 2$, the first rigorous proof of localization was provided by Fr{\"o}hlich and Spencer \cite{frohlich1983absence} using multi-scale analysis (see also \cite{frohlich1985constructive,von1989new,spencer1988localization}). A simpler alternative proof, based on the fractional moment method, was later introduced by Aizenman and Molchanov \cite{aizenman1994localization,aizenman1993localization}. 
The localization result has also been extended to the more challenging case of singular or even discrete potentials \cite{bourgain2005localization,klein2012comprehensive,carmona1987anderson,ding2020localization,li2022anderson}. 
Despite these remarkable advances, the complete localization conjecture in dimension $d=2$ remains unsolved; current results only establish localization under strong disorder or for extreme energies near the spectral edges. 
In dimensions $d\ge 3$, the picture is even more incomplete: the existence of a delocalized phase has not yet been rigorously proved in any dimension, and establishing the existence of a mobility edge is even more challenging. 


To approach the delocalized regime and investigate the existence of mobility edges, one strategy is to study the Anderson model on lattices with simpler topology than $\Z^d$, which allows for more explicit analysis. A prominent example is the infinite $d$-regular tree with $d\ge 3$, also referred to as the Bethe lattice in the literature. For the Bethe lattice, the existence of a delocalized phase has been rigorously established in \cite{Bethe_PRL, Bethe_JEMS}, and the presence of a mobility edge was recently proved in  \cite{aggarwal2025mobilityedgeandersonmodel}.

The Bethe lattice can be viewed as an $\infty$-dimensional analogue of $\Z^d$. To understand Anderson delocalization and mobility edges in finite dimensions, one alternative approach is to consider some ``simpler" variants of the Anderson model---simpler in the sense of showing delocalization---that still capture its essential physical features. One such example is the celebrated \emph{random band matrix} (RBM) ensemble \cite{scalingabndCGMLIF1990PRL,ConJ-Ref2,ScalingPropertyBandMatrixFYMA1991PRL}, sometimes referred to as the \emph{Wegner orbital model} \cite{Wegner1,Wegner2, Wegner3}. 
This is a finite-volume model defined on a $d$-dimensional discrete torus of linear size $L\to \infty$. The RBM is a Wigner-type random matrix in which non-negligible hopping occurs only between sites whose distance is less than a specified band width $W \ll L$. 
Heuristically, the RBM and the Anderson model are believed to exhibit similar qualitative behavior when $\lambda \asymp  W$. In particular, the RBM is also expected to display a localization–delocalization transition as the band width $W$ increases, with mobility edges emerging for certain ranges of $W$. 

Significant progress has been made in understanding Anderson localization and delocalization for the RBM or Wegner orbital model. In dimension 1, delocalization has been proven under the sharp condition $W\gg L^{1/2}$ on the band width, assuming the random entries are Gaussian distributed \cite{yau2025delocalizationonedimensionalrandomband}. A similar result has also been established under a weaker condition $W\gg L^{3/4}$ without the Gaussian assumption \cite{PartI,PartII,Band1D_III}. 
A more detailed review of the advances regarding the delocalized phase of one-dimensional (1D) RBMs can be found in the references therein. The localization for 1D RBMs has been shown under the condition $W\ll L^{1/4}$, as established in a series of works \cite{Sch2009,Wegner,CS1_4,CPSS1_4}. 
The delocalization has been proved under the assumption $W\ge L^\e$ (for an arbitrarily small constant $\e>0$) for RBMs in dimension $d=2$ \cite{Band2D} and in dimensions $d\ge 7$ \cite{CFHJBulkBandAOP2024,yang2021delocalizationquantumdiffusionrandom,YYYTexpansion2022CMP}, again assuming Gaussian distribution for the random entries. 
However, the localization result for RBM in dimensions $d\ge 2$ remains absent from the literature.
Most of the aforementioned works have focused on the bulk regime of the RBM. Around the spectral edges, Sodin proved a remarkable result regarding a phase transition in the edge eigenvalue statistics of 1D RBM when $W$ crosses the threshold $L^{5/6}$ \cite{Sod2010}, a result that was later extended to higher dimensions in \cite{Band_Edge123}. 
However, the localization or delocalization of the edge eigenvectors of RBM has yet to be established in any dimension, and the mobility edge phenomenon (conjectured to exist in dimensions $1\le d\le 5$) remains unproven.

\subsection{Overview of the main results} 


To investigate the Anderson localization–delocalization transition and the presence of mobility edges from a random matrix theory perspective, we consider another variant of the Anderson model that naturally interpolates between the 1D Anderson model and the Wigner ensemble \cite{Wigner}.
More precisely, we study a random block matrix model introduced in \cite{stone2024randommatrixmodelquantum}. 
Fix any integer $D\ge 2$. We consider $D$ independent random subsystems, each modeled by an $N\times N$ Wigner matrix whose entries have mean zero, variance $N^{-1}$, and satisfy certain moment conditions. Without introducing interactions, this system is represented by a block-diagonal matrix $H$ with diagonal blocks being independent Wigner matrices $H_a$ for $a=1,\ldots, D$. 
To introduce interactions, we assume that neighboring subsystems are coupled via an arbitrary deterministic $N\times N$ matrix $A$. For simplicity, we impose periodic boundary conditions---that is, the subsystems are arranged in a cycle so that the first and $D$-th subsystems are also neighbors. The interaction Hamiltonian $\Lambda$ is then a block tridiagonal matrix, with off-diagonal blocks given by $A$ or $A^*$, reflecting the coupling between adjacent subsystems. The full system, incorporating both the random subsystems and their interactions, is denoted by $H_{\Lambda}$: 
\be\label{eq:def_model}
H_{\Lambda}=H+\Lambda.
\ee
In matrix notation, $H$ and $\Lambda$ are $D\times D$ block matrices defined as:
\be\label{matrix_H}
H = \begin{pmatrix}
H_1 & 0 & 0 & \cdots & 0 & 0 \\
0 & H_2 & 0 & \cdots & 0 & 0 \\
0 & 0 & H_3 & \cdots & 0 & 0 \\
\vdots & \vdots & \vdots & \ddots & \vdots & \vdots \\
0 & 0 & 0 & \cdots & H_{D-1} & 0 \\
0 & 0 & 0 & \cdots & 0 & H_D
\end{pmatrix},\quad \Lambda = \begin{pmatrix}
0 &A & 0 & \cdots & 0 &A^* \\
A^* & 0 &A & \cdots & 0 & 0 \\
0 & A^* & 0 & \cdots & 0 & 0 \\
\vdots & \vdots & \vdots & \ddots & \vdots & \vdots \\
0 & 0 & 0 & \cdots & 0 &A \\
A & 0 & 0 & \cdots & A^* & 0
\end{pmatrix}.
\ee
In the terminology of \cite{yang2025delocalizationgeneralclassrandom,truong2025localizationlengthfinitevolumerandom,Wegner}, this model is referred to as a (1D) \emph{block Anderson model} or a \emph{random block Schr{\"o}dinger operator}. Informally, $H$ can be interpreted as a block potential, where the i.i.d.~scalar potential in \eqref{eq;Anderson} is replaced by an i.i.d.~block potential. Meanwhile, the interaction term $-\lambda\Delta$ in \eqref{eq;Anderson} is replaced by a block matrix $\Lambda$, which governs the hopping between neighboring blocks.


In this paper, we assume that $H_{\Lambda}$ is a perturbation of $H$, i.e., $\|A\|\ll \E\|H\|\sim 1$. 
Hence, the limiting spectrum of $H_{\Lambda}$ can be viewed as a perturbation of that of $H$, which is governed by Wigner's semicircle law. 
A localization-delocalization transition for $H_{\Lambda}$ was established in \cite{stone2024randommatrixmodelquantum} within the bulk of the spectrum, specifically in the interval $[-2+\kappa,2-\kappa]$ for an arbitrarily small constant $\kappa>0$, as $\|A\|_{\HS}$ crosses the threshold 1. In this paper, we extend that result to the entire spectrum, with a particular focus on the edge regime,
and establish a full characterization of the localization–delocalization transition for the corresponding eigenvectors.
For simplicity of presentation, we define the index sets ${\cal I}_a:=\llbracket (a-1)N+1,aN\rrbracket$, $a \in \{ 1,\ldots,D\}$, for the subsystems, and let $\cal I:=\llbracket DN\rrbracket$ be the index set for the entire system. 
Hereafter, for any $n,m\in \R$, we denote $\llbracket n, m\rrbracket: = [n,m]\cap \Z$ and $\llbracket n\rrbracket:=\llbracket1, n\rrbracket$. 
We denote the eigenvalues of $H_{\Lambda}$ by $\lambda_1\geq \lambda_2\geq \cdots \geq \lambda_{DN}$ and the corresponding (unit) eigenvectors by ${\bf v}_1,\mathbf{v}_2,\ldots, {\bf v}_{DN}$. Given $k\in\cal I$, we denote \be\label{eq:rk}\mathfrak{r}\pa{k}:=k\wedge \pa{DN+1-k}.\ee 
Roughly speaking, we find that the localization-delocalization transition of the $k$-th eigenvector occurs at $\|A\|_{\HS}\sim N^{1/3}/\mathfrak{r}\pa{k}^{1/3}$:
\begin{itemize}
    \item \textbf{Delocalized phase:} 
    If $\|A\|_{\HS}\gg N^{1/3}/\mathfrak{r}\pa{k}^{1/3}$, then the $k$-th eigenvector $\bv_k$ is delocalized in the following sense: with probability $1-\oo\pa{1}$, \be\label{eq:intro_delocal}\sum_{i\in\cI_a}\absa{\bv_{k}\pa{i}}^2=D^{-1}+\oo\pa{1} \quad \text{for each block $\cI_a$}.
    \ee
    In other words, the $\ell_2$-mass of $\bv_k$ is approximately evenly distributed across the $D$ subsystems. Furthermore, if $\|A\|_{\HS}\gg N^{1/3}$, the edge eigenvalue statistics of $H_\Lambda$ asymptotically match those of the Gaussian Unitary Ensemble (GUE). 
    More precisely, let $\qa{E^-,E^+}$ denote the support of the limiting spectrum of $H_\Lambda$. Then, the largest (resp.~smallest) eigenvalue around $E^+$ (resp.~$E^-$) converges in distribution to the celebrated Tracy-Widom (TW) law \cite{TW,TW1} under the $(DN)^{2/3}$ scaling.  
        
    \item \textbf{Localized phase:} If $\|A\|_{\HS}\ll N^{1/3}/\mathfrak{r}(k)^{1/3}$, then the $k$-th eigenvector $\bv_k$ is concentrated in a single subsystem in terms of its $\ell_2$-mass. More precisely, with probability $1-\oo\pa{1}$, there exists a block $\cal I_a$ such that \smash{$\sum_{i\in\cI_a}\absa{\bv_{k}\pa{i}}^2=1+\oo\pa{1}$}. Furthermore, the $k$-th eigenvalue of $H_\Lambda$ differs from that of $H$ (up to a deterministic shift) by a negligible amount compared to the typical fluctuation scale of $\lambda_k$, which is $N^{-2/3}\mathfrak{r}(k)^{-1/3}$.  
\end{itemize}
Let $\kappa_E:=|E-E^+|\wedge |E-E^-|$ denote the distance of an energy level $E$ from the spectral edges. 
It is known that the typical distance of the $k$-th eigenvalue $\lambda_k$ from the spectral edges $E^\pm$ is of order $\kappa_{\lambda_k}\sim (\mathfrak{r}\pa{k}/N)^{2/3}$. Therefore, the results above can also be interpreted as follows. For a fixed interaction matrix $A$ satisfying $1\ll \|A\|_{\HS} \ll N^{1/3}$, the eigenvectors corresponding to eigenvalues within the edge regime, defined by $\left\{E\in \R:\kappa_E\ll \|A\|_{\HS}^{-2}\right\}$, are localized, while those corresponding to eigenvalues in the bulk regime, defined by $\left\{E\in[E^-,E^+]:\kappa_E\gg \|A\|_{\HS}^{-2}\right\}$, are delocalized. This characterizes a localization–delocalization transition as the energy level $E$ crosses the critical regime where \smash{$\kappa_E\sim \|A\|_{\HS}^{-2}$}. In particular, it implies the existence of mobility edges near $E^\pm$.

This paper focuses on a simplified setting where $D$ remains fixed as $N\to \infty$. However, to gain a deeper understanding of the Anderson localization/delocalization phenomenon, it is also important to consider the regime $D\to \infty$, where the random block matrix model becomes increasingly "non-mean-field" as $D$ grows. Such extensions have been studied in the context of block Anderson models  \cite{yang2025delocalizationgeneralclassrandom,truong2025localizationlengthfinitevolumerandom,Wegner}. Roughly speaking, assuming $W\ge D^\e$ for some constant $\e>0$, certain results on delocalization and the order of localization length were established in dimensions 1 and 2 in  \cite{truong2025localizationlengthfinitevolumerandom}, and in dimensions 7 and higher in \cite{yang2025delocalizationgeneralclassrandom}. Conversely, a localization result was proved in \cite{Wegner} for the case where the matrix $A$ is a scalar matrix. 

\begin{figure}[ht]
    \centering
    \begin{subfigure}[b]{0.40\textwidth}
        \includegraphics[width=\textwidth]{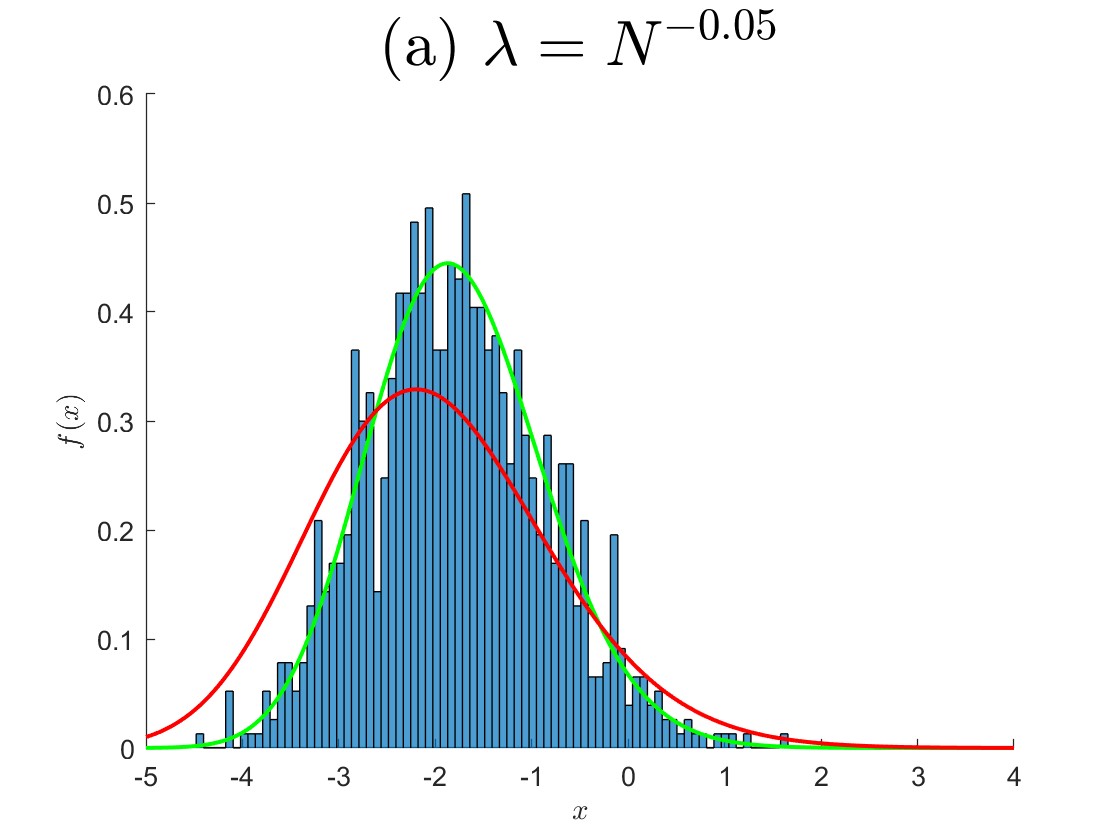}
    \end{subfigure}
    \begin{subfigure}[b]{0.40\textwidth}
        \includegraphics[width=\textwidth]{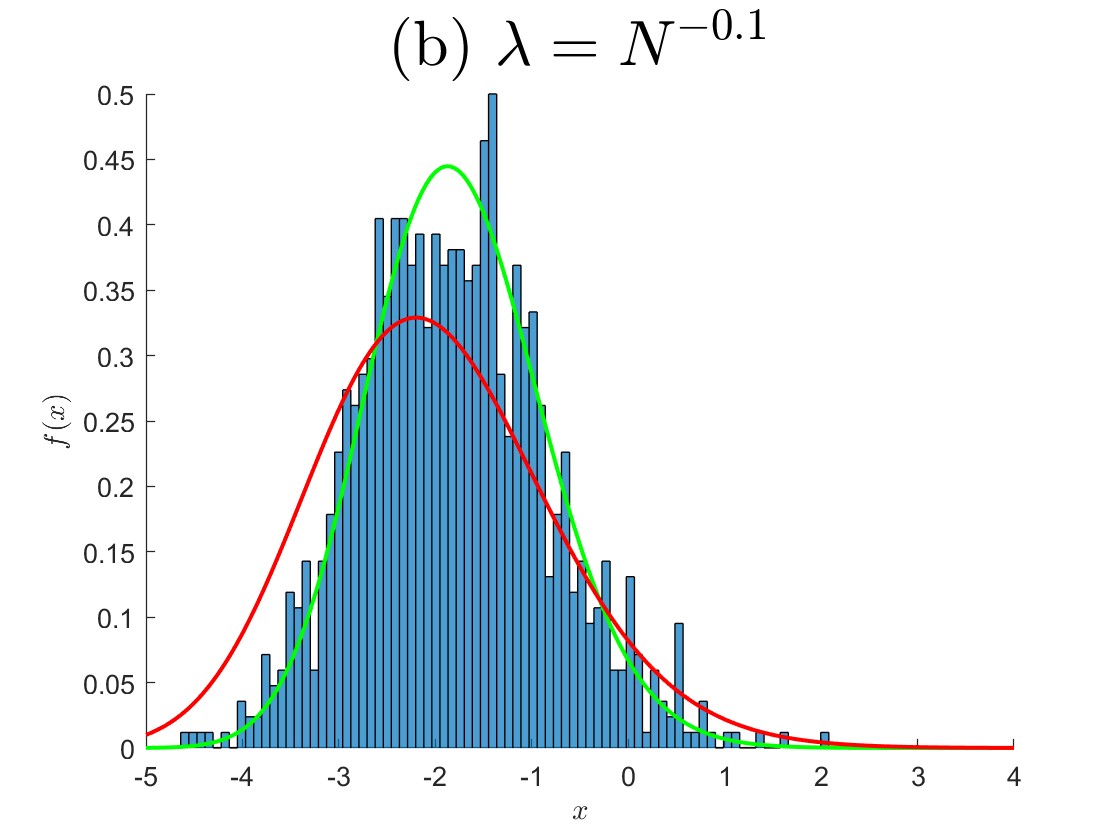}
    \end{subfigure}


    \begin{subfigure}[b]{0.40\textwidth}
        \includegraphics[width=\textwidth]{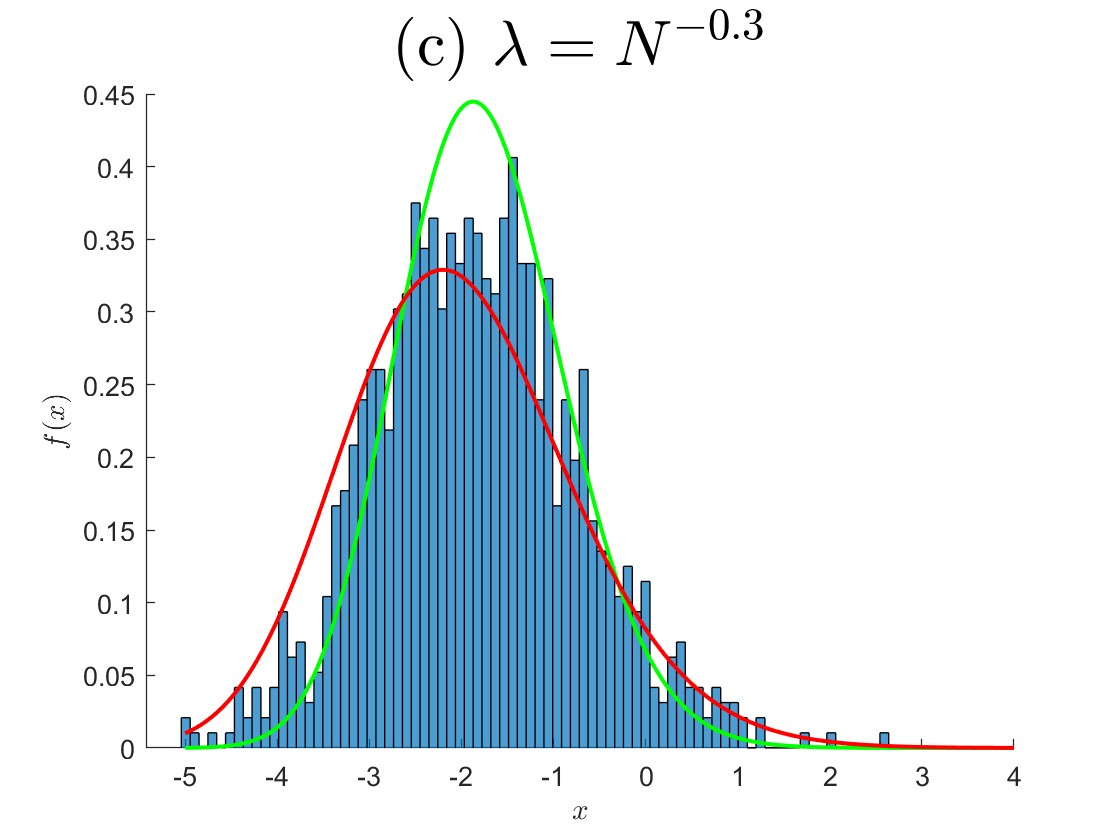}
    \end{subfigure}
    \begin{subfigure}[b]{0.40\textwidth}
        \includegraphics[width=\textwidth]{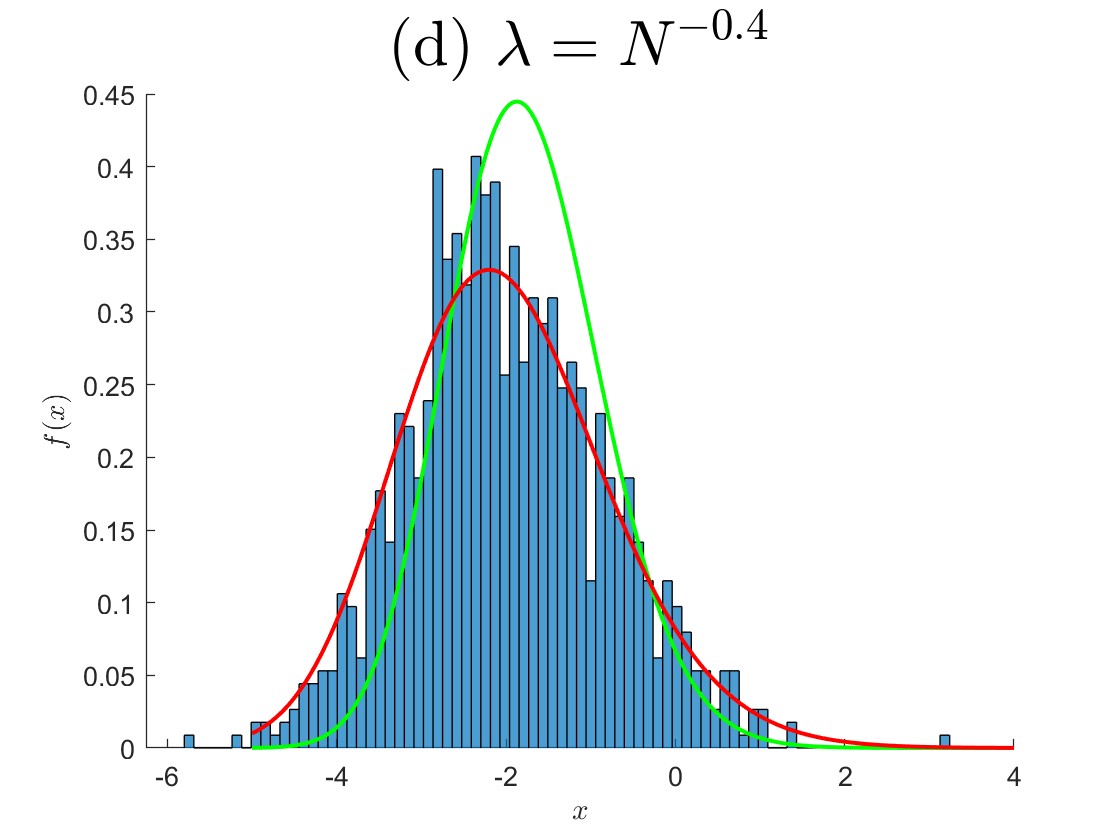}
    \end{subfigure}
    \caption{Distribution of the largest eigenvalue of $H_\Lambda$, where we take $N=400$ and $D=2$. 
    The normalized histograms in (a) and (b) display the simulated distribution of  \smash{$\gamma\pa{DN}^{2/3}\pa{\lambda_1-E^+}$} (where $\gamma$ is defined in \eqref{eq:def_gamma} below), while those in (c) and (d) show the simulated distribution of \smash{$\pa{DN}^{2/3}\pa{\lambda_1-E^+}$}. 
    The green curve plots the probability density function (PDF) for the TW-2 distribution, and the red curve plots the PDF for the maximum of two independent TW-2 distributions. 
    Note that the $\lambda=N^{-0.3}$ case does not align well with the red curve; we attribute this discrepancy to finite-$N$ effects.    }\label{Tracy_Widom_fig}
\end{figure}

Compared to \cite{yang2025delocalizationgeneralclassrandom,truong2025localizationlengthfinitevolumerandom,Wegner}, the present work provides a more comprehensive result in several respects. 
The delocalization results in \cite{yang2025delocalizationgeneralclassrandom,truong2025localizationlengthfinitevolumerandom} are restricted to the bulk of the spectrum, while \cite{Wegner} considers only the strong disorder regime, where no mobility edges arise. In contrast, our analysis covers the entire spectrum, including the spectral edges. Furthermore, the aforementioned works assume Gaussian-distributed blocks for the block potential, whereas we impose only general moment conditions on the entries of $H$. Finally, while \cite{yang2025delocalizationgeneralclassrandom,Wegner} assume the interaction matrix $A$ is proportional to the identity, we allow general $A$, subject only to bounds on $\|A\|$ and $\|A\|_{\HS}$. 
The main reason we are able to provide such a complete characterization of the localization-delocalization transition and the mobility edge is the availability of a sharp local law for the Green’s function (or resolvent) of $H_\Lambda$ under the simplifying assumption $D=\OO(1)$; see \Cref{lem_loc} below. 
This enables us to develop and exploit more intricate multi-resolvent local laws, which in turn allow us to establish localization or delocalization results across different parameter regimes for $\|A\|_{\HS}$. 
On the other hand, in the $D\to \infty$ case, establishing even a single-resolvent local law becomes a significant challenge.

Finally, we support our results with simulations. Let $\ha{H_a}_{a=1}^D$ be $D$ independent copies of $N\times N$ GUE, and let $A=\lambda I_N$, such that \smash{$\norm{A}_{\HS}=\lambda N^{1/2}$}. In \Cref{Tracy_Widom_fig}, we depict the distribution of the (centered and rescaled) largest eigenvalue $\lambda_1$ as $\lambda$ cross the transition threshold $\lambda=N^{-1/6}$. In the delocalized regime (plots (a) and (b)), the simulated distribution coincides with the TW-2 distribution. In contrast, in the localized regime (plots (c) and (d)), the distribution aligns with that of the maximum of $D$ independent TW-2 distributions, which represents the asymptotic distribution of the largest eigenvalue of $H$. 
In \Cref{mobility_edge_fig}, we illustrate the localization-delocalization transition from bulk energies to edge energies. In the bulk regime, the eigenvectors are delocalized in the sense of \eqref{eq:intro_delocal}. As the energy shifts from the bulk to the spectral edges, the $\ell_2$-mass of the eigenvector increasingly concentrated within a single block, indicating a transition to the localized phase. This demonstrates the mobility edge phenomenon predicted by our theory.

\begin{figure}[htb]
\begin{center}
\includegraphics[width=10cm]{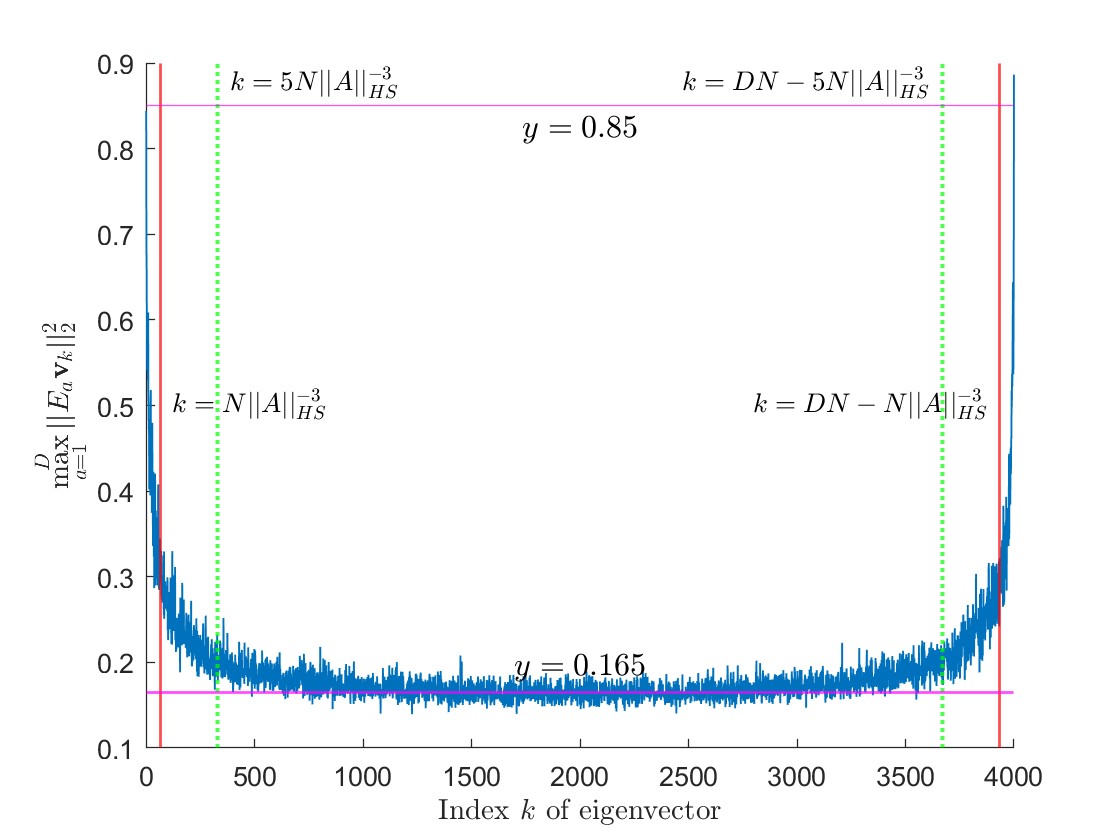}
\caption{Localization-delocalization transition across the entire spectrum. The horizontal axis represents the eigenvector index $k$, and the vertical axis shows the maximum squared  $\ell_2$-mass of $\bv_k$ over the $D$ blocks. We set $N=400$, $D=10$, and $\lambda=N^{-0.4}$, so that $\norm{A}_{\txt{HS}}=N^{1/10}$. The region between the green lines corresponds to delocalized energies, the regions between the red and green lines indicate transition regimes, and the regions outside the red lines represent localized energies. The purple lines illustrate the extent of localization or delocalization.}
\label{mobility_edge_fig}
\end{center}
\end{figure}

\subsection*{Organization of the remaining text}
In \Cref{sec:main_results}, we present the main results of this paper. In the delocalized phase, we state the delocalization of eigenvectors in \Cref{mix} and the Tracy-Widom statistics for the edge eigenvalues in \Cref{MixEV}. In the localized phase, we state the localization of eigenvectors in \Cref{theorem_localized_eigenvector_localization} and describe the eigenvalue statistics in \Cref{theorem_localized_eigenvalue}. 
The proofs of \Cref{mix,MixEV} are provided in \Cref{sec:delocalized_case_eigenvector,sec:delocalized_case_eigenvalue}, respectively, while \Cref{sec:localized_case} is devoted to the proofs of \Cref{theorem_localized_eigenvector_localization,theorem_localized_eigenvalue}. 
Additional auxiliary estimates used in the main proofs are collected in \Cref{appendix}.


\subsection*{Notations} 
To facilitate the presentation, we introduce some necessary notations that will be used throughout this paper. In this paper, we are interested in the asymptotic regime with $N\to \infty$. When we refer to a constant, it will not depend on $N$. Unless otherwise noted, we will use $C$ to denote generic large positive constants, whose values may change from line to line. Similarly, we will use $\epsilon$, $\delta$, $\tau$, $c$ etc.~to denote generic small positive constants. 
For any two (possibly complex) sequences $a_N$ and $b_N$ depending on $N$, $a_N = \OO(b_N)$ or $a_N \lesssim b_N$ means that $|a_N| \le C|b_N|$ for a constant $C>0$, whereas $a_N=\oo(b_N)$ or $|a_N|\ll |b_N|$ means that $\lim_{N\to \infty}|a_N| /|b_N| \to 0$. We say that $a_N\sim b_N$ if $a_N = \OO(b_N)$ and $b_N = \OO(a_N)$. For any $a,b\in\R$, we denote 
$a\vee b:=\max\{a, b\}$ and $a\wedge b:=\min\{a, b\}$. 
For an event $\Xi$, we let $\mathbf 1_\Xi$ or $\mathbf 1(\Xi)$ denote its indicator function. 
Given a vector $\mathbf v$, $\|\mathbf v\|\equiv \|\mathbf v\|_2$ denotes the Euclidean norm and $\|\mathbf v\|_p$ denotes the $\ell_p$-norm. 
Throughout this paper, we use ``$*$" to denote the Hermitian conjugate of a matrix.
Given a matrix $B = (B_{ij})$, we use $\|B\|$, $\|B\|_{\HS}$, and $\|B\|_{\max}:=\max_{i,j}|B_{ij}|$ to denote the operator, Hilbert-Schmidt, and maximum norms, respectively. We also adopt the notion of generalized entries: $B_\mathbf{uv}\equiv \mathbf u^* B \mathbf v$ for vectors $\bu,\bv$.

\medskip
\noindent {\bf Acknowledgement}.   
Fan Yang is supported in part by the National Key R\&D Program of China (No. 2023YFA1010400). 

\section{Main results}
\label{sec:main_results}

\subsection{The model and main results}

We consider the random block matrix model in \eqref{eq:def_model}. Fix any integer $D\geq 2$, let $H_1,H_2,\ldots, H_D$ be $D$ independent copies of $N\times N$ Wigner matrices, i.e., the entries of $H_a$ are independent (up to Hermitian symmetry $H=H^*$) random variables satisfying that 
\begin{equation}\label{eq:meanvar}
    \begin{aligned}
        \E(H_{a})_{ij}=0,\quad \E |(H_{a})_{ij}|^2 = {N}^{-1},\quad \forall a \in \qq{D},\ \ i,j \in \qq{N}.
    \end{aligned}
\end{equation}
For the definiteness of notation, we consider the complex Hermitian case in this paper, while the real case can be proved in the same way with some minor changes in notations. In the complex case, we assume additionally that 
\begin{equation}\label{eq:meanvar2}
    \begin{aligned}
        \E [(H_{a})^2_{ij}] = 0,\quad \forall a \in \qq{D},\ \ i\ne j \in \qq{N}.
    \end{aligned}  
\end{equation}
We assume that the diagonal entries are i.i.d.~real random variables and the entries above the diagonal are i.i.d.~complex random variables. Let $A$ be an arbitrary $N\times N$ (real or complex) deterministic matrix. Then, we consider the block random matrix model $H_{\Lambda}$ defined in \eqref{eq:def_model} with $H$ and $\Lambda$ given in \eqref{matrix_H}.

\begin{assumption}\label{main_assm}
Fix any integer $D\ge 2$, we consider the model \eqref{eq:def_model}, where $A$ is an arbitrary $N\times N$ deterministic matrix with $\|A\|\le N^{-\delta_A}$ for a constant $\delta_A>0$, and $H_1,H_2,\ldots, H_D$ are $D$ i.i.d.~$N\times N$ complex Hermitian Wigner matrices satisfying \eqref{eq:meanvar}, \eqref{eq:meanvar2}, and the following high moment condition: for any $p\in \N$, there exists a constant $C_p>0$ such that 
\begin{equation}\label{eq:highmoment}
    \begin{aligned}
       \E |H_{11}|^p  + \E |H_{12}|^p \leq C_pN^{-p/2}. 
    \end{aligned}
\end{equation}
\end{assumption}

Recall that the eigenvalues and corresponding eigenvectors of $H_{\Lambda}$ are denoted by $\lambda_1\geq \lambda_2\geq \cdots \geq \lambda_{DN}$ and ${\bf v}_1,\mathbf{v}_2,\ldots, {\bf v}_{DN}$. Let $p_{H_{\Lambda}}(\lambda_1,\ldots, \lambda_{DN})$ denote the joint \emph{symmetrized} probability density function of the eigenvalues of $H_{\Lambda}$. For any $1\le n \le DN$, define the $n$-point correlation function by
$$
p_{H_{\Lambda}}^{(n)}\left(\lambda_1, \ldots, \lambda_n\right)
:=\int_{\R^{DN-n}} p_{H_{\Lambda}}\left(\lambda_1, \ldots, \lambda_{DN}\right) \mathrm{d} \lambda_{n+1} \cdots \mathrm{d} \lambda_{DN}.
$$ 
Similarly, denote the eigenvalues of $H$ by $\lambda_1(H)\ge \cdots \ge \lambda_{DN}(H)$, and 
let $p_H^{(n)} $ represent the $n$-point correlation function of them. 
Recall that $\mathfrak{r}\pa{k}$ is defined in \eqref{eq:rk}. 
Now, we state our main results.

\begin{theorem}[Delocalized regime: eigenvectors]\label{mix} 
Under \Cref{main_assm}, given any $k\in\qq{1,DN}$, suppose there exists a constant $\varepsilon_A>0$ such that
\begin{equation}\label{eq:condA1}
\|A\|_{\HS}\ge N^{1/3+\varepsilon_A}\fr\pa{k}^{-1/3} .
\end{equation}
Then, there exists a constant $c>0$ such that 
\begin{equation}\label{eq:main_evector1}
\mathbb P\left( \max_{a\in\llbracket D\rrbracket}\left|{\bf v}_k^* E_{a}  {\bf v}_k-D^{-1}\right|\ge N^{-c}\right)\le N^{-c},
\end{equation}
where $E_a\in \C^{DN\times DN}$ denotes the block identity matrix restricted to $\cal I_a$, i.e., $(E_{a})_{ij}={\bf 1}(i=j\in {\cal I}_a)$.
\end{theorem}

We will define the limiting spectral density $\rho_N$ for the eigenvalues of $\VV$ in \eqref{eq:density_rhoN} below, and denote its support by $[E^-,E^+]$, where $E^\pm$ represent the spectral edges.
According to \cite[Lemma 4.3]{lee2015edge},
the density $\rho_N$ exhibits a square-root behavior near the edges $E^\pm$. Based on this behavior, we define the curvature parameters $\gamma_{\pm}$ at $E^\pm$ as:
\begin{equation}\label{eq:def_gamma}
    \begin{aligned}
        \lim_{E\uparrow E^+} \frac{\rho_N(E)}{\sqrt{E^+-E}}=\frac{\gamma_+^{3/2}}{\pi}, \quad \lim_{E\downarrow E^-} \frac{\rho_N(E)}{\sqrt{E-E^-}}=\frac{\gamma_-^{3/2}}{\pi}.
    \end{aligned}
\end{equation}

\begin{theorem}[Delocalized regime: eigenvalues]\label{MixEV}
In the setting of \Cref{mix}, let $O \in C_c^{\infty}\left(\mathbb{R}^n\right)$ be an arbitrary smooth, compactly supported function. If \eqref{eq:condA1} holds for $k=1$, then for any fixed $n\in \N$, there exists a constant $c>0$ so that 
\begin{align}
 &\left|	\E O\pa{\gamma_+ \pa{D N}^{2/3}\pa{E^+-\lambda_1},\ldots,\gamma_+ \pa{D N}^{2/3}\pa{E^+-\lambda_n}}\right.\nonumber \\
 &\left.-\E O\pa{ \pa{D N}^{2/3}\pa{2-\mu_1},\ldots, \pa{D N}^{2/3}\pa{2-\mu_n}}   \right|\leq N^{-c} ,
\end{align}
where $\mu_1\geq\mu_2\geq\cdots\geq\mu_n$ denote the largest $n$ eigenvalues of a $DN\times DN$ GUE. The corresponding edge universality result also holds for $\gamma_-(DN)^{2/3}(\lambda_{DN}-E^-,\ldots,\lambda_{DN-n}-E^- )$ at the left edge $E^-$.
\end{theorem}

\begin{remark}\label{rmk_only_edge}
The universality of eigenvalue statistics around any energy level $E \in [E^-, E^+]$ is expected to hold in the delocalized phase. In particular, this has been rigorously established in the bulk regime (i.e., for $E\in[E^-+\e, E^+-\e]$ with some small constant $\e>0$), as shown in \cite{stone2024randommatrixmodelquantum}. However, the local eigenvalue statistics in the transition regime between the spectral edges and the bulk (characterized by $N^{-2/3}\ll |E^+-E|\wedge |E-E^-|\ll 1$) have not yet been studied in the literature for GUE. Therefore, we focused only on the universality of the edge eigenvalue statistics in \Cref{MixEV}.
\end{remark}

\begin{theorem}[Localized regime: eigenvectors]\label{theorem_localized_eigenvector_localization}
Under \Cref{main_assm}, given any $k\in\qq{1,DN}$, suppose there exists a positive constant $\varepsilon_A$ such that
\begin{equation}\label{eq:condA2}
\|A\|_{\HS}\le N^{1/3-\varepsilon_A}\fr\pa{k}^{-1/3} . 
\end{equation}
Then, for any small constant $\varepsilon\in(0,2\e_A)$, there exists a constant $\varepsilon_0=\varepsilon_0\pa{\varepsilon}>0$ such that
\begin{equation}\nonumber
    \begin{aligned}
        \mathbb P\left(\max_{a=1}^D \|E_a{\bf v}_k\|^2 \le 1 - N^{-2/3+\varepsilon}\fr\pa{k}^{2/3}\norma{A}_{\HS}^2\right)\le N^{-\varepsilon_0}.
    \end{aligned}
\end{equation}
As a consequence, it implies that 
\begin{equation}\label{eq:main_evector2}
\mathbb P\left(\max_{a=1}^D \|E_a{\bf v}_k\|^2 \le 1 - N^{-c}\right)\le N^{-c}
\end{equation}
for a constant $c>0$ depending on $\e_A$.
\end{theorem}

\begin{theorem}[Localized regime: eigenvalues]\label{theorem_localized_eigenvalue}
In the setting of \Cref{theorem_localized_eigenvector_localization}, for any constants $\varepsilon>0$ and $\varepsilon_0\in\pa{0,2\varepsilon}$, we have that
    \begin{equation}
        \label{eq:local_eigenvalue_perturb}\begin{aligned}
            \P\left( \left|\pa{\lambda_k-\gamma_k}-\pa{\lambda_k(H)-\gamma_{k}^{\tsc}}\right| \ge N^{-1+\varepsilon}\norma{A}_{\HS}\right)\le N^{-\varepsilon_0},
        \end{aligned}
    \end{equation}
where $\gamma_k$ and $\gamma_k^{\tsc}$ denote the $k$-th quantiles of $\rho_N$ and the semicircle law, respectively, as defined in \eqref{eq:gammak} below. From \eqref{eq:local_eigenvalue_perturb}, there exists a constant $c>0$ depending on $\e_A$ such that the following estimate holds:
\begin{equation}\label{eq:main_perurb}
    \begin{aligned}
        \P\left( \left|\pa{\lambda_k-\gamma_k}-\pa{\lambda_k(H)-\gamma_{k}^{\txt{sc}}}\right| \ge N^{-2/3-c}\fr\pa{k}^{-1/3}\right)\le N^{-c}.
    \end{aligned}
\end{equation}
As a consequence of \eqref{eq:main_perurb}, for any $k\in \qq{1,DN}$ such that \eqref{eq:condA2} holds, and for any fixed $n\in \N$ and smooth, compactly supported test function $O \in C_c^{\infty}\left(\mathbb{R}^n\right)$, there exists a constant $c>0$ depending on $\e_A$ such that
\begin{equation*}
    \begin{aligned}
        &\left|	\int_{\mathbb{R}^n} \mathrm{~d} \boldsymbol{\alpha}\; O(\boldsymbol{\alpha}) p_{H_{\Lambda}}^{(n)}\left(\gamma_k+\frac{\alpha_1}{\pa{DN}^{2/3}\fr\pa{k}^{1/3}  }, \ldots, \gamma_k+\frac{\alpha_n}{\pa{DN}^{2/3}\fr\pa{k}^{1/3}  }\right)\right.\\
        &\left.-\int_{\mathbb{R}^n} \mathrm{~d} \boldsymbol{\alpha}\; O(\boldsymbol{\alpha})p_{H}^{(n)}\left(\gamma_k^{\txt{sc}}+\frac{\alpha_1}{\pa{DN}^{2/3}\fr\pa{k}^{1/3}  }, \ldots, \gamma_k^{\txt{sc}}+\frac{\alpha_n}{\pa{DN}^{2/3}\fr\pa{k}^{1/3}  }\right)   \right|\leq N^{-c},
    \end{aligned}
\end{equation*}
where $\boldsymbol{\alpha}$ denotes $\boldsymbol{\alpha}=\left(\alpha_1, \ldots, \alpha_n\right)$.
\end{theorem}

\subsection{Local law of Green's function}

A basic tool in our proof is the local law for the Green's function (or resolvent) of $H_\Lambda$, defined by
\be\label{def_resolv}
G(z)\equiv G(z,H,\Lambda):=(H_\Lambda-z)^{-1},\quad z\in \mathbb C_+:=\{z\in \C:\im z>0\},
\ee
as we will state in \Cref{lem_loc} below. To state it, we first introduce some notations. Note the model \eqref{eq:def_model} can be regarded as a deformed generalized Wigner matrix. As $N\to \infty$, $G(z)$ converges to a deterministic matrix $M(z)\equiv M(z,\Lambda)$, which satisfies the \emph{matrix Dyson equation}: 
\be\label{def_M}
\left( \cal S(M) +z-\Lambda\right)M+I=0, 
\ee
where $\cal S(\cdot)$ is a linear operator acting on $DN\times DN$ matrices such that $\cal S(M)$ is a diagonal matrix with 
$$\cal S(M)_{ij}=\mathbf 1(i=j)\sum_x s_{ix}M_{xx} =\mathbf 1(i=j)D\avg{ME_a},\quad \forall i,j\in \cal I_a. $$
Here, $s_{ij}$ denotes the variance of the $(i,j)$-th entry of $H$:
\be\label{eq:sij}
s_{ij}=\E|H_{ij}|^2= N^{-1}\mathbf 1(i,j\in \cal I_a \  \text{for some } a \in \llbracket D\rrbracket),
\ee
and we define the variance matrix by $S=(s_{ij}:i,j\in \cal I)$. We will use $\langle B\rangle := (DN)^{-1}\tr B$ to denote the normalized trace of a $DN\times DN$ matrix $B$. 
Due to the block translation symmetry of $S$ and $\Lambda$, we see that $M$ is also block translationally invariant, which implies that $\cal S(M)$ should be a scalar matrix 
$\cal S(M)=mI$, where $m(z)$ is defined as $m(z):=\langle M(z)\rangle$.

\begin{remark}
When $D=2$, the block translation symmetry may not hold. In this case, we denote
$$M=\begin{pmatrix}
M_{(11)} & M_{(12)} \\
M_{(21)} & M_{(22)}
\end{pmatrix}.
$$
Then, we can derive directly from equation \eqref{def_M} that 
\be\label{eq:MD=2}
\begin{aligned}
    M_{(11)}=\frac{m+z}{AA^*-(m+z)^2},\quad &   M_{(22)}=\frac{m+z}{A^*A-(m+z)^2},\\  M_{(12)}=\frac{1}{AA^*-(m+z)^2}A,\quad & M_{(21)}=\frac{1}{A^*A-(m+z)^2}A^*,
\end{aligned} 
\ee
where $m(z)$ satisfies the self-consistent equation 
$m(z)=N^{-1}\tr M_{(11)}(z)=N^{-1}\tr M_{(22)}(z).$
\end{remark}

 \begin{definition}[Matrix limit of $G$]\label{defn_Mm}
 We define $m(z)\equiv m_N(z)$ as the unique solution to 
 \be\label{self_m}
m(z)= \big\langle \left(\Lambda -z- m(z)\right)^{-1} \big\rangle ,
\ee
 such that $\im m(z)>0$ whenever $z\in \C_+$.  
 Then, we define the matrix $M(z)\equiv M_N(z,\Lambda)$ as 
\be\label{def_G0}
M(z):= \left(\Lambda -z- m(z)\right)^{-1}.
 \ee
Since $\Lambda$ is Hermitian, we have that $m(\bar z)= \overline {m(z)}$ and $M(\bar z)=M(z)^*$.
 \end{definition}

Under this definition, $m(z)$ is actually the Stieltjes transform of a probability measure $\mu_{N}$, called the \emph{free convolution} of the empirical measure of $\Lambda $ and the semicircle law with density 
\be\label{eq:semidensity}\rho_{\tsc}(x)=\frac{1}{2\pi} \sqrt{4-x^2}\mathbf 1_{x\in [-2,2]}.\ee
Moreover, the probability density $\rho_{N}$ of $\mu_{N}$ is determined from $m(z)$ by \be\label{eq:density_rhoN}
\rho_{N}(x)=\pi^{-1}\lim_{\eta\downarrow 0}\im m(x+\ii \eta).
\ee
Under the assumption $\|A\|=\OO(N^{-\delta_A})$, \cite[Lemma 4.3]{lee2015edge} shows that the support of $\rho_{N}$ is a single interval $[E^-,E^+]$, where $|E^+-2|+|E^-+2|=\oo(1)$. Moreover, from \eqref{eq:msc} below, we have $|m(z)-m_{\tsc}(z)|=\oo(1)$, where $m_{\tsc}(z)$ is the Stieltjes transform of $\rho_{\tsc}$, given by
$m_{\tsc}(z)=(-z+\sqrt{z^2-4})/{2}$. 
For any $k\in\qq{1,DN}$, we denote by $\gamma_k$ and $\gamma_k^{\txt{sc}}$ the $k$-th quantiles of $\rho_N$ and $\rho_{\txt{sc}}$, respectively, defined as:
\begin{equation}\label{eq:gammak}
    \begin{aligned}
        \gamma_k:=\sup_{x\in \R}\ha{\int_{x}^{+\infty}\rho_N(E)\rd E\geq \frac{k-1/2}{DN}},\quad \gamma_k^{\txt{sc}}:=\sup_{x\in \R}\ha{\int_{x}^{+\infty}\rho_{\txt{sc}}(E)\rd E\geq \frac{k-1/2}{DN}}.
    \end{aligned}
\end{equation}
We define the distance from an energy $E$ to the spectral edges $E^\pm$ by $\kappa \equiv \kappa_E:=\absa{E^+ -E}\wedge \absa{E-E^-} $. Some basic properties of $m$ and the density $\rho_N$ are collected in \Cref{lemma_usual_properties_of_m_and_mu} in the appendix. In particular, the square-root behavior of $\rho_N$ described in \eqref{square_root_density} implies that
\begin{equation}\label{eq:kappa_gamma_k}
    \begin{aligned}
        \absa{\gamma_k-E^+}\sim {k}^{2/3}/N^{2/3},\quad \absa{\gamma_{DN+1-k}-E^-}\sim {k}^{2/3}/N^{2/3},\quad \forall k\in\qq{1,DN}.
    \end{aligned}
\end{equation}

To state the local law and streamline the presentation, in this paper, we adopt the following convenient notion of stochastic domination introduced in \cite{Average_fluc}. 


\begin{definition}[Stochastic domination and high probability event]\label{stoch_domination}
	{\rm{(i)}} Let
	\[\xi=\left(\xi^{(N)}(u):N\in\mathbb N, u\in U^{(N)}\right),\hskip 10pt \zeta=\left(\zeta^{(N)}(u):N\in\mathbb N, u\in U^{(N)}\right),\]
	be two families of non-negative random variables, where $U^{(N)}$ is a possibly $N$-dependent parameter set. We say $\xi$ is stochastically dominated by $\zeta$, uniformly in $u$, if for any fixed (small) $\tau>0$ and (large) $D>0$, 
	\[\mathbb P\bigg(\bigcup_{u\in U^{(N)}}\left\{\xi^{(N)}(u)>N^\tau\zeta^{(N)}(u)\right\}\bigg)\le N^{-D}\]
	for large enough $N\ge N_0(\tau, D)$, and we will use the notation $\xi\prec\zeta$. 
	If for some complex family $\xi$ we have $|\xi|\prec\zeta$, then we will also write $\xi \prec \zeta$ or $\xi=\OO_\prec(\zeta)$. 
	
	\vspace{5pt}
	\noindent {\rm{(ii)}} As a convention, for two deterministic quantities $\xi$ and $\zeta$, we will write $\xi\prec\zeta$ if and only if $|\xi|\le N^\tau |\zeta|$ for any constant $\tau>0$. 

 \vspace{5pt}
	\noindent {\rm{(iii)}} Let $A$ be a family of random matrices and $\zeta$ be a family of non-negative random variables. Then, we use $A=\OO_\prec(\zeta)$ to mean that $\|A\|\prec \zeta$, where $\|\cdot\|$ denotes the operator norm. 
	
	\vspace{5pt}
	\noindent {\rm{(iv)}} We say an event $\Xi$ holds with high probability (w.h.p.) if for any constant $D>0$, $\mathbb P(\Xi)\ge 1- N^{-D}$ for large enough $N$. More generally, we say an event $\Omega$ holds $w.h.p.$ in $\Xi$ if for any constant $D>0$,
	$\P( \Xi\setminus \Omega)\le N^{-D}$ for large enough $N$.
\end{definition}

\begin{lemma}[Local laws and rigidity of eigenvalues, Lemma 2.9 in \cite{stone2024randommatrixmodelquantum}]\label{lem_loc}
Under \Cref{main_assm}, for any small constant $\tau>0$, the following local laws hold uniformly in $z=E+\ii\eta $ with $|z|\le \tau^{-1}$ and $ \eta \ge N^{-1+\tau}$. 
\begin{itemize}
    \item  \textnormal{\textbf{Anisotropic local law:}} For any deterministic unit vectors $\mathbf{u}, \mathbf{v}\in \C^{DN}$, we have
\begin{equation}\label{eq:aniso_local}
 \left(G(z)-M(z)\right)_{\mathbf{u}\mathbf{v}}\prec \sqrt{\frac{\im m(z)}{N\eta}}+\frac{1}{N\eta}.
\end{equation}

\item  \textnormal{\textbf{Averaged local law:}} For any deterministic matrix $B \in \C^{DN\times DN}$ with $\|B\|\le 1$, we have 
\begin{equation}\label{eq:aver_local}
 \left\langle \left(G-M\right) B\right\rangle \prec \frac{1}{ N\eta } .
\end{equation}

\end{itemize}
As a consequence of \eqref{eq:aver_local} when $B=I$, we have the rigidity of eigenvalues:
\begin{equation}\label{eq:rigidity}
\abs{\lambda_k-\gamma_k} \prec N^{-2/3}\mathfrak{r}\pa{k}^{-1/3},\quad \forall 1\le k\le DN. 
\end{equation}
In addition, all the above estimates remain valid even if we do not assume identical distributions for the diagonal and off-diagonal entries of $H$.
\end{lemma}

From the local law \eqref{eq:aniso_local}, we can derive some more general estimates for products of resolvents, which will be stated as \Cref{lemma_necessary_estimates} in \Cref{appendix}. These estimates will serve as the basic tools for our proofs.

\subsection{Preliminaries}
In the main proofs, the perturbation matrix $\Lambda$ may evolve with parameter $t$. For convenience, we introduce the following notations.

\begin{definition}\label{def_moving_Lambda}
    Suppose $\Lambda_t:[a, b] \rightarrow \mathbb{C}^{D N \times D N}$ is a continuous map such that $\Lambda_t$ satisfies \Cref{main_assm} throughout the evolution. 
     We define $m_t(z)$ as the unique solution to 
 \be\nonumber
m_t(z)= \big\langle \left(\Lambda_t -z- m_t(z)\right)^{-1} \big\rangle ,
\ee
such that $\im m_t(z)>0$ whenever $z\in \C_+$.  
Then, we define $M_t=M_t\pa{z,\Lambda_t}$ as
\begin{equation}\nonumber
        \begin{aligned}
            M_t(z)=
            \left(\Lambda_t-z-m_t(z)\right)^{-1},
        \end{aligned}
    \end{equation}
noting that $m_t(z)=\left\langle M_t(z)\right\rangle$. 
The associated probability density is given by
    \begin{equation}\nonumber
        \begin{aligned}
            \rho_t(E)=\frac{1}{\pi} \lim _{\eta \downarrow 0} \operatorname{Im} m_t(E+i \eta).
        \end{aligned}
    \end{equation}
We denote the support of $\rho_t$ by $\left[E_t^{-}, E_t^{+}\right]$, where $E_t^{\pm}$ represent the spectral edges. For $z=E+\ii \eta$, we also define $ \kappa_t\equiv \kappa_t(E):=|E_t^{+}-E|\wedge |E-E_t^{-}|$. Finally, we define the quantiles $\gamma_k\pa{t}$ of $\rho_t$ as in \eqref{eq:gammak}.
\end{definition}




Our proofs rely on the following identity, derived directly from the definitions of $G$ and $M$ in \eqref{def_M}: 
\be\label{eq:G-M} 
G - M= - G(H+m)M= - M(H+m)G,
\ee
together with the complex cumulant expansion formula. We use the version stated in \cite[Lemma 7.1]{He:2017wm}. 
	\begin{lemma}(Complex cumulant expansion) \label{lem:complex_cumu}
		Let $h$ be a complex random variable all of whose moments exist. The $(p,q)$-cumulant of $h$ is defined as
		$$
		\mathcal{C}^{(p,q)}(h)\deq (-\ii)^{p+q} \cdot \left(\frac{\partial^{p+q}}{\partial {s^p} \partial {t^q}} \log \E e^{\mathrm{i}sh+\mathrm{i}t\bar{h}}\right) \bigg{|}_{s=t=0}\,.
		$$
		Let $f:\mathbb C^2 \to \C$ be a smooth function, and we denote its holomorphic  derivatives by
		$$
		f^{(p,q)}(z_1,z_2)\deq \frac{\partial^{p+q}}{\partial z_1^p \partial z_2^q} f(z_1,z_2)\,.
		$$ Then, for any fixed $l \in \N$, we have
		\begin{equation} \label{5.16}
		\E f(h,\bar{h})\bar{h}=\sum\limits_{p+q=0}^l \frac{1}{p!\,q!}\mathcal{C}^{(p,q+1)}(h)\E f^{(p,q)}(h,\bar{h}) + R_{l+1}\,,
		\end{equation}
		given all integrals in (\ref{5.16}) exist. Here, $R_{l+1}$ is the remainder term depending on $f$ and $h$, and for any $\tau>0$, we have the estimate
		$$
			\begin{aligned}
			R_{l+1}=&\, \OO(1)\cdot \E \big|h^{l+2}\mathbf{1}_{\{|h|>N^{\tau-1/2}\}}\big|\cdot \max\limits_{p+q=l+1}\big\| f^{(p,q)}(z,\bar{z})\big\|_{\infty} \\
			+&\,\OO(1) \cdot \E |h|^{l+2} \cdot \max\limits_{p+q=l+1}\big\| f^{(p,q)}(z,\bar{z})\cdot \mathbf{1}_{\{|z|\le N^{\tau-1/2}\}}\big\|_{\infty}\,.
			\end{aligned}
		$$
	\end{lemma}

With assumptions \eqref{eq:meanvar}, \eqref{eq:meanvar2}, and \eqref{eq:highmoment}, we can show that for $i, j\in \cI$,
$$ \mathcal{C}^{(0,1)}(H_{ij})=\mathcal{C}^{(1,0)}(H_{ij})=0,\quad \mathcal{C}^{(1,1)}(H_{ij}) =s_{ij},\quad \mathcal{C}^{(0,2)}(H_{ij})=\mathcal{C}^{(2,0)}(H_{ij})=s_{ij}\delta_{ij},$$
and that for any fixed $p,q\in \N$ with $p+q\ge 3$, there exists a constant $C>0$ such that
\be\label{eq:bdd_cumu}
\max_{i,j\in \cI}|\mathcal{C}^{(p,q)}(H_{ij})| \le \left(CN\right)^{-(p+q)/2}.
\ee
We also adopt the following notation from \cite[equation (42)]{cipolloni-erdos2021}.

\begin{definition}\label{def:underline}
Suppose that $f$ and $g$ are matrix-valued functions. Define
\be
\underline{g(H)Hf(H)} := g(H)Hf(H)-\widetilde{\E}g(H)\widetilde{H}(\partial_{\widetilde{H}}f)(H) - \widetilde{\E}(\partial_{\widetilde{H}}g)(H)\widetilde{H}f(H),
\ee
where $\widetilde{H}$ is an indepdent copy of $H$, $\widetilde{\E}$ denotes the partial expectation with respect to $\wt H$, and $(\partial_{\widetilde{H}}f)(H)$ denotes the directional derivative of the function $f$ in the direction $\wt H$ at the point $H$, i.e.,
\be
[(\partial_{\widetilde{H}}f)(H)]_{xy} = (\widetilde{H}\cdot \nabla f(H))_{xy} := \sum_{\alpha,\beta\in \cI}\widetilde{H}_{\alpha\beta}\frac{\partial f(H)_{xy}}{\partial H_{\alpha\beta}}.
\ee
\end{definition}

The terms subtracted from $g(H)Hf(H)$ are precisely the second-order term in the cumulant expansion. In particular, if all entries of $H$ are Gaussian, we have $\E\underline{g(H)Hf(H)} = 0$. Moreover, if we take $g(H) = I$ and $f(H) = G$, we have that
\be\label{eq:HGline}
\underline{HG} = HG + \widetilde{\E}[\widetilde{H}G\widetilde{H}]G,\quad \text{with}\quad \widetilde{\E}[\widetilde{H}G\widetilde{H}] = \sum_{a=1}^D D\langle GE_a\rangle E_a.
\ee

We will frequently use the Cauchy-Schwarz inequality and the following Ward's identity to bound various quantities involving the resolvent.

\begin{lemma}[Ward's identity]\label{lem-Ward}
Let $\cal A$ be a Hermitian matrix. Define its resolvent as $R(z):=(\cal A-z)^{-1}$ for any $z= E+ \ii \eta\in \C_+$. Then, we have 
    \be\label{eq_Ward0}
    \begin{split}
\sum_x \overline {R_{xy'}}  R_{xy} = \frac{R_{y'y}-\overline{R_{yy'}}}{2\ii \eta},\quad
\sum_x \overline {R_{y'x}}  R_{yx} = \frac{R_{yy'}-\overline{R_{y'y}}}{2\ii \eta}.
\end{split}
\ee
As a special case, if $y=y'$, we have
\be\label{eq_Ward}
\sum_x |R_{xy}|^2 =\sum_x |R_{yx}|^2 = \frac{\im R_{yy}}{ \eta}.
\ee
\end{lemma}
\begin{proof}
These identities follow directly from the algebraic identity $R-R^*=2\ii \eta RR^*=2\ii \eta R^*R$.
\end{proof} 

\subsection{Proof ideas}\label{sec_idea}

In this subsection, we outline the core ideas underlying the proof of our main theorems. Without loss of generality, we assume that $k\in\qq{1,DN/2}$ so that  $\fr\pa{k}=k$.

\medskip

\paragraph{\bf Delocalized regime}

Our proofs in the delocalized phase largely follow the framework developed in \cite{stone2024randommatrixmodelquantum} for the bulk of the eigenvalue spectrum, with necessary modifications in the regime near the spectral edges. By Markov's inequality, the delocalization estimate \eqref{eq:main_evector1} follows directly from the second moment bound $\E[\|E_a \bv_k\|^2-D^{-1}]^2\le N^{-\delta}$ for some constant $\delta>0$ depending on $\varepsilon_A$.
Using the spectral decomposition of $G(z)$ and the eigenvalue rigidity \eqref{eq:rigidity}, the proof can reduce to establishing the two-resolvent bound:
\be\label{eq:normtrace}
\E \avg{\im G(z) (E_a -D^{-1}) \im G(z) (E_a -D^{-1})} \le N^{-1-\delta}\eta^{-2}, \quad \forall a\in \qq D, 
\ee
where $z=\gamma_k+\ii\eta$ and $\eta=N^{-2/3+\varepsilon}k^{-1/3}$, with $\e>0$ an arbitrarily small constant.  
Similar to \cite{stone2024randommatrixmodelquantum}, we prove \eqref{eq:normtrace} using the \emph{characteristic flow} method---a dynamic approach for estimating resolvents along a flow of the spectral parameter $z$, which corresponds to the characteristic flow of the underlying complex Burgers equation.  
This method was first introduced in \cite{lee2015edge} and has since been applied to various models \cite{HL2019_PTRF,Landon2024_PTRF,Adhikari2020_PTRF,Adhikari2023_PTRF,LS_2022,Bourgade_JEMS} to establish single-resolvent local laws (or closely related quantities), as well as more general multi-resolvent local laws, as in  \cite{BF_LQG,CEH2023,CEX2023,CGLJ2024_AJMP,erdHos2024eigenstate,cipolloni2024mesoscopic,campbell2024spectraledgenonhermitianrandom,erdos2024cuspuniversalitycorrelatedrandom}. 
It consists of three main steps:
\begin{enumerate}
\item[(1)] establishing a global law for $G(z)$ when $z$ lies away from the limiting spectrum $[E^-,E^+]$;

\item[(2)] propagating the estimates from large scales of $\im z$ to smaller scales along the characteristic flow, while introducing a Gaussian component into the original matrix model;

\item[(3)] eliminating the Gaussian component using a Green's function comparison argument.
\end{enumerate}
Steps (1) and (3) follow almost identically to the approach in \cite{stone2024randommatrixmodelquantum}. In Step (2), to extend the argument of \cite{stone2024randommatrixmodelquantum} to the spectral edge regime, it is crucial to carefully track the factors involving $\im m(z)$ in the estimates. This allows us to cancel certain singularities arising near the spectral edges; see \Cref{sec:delocalized_case_eigenvector} for further details.

After establishing the delocalization of the edge eigenvectors in \Cref{mix}, we can then prove \Cref{MixEV} by adopting an idea from \cite{CFHJBulkBandAOP2024}. 
Specifically, we utilize the estimate \eqref{eq:main_evector1}---referred to as a \emph{quantum unique ergodicity} estimate in \cite{CFHJBulkBandAOP2024}---to facilitate the Green's function comparison in the classical three-step strategy for proving the universality of eigenvalue statistics (see \cite{Erds2017ADA} for a review of this strategy).
Our argument closely resembles that in \cite{stone2024randommatrixmodelquantum}. However, near the spectral edges, we must conduct a comparison argument for a more complex function of $G(z)$, which requires a deeper exploration of its algebraic structures. For more details, see \Cref{sec:delocalized_case_eigenvalue}.

\medskip

\paragraph{\bf Localized regime} 
Despite the similarities to \cite{stone2024randommatrixmodelquantum} concerning the proofs in the delocalized phase, the proofs for the localized phase are significantly more challenging and technically demanding in our context, particularly near the spectral edges. 
In the remainder of this subsection, we will focus on explaining the key ideas behind the proofs of \Cref{theorem_localized_eigenvector_localization,theorem_localized_eigenvalue}.
The detailed proof will be presented in \Cref{sec:localized_case}.

For the proof of \Cref{theorem_localized_eigenvalue}, we define a sequence of interpolating matrices as 
\be\label{eq:VVtheta}
\VV(t):=H+t \Lambda, \quad t\in [0,1], \quad \text{with}\quad \VV(0)=H, \quad \VV(1)=\VV.
\ee
By standard perturbation theory for eigenvalues, we have 
$\lambda_k'(t)=\bv_k(t)^* \Lambda \bv_k(t)$, where $\lambda_k(t)$ denotes the $k$-th eigenvalue of $\VV(t)$, and $\bv_k(t)$ represents the corresponding eigenvector.
Thus, we can control the difference between the $k$-th eigenvalues of $H_\Lambda$ and $H$ by bounding $\bv_k(t)^* \Lambda \bv_k(t)$ for each $t\in[0,1]$. It is desirable to show that this quantity is much smaller than the typical fluctuation $N^{-2/3}k^{-1/3}$ of $\lambda_k$. This holds true within the bulk of the limiting spectrum, as shown in \cite{stone2024randommatrixmodelquantum}. 
However, it fails in the edge regime, where the perturbation $\Lambda$ induces a non-negligible shift in the quantiles $\gamma_k$. Incorporating this shift, given by $\gamma_k-\gamma_k^{\mathrm{sc}}$, we have that 
\begin{equation}\label{eq:lambdak-gammak}
    \begin{aligned}
  \E\absa{\pa{\lambda_k-\gamma_k}-\pa{\lambda_k\pa{H}-\gamma_k^\txt{sc}}}^2&=\E\absa{\int_{0}^1\left[\lambda_k'\pa{t}-\gamma_k'\pa{t}\right]\,\rd t}^2\leq  \int_{0}^1\E\absa{\lambda_k'\pa{t}-\gamma_k'\pa{t}}^2\,\rd t\\  &=\int_{0}^1\E\absa{\bv_k^*\pa{\Lambda-\gamma_k'\pa{t}}\bv_k}^2\,\rd t,
    \end{aligned}
\end{equation}
where $\gamma_k\pa{t}$ is the quantile defined as in \Cref{def_moving_Lambda} with $\Lambda_t=t\Lambda$. 
Let $z_t=\gamma_k\pa{t}+\ii\eta$, where $\eta=N^{-2/3+\varepsilon}k^{-1/3}$ for an arbitrarily small constant $\e>0$. By applying the spectral decomposition of $G_t=(\VV(t)-z_t)^{-1}$ along with the rigidity estimate for $\lambda_k(t)-\gamma_k(t)$, we can obtain that (see \eqref{bdwgy} below)
\begin{equation}\label{eq:2Gloop}
    \begin{aligned}     \E\absa{\bv_k(t)^*\pa{\Lambda-\gamma_k'\pa{t}}\bv_k(t)}^2\prec N\eta^2\E\avga{\pa{\im G_t}\pa{\Lambda-\gamma_k'\pa{t}}\pa{\im G_t}\pa{\Lambda-\gamma_k'\pa{t}}}.
    \end{aligned}
\end{equation}
Hence, to bound \eqref{eq:lambdak-gammak}, it suffices to control the right-hand side (RHS) of \eqref{eq:2Gloop}, which we refer to as a two-resolvent loop. 
One technical challenge in the proof is that  $\gamma_k'\pa{t}$ takes a complicated and implicit form. Fortunately, under the assumption \eqref{eq:condA2}, we can approximate $\gamma_k'\pa{t}$ with a more explicit quantity
\begin{equation*}
\Delta\pa{t}:=\frac{\avga{M_t\p{z_t}\Lambda M_t^*\p{z_t}}}{\avga{M_t\p{z_t}M_t^*\p{z_t}}},
\end{equation*}
with an error that is much smaller than the typical fluctuation $N^{-2/3}k^{-1/3}$. Here, $M_t$ is defined as in  \Cref{def_moving_Lambda} with $\Lambda_t=t\Lambda$. This expression allows us to derive a key deterministic cancellation (as detailed in the estimate \eqref{regular_estimate_localized_eigenvalue} below), which is crucial for establishing the following two-resolvent estimate for some constant $C>0$ that does not depend on $\e$:
\begin{equation}\label{eq:2-resolven1}
        \E\avga{\pa{\im G_t}\pa{\Lambda-\Delta\pa{t}}\pa{\im G_t}\pa{\Lambda-\Delta\pa{t}}}\prec N^{C\varepsilon}N^{-5/3}k^{2/3}\norma{A}_{\HS}^2.
\end{equation}
Substituting this into \eqref{eq:2Gloop} and subsequently into \eqref{eq:lambdak-gammak} yields  
$$\E\absa{\pa{\lambda_k-\gamma_k}-\pa{\lambda_k\pa{H}-\gamma_k^\txt{sc}}}^2\prec N^{-2+(C+2)\e}\|A\|_{\HS}^2.$$
Together with Markov's inequality, this completes the proof of \Cref{theorem_localized_eigenvalue} since $\e$ is arbitrary.


For the proof of \Cref{theorem_localized_eigenvector_localization}, we adopt a similar idea as in \cite[Section 7]{stone2024randommatrixmodelquantum}, but we need to incorporate the shift of the quantiles $\gamma_k-\gamma_k^{\mathrm{sc}}$, as inspired by the above discussion for the proof of \Cref{theorem_localized_eigenvalue}. 
To illustrate this idea, we consider the case $D=2$ for simplicity. By \Cref{theorem_localized_eigenvalue}, we know that $\lambda_k-\gamma_k + \gamma_k^{\mathrm{sc}}$ is a small perturbation of $\lambda_k(H)$ compared to the typical fluctuation $N^{-2/3}k^{-1/3}$.
Without loss of generality, suppose that $\lambda_k(H)$ is the eigenvalue of the block $H_1$. 
Then, by the level repulsion estimates for the Wigner matrix $H_2$ (see e.g.,  \cite{BENIGNI2022108109}), we know that the eigenvalue spectrum of $H_2$ is separated from $\lambda_k-\gamma_k + \gamma_k^{\mathrm{sc}}$ by a distance of order $N^{-2/3}k^{-1/3}$ with probability $1-\oo(1)$.
Suppose the $k$-th eigenvector of $\VV$ can be written as $\bv_k=(\bu_k^\top, \bw_k^\top)^\top$, where $\bu_k,\bw_k\in \C^N$.
From the eigenvalue equation $\VV \bv_k =\lambda_k \bv_k$, we get 
\begin{equation*}
        \begin{pmatrix}
            H_1 & A\\
            A^* & H_2
        \end{pmatrix}\begin{pmatrix}
    \bu_k\\
    \bw_k
\end{pmatrix}
-\left(\gamma_k - \gamma_k^{\mathrm{sc}}\right) \begin{pmatrix}
    \bu_k\\
    \bw_k
\end{pmatrix}
=\pa{\lambda_k- \gamma_k + \gamma_k^{\mathrm{sc}} }\begin{pmatrix}
    \bu_k\\
    \bw_k
\end{pmatrix},
\end{equation*}
which implies 
\begin{equation}\label{eq:bwk}
\bw_k=-\cG_2\pa{\lambda_k- \Delta_k}\pa{A^*\bu_k - \Delta_k\bw_k},
\quad \bu_k=-\cG_1 \pa{\lambda_k- \Delta_k} \pa{A\bw_k- \Delta_k\bu_k}.
\end{equation}
Here, we denote $\Delta_k:=\gamma_k -\gamma_k^{\mathrm{sc}}$ and $\cG_i (z) :=\pa{H_i-z}^{-1}$ as the resolvent of $H_i$ for $i\in\{1,2\}$. 


One insight from \cite{stone2024randommatrixmodelquantum} is that in the localized regime, $A$ is a small perturbation, so $H_2$ and $\bu_k$ should be nearly independent. This implies that when $\dist(\lambda_k-\Delta_k,\spec(H_2)) \gtrsim N^{-2/3}k^{-1/3}$, $\|\cG_2\pa{\lambda_k- \Delta_k}\pa{A^*\bu_k}\|$ should be small, while the other term $\|\cG_2\pa{\lambda_k- \Delta_k}\pa{ \Delta_k\bw_k}\|$ is also small since $\Delta_k$ represents a small shift. However, this argument cannot reach the optimal threshold for $\|A\|_{\HS}$. 
If we were to naively apply the strategy from \cite{stone2024randommatrixmodelquantum} to bound $\|\cG_2\pa{\lambda_k- \Delta_k}\pa{A^*\bu_k}\|$, we would get expressions that are properly bounded only when $\|A\|_{\HS}\ll N^{1/6}/k^{1/6}$. 
To address this issue, we need to bound the term $\|\cG_2\pa{\lambda_k- \Delta_k}\pa{A^*\bu_k - \Delta_k\bw_k}\|$ as a whole, so that the leading terms cancel in the proof. This cancellation leads to the critical threshold $\|A\|_{\HS}\ll N^{1/3}/k^{1/3}$.

Let $G_0(z):=(H-z)^{-1}$ denote the resolvent of $H$, and let $z=\gamma_k+\ii\eta$, where $\eta=N^{-2/3+\e}k^{-1/3}$ for an arbitrarily small constant $\e>0$. By applying the spectral decompositions of $G$ and $G_0$ along with the eigenvalue rigidity estimate for $\lambda_k$ and the level repulsion estimates for Wigner matrices, we can bound the vectors in \eqref{eq:bwk} as (see \eqref{eq:boundE1} below):  
\begin{equation}
\label{eq:ukwk}    \E\pa{\norm{\bu_k}^2\wedge\norm{\bw_k}^2}\prec N\E\avga{\pa{\im G_0\pa{z-\Delta_{k}}}\pa{\Lambda-\Delta_{k}}\pa{\im G (z) }\pa{\Lambda-\Delta_{k}}}.
\end{equation}
One technical issue is that the shift $\Delta_k$ also takes on a complicated and implicit form. 
However, under \eqref{eq:condA2}, we can approximate it with the following quantity, with an error that is much smaller than the typical fluctuation $N^{-2/3}k^{-1/3}$:
\begin{equation}\label{eq:shiftEV}     \Delta_{\txt{ev}}\equiv \Delta_{\txt{ev}}(z)=\mathrm{Re}\pa{z+m (z) +\frac{1}{m (z) }}.
\end{equation}
Again, this expression enables us to derive a key deterministic cancellation (as we will discuss in \eqref{intro_regular_bound_eigenvector} below), which is crucial for establishing the following two-resolvent estimate for some constant $C>0$ that does not depend on $\e$: 
\begin{equation}\label{eq:2-resolven2}
    \begin{aligned}
        \E\avga{\pa{\im G_0\pa{z-\Delta_{\txt{ev}}}}\pa{\Lambda-\Delta_{\txt{ev}}}\pa{\im G (z) }\pa{\Lambda-\Delta_{\txt{ev}}}}\prec N^{C\varepsilon}N^{-5/3}k^{2/3}\norma{A}_{\HS}^2.
    \end{aligned}
\end{equation}
Applying this estimate to \eqref{eq:ukwk} will complete the proof of \Cref{theorem_localized_eigenvector_localization}.

The main technical challenge for our proofs within the localized regime is to establish the two-resolvent estimates \eqref{eq:2-resolven1} and \eqref{eq:2-resolven2}. These two estimates have similar forms, and their proofs are nearly identical. For the sake of discussion, we will focus on the estimate \eqref{eq:2-resolven2}.  
To bound the left-hand side (LHS) of \eqref{eq:2-resolven2}, we will expand it using \eqref{eq:G-M} and the cumulant expansion in \Cref{lem:complex_cumu}, following a specific  expansion strategy developed in \cite{stone2024randommatrixmodelquantum}.  
To illustrate this, denote \smash{$\wt \Lambda=\Lambda-\Delta_\txt{ev}$}, $z_1 \equiv z =\gamma_k +\ii \eta
$ with $\eta=N^{-2/3+\varepsilon}k^{-1/3}$, and $z_0=z_1-\Delta_{\txt{ev}}$. We abbreviate that $G_0\equiv G_0\pa{z_0}$, $m_0\equiv m_{\mathrm{sc}}(z_0)$, $M_0\equiv m_0 I$, and $G_1\equiv G_1(z_1)$, $M_1\equiv M(z_1)$, $m_1=\avga{M_1}$. 
Using $\im G=\frac{1}{2\ii}\pa{G-G^*}$, we can decompose the LHS of \eqref{eq:2-resolven2} into four parts:
\begin{equation}\label{imaginary_part_4_fold_linearization}
    \begin{aligned}
\avg{\im G_0\cdot\tLambda\cdot\im G_1\cdot\tLambda}=-\frac 14 \left(\avg{G_0\tLambda G_1\tLambda}+\avg{G_0^*\tLambda G_1^*\tLambda}-\avg{G_0^*\tLambda G_1\tLambda}-\avg{G_0\tLambda G_1^*\tLambda}\right).
    \end{aligned}
\end{equation}
Next, we expand these terms using the following identities:
\be\label{eq:id_G0G1}
\begin{aligned}
&G_0=M_0-G_0\pa{H+m_{0}}M_0=M_0-M_0\pa{H+m_{0}}G_0 \, ,\\ 
&G_1=M_1-G_1\pa{H+m_1}M_1=M_1-M_1\pa{H+m_1}G_1 \, .    
\end{aligned}
\ee
In each step, we apply \eqref{eq:id_G0G1} to a carefully selected $G_0$ or $G_1$ entry, generating a more deterministic term with $G_0$ or $G_1$ replaced by $M_0$ or $M_1$, along with a term that factors out an $H$ entry. We then apply the cumulant expansion \eqref{5.16} to the latter term with respect to the $H$ entry. 
This yields a linear combination of leading terms that are ``more deterministic", higher-order terms whose sizes are reduced compared to the original expression by a factor of $N^{-c}$ for some constant $c>0$, and some negligible error terms corresponding to the remainder term $R_{l+1}$ in \eqref{5.16}. 
If a leading term becomes ``deterministic enough" (in a sense we will describe in \Cref{subsection:proof_of_lemma_localized_eigenvector_local_law_and_lemma_localized_eigenvalue_local_law} below) or if a higher-order term has sufficiently small size, then we will stop the expansion. 
Otherwise, we continue the process by selecting another $G_0$ or $G_1$ entry according to a specific rule, decomposing it as in \eqref{eq:id_G0G1}, and applying the cumulant expansion again.  
By repeating this procedure for $\OO(1)$ many steps, we finally obtain a linear combination of higher-order terms that can be directly bounded, along with some leading terms that are ``deterministic enough".


Compared to the proof in \cite{stone2024randommatrixmodelquantum}, which focuses on the bulk regime, our proof in the edge regime is much more involving and delicate due to the possibly diverging factor $\|A\|_{\HS}$ (recall \eqref{eq:condA2}) when $k$ is small. To cancel these singular factors, as has been done in many previous works addressing local laws of random matrices near spectral edges (see e.g., \cite{erdHos2012rigidity}),  
we need to obtain additional small factors $\im m(z)$, that arise from the vanishing spectral density near edges. This adds significant technical complexity to the proof in several ways.

One major technical challenge involves estimating the leading terms from our expansion strategy that are ``deterministic enough". In the bulk regime, these leading terms can be bounded directly, as demonstrated in \cite{stone2024randommatrixmodelquantum}. However, in our setting, the main leading terms will include additional powers of $N^{1/3}/k^{1/3}$, which makes the estimate too weak for our proof. 
Thus, we must explicitly enumerate these troublesome terms and identify cancellations in them. 
One type of cancellation arises from the polarization identity in \eqref{imaginary_part_4_fold_linearization}---in some expressions from the expansions, a leading term containing $M_0$ (or $M_1$) cancels with a corresponding term that has the same form but with $M_0$ (or $M_1$) replaced by $M_0^*$ (or $M_1^*$), resulting in an extra $\im m_0$ or $\im m_1$ factor. 
Another type of cancellation occurs in expressions that include a factor of the form \smash{$\avg{\sM_0\wt \Lambda \sM_1E_a}$}, where $a\in\qq D$, $\sM_0\in \ha{m_\txt{sc}(z_0)I,\bar{m}_\txt{sc}(z_0)I}$, and  $\sM_1\in \ha{M(z_1),M^*(z_1)}$.
For this factor, we have the following estimate (see \Cref{lem_regular_estimate_localized_eigenvector} below for the proof):
\begin{equation}\label{intro_regular_bound_eigenvector}
    \begin{aligned}
        \avg{\sM_0\wt \Lambda \sM_1E_a}=\OO\pa{\im m_1\cdot \avga{\Lambda^2}}.
    \end{aligned}
\end{equation} 
We remark that without introducing the shift  $\Delta_{\mathrm{ev}}$, the correct bound for $\avg{\sM_0 \Lambda \sM_1E_a}$ is of order $\OO(\avga{\Lambda^2})$, as indicated by the estimate \eqref{add_one_more_Lambda} below. The introduction of the shift $\Delta_{\mathrm{ev}}$ results in a cancellation that improves the bound by an additional factor of $\im m_1$. 
Finally, we mention that similar improved estimates have been discussed in a series of works \cite{CEH2023,erdHos2024eigenstate,CGLJ2024_AJMP,CGEHJD2024_JFA} concerning the proofs of certain optimal multi-resolvent local laws via the characteristic flow method, where it is referred to as a \emph{regularity condition}. However, our estimate in \eqref{intro_regular_bound_eigenvector} has a somewhat different basis than the regularity conditions presented in those works.

Another technical challenge involves managing the cumulant expansions and a more intricate expansion strategy. Similar to \cite{stone2024randommatrixmodelquantum}, we divide the terms from the cumulant expansion \eqref{5.16} into two parts: the leading part with $p+q=1$ (which corresponds to an application of Gaussian integration by parts) and the remaining higher-order cumulant terms. 
Our treatment of the Gaussian integration by parts terms largely follows the approach in \cite{stone2024randommatrixmodelquantum}, with the additional need to exploit the cancellation mechanisms discussed above. 
On the other hand, unlike in \cite{stone2024randommatrixmodelquantum}, the higher-order cumulant terms with $p+q>1$ in our setting cannot be handled as straightforwardly through direct estimation. 
While the higher-order cumulant terms with $p+q\ge 3$, despite their complicated structure, can still be estimated directly, the $p+q=2$ terms cannot be controlled using the desired bounds and thus require a more delicate analysis. 
We need to further expand these terms using \eqref{eq:id_G0G1} and \eqref{5.16} according to a newly designed expansion strategy. 
These expansions again yield higher-order terms that can be directly bounded, along with some leading terms that are ``deterministic enough". Estimating the leading terms is particularly involved, as it requires tracking their detailed structures and exploring the cancellations mentioned earlier. For more details on the argument, readers can refer to \Cref{subsection:proof_of_claim_estimate_remainder_terms}.

\section{Delocalized phase: eigenvectors}\label{sec:delocalized_case_eigenvector}

In this section, we prove \Cref{mix}. For convenience, we will consider the case $k\in\qq{1,DN/2}$ such that $\fr\pa{k}=k$. We begin by defining the following notations.

\begin{definition}
    \label{def_wtM} 
Define the spectral domain $\mathbf D(\tau):=\{z=E+\ii \eta\in \C :|z|\le \tau^{-1}, |\eta| \ge N^{-1+\tau}\}$ for an arbitrarily small constant $\tau>0$.
For $z_1,z_2\in \mathbf D(\tau)$, we define the $D\times D$ matrices $\wh M$ and $L$ as
\begin{align}\label{def_ML}
    \wh M_{ab}(z_1,z_2,\Lambda):=D\langle M(z_1) E_{a} M(z_2) E_{b}\rangle ,\quad L_{ab}(z_1,z_2, H, \Lambda):= D\langle G(z_1)E_{a}G(z_2)E_{b}\rangle , 
\end{align}
for  $a,b\in \llbracket D\rrbracket$, and define the $D\times D$ matrix $K$ as
\begin{align}\label{def_K}   K(z_1,z_2, \Lambda ):= \left[1-\wh M(z_1,z_2,\Lambda)\right]^{-1} \wh M(z_1,z_2,\Lambda).
\end{align}
\end{definition}

For ease of presentation, we introduce the following simplified notations: given a matrix-valued function (e.g., $G$, $M$, \smash{$\wh M$}, $L$, and $K$) of $z$, we use subscripts to indicate its dependence on the spectral parameters. For example, we will denote $G_i:= G(z_i,H,\Lambda)$, $M_{i}:=M(z_{i},\Lambda),$ \smash{$\wh M_{(1,2 )}:=\wh M(z_{1 },z_{2 } ,\Lambda )$}, $L_{(1,2)}:=L(z_{1 },z_{2 }, H ,\Lambda ),$ and $K_{(1,2 )}:=K(z_{1 },z_{2 } ,\Lambda )$. We also need the following notations that are similar to those in \Cref{def_wtM} but with three $z$ arguments.

\begin{definition}\label{defLK3}
Define the $D\times D\times D$ tensors $L$ and $K$ as
$$
\left[L(z_1, z_2, z_3, H, \Lambda)\right]_{a_1 a_2 a_3}:= D\langle G_1E_{a_1}G_2E_{a_2}G_3E_{a_3}\rangle,
$$
$$\left[K(z_1, z_2, z_3, \Lambda)\right]_{a_1 a_2 a_3} = \sum_{b_1, b_2, b_3}\big(I - \wh M_{(1,2)}\big)^{-1}_{a_1b_1}\big(I - \wh M_{(2,3)}\big)^{-1}_{a_2b_2}\big(I - \wh M_{(3,1)}\big)^{-1}_{a_3b_3}D\langle M_1E_{b_1}M_2E_{b_2}M_3E_{b_3}\rangle ,
$$
for $a_1,a_2,a_3\in \qq D$. Here, we have abused the notations a little bit and still use $L$ and $K$ to denote these tensors. Moreover, we will also abbreviate these notations as $L_{(1,2,3)}$ and $K_{(1,2,3)}$.
\end{definition} 

\subsection{Proof of \Cref{mix}}
Since the proof is similar to that in the bulk regime  \cite{stone2024randommatrixmodelquantum}, we will outline only the main differences from the proof in \cite{stone2024randommatrixmodelquantum}, without writing the full details. The key is to establish the following lemma. 

\begin{lemma}\label{key_lemma_delocalized_eigenvector}
Take $z=E+\ii \eta\in \mathbf D\pa{\tau}$ with $E=\gamma_k\in \qa{E^{-},E^+}$ and $\eta\sim N^{-2/3+\varepsilon_L}k^{-1/3}$ for a small constant $\varepsilon_L>0$ (recall that we assume $k\in\qq{1,DN/2}$). Under the assumptions of \Cref{mix}, there exists a constant $c_L>0$ (depending on $\varepsilon_L,\delta_A,\varepsilon_A$) such that 
\be\label{main_EGG}
\max_{z_1,z_2\in\{z,\overline z\}}\max_{a,b\in \qq{D}}\left|\left(\E L_{(1,2)}- K_{(1,2)}\right)_{ab}\right|
=\OO\pa{N^{-1-c_L}\eta^{-2}} .
\ee
\end{lemma}

As discussed in the proof of \cite[Theorem 2.2]{stone2024randommatrixmodelquantum}, \Cref{key_lemma_delocalized_eigenvector} actually implies a slightly stronger estimate than \eqref{eq:main_evector1}: for some constant $c>0$,
\begin{equation}\label{eq:extend:main_evector1}
\mathbb P\left(\max _{i,j\in \qq{k-N^c,k+N^c}}\max_{a\in \qq D}
\left| {\bf v}_i^* (E_a-D^{-1}){\bf v}_j \right|\ge N^{-c}\right) \le  N^{-c}.
\end{equation}
This estimate will play a key role in the proof of \Cref{MixEV}. 
Now, we provide the proof of \Cref{mix} and \eqref{eq:extend:main_evector1} using \Cref{key_lemma_delocalized_eigenvector}. It is similar to the proof of Theorem 2.2 in \cite{stone2024randommatrixmodelquantum}, and for the convenience of the readers, we will repeat the argument here.

\begin{proof}[\bf Proof of \Cref{mix} and equation \eqref{eq:extend:main_evector1}]

Recall that we assume $k\in\qq{1,DN/2}$ without loss of generality. For $z=E+\ii \eta$, using the spectrum decomposition of $\im G(z)$,
we get that for any $DN\times DN$ matrix $B$,
$$
\tr \left[ \im G (z) B\im G(z)  B^* \right]=\eta^2\sum_{i,j\in \cI}\frac{|\bv^*_i B \bv_j|^2}{|\lambda_i-z|^2 |\lambda_j-z|^2 }.
$$
In particular, choosing $B=E_{a}-D^{-1}I$ and $z_k=\gamma_k+\ii\eta$ with $\eta=N^{-2/3+\varepsilon_L}k^{-1/3}$,  and using the rigidity of eigenvalues in \eqref{eq:rigidity}, we get from the above equation that for any constant $c\in (0,\e_L/100)$, 
\be\label{eq:bdd_spec}
 \max_{i,j\in \qq{k-N^c,k+N^c}}|\bv^*_i (E_{a}-D^{-1}) \bv_j|^2 \prec \eta^2 \tr \left[ \im G (z_k) (E_{a}-D^{-1}I)\im G(z_k) (E_{a}-D^{-1}I) \right].
\ee
Choosing $z_1=z_k$, $z_2=\bar z_k$ and applying \eqref{main_EGG}, we can estimate the expectation of the RHS of \eqref{eq:bdd_spec} as
\begin{align}
   &-\frac{1}{4}\eta^2 \E\tr \Big[ (G_1-G_2)\Big(E_{a}-D^{-1}\sum_b E_b\Big)(G_1-G_2) \Big(E_{a}-D^{-1}\sum_{b'} E_{b'}\Big) \Big] \nonumber\\
   &=N\eta^2 \bigg(\E\cal L_{aa} - \frac{2}{D} \sum_{b=1}^D \E\cal L_{ab} + \frac{1}{D^2} \sum_{b,b'=1}^D \E\cal L_{bb'}\bigg)\nonumber\\
   &=N\eta^2 \bigg(\cal K_{aa} - \frac{2}{D} \sum_{b=1}^D \cal K_{ab} + \frac{1}{D^2} \sum_{b,b'=1}^D \cal K_{bb'}\bigg) + \OO\left(N^{-c_L}\right),\label{eq:bdd_spec2}
\end{align}
where the $D\times D$ matrices $\cal L$ and $\cal K$ are defined as $\cal L:=(L_{(1,2)}+L_{(2,1)}-L_{(1,1)}-L_{(2,2)})/4$ and $\cal K:=(K_{(1,2)}+K_{(2,1)}-K_{(1,1)}-K_{(2,2)})/4$, respectively. 
On the other hand, by \eqref{flat_M}, we have that for $i,j\in \{1,2\}$, 
\be\label{eq:est_M-M}
\max_{a,b,a',b'\in \qq D}\left|\left(K_{(i,j)}\right)_{ab}-\left(K_{(i,j)}\right)_{a'b'}\right|=\OO \left({N}/{\|A\|_{\HS}^2}\right).
\ee
With \eqref{eq:est_M-M} and the condition \eqref{eq:condA1}, we obtain that
\be\label{eq:bdd_spec3} N\eta^2 \bigg(\cal K_{aa} - \frac{2}{D} \sum_{b=1}^D \cal K_{ab} + \frac{1}{D^2} \sum_{b,b'=1}^D \cal K_{bb'}\bigg) \lesssim N^{-2\e_A + 2\e_L} .
\ee
Combining \eqref{eq:bdd_spec}, \eqref{eq:bdd_spec2}, and \eqref{eq:bdd_spec3}, we obtain that for any small constant $\e>0$, 
\be\label{eq:bdd_spec4} \E  \max_{i,j\in \qq{k-N^c,k+N^c}}|\bv^*_i (E_{a}-D^{-1}) \bv_j|^2 \le N^{-c_L +\e} +N^{-2\e_A + 2\varepsilon_L+\e}.
\ee
If we take $\varepsilon_L<\e_A/2$ and $\e< (c_L\wedge \e_A)/2$, this gives that 
$$ \E  \max_{i,j\in \qq{k-n^c,k+n^c}}|\bv^*_i (E_{a}-D^{-1}) \bv_j|^2 \le N^{-c_L/2} +N^{-\e_A/2}.$$
Then, applying Markov's inequality and a simple union bound over $a\in \qq D$ concludes \eqref{eq:extend:main_evector1}. Taking $i=j=k$, we obtain \eqref{eq:main_evector1}.
\end{proof}

\subsection{Proof of \Cref{key_lemma_delocalized_eigenvector}}
The remainder of this section is devoted to the proof of \Cref{key_lemma_delocalized_eigenvector}. We begin by introducing the characteristic flow, a key tool that enables the propagation of resolvent bounds from large to small scales in the spectral parameter $\eta$.

\begin{definition}[Characteristic flow]\label{flow} 
Given a starting time $t_0\in \R$ and initial values $(z_{t_0},\Lambda_{t_0})$, we define flows of $z$ and $\Lambda$ as 
\be\label{eq:deterflow} \frac{\dd}{\dd t}z_t=-\frac12 z_t- \langle M_t\rangle,\quad   \frac{\dd}{\dd t}\Lambda_t=-\frac12 \Lambda_t,\quad t\ge t_0, 
\ee
where $M_t:= M(z_t, \Lambda_t)$ is the solution to \eqref{def_M} with $z$ and $\Lambda$ replaced by $z_t$ and $\Lambda_t$. Let $t_c:=\inf\{t\ge t_0: \im z_{t_c}=0\}$ be the first time $\im z_t$ vanishes. We also introduce the function $Z:\C\times \C^{DN\times DN}\to\C^{DN\times DN} $ as $Z(z ,\Lambda ):=z I - \Lambda$ and abbreviate that $Z_t:=Z(z_t , \Lambda_t )$. Note that $Z_t$ satisfies
\be\label{eq:Zt}
 \frac{\dd}{\dd t}Z_t=-\frac12 Z_t- m_t,\quad \text{where}\quad m_t:=\langle M_t\rangle .
\ee
Given the initial random matrix $H_{t_0}$ satisfying \Cref{main_assm} with diagonal blocks $(H_a)_{t_0}$, $a\in \qq D$, we define the flow $H_t$ as a $DN\times DN$ random matrix with diagonal blocks $(H_a)_{t}$ being matrix-valued OU processes 
\be\label{eq:Ht} \dd (H_a)_t = -\frac12(H_a)_t\dd t + \frac{1}{\sqrt{N}}\dd (B_a)_t,\ee
where $(B_a)_t$, $a\in \qq D$, are independent complex Hermitian matrix Brownian motions (i.e., $\sqrt{2} \re (B_a)_{ij}$ and $ \sqrt{2} \im (B_a)_{i j}$, $i<j$, and $(B_a)_{i i}$ are independent standard Brownian motions and $(B_a)_{j i}=(\overline{B}_a)_{i j}$). In particular, for each $t\ge t_0$, $(H_a)_t$ has the same law as  
\be\label{eq:Gauss_div}
e^{-(t-t_0)/2}\cdot  H_a^{(0)}+ \sqrt{1-e^{-(t-t_0)}}\cdot H_a^{(g)},
\ee
where $H_a^{(g)}$, $a\in \qq D$, are i.i.d.~GUE. Then, we define the Green's function flow 
$G_t=\left(H_t+\Lambda_t-z_t\right)^{-1}.$
Finally, with $(z_{i})_t$, $i\in \{1,2,3\}$, $\Lambda_t$, $H_t$, and $M_t$, we can define 
$$ \wh M_{(1,2),t}=\wh M((z_{1})_t,(z_{2})_t,\Lambda_t),\quad L_{(1,2),t}=L((z_{1})_t,(z_{2})_t,H_t, \Lambda_t),\quad K_{(1,2),t}=K((z_{1})_t,(z_{2})_t,\Lambda_t)$$ 
as in \Cref{def_wtM}, and define
$$L_{(1,2,3),t}=L((z_{1})_t, (z_{2})_t, (z_{3})_t, H_t, \Lambda_t),\quad K_{(1,2,3),t}=K((z_{1})_t, (z_{2})_t, (z_{3})_t, \Lambda_t)$$ 
as in \Cref{defLK3}. 
\end{definition}

We now collect some basic properties of the characteristic flow in \Cref{flow}. 
\begin{lemma}
[Basic properties for the flow]
\label{lem_prop_flow} 
Under \Cref{flow}, the following properties hold for $t\in \qa{t_0,t_c}$.

\begin{itemize}
    \item If $t_c-t=\oo(1)$, we have that (recall $m_t$ defined in \eqref{eq:Zt})
    \be\label{t_eta}
t_c-t=\frac{\im z_t}{\im m_t}(1+\oo(1)).
 \ee

 \item $M_t$ satisfies the following equation:
 \be\label{devM}
  \frac{\dd }{\dd t} M( z_t, \Lambda_t)=\frac{1}{2}M( z_t, \Lambda_t).  
\ee
From this equation, we easily see that \begin{equation}\label{im_m_t_sim_through_evolution}
    \begin{aligned}
        \im m_t\sim \im m_{t_0} \quad \text{whenever}\quad t-t_0=\OO\pa{1} \, .
    \end{aligned}
\end{equation}

 \item {\bf Conjugate flow}: We have $Z_t^*=Z(\bar z_t, \Lambda_t)$, $M_t^*=M(\bar z_t, \Lambda_t)$, and $\overline m_t=m_t(\bar z_t,\Lambda_t)$. Moreover, they satisfy the following equations under the conjugate flows $(\bar z_t, \Lambda_t)$:
\be\label{Conflow}
 \frac{\dd}{\dd t}Z(\bar z_t, \Lambda_t )=-\frac12 Z(\bar z_t, \Lambda_t )- m_t(\bar z_t, \Lambda_t ) ,\quad \frac{\dd}{\dd t} M( \bar z_t, \Lambda_t)=\frac{1}{2}M( \bar z_t, \Lambda_t).
\ee

   \item  For any $   (z_{i})_{t}\in \{z_t, \bar z_t\}$, $i\in\{1,2,3\}$, $\wh M_{(1,2),t}$ and $K_{(1,2),t}$ satisfy the equations
   \be\label{eq:whMK}
    \frac{\dd}{\dd t}\wh M_{(1,2),t} =\wh M_{(1,2),t},\quad \frac{\dd}{\dd t} K_{(1,2),t}= \left(K_{(1,2),t}\right)^2+K_{(1,2),t},
    \ee
     and $K_{(1,2,3),t}$ satisfies the following equation for any $a_1,a_2,a_3\in \qq D$: 
    \begin{align}
        \frac{\dd}{\dd t} (K_{(1,2,3),t})_{a_1 a_2 a_3}= \frac{3}{2}K_{(1,2,3),t} + \sum_{a=1}^D &\left[(K_{(1,2),t})_{a_1 a}(K_{(1,2,3),t})_{a a_2 a_3}  +(K_{(2,3),t})_{a_2 a}(K_{(1,2,3),t})_{a_1 a a_3} \right. \nonumber\\
    &\left.+ (K_{(3,1),t})_{a_3 a}(K_{(1,2,3),t})_{a_1 a_2 a}\right] .   \label{eq:whMK3} 
    \end{align} 
 
\end{itemize} 
  
\end{lemma}

\begin{proof}
In the following, we prove only \eqref{t_eta}, as the remaining properties have already been established in \cite[Lemma 4.5]{stone2024randommatrixmodelquantum}. Denoting $\eta_t:=\im z_t$ and $q_t:=\eta_t/\im m_t$, we get from \eqref{eq:deterflow} and \eqref{devM} that
    \begin{equation}\nonumber
        \begin{aligned}
            q_t'=\frac{1}{\pa{\im m_t}^2}\pa{\eta_t'\im m_t-\eta_t \im m_t'}= \frac{1}{\pa{\im m_t}^2}\pa{\pa{-\frac{\eta_t}{2}-\im m_t}\im m_t-\eta_t \frac{\im m_t}{2}}=-q_t-1.
        \end{aligned}
    \end{equation}
    Then, we can derive \eqref{t_eta} by solving this differential equation.
\end{proof}

To prove \Cref{key_lemma_delocalized_eigenvector} for $z=E+\ii \eta$ with $E=\gamma_k$ and $\eta\sim N^{-2/3+\varepsilon_L}k^{-1/3}$, we need to construct a characteristic flow starting at $z_{t_0}$ and terminating at $z_{t_f}=z$. Then, we will establish a sufficiently sharp bound at $z_{t_0}$ and propagate it along the flow to $z_{t_f}=z$. From \eqref{eq:Gauss_div}, propagating bounds along the flow introduces a small GUE component of magnitude \smash{$\sqrt{1-\r e^{t_f-t_0}}\sim \sqrt{t_f-t_0}$}. To get the corresponding result for the original matrix, we invoke a comparison argument. For this purpose, we need the Gaussian component to be small. Consequently, we select $t_f-t_0\sim N^{-\varepsilon_g}$ for some small constant $\varepsilon_g>0$. On the other hand, by \eqref{t_eta}, \eqref{eq:kappa_gamma_k}, and \eqref{square_root_density} below, $t_f$ satisfies 
\[t_c-t_f\sim \eta/\im m (z) \sim \pa{N^{-1/3+\varepsilon_L}k^{-2/3}}\wedge \pa{N^{-1/3+\varepsilon_L/2}k^{-1/6}}\ll N^{-\varepsilon_g},\] 
which yields that $t_c-t_0\sim N^{-\varepsilon_g}$.
We now list the key lemmas leading to the proof of \Cref{key_lemma_delocalized_eigenvector}. We begin with the large $\eta$ estimates in \Cref{lemma_G_large,lemma_G_large_2}. 
In these estimates, $\norm{\cdot}_2$ denotes the $\ell_2$-norm, where matrices and tensors are viewed as vectors (i.e., for matrices, this coincides with the Hilbert–Schmidt norm).

\begin{lemma}\label{lemma_G_large}
    In the setting of \Cref{mix}, let $z=E+\ii \eta \in \bf D(\tau)$ with $\eta\gtrsim N^{-1/3}$ and $\eta/\im m (z) \sim N^{-\varepsilon_g}$ for a small constant $\e_g\in\pa{0,\delta_A/4}$. 
    If $z_1,z_2,z_3\in\ha{z,\bar{z}}$ are not all equal, then we have
    \begin{align}\label{largG_Ent}
            \left\|L_{(1,2)}-K_{(1,2)}\right\|_2 &\prec N^{-1} \eta^{-2} \cdot\Big\|\left(1-\widehat{M}_{\pa{1,2}}\right)^{-1}\Big\|,\\
            \left\|L_{(1,2,3)}-K_{(1,2,3)}\right\|_2 &\prec N^{-1}\eta^{-3}N^{\varepsilon_g}.\label{largG_3G_1}
    \end{align}
    On the other hand, if $z_1=z_2=z_3$, then we have 
    \begin{equation}\label{largG_3G_2}
        \begin{aligned}
            \left\|L_{(1,2,3)}-K_{(1,2,3)}\right\|_2 \prec N^{-1}\eta^{-3}\pa{\frac{1}{\im m (z) }\wedge N^{\varepsilon_g}}.
        \end{aligned}
    \end{equation}
More generally, we can prove the following extension of \eqref{largG_Ent}: given an arbitrary deterministic matrix $B$ with $\|B\|\le 1$, we have 
    \begin{equation}\label{largG_Ent_BB}
        \begin{aligned}
            \avga{G_1E_aG_2B}=\sum_{x=1}^D\pa{1-\wh M_{\pa{1,2}}}^{-1}_{ax}\avga{M_1E_xM_2B}+\opr{N^{-1}\eta^{-2}\cdot\Big\|\pa{1-\wh M_{\pa{1,2}}}^{-1}\Big\|}.
        \end{aligned}
    \end{equation}
\end{lemma}

\begin{proof}
    The proof of \cref{lemma_G_large} follows a similar approach to that in the proof of \cite[Lemma 4.2]{stone2024randommatrixmodelquantum}, with some minor modifications.
    More specifically, the proof of \cite[Lemma 4.2]{stone2024randommatrixmodelquantum} relies on the local laws in \cite[Lemma 2.11]{stone2024randommatrixmodelquantum}, which can be replaced by our estimate \eqref{entprodG} in \Cref{lemma_necessary_estimates} below. Additionally, whenever we need to bound the operator norm of \smash{$\p{1-\wh M_{\pa{1,2}}}^{-1}$}, we will apply \eqref{1-M} and \eqref{1-M-2} from \Cref{lemma_usual_properties_of_m_and_mu}, instead of the bounds in  \cite[Lemma A.1]{stone2024randommatrixmodelquantum}. 
    We omit the details for brevity.
\end{proof}


\begin{lemma}\label{lemma_G_large_2}
In the setting of \Cref{mix}, let $z=E+\ii \eta$ with $\eta\gtrsim N^{-1/3+\tau_e}$ and $\eta/\im m (z) \sim N^{-\varepsilon_g}$ for some small constants $\tau_e,\e_g>0$. If $\varepsilon_g<\p{1/8}\wedge \p{\delta_A/4}$, then for any $z_1,z_2\in\ha{z,\bar{z}}$, we have that
    \begin{equation}\label{largG_1G}
        \begin{aligned}
            \max_{a\in \qq D}\absa{\E\avga{\pa{G (z) -M (z) }E_a}}\prec N^{-1}\pa{\im m (z) }^{-1},
        \end{aligned}
    \end{equation}
    \begin{equation}\label{largG_Ex}
        \begin{aligned}
            \left\|\E L_{(1,2)}-K_{(1,2)}\right\|_2\prec N^{-1}\eta^{-2}\pa{N^{-\tau_e\wedge \varepsilon_g}}.
        \end{aligned}
    \end{equation}
\end{lemma}

The proof of \Cref{lemma_G_large_2} follows a similar approach to that in the proof of \cite[Lemma 4.3]{stone2024randommatrixmodelquantum}. We will outline its proof in \Cref{proof_of_lemma_G_large_2}.
Now, the proof of \Cref{key_lemma_delocalized_eigenvector} will be based on the following key lemmas, Lemmas \ref{lem_flow_1}--\ref{lem_flow_3}, which control the evolutions of the relevant quantities (i.e., $L_{(1,2), t}-K_{(1,2), t}$, $L_{(1,2,3), t}-K_{(1,2,3), t}$, and $\left\langle\left(G_t-M_t\right) E_a\right\rangle$) along the characteristic flow defined in \Cref{flow}.


\begin{lemma}\label{lem_flow_1}
Suppose that $H_{t_0}$ and $ \Lambda_{t_0}$ satisfy the assumptions (for $H$ and $\Lambda$) in \Cref{mix}. Consider the characteristic flow in  \Cref{flow} with $z_{t_0}=E_{t_0}+\ii \eta_{t_0} \in \bf D\pa{\tau}$, where \smash{$\eta_{t_0}\gtrsim N^{-1/3+\tau_e}$} for a constant $\tau_e>0$ and $t_c-t_0 \sim N^{-\epsilon_g}$ for a constant $\varepsilon_g\in \pa{0,\p{1/8}\wedge \p{\delta_A/4}}$. Define 
\[t_m:=\inf\ha{t\geq t_0: N\eta_t\im m\pa{z_t}\leq N^{C_0\varepsilon_g}}\] 
for an absolute constant $C_0>6$. Then, for any $(z_1)_t, (z_2)_t\in \{z_t, \bar z_t\}$ and all $t\in [t_0,t_m]$, we have that
    \begin{equation}\label{flow_1_result}
        \begin{aligned}
            \left\|L_{(1,2), t}-K_{(1,2), t}\right\|_2 \prec \frac{\left(t_c-t_0\right)^2}{\left(t_c-t\right)^2}\left\|L_{(1,2), t_0}-K_{(1,2), t_0}\right\|_2+\frac{1}{N\left(t_c-t\right)^2\left(\operatorname{Im} m_t\right)^2}.
        \end{aligned}
    \end{equation}
Combining this with \eqref{t_eta}, \eqref{im_m_t_sim_through_evolution}, and \eqref{largG_Ent}, and applying the estimates \eqref{1-M} and \eqref{1-M-2} below to bound \smash{$\|\p{1-\wh M_{\pa{1,2}}}^{-1}\|$}, we obtain that for all $t\in\qa{t_0,t_m}$,
    \begin{equation}\label{flow_1_result_2}
        \begin{aligned}
            \left\|L_{(1,2), t}-K_{(1,2), t}\right\|_2 \prec \frac{N^{\epsilon_g}}{N\left(t_c-t\right)^2\left(\operatorname{Im} m_t\right)^2}\ .
        \end{aligned}
    \end{equation}
\end{lemma}

\begin{lemma}\label{lem_flow_2}
    Under the assumptions of \Cref{lem_flow_1}, let $\left(z_1\right)_t,\left(z_2\right)_t,\left(z_3\right)_t \in$ $\left\{z_t, \bar{z}_t\right\}$ for $t\in \qa{t_0,t_c}$. Then, we have that for all $t\in [t_0,t_m]$,
    \begin{equation}\label{flow_3G_res}
        \begin{aligned}
            \left\|L_{(1,2,3), t}-K_{(1,2,3), t}\right\|_2 \prec \frac{\left(t_c-t_0\right)^3}{\left(t_c-t\right)^3}\left\|L_{(1,2,3), t_0}-K_{(1,2,3), t_0}\right\|_2+\frac{N^{\epsilon_g}}{N\left(t_c-t\right)^3\left(\operatorname{Im} m_t\right)^3}.
        \end{aligned}
    \end{equation}
Combining this with \eqref{t_eta}, \eqref{im_m_t_sim_through_evolution}, \eqref{largG_3G_1}, and \eqref{largG_3G_2}, we obtain that for all $t\in\qa{t_0,t_m}$,
    \begin{equation}\label{flow_3G_res_2}
        \begin{aligned}
            \left\|L_{(1,2,3), t}-K_{(1,2,3), t}\right\|_2 \prec \frac{N^{\epsilon_g}}{N\left(t_c-t\right)^3\left(\operatorname{Im} m_t\right)^3} \ .
        \end{aligned}
    \end{equation}
\end{lemma}

\begin{lemma}\label{lem_flow_1.5}
    Under the assumptions of \Cref{lem_flow_1}, we have that for all $t\in [t_0,t_m]$,
    \begin{equation}\label{flow_1G_res}
        \begin{aligned}
            \max _{a \in \llbracket D \rrbracket}\left|\mathbb{E}\left\langle\left(G_t-M_t\right) E_a\right\rangle\right| \prec \frac{t_c-t_0}{t_c-t} \max _{a \in \llbracket D \rrbracket}\left|\mathbb{E}\left\langle\left(G_{t_0}-M_{t_0}\right) E_a\right\rangle\right|+\frac{N^{\epsilon_g}}{N^2\left(t_c-t\right)^2\left(\operatorname{Im} m_t\right)^3}.
        \end{aligned}
    \end{equation}
Combining this with \eqref{t_eta}, \eqref{im_m_t_sim_through_evolution}, and \eqref{largG_1G}, and using the definition of $t_m$, we obtain that for all $t\in\qa{t_0,t_m}$,
    \begin{equation}\label{flow_1G_res_2}
        \begin{aligned}
            \max _{a \in \llbracket D \rrbracket}\left|\mathbb{E}\left\langle\left(G_t-M_t\right) E_a\right\rangle\right| \prec& \frac{N^{-\varepsilon_g}}{N\pa{t_c-t}\im m_t}+\frac{N^{\epsilon_g}}{N^2\left(t_c-t\right)^2\left(\operatorname{Im} m_t\right)^3}\sim \frac{N^{-\varepsilon_g}}{N\pa{t_c-t}\im m_t}.
        \end{aligned}
    \end{equation}
\end{lemma}

\begin{lemma}\label{lem_flow_3}
    Under the assumptions of \Cref{lem_flow_1}, we have that for all $t\in [t_0,t_m]$,
    \begin{equation}\label{flow_3_result}
        \begin{aligned}
            \left\|\mathbb{E} L_{(1,2), t}-K_{(1,2), t}\right\|_2 \prec&~ \frac{\left(t_c-t_0\right)^2}{\left(t_c-t\right)^2}\left\|\mathbb{E} L_{(1,2), t_0}-K_{(1,2), t_0}\right\|_2\\
            &~+\frac{N^{-\epsilon_g}}{N\left(t_c-t\right)^2\left(\operatorname{Im} m_t\right)^2}+\frac{N^{2 \epsilon_g}}{N^2\left(t_c-t\right)^3\left(\operatorname{Im} m_t\right)^4}.
        \end{aligned}
    \end{equation}
    Combining this with \eqref{t_eta}, \eqref{im_m_t_sim_through_evolution}, and \eqref{largG_Ex}, and using the definition of $t_m$, we obtain that for all $t\in\qa{t_0,t_m}$,
    \begin{equation}\label{flow_3_result_2}
        \begin{aligned}
            \left\|\mathbb{E} L_{(1,2), t}-K_{(1,2), t}\right\|_2\prec&~ \frac{N^{-\tau_e\wedge\epsilon_g}}{N\left(t_c-t\right)^2\left(\operatorname{Im} m_t\right)^2}+\frac{N^{-\epsilon_g}}{N\left(t_c-t\right)^2\left(\operatorname{Im} m_t\right)^2}+\frac{N^{2 \epsilon_g}}{N^2\left(t_c-t\right)^3\left(\operatorname{Im} m_t\right)^4}\\
            \sim&~ \frac{N^{-\tau_e\wedge\epsilon_g}}{N\left(t_c-t\right)^2\left(\operatorname{Im} m_t\right)^2}.
        \end{aligned}
    \end{equation}
\end{lemma}

The proof of the above lemmas, Lemmas \ref{lem_flow_1}--\ref{lem_flow_3}, will be described in \Cref{proof_of_lem_flow}. 
With these lemmas, we immediately obtain \Cref{key_lemma_delocalized_eigenvector} for matrices with small Gaussian components, specifically the \emph{Gaussian divisible matrices}. This is stated as the following lemma.

\begin{lemma}\label{main_lemma_G}
In the setting of \Cref{mix}, suppose $H_a$, $a\in \qq D$, take the form
\be\label{Gcom_a}
H_a = \sqrt{1-N^{-\e_g}}\cdot H_a^{(0)}+ {N^{-\e_g/2}} H_a^{(g)},
\ee
where $H_a^{(0)}$ are independent Wigner matrices satisfying the assumptions for $H_a$ in \Cref{main_assm}, and $H_a^{(g)}$ are i.i.d.~GUE satisfying \eqref{eq:meanvar} and \eqref{eq:meanvar2}. Then, for small enough constant $\e_g>0$ (depending on $\delta_A$ and $\e_A$) and $z=E+\ii\eta $ with $E=\gamma_k$ for some $k\leq DN/2$, there exists an absolute constant $C>8\vee C_0$ such that 
\be\label{main_G_es}
\|\E L_{(1,2)}- K_{(1,2)}\|_{2}\prec N^{-1-\e_g}\eta^{-2},\quad \forall \  N^{-2/3+C\e_g }k^{-1/3} \le \eta \le N^{-C\e_g }\ , \ z_1,z_2\in \{z, \bar z\}.
\ee
\end{lemma}

\begin{proof}
For $z=E+\ii \eta$ with $E=\gamma_k$ and $N^{-2/3+C\varepsilon_g}k^{-1/3}\le \eta \le N^{-C\varepsilon_g}$, by \eqref{eq:kappa_gamma_k} and \eqref{square_root_density} below, we have that $\im m (z) \sim \sqrt{\kappa+\eta}$ and $\kappa\sim k^{2/3}/N^{2/3}$. Let $t_f=0$ and $t_0=t_f- N^{-\e_g}/2$. Then, we can find initial values $z_{t_0}$ and $\Lambda_{t_0}$ such that $z_{t_f}=z$ and $\Lambda_{t_f}=\Lambda$ at $t=t_f$. In fact, we can first solve the second equation in \eqref{eq:deterflow} as \smash{$\Lambda_t=e^{(t_f-t)/2}\Lambda$} and then plug it into the first equation in \eqref{eq:deterflow}. In the resulting equation, the RHS is a locally Lipschitz function in $t$ and $z$, so there exists a solution $z_{t_0}$ at $t=t_0$.

At $t=t_f$, we have $\im m_{t_f}(z_{t_f})=\im m (z) \sim \sqrt{\kappa+\eta}$. Thus, by \eqref{t_eta}, we know that 
\[t_c-t_f\sim \eta/\im m_{t_f}\pa{z_{t_f}}\lesssim \sqrt{\eta}\le N^{-C\varepsilon_g/2} ,\] 
which also implies that $t_c-t_0=(t_f-t_0)(1+\oo(1))=N^{-\e_g}(1/2+\oo(1))$. Using \eqref{t_eta} again and the fact that $\im m_{t_0}\pa{z_{t_0}}\sim \im m_{t_f}\pa{t_f}=\im m (z) $ by \eqref{im_m_t_sim_through_evolution}, we get
    \begin{equation}\nonumber
        \begin{aligned}
            \eta_{t_0}\sim N^{-\varepsilon_g}\im m (z) \gtrsim N^{-\varepsilon_g}\sqrt{k^{2/3}/N^{2/3}+N^{-2/3+C\varepsilon_g}k^{-1/3}}\gtrsim N^{-1/3+\pa{C/3-1}\varepsilon_g}.
        \end{aligned}
    \end{equation}
 Since $C>8$, this implies $\eta_{t_0}\gtrsim N^{-1/3+\varepsilon_g}$.

To complete the proof using \eqref{flow_3_result_2} from \Cref{lem_flow_3}, we only need to check that $t_f\leq t_m$. It suffices to prove that $N\eta_t\im_t\pa{z_t}> N^{C\varepsilon_g}$ for all $t\in \qa{t_0,t_f}$. In fact, by \eqref{t_eta}, \eqref{im_m_t_sim_through_evolution}, and \eqref{square_root_density}, we have that 
    \begin{equation}\nonumber
        \begin{aligned}
            N\eta_t\im m_t\pa{z_t}\gtrsim& N\pa{t_c-t}\pa{\im m_t\pa{z_t}}^2\gtrsim N\pa{t_c-t_f}\pa{\im m_{t_f}\pa{z_{t_f}}}^2\sim N \eta\im m (z) \\
            \gtrsim& N\cdot \pa{N^{-2/3+C\varepsilon_g}k^{-1/3}}\cdot \pa{k^{1/3}/N^{1/3}}=N^{C\varepsilon_g}\gg N^{C_0\varepsilon_g},\quad \forall t\in\qa{t_0,t_f}.
        \end{aligned}
    \end{equation}
    Thus, we conclude that $t_f\leq t_m$, thereby completing the proof of \Cref{main_lemma_G} using \Cref{lem_flow_3}. 
\end{proof}


With \Cref{main_lemma_G}, we can employ a standard Green's function comparison argument to deduce \eqref{main_EGG} for the original model, as shown by the following lemma. 
The proof of \Cref{main_lemma_com} is identical to that of \cite[Lemma 3.4]{stone2024randommatrixmodelquantum} and is therefore omitted here.

\begin{lemma}\label{main_lemma_com}
Let $H$ and $\wt H$ be two matrices satisfying \Cref{main_assm}. Suppose they satisfy the following moment-matching conditions: for $i,j\in \cal I$ and integers $l,\kk\ge 0$,  
\be\label{comH0}
\E (H_{ij})^l (H^*_{ij})^{\kk} - \E (\wt H_{ij})^l (\wt H^*_{ij})^{\kk}=0 \ \ \text{for} \ \ l+\kk\le 3, 
\ee
and there exists a constant $\delta \in (0,1/2)$ such that 
\be\label{comH}
\left|\E (H_{ij})^l (H^*_{ij})^{\kk} - \E (\wt H_{ij})^l (\wt H^*_{ij})^{\kk}\right| \lesssim N^{-2-\delta}  \ \ \text{for} \ \ l+\kk= 4. 
\ee
Then, for any $z\in \mathbf D(\tau)$, $z_1,z_2\in \{z, \bar z\}$, and $a,b \in \llbracket D\rrbracket$, we have
\be\label{eq:cpmparison}
\mathbb E \langle G_1E_{a}G_2E_{b}\rangle-\mathbb E \langle \wt G_1E_{a}\wt G_2E_{b}\rangle\prec N^{-1-\delta}\eta^{-2},
\ee
where $\wt G_i \equiv G(z_i,\wt H,\Lambda)$, $i\in \{1,2\}$, denote the Green's functions of $\wt H$.
\end{lemma}


\begin{proof}[\bf Proof of \Cref{key_lemma_delocalized_eigenvector}]
     Given the matrix $H$ considered in \Cref{key_lemma_delocalized_eigenvector}, we can construct another random matrix \smash{$\wt H$} satisfying the setting in \Cref{main_lemma_G} and such that the moment-matching conditions \eqref{comH0} and \eqref{comH} hold with $\delta=\e_g$ (see e.g., Lemma 6.5 in \cite{erdHos2012bulk}).  
     By \Cref{main_lemma_G}, as long as we choose $\e_g$ small enough such that $C\e_g \le \varepsilon_L \le 2/3-C\e_g$, there is 
     $$D\mathbb E \langle \wt G_1E_{a}\wt G_2E_{b}\rangle- (K_{(1,2)})_{ab}
     \prec N^{-1-\e_g}\eta^{-2},$$
     for $\eta = N^{-2/3+ \varepsilon_L} k^{-1/3}$. On the other hand, by \Cref{main_lemma_com}, we have that 
     $$\mathbb E \langle  G_1E_{a} G_2E_{b}\rangle- \mathbb E \langle \wt G_1E_{a}\wt G_2E_{b}\rangle \prec N^{-1-\e_g}\eta^{-2}.$$
     Combining the above two estimates, we conclude \Cref{key_lemma_delocalized_eigenvector} by choosing $c_L=\e_g$.
\end{proof}

\subsection{Proof of \Cref{lemma_G_large_2}}\label{proof_of_lemma_G_large_2}

For any $z_1, z_2\in \{z, \overline{z}\}$, we abbreviate that 
$$\wh M\equiv \wh M_{(1,2)},\quad L\equiv L_{(1,2)},\quad K\equiv K_{(1,2)},\quad \text{and}\quad \wt M\equiv \wh M_{(2,1)},\quad \wt L\equiv L_{(2,1)}, \quad \wt K\equiv K_{(2,1)}.$$
Moreover, given any deterministic matrix $B\in \C^{DN\times DN}$, we denote 
$$L_{ab}(B):=D\langle G_1 E_a G_2 E_b B\rangle,\quad K_{ab}(B):=\sum_x (1-\wh M)^{-1}_{ax}D\langle M_1 E_x M_2 E_b B\rangle. $$
Similarly, we define \smash{$\wt L_{ab}(B)$ and $\wt K_{ab}(B)$} by exchanging $1$ and $2$. Applying the identity 
\begin{equation}\nonumber
    \begin{aligned}
        G - M = -M(m + H)G = -M\underline{HG} + M(\widetilde{\E}[\widetilde{H}G\widetilde{H}] - m)G
    \end{aligned}
\end{equation}
to $G_2$ in $L_{ab}=D\avga{G_1E_aG_2E_b}$ and using the notation in \Cref{def:underline}, we can show that
\begin{align}
\begin{split}
L_{ab} = &\,D\langle G_1E_aM_2E_b\rangle - D\langle \underline{G_1E_aM_2HG_2E_b}\rangle \\
&+ D \sum_{x=1}^D \langle G_1E_aM_2E_x\rangle L_{xb} + D^2\sum_{x=1}^D \langle (G_2 - M_2) E_x\rangle \langle G_1E_aM_2E_xG_2E_b\rangle
\end{split}\label{eq:aftertrace}
\end{align}
through a direct computation. Taking the expectation on both sides of \eqref{eq:aftertrace}, we obtain that
\begin{align}
\E L_{ab} &= \wh{M}_{ab} + D\E\langle (G_1-M_1)E_aM_2E_b\rangle - D\E\langle \underline{G_1E_aM_2HG_2E_b}\rangle + \sum_{x=1}^D \wh M_{ax} \E L_{xb} \nonumber\\
&\quad + D\sum_{x=1}^D \E\langle (G_1 - M_1)E_aM_2E_x\rangle L_{xb} + D^2\sum_{x=1}^D \E\langle (G_2-M_2) E_x\rangle \langle G_1E_aM_2E_xG_2E_b\rangle \nonumber\\
&= \wh{M}_{ab} + D\E\langle (G_1-M_1)E_aM_2E_b\rangle - D\E\langle \underline{G_1E_aM_2HG_2E_b}\rangle + \sum_{x=1}^D \wh M_{ax} \E L_{xb} \nonumber\\
&\quad + D\sum_{x=1}^D \E\langle (G_1 - M_1)E_aM_2E_x\rangle K_{xb} + D\sum_{x=1}^D \E\langle (G_2-M_2) E_x\rangle \wt K_{ba}(M_2E_x)\nonumber\\
&\quad+\OO_\prec \pa{N^{-2}\eta^{-3}\big\|\big({1-\wh M}\big)^{-1}\big\|},\label{ELK}
\end{align}
where we used the averaged local law \eqref{eq:aver_local} and the two-resolvents local laws \eqref{largG_Ent} and \eqref{largG_Ent_BB} in the above derivation. Now, the proof of \Cref{lemma_G_large_2} is based on \eqref{ELK} and the following two lemmas. The proofs of \Cref{lem:fourthcumu1,lem:fourthcumu2} are nearly the same as those of \cite[Lemmas 4.13 and 4.14]{stone2024randommatrixmodelquantum}. 
More precisely, as explained in the proof of \Cref{lemma_G_large}, we will use \eqref{entprodG} to replace the resolvent estimates in \cite[Lemma 2.11]{stone2024randommatrixmodelquantum} and use \eqref{1-M} and \eqref{1-M-2} instead of those in \cite[Lemma A.1]{stone2024randommatrixmodelquantum} to bound the operator norm of \smash{$\p{1-\wh M}^{-1}$}. Hence, we again omit further details.


\begin{lemma}\label{lem:fourthcumu1}
In the setting of \Cref{lemma_G_large_2}, we have that     
 \begin{align}
&- D\E\langle \underline{G_1E_aM_2HG_2E_b}\rangle=\opr{N^{-3/2}\eta^{-2} +N^{-2}\eta^{-2}\big\|\big({1-\wh M}\big)^{-1}\big\|} \nonumber\\ &+\frac{D\kappa^{(2, 2)}}{N}\sum_{x=1}^D \left[  { \langle \diag(M_2)^2E_x\rangle }\wt K_{ba}(M_2\diag(M_2)E_x) +\langle M_1 \diag(M_2)E_x\rangle  \wt K_{bx}(\diag(M_1E_aM_2))\right]\nonumber\\
&+\frac{D\kappa^{(2, 2)}}{N}\sum_{x=1}^D   \left[\langle M_1E_aM_2 \diag(M_1)E_x\rangle  \wt K_{bx}(\diag(M_1))+\langle M_1E_aM_2 \diag(M_2)E_x\rangle  \wt K_{bx}(\diag(M_2))\right]  ,\label{eq:G1HG2}
    \end{align}
where $\kappa^{(2, 2)}$ is the normalized $(2,2)$-cumulant of $h_{12}$, defined as $\kappa^{(2, 2)}:=N^2\cal C_{12}^{(2,2)}$, and $\diag(B)$ is the diagonal matrix consisting of the diagonal entries of the given matrix $B$.  
\end{lemma}

\begin{lemma}\label{lem:fourthcumu2}
In the setting of \Cref{lemma_G_large_2}, let $B$ be an arbitrary deterministic matrix with $\|B\|\le 1$. Then, we have that
    \be\label{eq:Exp_G1-M1}
    \begin{aligned}
    \E \langle(G_1 - M_1)B\rangle 
    =&~\frac{\kappa^{(2,2)}  \langle \diag(M_1)^2 \rangle }{N}\left[ \langle M_1BM_1\diag(M_1)\rangle   +   \frac{ \langle M_1^2 \diag(M_1)\rangle}{1-\langle M_1^2\rangle } \langle M_1^2B\rangle\right]\\
    &~ + \OO_\prec\qa{\pa{\frac{1}{\im m (z) }\wedge N^{\varepsilon_g}}\cdot{\pa{ N^{-3/2}\eta^{-1}+N^{-2}\eta^{-2}}}}.
    \end{aligned}
    \ee
\end{lemma}

We abbreviate $M=M (z) $ and $m=m (z) $ in the following derivation. By \eqref{1-M-3} below, we have that
\begin{equation}\nonumber
    \begin{aligned}
        \absa{1-\avga{M_1^2}}^{-1}\lesssim \pa{\im m}^{-1}.
    \end{aligned}
\end{equation}
Then, we get from \eqref{eq:Exp_G1-M1} that
\begin{equation}\label{proof_largG_1G}
    \begin{aligned}
        \absa{\E \langle(G_1 - M_1)B\rangle }\prec \pa{\im m}^{-1}\cdot  \pa{N^{-1}+\eta^{-1}N^{-3/2}+\eta^{-2}N^{-2}}\sim N^{-1}\pa{\im m}^{-1},
    \end{aligned}
\end{equation}
which gives \eqref{largG_1G}. It remains to show \eqref{largG_Ex}.

We first consider the case $z_1=z_2\in \ha{z,\bar{z}}$. Applying \eqref{1-M-2}, \eqref{eq:G1HG2}, and \eqref{proof_largG_1G} to \eqref{ELK}, we get that
\begin{equation}\nonumber
    \begin{aligned}
        \E L_{ab} &= \wh{M}_{ab} + \sum_{x=1}^D \wh M_{ax} \E L_{xb}+ \opr{N^{-1}\pa{\im m}^{-2}+N^{-2+\varepsilon_g}\eta^{-3}+N^{-3/2}\eta^{-2}}.
    \end{aligned}
\end{equation}
Solving for $\E L_{ab}$ and using \eqref{1-M-2} again, we obtain that
\begin{equation}\nonumber
    \begin{aligned}
        \E L_{ab}= K_{ab}+\opr{N^{-1+\varepsilon_g}\pa{\im m}^{-2}+N^{-2+2\varepsilon_g}\eta^{-3}+N^{-3/2+\varepsilon_g}\eta^{-2}}
        = K_{ab}+\opr{N^{-1-\varepsilon_g}\eta^{-2}}.
    \end{aligned}
\end{equation}

Next, we consider the case $z_1=\bar{z}_2\in \ha{z,\bar{z}}$. Without loss of generality, suppose that $z_1=\bar{z}_2=z$. Plugging \eqref{eq:G1HG2} and \eqref{eq:Exp_G1-M1} back into \eqref{ELK}, using \eqref{proof_largG_1G} to control the term $D\E\avga{\pa{G_1-M_1}E_aM_2E_b}$, and applying \eqref{1-M} to bound the operator norm of \smash{$\p{1-\wh M}^{-1}$}, we obtain that
\begin{align}
\E L_{ab}
&= \wh{M}_{ab}+ \sum_{x=1}^D \wh M_{ax} \E L_{xb} +\OO_\prec \pa{N^{-2}\eta^{-4}\im m+N^{-1}\pa{\im m}^{-1}+N^{-3/2}\eta^{-2}}\nonumber\\
&\quad+\frac{D\kappa^{(2, 2)}}{N}\sum_{x=1}^D \left[  { \langle \diag(M_2)^2E_x\rangle }\wt K_{ba}(M_2\diag(M_2)E_x) +\langle M_1 \diag(M_2)E_x\rangle  \wt K_{bx}(\diag(M_1E_aM_2))\right]\nonumber\\
&\quad+\frac{D\kappa^{(2, 2)}}{N}\sum_{x=1}^D   \left[\langle M_1E_aM_2 \diag(M_1)E_x\rangle  \wt K_{bx}(\diag(M_1))+\langle M_1E_aM_2 \diag(M_2)E_x\rangle  \wt K_{bx}(\diag(M_2))\right] \nonumber\\
&\quad+\frac{D\kappa^{(2,2)}  \langle \diag(M_1)^2 \rangle }{N}\sum_{x=1}^D\left[ \langle M_1E_aM_2E_xM_1\diag(M_1)\rangle   +   \frac{ \langle M_1^2 \diag(M_1)\rangle}{1-\langle M_1^2\rangle } \langle M_1^2E_aM_2E_x\rangle\right]K_{xb}\nonumber\\
&\quad+\frac{D\kappa^{(2,2)}  \langle \diag(M_2)^2 \rangle }{N}\sum_{x=1}^D\left[ \langle M_2E_xM_2\diag(M_2)\rangle   +   \frac{ \langle M_2^2 \diag(M_2)\rangle}{1-\langle M_2^2\rangle } \langle M_2^2E_x\rangle\right]\wt K_{ba}(M_2E_x).
\end{align}
To simplify the expression, we first replace all $M_i$, $i\in\ha{1,2}$, in the second and third lines and all $\mathrm{diag}\pa{M_i}$, $i\in\ha{1,2}$, in the last two lines with $m_i=m(z_i),$ up to an error of order $\OO\pa{N^{-\delta_A/2}}$ by the estimate \eqref{eq:msc} below. This leads to that:

\begin{align}
    \mathbb{E} L_{a b}
    = &~ \widehat{M}_{a b}+\sum_{x=1}^D \widehat{M}_{a x} \mathbb{E} L_{x b}+\opr{N^{-\delta_A / 2}(N \eta)^{-1}+N^{-1}(\operatorname{Im} m)^{-1}+N^{-3 / 2}\eta^{-2} +N^{-2}\eta^{-4} \im m} \nonumber\\
    &~ +\frac{\kappa^{(2,2)}}{N}\left[{\overline{m}}^4+\left|m\right|^4+\left|m\right|^2 m^2+\left|m\right|^2{\overline{m}}^2\right] K_{a b} +\frac{\kappa^{(2,2)}}{N}\left[\frac{m^4\left|m\right|^2}{1-\left\langle M^2\right\rangle}+\frac{{\overline{m}}^6}{1-\langle\left(M^*\right)^2\rangle}\right] K_{a b}\label{pre_cancellation},
\end{align}
where we also used the bounds \eqref{1-M}, \eqref{1-M-2}, \eqref{1-M-3}, and \eqref{proof_largG_1G} in the above derivation. By \eqref{equation_z_im_m} and \eqref{eq:msc}, we have that
\begin{equation}\label{estimate_1_minus_abs_m_2}
    \begin{aligned}
        1-\left\langle M^* M\right\rangle=1-|m|^2+\OO\big(N^{-\delta_A / 2}\big) ,\quad \text{and}\quad 1-\left\langle M^* M\right\rangle=\frac{\eta}{\eta+\operatorname{Im} m} \sim \frac{\eta}{\operatorname{Im} m}.
    \end{aligned}
\end{equation}
Together with the assumption $\eta / \operatorname{Im} m \sim N^{-\epsilon_g} \gg N^{-\delta_A / 2}$, it implies that $1-|m|^2=\pa{1+\oo\pa{1}}\pa{1-\avga{MM^*}} \sim {\eta}/{\operatorname{Im} m}$. With \eqref{eq:msc}, \eqref{1-M-3}, and \eqref{estimate_1_minus_abs_m_2}, we then obtain that
\begin{align*}
        &~\bar{m}^4 +|m|^4+|m|^2 m^2+|m|^2 \bar{m}^2+\frac{m^4|m|^2}{1-\left\langle M^2\right\rangle}+\frac{\bar{m}^6}{1-\langle\left(M^*\right)^2\rangle} \\
        =&~ 1+m^2+\frac{m^4}{1-\left\langle M^2\right\rangle}+\bar{m}^2+\bar{m}^4+\frac{\bar{m}^6}{1-\langle\left(M^*\right)^2\rangle}+\OO\left(\frac{\eta}{(\operatorname{Im} m)^2}\right) \\
        =&~ \frac{\bar{m}^4}{1-\left\langle M^2\right\rangle}+\frac{\bar{m}^6}{1-\langle\left(M^*\right)^2\rangle}+\OO\left(\frac{\eta}{(\operatorname{Im} m)^2}+\frac{N^{-\delta_A / 2}}{\operatorname{Im} m}\right) \\
        =&~ \bar{m}^4\frac{\left(1-|m|^2\right)\left(1+|m|^2\right)}{\left(1-\left\langle M^2\right\rangle\right)(1-\langle\left(M^*\right)^2\rangle)}+\OO\left(\frac{\eta}{(\operatorname{Im} m)^2}+\frac{N^{-\delta_A / 2}}{(\operatorname{Im} m)^2}\right) = \OO\left(\frac{\eta}{(\operatorname{Im} m)^3}+\frac{N^{-\delta_A / 2}}{(\operatorname{Im} m)^2}\right).
    \end{align*}
Plugging this back into \eqref{pre_cancellation} and using $\absa{K_{ab}}\lesssim \im m/\eta$ by \eqref{1-M}, we get that
\begin{equation}\nonumber
    \begin{aligned}
        \mathbb{E} L_{a b}=&\widehat{M}_{a b}+\sum_{x=1}^D \widehat{M}_{a x} \mathbb{E} L_{x b} +\OO_{\prec}\left(N^{-1}(\operatorname{Im} m)^{-2}+N^{-\delta_A / 2}(N \eta)^{-1}(\operatorname{Im} m)^{-1}+ N^{-3 / 2}\eta^{-2}+N^{-2}\eta^{-4} \im m\right) .
    \end{aligned}
\end{equation}
Solving for $\mathbb{E} L_{(1,2)}$ and using \eqref{1-M} again, we obtain that
\begin{equation}\nonumber
    \begin{aligned}
        \mathbb{E} L_{a b}=K_{a b}+\OO_{\prec}\left(\p{N \eta}^{-1}(\operatorname{Im} m)^{-1}+N^{-\delta_A / 2} N^{-1} \eta^{-2}+N^{-3 / 2} \eta^{-3} \operatorname{Im} m+N^{-2} \eta^{-5} \pa{\operatorname{Im} m}^2\right).
    \end{aligned}
\end{equation}
This completes the proof of \eqref{largG_Ex} for the case $z_1=\bar{z}_2\in\ha{z,\bar{z}}$ by using $\eta / \operatorname{Im} m \sim N^{-\epsilon_g}$ and $\eta \gtrsim N^{-1 / 3+\tau_e}$.

\subsection{Proof of \Cref{lem_flow_1,lem_flow_2,lem_flow_1.5,lem_flow_3}}\label{proof_of_lem_flow}

In this subsection, we present the proofs of \Cref{lem_flow_1,lem_flow_2,lem_flow_1.5,lem_flow_3}, based on an extension of the arguments used in the proofs for \cite[Lemmas 4.6--4.9]{stone2024randommatrixmodelquantum}.
Since the proofs of these lemmas share a similar structure, we provide a detailed proof only for \Cref{lem_flow_1} to avoid redundancy. The remaining three lemmas follow from analogous---and in some cases simpler---adaptations of the corresponding arguments in \cite{stone2024randommatrixmodelquantum}.

Let $\mathbf{B}_t=\left(b_{i j}(t)\right)_{i,j\in \cI}$ be a $D  \times D $ block matrix Brownian motion consisting of the diagonal blocks $(B_a)_t$ in \eqref{eq:Ht}. Then, by \eqref{eq:Ht}, $H_t=(h_{ij}(t))_{i,j\in \cI}$ satisfies the equation 
$$
\dd h_{i j}=-\frac{1}{2} h_{i j} \dd t+\frac{1}{\sqrt{N}} \dd b_{i j}(t) ,
$$
with initial data $H_{t_0}$. Let $F$ be any function of $t$ and $H$ with continuous second-order derivatives. Then, by It{\^o}'s formula, we have that
\be\label{Ito}
\dd F=\partial_t F \dd t+ \sum_{a=1}^D \sum_{l,\kk \in \mathcal{I}_a} \partial_{h_{l\kk}} F \dd h_{l\kk}+\frac{1}{2 N} \sum_{a=1}^D \sum_{l,\kk \in \mathcal{I}_a} \partial_{h_{l\kk}}\partial_{ h_{\kk l}}  F \dd t.
\ee
We will apply this equation to functions of the resolvents $G_{i,t}\equiv (G_i)_t =(H_t-Z_{i,t})^{-1}$ with $Z_{i,t}=(z_{i})_{t}-\Lambda_t$ for $z_{i}\in \{z, \bar z\}$. 
Using the formula (with the simplified notation $\partial_{l\kk}\equiv \partial_{h_{l\kk}}$)
\begin{equation} \label{Ito_pro_1}
        \partial_{l_1\kk_1} \left(G_{i,t}\right)_{l_2\kk_2}= -\left(G_{i,t}\right)_{l_2l_1}\left(G_{i,t}\right)_{\kk_1\kk_2},\quad l_2,\kk_2 \in \cI, \ l_1,\kk_1\in \cal I_a, \ a\in \qq D,
\end{equation}
we can easily obtain the following identities (with $M_{i,t}\equiv (M_i)_t$):  
\begin{align} \label{Ito_pro_2}
&        \partial_{t}  G_{i,t}= G_{i,t}  \left(\frac{\dd}{\dd t}Z _{i,t}\right) G_{i,t},\quad \text{with}\quad \frac{\dd}{\dd t}Z_{i,t}=-\frac12 Z_{i,t}-\langle M_{i,t}\rangle; \\
 &\sum_{a=1}^D \sum_{l,\kk \in \mathcal{I}_a}  h_{l\kk}\partial_{l\kk} G_{i,t}  =-G_{i,t}H_t G_{i,t}=-G_{i,t}-G_{i,t}Z_{i,t}G_{i,t}; \label{Ito_pro_3}\\
 &       \sum_{l,\kk \in \mathcal{I}_a} \partial_{l\kk} \left(G_{i,t}\right)_{l_1\kk_1} \cdot \partial_{\kk l} \left(G_{i',t}\right)_{l_2\kk_2}  = \left(G_{i,t}E_{a}G_{i',t}\right)_{ l_1\kk_2} \left(G_{i',t}E_{a}G_{i,t}\right)_{l_2\kk_1},\quad l_1,\kk_1,l_2,\kk_2 \in \cI. \label{Ito_pro_4}
\end{align}

\begin{proof}[\bf Proof of \Cref{lem_flow_1}]
    For simplicity of notations, we abbreviate $\wh M_{(1,2),t}$, $L_{(1,2),t}$, and $K_{(1,2),t}$ as $\wh M_{t}$, $L_t$, and $K_t$, respectively. Moreover, we denote $z_t=E_t+\ii \eta_t$ and
\be\label{eq:wtL12t}
\wt L_t\equiv \wt L_{(1,2),t}:= (t_c-t) L_t, \quad \wt K_t\equiv \wt K_{(1,2),t}:= (t_c-t) K_t \, .
\ee
Using Itô's formula \eqref{Ito} and the identities \eqref{Ito_pro_1}--\eqref{Ito_pro_4}, we can calculate that for $ x,y \in \qq D$, 
\begin{align*}
\dd (\widetilde{L}_{t})_{xy} &= - (L_t)_{xy} \dd t  + \frac{1}{\sqrt{N}} \sum_{a=1}^D \sum_{ l,\kk \in \mathcal{I}_a} \partial_{ l\kk} (\widetilde{L}_{t})_{x y} \dd b_{ l\kk} + D(t_c-t)\left\langle G_{1,t} E_x G_{2,t} E_y\right\rangle \dd t \\
&+ D^2(t_c-t) \sum_{a=1}^D\left\langle G_{1,t} E_x G_{2,t} E_a\right\rangle\left\langle G_{2,t} E_y G_{1,t} E_a\right\rangle \dd t  \\
& + D^2(t_c-t) \sum_{a=1}^D
 \left\langle \left(G_{1,t}-M_{1,t} \right) E_a\right\rangle\left\langle G_{1,t} E_x G_{2,t} E_y G_{1,t} E_a\right\rangle \dd t\\
 & + D^2(t_c-t) \sum_{a=1}^D \left\langle \left(G_{2,t}-M_{2,t} \right) E_a\right\rangle\left\langle G_{2,t} E_y G_{1,t} E_x G_{2,t} E_a\right\rangle \dd t  .
 \end{align*}
 Using the definitions of $\wt L_t$ and $L_{(1,2,3),t}$, we can rewrite the above equation as 
 \begin{align}
\dd (\widetilde{L}_{t})_{xy}&= \frac{1}{\sqrt{N}} \sum_{a=1}^D \sum_{ l,\kk \in \mathcal{I}_a} \partial_{ l\kk} (\widetilde{L}_{t})_{x y} \dd b_{ l\kk} + \left(1 - \frac{1}{t_c-t}\right) (\widetilde{L}_{t})_{x y} \dd t + \frac{1}{t_c-t} \sum_{a=1}^D (\widetilde{L}_{t})_{x a} (\widetilde{L}_{t})_{a y} \dd t  \nonumber\\
&+ D(t_c-t) \sum_{a=1}^D \left\{\left\langle \left(G_{1,t}-M_{1,t} \right) E_a\right\rangle [L_{(1,2,1),t}]_{xya} + \left\langle \left(G_{2,t}-M_{2,t} \right) E_a\right\rangle [L_{(2,1,2),t}]_{yxa} \right\} \dd t . \label{ItowtL}
\end{align}
Next, with the averaged local law \eqref{eq:aver_local} and the estimate \eqref{entprodG}, we can bound the last term by
\begin{equation}\nonumber
    \begin{aligned}
        \OO_\prec\pa{(t_c-t)\cdot N^{-1}\eta_t^{-3}\im m_t}=\OO_\prec\pa{N^{-1}(t_c-t)^{-2}\pa{\im m_t}^{-2}},
    \end{aligned}
\end{equation}
where we used $\eta_t/\im m_t\sim t_c-t$ by \eqref{t_eta}. Hence, we can rewrite \eqref{ItowtL} as 
\be\label{mainwtL}
\begin{aligned}
\dd  \widetilde{L}_{t }
=&~ \frac{1}{\sqrt{N}} \sum_{a=1}^D \sum_{ l,\kk \in \mathcal{I}_a} \partial_{ l\kk} \widetilde{L}_{t } \dd  b_{ l\kk} +\left[  \left(1 - \frac{1}{t_c-t}\right) \widetilde{L}_{t} + \frac{1}{t_c-t} (\widetilde{L}_{t })^2   \right] \dd  t\\
&~ +\OO_\prec\left(N^{-1}(t_c-t)^{-2}\pa{\im m_t}^{-2}\right)\dd t .
\end{aligned}
\ee
On the other hand, by \eqref{eq:whMK}, we see that $\wt K_t$ satisfies the following equation:
\be\label{mainwtK}
\frac{\dd}{\dd t} \wt K_t= \left(1 - \frac{1}{t_c-t}\right) \widetilde{K}_{t} + \frac{1}{t_c-t} (\widetilde{K}_{t })^2, 
\ee
 which matches the drift term in \eqref{mainwtL}.

 We now study the martingale term in \eqref{mainwtL}, which is denoted as $\cal L_t$: 
$$ \dd\mathcal{L}_t= \frac{1}{\sqrt{N}}\sum_{a=1}^D\sum_{l,\kk\in\cal I_a}\partial_{l\kk}\wt L_{t}\dd b_{l\kk} \quad \text{with}\quad \cal L_{t_0}=0.
$$
The quadratic variation of $(\cal L_t)_{xy}$, $x,y \in \qq D$, is given by
\begin{align}\label{eq:quad_var}
[ \cal L_{xy}]_t &= \frac{1}{N}\int_{t_0}^t \sum_{a=1}^D\sum_{l,\kk\in\cal I_a}|\partial_{l\kk}(\wt L_s)_{xy}|^2\dd s.
\end{align}
Using \eqref{Ito_pro_1}, we can calculate the integrand as
\begin{align*}
\sum_{a=1}^D\sum_{l,\kk\in\cal I_a}|\partial_{l\kk}(\wt L_s)_{xy}|^2  =& \frac{(t_c-s)^2}{N^2}\sum_{a=1}^D\sum_{l,\kk\in\cal I_a}\Big(\left|(G_{1,s}E_xG_{2,s}E_yG_{1,s})_{\kk l}\right|^2 + \left|(G_{2,s}E_yG_{1,s}E_xG_{2,s})_{\kk l}\right|^2\\
&\qquad \qquad \qquad + 2\re\left[(G_{1,s}E_xG_{2,s}E_yG_{1,s})_{\kk l}\overline{(G_{2,s}E_yG_{1,s}E_xG_{2,s})_{\kk l}}\right]\Big)\\
=& \frac{D(t_c-s)^2}{N }\sum_{a=1}^D \Big( \langle G_{1,s}E_xG_{2,s}E_yG_{1,s}E_aG_{1,s}^*E_yG_{2,s}^*E_xG_{1,s}^*E_a\rangle\\
&\qquad \qquad \qquad +  \langle G_{2,s}E_yG_{1,s}E_xG_{2,s}E_aG_{2,s}^*E_xG_{1,s}^*E_yG_{2,s}^*E_a\rangle\\
&\qquad \qquad \qquad + 2 \re\langle G_{1,s}E_xG_{2,s}E_yG_{1,s}E_aG_{2,s}^*E_xG_{1,s}^*E_yG_{2,s}^*E_a\rangle\Big).
\end{align*}
Applying \eqref{t_eta} and the estimate \eqref{entprodG} below, we obtain that if $t_0 \le s \le t_m$, then
\be\label{eq:quad_integrand}
\sum_{a=1}^D\sum_{l,\kk\in\cal I_a}|\partial_{l\kk}(\wt L_s)_{xy}|^2  \prec \frac{|t_c-s|^2}{N }\cdot \frac{\im m_s}{\eta_s^5} \lesssim \frac{1}{N(t_c-s)^3\pa{\im m_s}^4} .
\ee
With a standard continuity argument, we obtain that this estimate holds uniformly in $s\in [t_0 , t_m]$ (i.e., we first show that \eqref{eq:quad_integrand} holds uniformly in $t$ belonging to an $N^{-C}$-net of $[t_0 , t_m]$ and then extend it uniformly to the whole interval using the Lipschitz continuity in $t$). Plugging \eqref{eq:quad_integrand} into \eqref{eq:quad_var}, we get the estimate
\be\label{eq:Ltxy}
[ \cal L_{xy}]_t  \prec \frac{1}{N^2(t_c-t)^2\pa{\im m_t}^4}, \quad \text{if}\quad t_0 \le t \le t_m.
\ee
On the other hand, we have the trivial bound $|[ \cal L_{xy}]_t |\le N$ by using $\|G_{i,t}\|\le \eta_t^{-1} \ll N$ for $t\in [t_0 , t_m]$. Together with \eqref{eq:Ltxy} and \Cref{stoch_domination}, it implies that for any constant $c>0$ and fixed $p\in \N$,
\[\mathbb E \left|[ \cal L_{xy}]_t\right|^p \le \left(\frac{N^c}{N^2(t_c-t)^2\pa{\im m_t}^4}\right)^p, \quad \forall \ t \in [t_0 , t_m]. \]
Applying the Burkholder-Davis-Gundy inequality, we obtain a $p$-th moment bound on $\sup_{s\in[ t_0,t] }\left|(\cal L_s)_{xy}\right|$. Then, applying Markov's inequality yields that for any $t\in [t_0 , t_m]$ and $x,y \in \qq D$,
\be\label{CIest}
\sup_{s\in[ t_0,t] }\left|(\cal L_s)_{xy}\right| \prec \frac{1}{  N(t_c-t)\pa{\im m_t}^2 }.
\ee

Inserting \eqref{CIest} back to \eqref{mainwtL}, we obtain that for any $t\in [t_0 , t_m]$ and $x,y \in \qq D$,
\be\label{mainwtL2}
 \widetilde{L}_{t }- \widetilde{L}_{t_0 }
=\int_{t_0}^{t} \left[\left(1 - \frac{1}{t_c-s}\right) \widetilde{L}_{s}  + \frac{1}{t_c-s}   (\widetilde{L}_{s} )^2 \right] \dd s+\OO_\prec \left(\frac1{N(t_c-t)\pa{\im m_t}^2 }\right) .
\ee
On the other hand, by \eqref{mainwtK}, we have
\be\label{mainwtL3} 
 \widetilde{K}_{t }- \widetilde{K}_{t_0 }
=\int_{t_0}^{t} \left[\left(1 - \frac{1}{t_c-s}\right) \widetilde{K}_{s}  + \frac{1}{t_c-s} (\widetilde{K}_{s})^2\right] \dd s.
\ee
For simplicity, we introduce the notation $ \wt \Delta_t:= \wt L_t-\wt K_t$ and define the linear operator 
$\cal T_t$ acting on $D\times D$ matrices as 
\be\label{eq:TtV}
\cal T_t(V):=\wt K_t  V+V\wt K_t - [1-(t_c-t)] V,\quad V\in \C^{D\times D}.
\ee
Then, subtracting \eqref{mainwtL3} from \eqref{mainwtL2}, we obtain that
\begin{align}
\widetilde{\Delta}_{t }- \widetilde{\Delta}_{t_0} &=\int_{t_0}^{t} \left[\left(1 - \frac{1}{t_c-s}\right) \widetilde{\Delta}_{s}  + \frac{1}{t_c-s} \left(\widetilde{K}_{s}\widetilde{\Delta}_{s}+ \widetilde{\Delta}_{s}\widetilde{K}_{s} + (\widetilde{\Delta}_{s})^2\right)\right] \dd s  +\OO_\prec \left(\frac1{N(t_c-t)\pa{\im m_t}^2 }\right)\nonumber \\
&=\int_{t_0}^{t} \left( \cal T_s( \widetilde{\Delta}_s)  +  (\widetilde{\Delta}_s)^2\right) 
 \frac{\dd s}{t_c-s}+\cal E_{t},\label{wtLKint}
\end{align}
where $\cal E_{t}$ is a $D\times D$ random matrix satisfying that $\|\cal E_{t}\|_{\HS} \prec [N(t_c-t)\pa{\im m_t}^2]^{-1} $ uniformly in $t\in [t_0 , t_m]$.
Denoting ${\wh \Delta}_t:=\wt\Delta_t-\cal E_t$ and noticing that $\cal E_{t_0}=0$, we can rewrite \eqref{wtLKint} as
\be\label{eq:deltat}
{{\wh \Delta}}_{t} - {\wh \Delta}_{t_0} = \int_{t_0}^{t} \left( \mathcal{T}_s({\wh \Delta}_s)+ \mathcal{T}_s(\mathcal{E}_s) + ({\wh \Delta}_s + \mathcal{E}_s)^2 \right) \frac{\dd s}{t_c-s }.
\ee
Let $\Phi\left(t ; t_0\right)$ be the standard Peano-Baker series corresponding to the linear operator $\cal  T_t/(t_c-t)$, i.e., it is the unique solution to the following linear integral equation
\be\label{eq:PBintegral}
\Phi \left(t ; t_0\right)=\mathbf{1}+\int_{t_0}^t \frac{\cal T_s}{t_c-s}\circ \Phi \left(s ; t_0\right) \mathrm{d} s ,
\ee
where $\mathbf{1}$ denotes the identity operator. By Duhamel's principle, the solution ${\wh \Delta}_t$ to \eqref{eq:deltat} can be written as 
\be\label{LKE}
{\wh \Delta}_{t }
=\Phi \left(t  ; t_0\right) {\wh \Delta}_{t_0}+\int_{t_0}^t \Phi \left(t ; s\right) \left(\frac{\mathcal{T}_s(\mathcal{E}_s) + ({\wh \Delta}_s + \mathcal{E}_s)^2}{t_c-s} \right) \mathrm{d} s.\ee

Suppose the space $\mathbb{C}^{D\times D}$ of $D\times D$ matrices is equipped with the Hilbert-Schmidt norm. Then, we claim that, as a linear operator on $\mathbb{C}^{D\times D}$, 
$\mathcal T_t$ has operator norm at most $1 + \oo(1)$:
\be\label{Tnorm} 
 \|\cal T_t\|_{op}\le 1+\oo(1).
 \ee
Before proving this estimate, we first use it to prove \eqref{flow_1_result}.
With \eqref{Tnorm}, we get from \eqref{eq:PBintegral} that 
\[\frac{\dd}{\dd t}\|\Phi(t;s)\|_{op}\le \frac{1+\oo(1)}{t_c-t} \|\Phi(t;s)\|_{op}.\]
Using Gr{\"o}nwall's inequality, we conclude that for $t_0\le s \le t\le t_m $,
\be\label{Phinorm} 
\|\Phi \left(t ; s\right)\|_{op}\prec  \frac{t_c-s}{t_c-t}  .
\ee
Applying \eqref{Tnorm} and \eqref{Phinorm} to \eqref{LKE} and using the bound on $\|\cal E_t\|_{\HS}$, we obtain that 
\[
\|{{\wh \Delta}}_{t}\|_{2}\prec \frac{t_c-t_0}{t_c-t} \|{{\wh \Delta}}_{t_0}\|_2 +\frac{1}{t_c-t}\int_{t_0}^t \|{{\wh \Delta}}_s+\cal E_s\|_2^2 \dd s+\int_{t_0}^t \frac{\dd s}{N(t_c-t)\pa{t_c-s}\pa{\im m_t}^2},\]
where we also used that $\im m_s\sim \im m_{t}$ by \eqref{im_m_t_sim_through_evolution}. From this estimate, writing ${\wh \Delta}_t=\wt{\Delta}_t-\cal E_t$, we obtain that for $t_c-t_0\sim N^{-\e_g}$ and $t_0\le t\le t_m$,
\be\label{eq:wtLKt}
\|\wt {{\Delta}}_{t }\|_2\prec \frac{t_c-t_0}{t_c-t} \|\wt {{\Delta}}_{t_0}\|_2 +\frac{1}{t_c-t}\int_{t_0}^t \|\wt{{\Delta}}_s\|_2^2 \dd s+\frac{1}{N(t_c-t)\pa{\im m_t}^2}.
\ee
By \eqref{largG_Ent}, \eqref{1-M}, and \eqref{1-M-2}, we have 
\[\pa{\im m_t}^2\|\widetilde{{\Delta}}_{t_0 }\|_2\prec \pa{\im m_t}^2\frac{t_c-t_0}{N\eta_{t_0}^{2}} \frac{\im m_{t_0}}{\eta_{t_0}} \lesssim \frac{N^{\e_g}}{N\pa{t_c-t_0}^{2}}\prec N^{-1+3\e_g},\]
where we used \eqref{t_eta} and \eqref{im_m_t_sim_through_evolution} in the second step. Then, from \eqref{eq:wtLKt}, we derive the the following self-improving estimate for $t\in [t_0,t_m]$ when $C_0>4$:  
\be\label{eq:self_imp}
\sup_{ s \in [t_0, t]}N(t_c-s)\pa{\im m_s}^2\|\wt{\Delta}_s\|_2\prec N^{3\e_g} \ \Rightarrow \ N(t_c-t)\pa{\im m_t}^2\|\wt{\Delta}_{t}\|_2\prec N^{2\e_g}+N^{(6-C_0)\e_g} \, ,
\ee
where we also used that \smash{$N\pa{t_c-t}\pa{\im m_t}^2\gtrsim N\eta_t\im m_t\geq N^{C_0\varepsilon_g}$} by \eqref{t_eta} and the definition of $t_m$. Moreover, defining the stopping time \smash{$T=\inf_{t\ge t_0} \{N(t_c-t)\pa{\im m_t}^2\|\widetilde{{\Delta}}_t\|_2\ge  N^{2\e_g+\e}\}$} for a constant $0<\e<\e_g$, we obtain from \eqref{eq:wtLKt} that 
\[
\|\wt {{\Delta}}_{t }\|_2\prec \frac{t_c-t_0}{t_c-t} \|\wt {{\Delta}}_{t_0}\|_2
 +\frac{1}{N(t_c-t)\pa{\im m_t}^2},\]
if $t\le T$ and $t_0\le t \le t_m$ with $C_0>6$. Now, applying a standard continuity argument with \eqref{eq:self_imp} gives that $T\ge t_m$ with high probability when $C_0>6$ and hence concludes the desired result \eqref{flow_1_result}.

Finally, we prove the bound \eqref{Tnorm}. By the estimate \eqref{1-M-2} below, we have
\be\label{eq:wtKt}
\| \wt K_t\|=(t_c-t)\| K_t\|\le (t_c-t)\| (1-\wh M_t)^{-1}\|\|\wh M_t\|\lesssim (t_c-t)\pa{\im m_t}^{-1} 
\ee
in the case where $(z_1)_t=(z_2)_t \in \{z_t,\bar z_t\}$. In this setting, if $E_t\in\q{E_t^- +\pa{\log N}^{-1},E_t^+ -\pa{\log N}^{-1}}$, then applying \eqref{square_root_density} yields the estimate $\| \wt K_t\|\lesssim \pa{t_c-t}\sqrt{\log N}$, from which the bound \eqref{Tnorm} follows immediately.
If, on the other hand, $E_t \notin [E_t^-, E_t^+]$, then by \eqref{t_eta} and \eqref{square_root_density}, we have
$t_c-t\sim \eta_t/\im m_t\sim \sqrt{\kappa_t+\eta_t}\geq \sqrt{\kappa_t}$, which implies that $\kappa_t=\oo\pa{1}$. Thus, it remains to consider the following two cases:
\begin{enumerate}
    \item $(z_1)_t=(\bar z_2)_t \in \{z_t,\bar z_t\}$;
    \item $(z_1)_t=(z_2)_t\in \ha{z_t,\bar{z}_t}$ with $\kappa_t=\oo\pa{1}$.
\end{enumerate}
In both case, since \smash{$\wh M_t$} is a circulant matrix, it has an eigendecomposition \smash{$\wh M_t = U_t D_t U_t^*$}, where $D_t$ is the diagonal matrix of eigenvalues and $U_t$ is a $D\times D$ unitary matrix. Then, \smash{$\wt K_t$} can be written as 
\[\wt K_t = U_t\Xi_t U_t^*,\quad \Xi_t:=(t_c-t)\frac{D_t}{1-D_t}.\]
Now, we define the linear operator $\wt{\cal T}_t$ as 
\[
\wt{\cal T}_t(V):=\Xi_t  V+V\Xi_t - [1-(t_c-t)] V,\quad V\in \C^{D\times D}.
\]
It is easy to see ${\cal T}_t(V)= U_t[\wt{\cal T}_t(U_t^* VU_t)]U_t^*$, which implies that $\|\cal T_t\|_{op}=\|\wt{\cal   T }_t\|_{op}$. From the definition of $\wt{\cal T}_t$, we see that 
  \be\label{T2aa}
  \|\wt{\cal T}_t\|_{op} \le 
    \max_{l,\kk\in \qq D} \left|(\Xi_t)_{ll}+(\Xi_t)_{\kk\kk}-1\right| + |t_c-t|.
  \ee
It remains to estimate the eigenvalues of $\wt K_t$. 

In case (i), since the entries of \smash{$\wh M_t$} are all non-negative, it has a Perron–Frobenius eigenvalue 
\[d_1=\frac{\im m_t(z_t)}{\im m_t(z_t) + \eta_t}\]
by equation \eqref{sumwtM} below. Moreover, by equation \eqref{eq:otherM}, the eigenvalues $d_l$ of \smash{$\wh M_t$} satisfy $d_l=d_1-a_l-\ii b_l$, $l \in \qq D$, for some $a_l\ge 0$ and $a_l+|b_l|=\oo(1)$. Thus, 
\begin{align}
(\Xi_t)_{ll}+(\Xi_t)_{\kk\kk}-1 & = (t_c-t)\left[\frac{d_1-a_l-\ii b_l}{(1-d_1)+a_l+\ii b_l} + \frac{d_1-a_\kk-\ii b_\kk}{(1-d_1)+a_{\kk}+\ii b_{\kk}}\right]-1 \nonumber\\
& =\frac{\eta_t}{\eta_t + a'_l + \ii b'_l} + \frac{\eta_t}{\eta_t + a'_{\kk} + \ii b'_{\kk}} -1 +\oo(1),\label{estimate_eigenvalue_tilde_K_t}
\end{align} 
where we used \eqref{t_eta} in the second step and abbreviated that $a'_l:=(\im m_t+\eta_t)a_l$ and $b'_l:=(\im m_t+\eta_t)b_l$. 
Together with the simple fact $|1/(1+z)-1/2|\le 1/2$ when $\re z\ge 0$, this equation implies $|(\Xi_t)_{ll}+(\Xi_t)_{\kk\kk}-1|\le 1+\oo(1)$. Plugging it into \eqref{T2aa} concludes \eqref{Tnorm} for case (i). 

The proof of \eqref{Tnorm} for case (ii) is similar. We only need to replace decomposition $d_l=d_1-a_l-\ii b_l$ with the decomposition \smash{$\wh d_l=d_1-\wh a_l -\ii \wh b_l$} in \eqref{eq:otherM2}, and bound the first term on the RHS of \eqref{T2aa} using the same argument as that in \eqref{estimate_eigenvalue_tilde_K_t}. 
Here, we again utilize the facts that \smash{$\wh a_l\geq 0$ and $\wh a_l+\abs{\wh b_l}=\oo\pa{1}$}, as shown in equation \eqref{eq:akbk2} below. This completes the proof of \eqref{Tnorm}.
\end{proof}

\section{Delocalized phase: eigenvalues}\label{sec:delocalized_case_eigenvalue}

Consider the matrix OU process ${H_{\Lambda}}(t)=H_t +\Lambda$, where $H_t=(h_{ij}(t))_{i,j\in \cI}$ satisfies the OU equation 
\be\label{defVt}
\dd h_{i j}=-\frac{1}{2} h_{i j} \dd t+\frac{1}{\sqrt{DN}} \dd b_{i j}(t),\quad \text{with}\quad H_{0}=H,
\ee
where \smash{$B_t=(b_{ij}(t))_{i,j\in \cI}$} denotes a Hermitian matrix whose upper triangular entries are independent complex Brownian motions with variance $t$. 
We denote the Green's function of $H_{\Lambda}\pa{t}$ by \smash{$G_t (z) :=\pa{H_{\Lambda}\pa{t}-z}^{-1}$}. Let $M_t(z)$ be the solution to the matrix Dyson equation \eqref{def_M} with the operator $\cal S$ replaced by $\cal S_t$: 
\[ \cal S_t(M_t):=e^{-t}\cal S(M_t)+(1-e^{-t}) \langle M_t\rangle.\]
Note that the self-consistent equation \eqref{self_m} for $m_t(z):=\langle M_t(z)\rangle$ is unchanged, so we always have $m_t(z)=m(z)$ and $M_t(z)=M(z)$ as given by \eqref{def_G0}.

\Cref{MixEV} follows immediately from the next two lemmas, Lemmas \ref{CVtG} and \ref{CVtV}. 

\begin{lemma}\label{CVtG} 
Under the assumptions of Theorem \ref{MixEV}, suppose $\mathfrak{t}= N^{-1/3+\fc}$ for a constant $\fc \in (0,1/10)$. Then, for any fixed $n\in \N$, there exist a constant $c_n=c_n(\fc,\delta_A,\varepsilon_A)>0$ such that 
\begin{align}
&\left|	\E O\pa{\gamma_+  \pa{DN}^{2/3}\pa{E^+-\lambda_1^{\ft}},\ldots,\gamma_+ \pa{DN}^{2/3}\pa{E^+-\lambda_n^{\ft}}}\right.\nonumber\\
&\left.-\E O\pa{ \pa{DN}^{2/3}\pa{2-\mu_1},\ldots, \pa{DN}^{2/3}\pa{2-\mu_n}}   \right|\leq N^{-c_n}   , \label{ls_CVtG}
\end{align}
where $\lambda_1^{\ft}\geq\cdots\geq \lambda_n^{\ft}$ and $\mu_1\geq\cdots\geq \mu_n$ denote respectively the largest $n$ eigenvalues of ${H_{\Lambda}}(\ft)$ and a $DN\times DN$ GUE. 
The corresponding result also holds for $\gamma_-(DN)^{2/3}(\lambda_{DN}^{\ft}-E^-,\ldots,\lambda_{DN-n}^{\ft}-E^- )$ at the left edge $E^-$. 
\end{lemma}

\begin{proof} 
Note that $H_t$ in \eqref{defVt} has law 
\be\label{eq:eq_in_law_H}
H_t \stackrel{d}{=} e^{-t/2}\cdot H+ \sqrt{1-e^{-t}}\cdot W,
\ee
where $\stackrel{d}{=}$ means ``equal in distribution" and $W$ is a $DN\times DN$ GUE independent of $H$. Let $V=\r e^{-\ft/2}H+\Lambda$.  Using the local laws in \Cref{lem_loc} and the estimate \eqref{square_root_density} below, we can check that $V$ satisfies the $\eta_*$-regular condition in the sense of \cite[Definition 2.1]{landon2017edgestatisticsdysonbrownian}. 
Then, applying \cite[Theorem 2.2]{landon2017edgestatisticsdysonbrownian}, we obtain that
\begin{equation}\label{delocalized_first_approximation_edge_statistics}
    \begin{aligned}
        &\left|	\E O\pa{\gamma_{\txt{fc}}^\ft \pa{DN}^{2/3}\pa{E_{\txt{fc},\ft}^+-\lambda_1^{\ft}},\ldots, \gamma_{\txt{fc}}^{\ft} \pa{DN}^{2/3}\pa{E_{\txt{fc},\ft}^+-\lambda_n^{\ft}}}\right.\\
        &\left.-\E O\pa{ \pa{DN}^{2/3}\pa{2-\mu_1},\ldots, \pa{DN}^{2/3}\pa{2-\mu_n}}   \right|\leq N^{-c}
    \end{aligned}
\end{equation}
for some constant $c>0$. Here, \smash{$\gamma_{\txt{fc}}^\ft$ and $E_{\txt{fc},\ft}^+$} are defined analogously to $\gamma_+$ and $E^+$, with the limiting density $\rho_N$ in their definitions replaced by $\rho_{\txt{fc},\ft}$, which is the probability density for the free convolution of the empirical spectral distribution of $V=\r e^{-\ft/2}H+\Lambda$ and the semicircle law generated by $\sqrt{1-\r e^{-\ft}}W$. In particular, \smash{$\gamma_{\txt{fc}}^\ft$ and $E_{\txt{fc},\ft}^+$} are random, depending on $V$. To be more precise, denoting $G_V (z) :=\pa{V-z}^{-1}$, we define the Stieltjes transform of $\rho_{\txt{fc},\ft}$, denoted by $m_{\txt{fc},\ft} (z) $, as the unique solution to 
\begin{equation}\nonumber
    \begin{aligned}
        m_{\txt{fc},\ft} (z) =\avga{G_V\pa{z+\pa{1-\r e^{-\ft}}m_{\txt{fc},\ft} (z) }},\quad \text{with}\quad \im m_{\txt{fc},\ft} (z) \ge 0.
    \end{aligned}
\end{equation}
Then, $\gamma_{\txt{fc}}^\ft$ and $E_{\txt{fc},\ft}^+$ are defined by (2.11) and (2.12) in \cite[Lemma 2.3]{landon2017edgestatisticsdysonbrownian}. 

By a similar argument as that in \cite[Section 6.1]{CorrelatedBandAOP}, we can establish that $\abs{\gamma_{\txt{fc}}^\ft-\gamma_+}\leq N^{-\varepsilon}$ and $\abs{E_{\txt{fc},\ft}^+-E^+}\leq N^{-2/3-\varepsilon}$ with high probability for some constant $\varepsilon>0$, which, together with \eqref{delocalized_first_approximation_edge_statistics}, concludes \eqref{ls_CVtG}.
\end{proof}

\begin{lemma}\label{CVtV} 
Under the assumptions of Theorem \ref{MixEV}, there exists a constant $\fc>0$ depending on $\e_A$ and $\delta_A$ such that the following holds for $\ft= N^{-1/3+\fc}$.  
For any fixed $n\in \N$, there exists a constant $c_n=c_n(\fc,\delta_A,\e_A)$ such that
\begin{align}
&\left|	\E O\pa{ \pa{DN}^{2/3}\pa{E^+-\lambda_1^{\ft}},\ldots, \pa{DN}^{2/3}\pa{E^+-\lambda_n^{\ft}}}\right.\nonumber\\
&\left.-\E O\pa{ \pa{DN}^{2/3}\pa{E^+-\lambda_1},\ldots, \pa{DN}^{2/3}\pa{E^+-\lambda_n}}   \right|\leq N^{-c_n}.\label{ls_CVtV}
\end{align}
The corresponding result also holds at the left edge $E^-$.
\end{lemma}

The remainder of this section is devoted to the proof of \Cref{CVtV}. Following an argument analogous to that in \cite[Section 17]{Erds2017ADA}, it suffices to establish the following correlation function comparison theorem.

\begin{lemma}[Green‘s function comparison theorem at the edge]\label{rescomY}
Under the assumptions of Theorem \ref{MixEV}, let $G$ and $G_{\ft}$ denote the resolvents of ${H_{\Lambda}}$ and ${H_{\Lambda}}(\ft)$, respectively. 
Let $F: \mathbb{R}^n \rightarrow \mathbb{R}$ be a function whose derivatives satisfy the following bound: for any fixed $l\in \Z_+$, there exists a constant $C_l>0$ such that
\begin{equation}\label{derivative_bound_F}
    \begin{aligned}
        \max_{\absa{\alpha}=1,2,\ldots,l}\max_x\left|F^{(\alpha)}(x)\right|(|x|+1)^{-C_l} \leqslant C_l.
    \end{aligned}
\end{equation}
Denote $\wh m=\avga{G}$ and $\wh m_t=\avga{G_t}$ for any $t\in \qa{0,\ft}$. Then, there exists a constant $\sigma_0>0$ such that for any constant $0<\sigma<\sigma_0$, and for any sequences of real numbers $\ha{E_1\pa{i}}_{i=1}^n$ and $\ha{E_2\pa{t}}_{i=1}^n$ satisfying
\begin{equation}\nonumber
    \begin{aligned}
        \left|E_1\pa{i}-E^+\right| \leqslant N^{-2 / 3+\sigma}, \quad\left|E_2\pa{i}-E^+\right| \leqslant N^{-2 / 3+\sigma}, \quad {i=1,2,\ldots,n},
    \end{aligned}
\end{equation}
setting $\eta=N^{-2 / 3-\sigma}$, we have
\begin{equation}\label{2comVVt}
    \begin{aligned}
        \Bigg|\mathbb{E} &F\left({DN} \int_{E_1\pa{1}}^{E_2\pa{1}} \mathrm{~d} y \operatorname{Im} \wh m(y+\ii \eta),\ldots,DN \int_{E_1\pa{n}}^{E_2\pa{n}} \mathrm{~d} y \operatorname{Im} \wh m(y+\ii \eta)\right)\\
        &-\mathbb{E} F\left(DN \int_{E_1\pa{1}}^{E_2\pa{1}} \mathrm{~d} y \operatorname{Im} \wh m_{\ft}(y+\ii \eta),\ldots,DN \int_{E_1\pa{n}}^{E_2\pa{n}} \mathrm{~d} y \operatorname{Im} \wh m_{\ft}(y+\ii \eta)\right)\Bigg| \lesssim N^{-\delta}
    \end{aligned}
\end{equation}
for some small constant $\delta>0$ depending only on $\delta_A,\ \varepsilon_A$, and the constants $C_l$.
\end{lemma}

Note that we have only proved \Cref{mix} for ${H_{\Lambda}}$, but it can be extended to any ${H_{\Lambda}}(t)$ with $t\in [0,\ft]$. (Heuristically, adding a GUE component will ``help" the QUE of eigenvectors, so there is no essential difficulty in making this extension.) We will bound the LHS of \eqref{2comVVt} using \Cref{lem:samefort}.

\begin{lemma}\label{lem:samefort}
For any $t\in [0,\ft]$, under the assumptions of \Cref{lem_loc}, the local laws \eqref{eq:aniso_local} and \eqref{eq:aver_local} hold with $G$ replaced by $G_t$, and the eigenvalue rigidity estimate \eqref{eq:rigidity} holds for the eigenvalues of ${H_{\Lambda}}(t)$. Moreover, under the assumptions of \Cref{mix}, the QUE estimate \eqref{eq:extend:main_evector1} holds for the eigenvectors of ${H_{\Lambda}}(t)$. 
\end{lemma}
\begin{proof}
    The estimates \eqref{eq:aniso_local}--\eqref{eq:rigidity} have been proved in Lemma 6.4 of \cite{stone2024randommatrixmodelquantum}. The proof of \eqref{eq:extend:main_evector1} is similar to that for \Cref{mix}, and we omit the details.
\end{proof}


\begin{proof}[\bf Proof of \Cref{rescomY}]
We provide the proof for $n=1$; the general case follows by a similar argument. For ease of presentation, we denote 
\begin{equation}\nonumber
    \begin{aligned}
        A_t:=DN \int_{E_1}^{E_2}  \operatorname{Im} \wh m_{t}(E+\ii \eta)\mathrm{~d} E ,\quad t\in \qa{0,\ft}.
    \end{aligned}
\end{equation}
Recall that we have $M_t(z)\equiv M(z)$ and $m_t(z)\equiv m(z)$ for all $t\in\qa{0,\ft}$. Then, by the averaged local law \eqref{eq:aver_local} for $H_\Lambda\pa{t}$ (as shown in \Cref{lem:samefort}) and the estimate \eqref{square_root_density} below, we obtain the rough estimate: 
\begin{equation}\label{rough_estimate_A_t}
    \begin{aligned}
        \absa{A_t}\prec N \int_{E_1}^{E_2} \pa{\im m\pa{E+\ii \eta}+\frac{1}{N\eta}} \rd E\lesssim N \int_{E_1}^{E_2}\sqrt{\absa{E-E^+}+\eta}\ \rd E+N^{2\sigma}\lesssim N^{2\sigma}.
    \end{aligned}
\end{equation}
To prove \eqref{2comVVt}, we apply It{\^o}'s formula and get that 
\begin{align*}
    &\partial_{t}\E  F\pa{A_t}  = \frac{1}{2DN}\E\sum_{x,y\in \cI} \partial_{xy}\partial_{yx}F\pa{ A_t} - \frac{1}{2}\E\sum_{x,y\in \cI} h_{xy}(t)\partial_{xy}F\pa{ A_t},
\end{align*}
where $\partial_{xy}$ denotes the partial derivative $\partial/\partial_{h_{xy}(t)}$. Then, applying the cumulant expansion from \Cref{lem:complex_cumu} to the second term on the RHS, we get that 
\begin{align}\label{eq:prodnGt}
    &\partial_{t}\E F\pa{ A_t}  = \frac{e^{-t}}{2}\E\sum_{x,y\in \cI} \left( \frac{1}{DN} - s_{xy}\right)\partial_{xy}\partial_{yx} F\pa{A_t} + \sum_{r=3}^l \cal F_r + \cal E_{l+1} ,
\end{align}
where we used that $E|h_{xy}(t)|^2=e^{-t}s_{xy}+(1-e^{-t})(DN)^{-1}$ by \eqref{eq:eq_in_law_H} (recall that $s_{xy}$ was defined in \eqref{eq:sij}). Here, $\cal F_r$ denotes the sum of all terms involving cumulants \smash{$\cal C^{(m,n)}(h_{xy}(t))$} with $m+n=r$, and $\cal E_{l+1}$ is the remainder term. By \eqref{derivative_bound_F}, we can choose $l$ sufficiently large so that the reminder satisfies $ \cal E_{l+1} \lesssim 1$. 

To estimate \eqref{eq:prodnGt}, we begin by analyzing the derivatives of $F(A_t)$. Abbreviating $G_i \equiv  G_t(E_i + \ii \eta)$, we can write that  
\begin{align}\label{partial_xy_F_A_t}
\partial_{xy}F\pa{A_t}&=-F'\pa{A_t}\int_{E_1}^{E_2} \pa{ \im G_t^2}_{yx}\pa{E+\ii \eta}\  \rd E=-F'\pa{A_t}\pa{ \pa{\im G_2}_{yx}- \pa{\im G_1}_{yx}},\\
\label{partial_xy_partial_yx_F_A_t}
        \partial_{xy}\partial_{yx}F\pa{A_t}&= F''\pa{A_t}\pa{\pa{\im G_2}_{yx}-\pa{\im G_1}_{yx}}\pa{\pa{\im G_2}_{xy}-\pa{\im G_1}_{xy}}\\
        & +F'\pa{A_t}\im \pa{\pa{G_2}_{xx}\pa{G_2}_{yy}-\pa{G_1}_{xx}\pa{G_1}_{yy}}.\nonumber
\end{align}
By continuing to differentiate $F(A_t)$ as described above, we obtain, for any fixed $m, n \geq 0$, that
\begin{equation}\nonumber
    \begin{aligned}
        \partial_{xy}^{m}\partial_{yx}^{n}F\pa{A_t}=\sum_{\alpha=1}^{m+n}F^{\pa{\alpha}}\pa{A_t}\sum_{p\in \mathscr{I}_{\alpha}}\Pi_{p} \, ,
    \end{aligned}
\end{equation}
where $\mathscr{I}_{\alpha}$ denotes the set of all possible terms associated with $F^{(\alpha)}$ in the expansion, and $\sup_{\alpha}\absa{\mathscr{I}_{\alpha}}=\OO\pa{1}$. 
For each $\alpha\in\qq{1,m+n}$ and $p\in\mathscr{I}_{\alpha}$, the term $\Pi_p$ is of the following form for some deterministic coefficient $c_p=\OO(1)$ and fixed integer $d_p\ge 1$: 
\begin{equation}\nonumber
    \begin{aligned}
        \Pi_p=c_p\prod_{u=1}^{d_p}\pi_p^u ,
    \end{aligned}
\end{equation}
where each $\pi_p^u$ is either of the form $\pi_p^{u}=\pa{\im G_i}_{\star\star}$ (if $l_{p,u} = 1$), or 
\smash{\(\pi_p^u=\im \left(\pa{G_{i_1}}_{\star\star}\cdots\right.\allowbreak\left.\pa{G_{i_{l{p,u}}}}_{\star\star}\right)\)}
if $l_{p,u} \geq 2$. Here, each $\star$ represents either $x$ or $y$, and each $i_*$ is an index in $\{1,2\}$. It is easy to verify by induction that
\[\sum_{u=1}^{d_p}l_{p,u}=m+n.\] 
By the anisotropic local law \eqref{eq:aniso_local} for $H_\Lambda(t)$ (as shown in \Cref{lem:samefort}) and the estimate \eqref{square_root_density}, we have
\begin{align}
        \absa{ \pa{\im G_i}_{\star_1\star_2}}&\prec \im m\pa{E_i+\ii \eta}+\sqrt{\frac{\im m\pa{E_i+\ii \eta}}{N\eta}}+\frac{1}{N\eta}\lesssim N^{-\frac 1 3+\sigma},\label{single_resolvent_local_law_bound}\\
        \absa{\im [\pa{ G_i}_{\star_1\star_2}]} &\lesssim \big|\im [\pa{ G_i}_{\bu_+\bu_+}]\big|+\big|\im [\pa{ G_i}_{\bu_-\bu_-}]\big| =\abs{ \pa{\im G_i}_{\bu_+\bu_+}}+\abs{ \pa{\im G_i}_{\bu_-\bu_-}}\prec N^{-\frac 1 3+\sigma},\label{single_resolvent_local_law_bound2}
\end{align}
where $\star_1,\star_2\in \{x,y\}$, $\bu_{\pm}:=\mathbf{ e}_{\star_1}\pm \mathbf{e}_{\star_2}$, and in the second equation, we applied the polarization identity. These directly imply that $\absa{\pi_p^u}\prec N^{-1/3+\sigma}$. Combining this bound with the structure of $\partial_{xy}^m \partial_{yx}^n F(A_t)$ described above, as well as the estimate \eqref{rough_estimate_A_t} and the condition \eqref{derivative_bound_F}, we conclude that
\begin{equation}\nonumber
    \begin{aligned}        \absa{\partial_{xy}^{m}\partial_{yx}^{n}F\pa{A_t}}\prec N^{-\frac 1 3+2C_{m+n}\sigma+\sigma}.
    \end{aligned}
\end{equation}
Then, for the terms $\cal F_r$ with $r\ge 3$, it is easy to check that
\begin{equation}\label{eq:Fk}
    \begin{aligned}
        \cal F_r \prec N^{-r/2+5/3+\wt C_l\sigma},\quad 3\le r\le l,
    \end{aligned}
\end{equation}
for a constant $\wt C_l>0$ that does not depend on $\sigma$.

It remains to bound the first term on the RHS of \eqref{eq:prodnGt}. We rewrite \eqref{partial_xy_partial_yx_F_A_t} as
\begin{equation}\nonumber
    \begin{aligned}
        \partial_{xy}\partial_{yx}F\pa{A_t}&=  F''\pa{A_t}\qa{\im \pa{G_2-G_1}}_{yx}\qa{\im \pa{G_2-G_1}}_{xy}\\
        & +F'\pa{A_t} \left(\pa{\im G_2}_{xx}\pa{G_2}_{yy}+\pa{G_2}_{xx}\pa{\im G_2}_{yy}-2\ii \pa{\im G_2}_{xx}\pa{\im G_2}_{yy}\right.\\
        &\left.\qquad\qquad -\pa{\im G_1}_{xx}\pa{G_1}_{yy}-\pa{G_1}_{xx}\pa{\im G_1}_{yy}+2\ii \pa{\im G_1}_{xx}\pa{\im G_1}_{yy}\right).
    \end{aligned}
\end{equation}
Thus, we can write the first term on the RHS of \eqref{eq:prodnGt} as $e^{-t}/2$ times 
\begin{equation}\nonumber
    \begin{aligned}
        \mathscr{F}_2&:= D\sum_{a\in \qq D} F''\pa{A_t}\avga{\im \pa{G_2-G_1}\cdot\pa{D^{-1}-E_a}\cdot\im \pa{G_2-G_1}\cdot E_a}\\
        & +2D^2N\sum_{a\in\qq D} F'\pa{A_t} \left(\avga{\im G_2\cdot\pa{D^{-1}-E_a}}\avga{G_2E_a}-\ii \avga{\im G_2\cdot\pa{D^{-1}-E_a}}\avga{\im G_2\cdot E_a}\right.\\
        &\left.\qquad\qquad\qquad -\avga{\im G_1\cdot\pa{D^{-1}-E_a}}\avga{G_1E_a}+\ii \avga{\im G_1\cdot\pa{D^{-1}-E_a}}\avga{\im G_1\cdot E_a}\right).
    \end{aligned}
\end{equation}
Using the block translation invariance of $M_t$ and the fact that $\sum_{a}(D^{-1}-E_a)=0$, we can rewrite $\mathscr{F}_2$ as
    \begin{align}
        & \mathscr{F}_2=D\sum_{a\in \qq D} F''\pa{A_t}\avga{\im \pa{G_2-G_1}\cdot\pa{D^{-1}-E_a}\cdot\im \pa{G_2-G_1}\cdot E_a}\nonumber\\
        & +2D^2N\sum_{a\in\qq D} F'\pa{A_t} \left(\avga{\im G_2\cdot\pa{D^{-1}-E_a}}\avga{\pa{G_2-M_2}E_a}-\ii \avga{\im G_2\cdot\pa{D^{-1}-E_a}}\avga{\im \pa{G_2-M_2}\cdot E_a}\right.\nonumber\\
        &\left.\qquad\qquad -\avga{\im G_1\cdot\pa{D^{-1}-E_a}}\avga{\pa{G_1-M_1}E_a}+\ii \avga{\im G_1\cdot\pa{D^{-1}-E_a}}\avga{\im \pa{G_1-M_1}\cdot E_a}\right),
    \end{align}
where $M_i\equiv M_t\pa{E_i+\ii \eta}$ for $i\in\{1,2\}$. It remains to bound the following terms for $i,j\in \{1,2\}$: 
\begin{align*}
	X(i,j;a)&:=F''\pa{A_t}\avga{\im G_i\cdot\pa{D^{-1}-E_a}\cdot\im G_j\cdot E_a},\\
	Y_1({i;a})&:=F'\pa{A_t} \avga{\im G_i\cdot\pa{D^{-1}-E_a}}\avga{\pa{G_i-M_i}E_a},\\
    Y_2({i;a})&:=F'\pa{A_t} \avga{\im G_i\cdot\pa{D^{-1}-E_a}}\avga{\im \pa{G_i-M_i}\cdot E_a}.
\end{align*}

With the average local law \eqref{eq:aver_local} and the bounds \eqref{derivative_bound_F}, \eqref{rough_estimate_A_t}, \eqref{single_resolvent_local_law_bound}, and  \eqref{single_resolvent_local_law_bound2}, 
we get the following rough bounds on $X$ and $Y$:
\begin{align}\label{eq:roughXY}
	X(i,j;a) \prec N^{1/3+2\sigma+2C_2\sigma},\quad 
	Y_1({i;a}) \prec N^{-2/3+2\sigma+2C_2\sigma},\quad
    Y_2({i;a}) \prec N^{-2/3+2\sigma+2C_2\sigma}.
\end{align}
To improve these estimates, we consider the eigendecompositions 
   \begin{align}
        \avga{\im G_i\cdot\pa{D^{-1}-E_a}\cdot\im G_j\cdot E_a} &=\frac{1}{DN}\sum_{r,s=1}^{DN}\eta^2 \frac{  {\bf v}_r^* (D^{-1}-E_a){\bf v}_s \cdot  {\bf v}_s^* E_a{\bf v}_r}
       {\pa{(\lambda_r-E_i)^2+\eta^2}\pa{(\lambda_s-E_j)^2+\eta^2}},\label{GEaD1}\\
        \avga{\im G_i\cdot\pa{D^{-1}-E_a}}
      & =\frac{1}{DN}\sum_{r=1}^{DN}\eta\frac{\bv_r^*\pa{D^{-1}-E_a}\bv_r}{\pa{\lambda_r-E_i}^2+\eta^2},\label{GEaD2}
   \end{align}
where $\lambda_k\equiv \lambda_k(t)$ and $\bv_k\equiv \bv_k(t)$ denote the eigenvalues and eigenvectors of $H_t + \Lambda$, respectively. Using the eigenvalue rigidity estimate \eqref{eq:rigidity}, the fact \eqref{eq:kappa_gamma_k}, and the QUE estimate \eqref{eq:extend:main_evector1} for ${H_{\Lambda}}(t)$ (as established in \Cref{lem:samefort}), we can bound \eqref{GEaD1} as follows: with probability $1-\OO(N^{-c})$, 
\begin{align}
\eqref{GEaD1}
&\lesssim \frac{1}{N}\left(\sum_{r,s\le N^\e} \frac{N^{-c} }{\eta^2} +\sum_{r\le N^\e, N^\e < s\le N^c} \frac{N^{-c} }{\pa{s/N}^{4/3}} + \sum_{N^\e < r, s\le N^c} \frac{N^{-c} \eta^2}{\pa{r/N}^{4/3}\pa{s/N}^{4/3}}\right. \nonumber\\
&    +\left. \sum_{N^\e < r\le N^c,s\ge N^{c}}\frac{\eta^2}{\pa{r/N}^{4/3}\pa{s/N}^{4/3}}+\sum_{r\le N^\e,s\ge N^{c}}\frac{1}{\pa{s/N}^{4/3}} +\sum_{r\ge N^c,s\ge N^{c}}\frac{\eta^2}{\pa{r/N}^{4/3}\pa{s/N}^{4/3}}\right)\nonumber \\
&\lesssim N^{1/3-c+2\sigma+2\e}+N^{1/3-c+2\varepsilon/3}+N^{1/3-c-2\sigma-2\varepsilon/3}+N^{1/3-c/3-2\sigma-\varepsilon/3}+N^{1/3-c/3+\varepsilon}+N^{1/3-2c/3-2\sigma},\nonumber\\
&    \lesssim N^{1/3-c/3+2\sigma+2\varepsilon},\label{GEaD3}
\end{align}
provided the positive constants $\sigma$ and $\varepsilon$ satisfy $0<\sigma+\e < c/6$. Similarly, we can bound \eqref{GEaD2} as 
\begin{equation}\label{GEaD4}
\P\left(\left|\avga{\im G_i\cdot\pa{D^{-1}-E_a}}\right| \ge N^{-1/3-c/3+\sigma+\e}\right)  \lesssim N^{-c} .
\end{equation}
Combining \eqref{GEaD3} and \eqref{GEaD4} with \eqref{eq:aver_local}, \eqref{derivative_bound_F}, and \eqref{rough_estimate_A_t}, we obtain that 
\[\mathbb P\left( \absa{\mathscr{F}_2}\ge N^{1/3-c/3+2\sigma+3\e+2C_2\sigma} \right)\lesssim N^{-c}.\] 
Together with the rough bound \eqref{eq:roughXY}, it yields that 
\be\label{eq:F2}
\E\absa{\mathscr{F}_2} \lesssim N^{1/3-c/3+2\sigma+3\e+2C_2\sigma} +  N^{1/3+2\sigma + 2C_2 \sigma +\varepsilon}\cdot N^{-c} \le  2N^{1/3-c/3+2\sigma+3\e+2C_2\sigma }.
\ee

Finally, by choosing the constants $\sigma$ and $\varepsilon$ sufficiently small depending on $c$, and integrating \eqref{eq:F2} and \eqref{eq:Fk} over $t\in [0, \ft]$, we complete the proof of \Cref{rescomY}.
\end{proof}

\section{Localized phase}\label{sec:localized_case}

In this section, we present the proof of \Cref{theorem_localized_eigenvector_localization} and \Cref{theorem_localized_eigenvalue}. Again, without loss of generality, we only consider the case $k\le DN/2$, while the other case $k> DN/2$ can be treated analogously. 
As discussed in \Cref{sec_idea}, the key step in the proof is to establish the two-resolvent estimates, namely \Cref{lemma_localized_eigenvector_local_law} and \Cref{lemma_localized_eigenvalue_local_law} below. 

\subsection{Localized regime: eigenvectors}

We begin by proving the localization of eigenvectors, \Cref{theorem_localized_eigenvector_localization}.
Throughout the following proof, we fix $k\le DN/2$ and set 
\be\label{eq:z1}
z_1\equiv z_1(k,\e):=E+\r i \eta,\quad \text{with}\quad E=\gamma_k, \ \ 
\eta = N^{-2/3+\varepsilon}k^{-1/3},
\ee
for a sufficiently small constant $\e>0$. 
As mentioned previously in \eqref{eq:shiftEV}, an appropriate shift for the spectral parameter is required, defined as
\begin{equation}\label{definition_shift_eigenvector}
z_0\equiv z_0(k,\e):=z_1-\Delta_{\txt{ev}},\quad \text{with}\quad 
\Delta_{\txt{ev}} \equiv   \Delta_{\txt{ev}}(z_1):=\re \pa{z_1+m(z_1)+\frac{1}{m(z_1)}}.
\end{equation}
We will abbreviate $M\equiv M(z_1)$, $M_{\txt{sc}}\equiv M_{\txt{sc}}(z_0):=m_{\txt{sc}}(z_0)I$, $m\equiv m(z_1)$, and $m_{\txt{sc}}\equiv m_{\txt{sc}}(z_0)$.
By the estimate \eqref{modification_shift_eigenvector} below, we know that $\gamma_k -\Delta_{\txt{ev}}$ is approximately equal to $\gamma_{k}^{\txt{sc}}$ up to a negligible error $\oo(N^{-2/3}k^{-1/3})$, which implies that:
\begin{equation}\label{sim_m_sc_and_m}
    \begin{aligned}
        \im m_{\txt{sc}}(z_0)\sim \im m(z_1).
    \end{aligned}
\end{equation}
Furthermore, the shift $\Delta_{\txt{ev}}$ plays a crucial role in the proof by introducing a key cancellation that gives the estimate \eqref{regular_estimate_localized_eigenvector} in the following lemma.
 \begin{lemma}\label{lem_regular_estimate_localized_eigenvector}
     Under the assumptions of \Cref{theorem_localized_eigenvector_localization}, the following bounds hold for any $a\in \qq{D}$ and $\sM_0\in\ha{M_{\txt{sc}}(z_0),M_{\txt{sc}}^*(z_0)}$, $\sM_1\in\ha{M(z_1),M^*(z_1)}$:   \begin{align}\label{shift_bound_localized_eigenvector}
\Delta_{\txt{ev}}(z_1)&=\OO\pa{\avga{\Lambda^2}},\\
\avg{\sM_0\tLambda\sM_1E_a}&=\im m(z_1)\cdot \OO\pa{\avga{\Lambda^2}} , \label{regular_estimate_localized_eigenvector}
\end{align}
where $z_0$ and $z_1$ are defined in \eqref{eq:z1} and \eqref{definition_shift_eigenvector}, respectively, and $\tLambda$ is defined as $\tLambda:=\Lambda-\Delta_{\txt{ev}}$. 
 \end{lemma}
 \begin{proof}
 Using equation \eqref{self_m} and the fact that $\avga{\Lambda}=0$, we obtain 
    \begin{equation}\nonumber
        \begin{aligned}
            m+\frac{1}{m+z_1}=\avga{\pa{\Lambda-m-z_1}^{-1}}+\frac{1}{m+z_1}=-\sum_{l=2}^{\infty} \pa{m+z_1}^{-l-1}\avga{\Lambda^l}=\OO\pa{\avga{\Lambda^2}},
        \end{aligned}
    \end{equation}
which implies \eqref{shift_bound_localized_eigenvector}: 
    \begin{equation}\nonumber
        \begin{aligned}
            \absa{\Delta_{\txt{ev}}}\leq \absa{\frac{m+z_1}{m}}\absa{m+\frac{1}{m+z_1}}\lesssim \avga{\Lambda^2}.
        \end{aligned}
    \end{equation}
    To prove \eqref{regular_estimate_localized_eigenvector}, note that since $\sM_0$ is a scalar matrix and $\sM_1\in\{M,M^*\}$ satisfies the block translation symmetry, it suffices to show that
    \begin{equation}\label{eq:cancel_pf}
        \begin{aligned}
            \avg{\tLambda M}=\im m \cdot \OO\pa{\avga{\Lambda^2}}.
        \end{aligned}
    \end{equation}
    
We first estimate the distance between $m_{\txt{sc}}$ and $m$ as:     \begin{equation}\label{estimate_self_consistent_equation_m_in_semicircle}
\begin{aligned}
z_0+m+\frac{1}{m}&=z_1+m+\frac{1}{m}-\Delta_{\txt{ev}}=\ii \im \pa{z_1+m+\frac{1}{m}}=\ii\pa{ \eta+\im m-\frac{\im m}{\absa{m}^2}}\\
&=\ii \im m\bigg(\frac{1}{\avga{MM^*}}-\frac{1}{|m|^2}\bigg)=\im m \cdot \OO\pa{\avga{\Lambda^2}},
\end{aligned}
\end{equation}
where in the fourth step, we used the identity \eqref{equation_z_im_m}, and in the last step, we applied \eqref{estimate_replace_M_by_m}. Then, we get  
    \begin{equation}\nonumber
        \begin{aligned}
            \absa{m_{\txt{sc}}(z_0)-m(z_1)}\lesssim \frac{\im m(z_1) \cdot \avga{\Lambda^2}}{\im m_{\txt{sc}}(z_0)}\lesssim\avga{\Lambda^2},
        \end{aligned}
    \end{equation}
where in the first step, we used \eqref{estimate_self_consistent_equation_m_in_semicircle} and the stability of the self-consistent equation for $m_{sc}$, and in the second step, we applied \eqref{sim_m_sc_and_m}. Then, from the definition \eqref{def_G0} and identity $m=-(z_0+m)^{-1}$, we obtain 
        \begin{align}
   \abs{\avg{\tLambda M}}\sim          \abs{\avg{M_{\txt{sc}}\tLambda M}}&=\absa{\avga{M_{\txt{sc}}-M+\pa{m-m_{\txt{sc}}}M_{\txt{sc}}M}}=\absa{m-m_{\txt{sc}}}\absa{1-m_{\txt{sc}}m}\nonumber\\
  &\lesssim \absa{m-m_{\txt{sc}}}^2+\absa{m-m_{\txt{sc}}}\absa{1-m^2}     \lesssim \sqrt{\kappa+\eta}\avga{\Lambda^2}\sim \im m \cdot \avga{\Lambda^2},
        \end{align}
    where $\kappa:=|E^+-E|\wedge|E-E^-|$, and in the fourth step, we also used $\absa{1-m^2}\lesssim \sqrt{\kappa+\eta}$ by \eqref{de_ir_1}, along with the fact that $\avga{\Lambda^2} \lesssim\|A\|_{\HS}^2/N\leq N^{-1/3-2\varepsilon_A}k^{-2/3}\ll\sqrt{\kappa+\eta}$ by \eqref{eq:condA2} and \eqref{eq:kappa_gamma_k}. This concludes \eqref{eq:cancel_pf}, which further completes the proof of \eqref{regular_estimate_localized_eigenvector}. 
 \end{proof}

 \Cref{theorem_localized_eigenvector_localization} follows from the following two-resolvent estimate, whose proof is deferred to \Cref{subsection:proof_of_lemma_localized_eigenvector_local_law_and_lemma_localized_eigenvalue_local_law}.
\begin{lemma}\label{lemma_localized_eigenvector_local_law}
    In the setting of \Cref{theorem_localized_eigenvector_localization}, and under the definitions in \eqref{eq:z1} and \eqref{definition_shift_eigenvector}, the following estimate holds for some constant $C>0$ that does not depend on $\varepsilon$:   \begin{equation}\label{estimate_localized_eigenvector_local_law}
        \begin{aligned}
            \E\avg{\pa{\im G_0}\tLambda\pa{\im G_1}\tLambda}\prec N^{C\varepsilon}N^{-5/3}k^{2/3}\norm{A}_{\HS}^2\leq N^{-1-2\varepsilon_A+C\varepsilon} ,
        \end{aligned}
    \end{equation}
   where $G_0$ and $G_1$ denote $G_0:=\pa{H-z_0}^{-1}$ and $G_1:=\pa{H_{\Lambda}-z_1}$.
\end{lemma}

\begin{proof}[\bf Proof of \Cref{theorem_localized_eigenvector_localization}]
For ease of presentation, we will assume $D=2$ in the following proof. The argument for the general case of $D$ is similar and will be sketched at the end.

For any $k\le DN/2$, we denote the $k$-th eigenvector by $\bv_k=(\bu_k^\top, \bw_k^\top)^\top$, where $\bu_k,\bw_k\in \C^N$. 
Then, from the eigenvalue equation
\[
H \begin{pmatrix}
\mathbf{u}_k \\
\mathbf{w}_k
\end{pmatrix}=\begin{pmatrix}
H_1 & A \\
A^* & H_2
\end{pmatrix}
\begin{pmatrix}
\mathbf{u}_k \\
\mathbf{w}_k
\end{pmatrix}
=
\lambda_k
\begin{pmatrix}
\mathbf{u}_k \\
\mathbf{w}_k
\end{pmatrix},
\]
we can derive similar equations as in \eqref{eq:bwk}:   
\be\label{eq:bwbu}\bw_k = -\cG_2(\lambda_k-\Delta_{\txt{ev}})  \pa{A^*\bu_k-\Delta_{\txt{ev}}\bw_k}, \quad \bu_k=-{\cal G}_1(\lambda_k-\Delta_{\txt{ev}}) \left( A\bw_k-\Delta_{\txt{ev}} \bu_k \right).\ee
Now, given an arbitrarily small constant $\delta>0$, we define the following events:
\begin{equation}\nonumber
    \begin{aligned}
        &\mathscr E_1\equiv \mathscr E_1\pa{\delta}:=\left\{\dist(\lambda_k-\Delta_{\txt{ev}}, \spec(H_1))\ge N^{-2/3-\delta}k^{-1/3}\right\},\\
        &\mathscr E_2\equiv \mathscr E_2 \pa{\delta}:=\left\{\dist(\lambda_k-\Delta_{\txt{ev}}, \spec(H_2))\ge N^{-2/3-\delta}k^{-1/3}\right\}.
    \end{aligned}
\end{equation}
We claim that there exists a constant $\delta_0=\delta_0\pa{\delta}>0$ depending on $\delta$ such that 
\be\label{eq:probab12}
\P \left( \mathscr E_1\cup \mathscr E_2\right) = 1-\OO(N^{-\delta_0}).
\ee
To prove this claim, notice that 
$$ \P\left( (\mathscr E_1\cup \mathscr E_2)^c \right) \le \P \left( \exists i,j \in \qq{N} \text{ such that } |\lambda_i^{(1)}-\lambda_j^{(2)}|\le 2N^{-2/3-\delta}k^{-1/3}\right),$$
where $\lambda_i^{(1)}$ and $\lambda_j^{(2)}$ denote the eigenvalues of $H_1$ and $H_2$, respectively. Using the rigidity of eigenvalues for Wigner matrices (see, e.g., \cite[Theorem 2.2]{erdHos2012rigidity} or \eqref{eq:rigidity} in the case of $D=1$), we get
\be\label{rigidity12} |\lambda_i^{(1)}-\gamma_{i,N}^{\txt{sc}}|+|\lambda_i^{(2)}-\gamma_{i,N}^{\txt{sc}}| \prec N^{-2/3}\min(i,N+1-i)^{-1/3},\quad i\in \qq{N}, \ee
where $\gamma_{i,N}^{\txt{sc}}$, $i\in \qq{N}$, denote the quantiles of the semicircle law as defined in \eqref{eq:gammak}, but with $N$ particles: 
\be\label{eq:gammak2}
\gamma_{i,N}^{\txt{sc}}:=\sup_{x\in \R}\left\{\int_{x}^{+\infty}\rho_{\txt{sc}}(x)\dd x \geq \frac{i-1/(2D)}{N}\right\}.\ee
Note that $\gamma_{i,N}^{\txt{sc}}$ is related to $\gamma_{Di}^{\txt{sc}}$ in \eqref{eq:gammak} through the relation $\gamma_{i,N}^{\txt{sc}}=\gamma_{Di}^{\txt{sc}}$.

Next, we present a level-repulsion estimate. Given any constant $\delta>0$, there exists a constant $\delta_0=\delta_0\pa{\delta}>0$ such that the following estimate holds on the event \[A_{j,\tau}:=\Big\{\lambda_j^{\pa{2}}\in [\gamma_k^{\txt{sc}}-N^{-2/3+\tau}k^{-1/3},\gamma_k^{\txt{sc}}+N^{-2/3+\tau}k^{-1/3}]\Big\}\] 
for sufficiently small constant $\tau>0$ (depending on $\delta$ and $\delta_0$):
\be\label{Wegner12} \P \left( \exists i\in \qq{N}, \left. |\lambda_i^{(1)}-\lambda_j^{(2)}|\le 2N^{-2/3-\delta}k^{-1/3}, A_{j,\tau} \right| H_2\right) \le N^{-2\delta_0} \, .\ee
In fact, \cite[Lemmas B.1 and B.12]{BENIGNI2022108109} show \eqref{Wegner12} when $H_1$ is a Gaussian divisible matrix with a small Gaussian component of order \smash{$N^{-\delta'}$}, where $\delta'>0$ is a small constant. Then, by applying the comparison theorem in \cite[Proposition 2.10]{Bourgade_JEMS}, we can conclude \eqref{Wegner12}. 
By \eqref{modification_shift_eigenvector} in the appendix, we have $\gamma_k^{\txt{sc}}+\Delta_{\txt{ev}}=\gamma_k+\oo\pa{N^{-2/3}k^{-1/3}}$. Together with the eigenvalue rigidity estimate \eqref{eq:rigidity}, it implies that 
\begin{equation}\label{shift_rigidity_estimate}
    \begin{aligned}
        \absa{\lambda_k-\Delta_{\txt{ev}}-\gamma_k^{\txt{sc}}}\prec N^{-2/3}k^{-1/3}.
    \end{aligned}
\end{equation}
We denote $k_0=k/D$ so that $\gamma_{k}^{\txt{sc}}=\gamma_{k_0,N}^{\txt{sc}}$ (noting that the definition \eqref{eq:gammak2} remains valid even if $k_0$ is not an integer). Combining estimates \eqref{rigidity12}, \eqref{Wegner12}, and \eqref{shift_rigidity_estimate}, we obtain that for any constants $\tau,C>0$, 
\begin{align*}
\P\left( (\mathscr E_1\cup \mathscr E_2)^c \right) &\le \P \left( { \exists i,j \in \qq{k_0-N^\tau , k_0+ N^\tau}  \text{ such that } |\lambda_i^{(1)}-\lambda_j^{(2)}|\le 2N^{-2/3-\delta}k^{-1/3} },\ A_{j,\tau}\right) + N^{-C}\\
&\le \sum_{j\in \qq{k_0-N^\tau,k_0+ N^\tau}\cap\qq{1,N}} \P \left(\exists i\in \qq{N}, |\lambda_i^{(1)}-\lambda_j^{(2)}|\le 2N^{-2/3-\delta}k^{-1/3},\ A_{j,\tau}\right) + N^{-C} \lesssim N^{-2\delta_0+\tau}.
\end{align*}
If we take $\tau<{\delta_0}$, this concludes \eqref{eq:probab12}.


Now, without loss of generality, suppose the event $\mathscr E_1$ holds. Recall $z_0$ and $z_1$ defined in \eqref{eq:z1} and \eqref{definition_shift_eigenvector}, respectively, and recall that $G_0=\pa{H-z_0}^{-1}$ and $G_1=\pa{H_{\Lambda}-z_1}$. We claim the following estimate: 
\begin{equation}\label{eq:boundE1}
    \E\left(\| {\cal G}_1(\lambda_k-\Delta_{\txt{ev}}) \left( A\bw_k-\Delta_{\txt{ev}} \bu_k \right)\|^2;\mathscr E_1\right)\lesssim N^{2(\e+\delta)} \E \tr[\pa{\im G_0}\tLambda \pa{\im G_1} \tLambda] .
\end{equation}
To see why \eqref{eq:boundE1} holds, using the spectral decomposition of $\im G_1$, we obtain that  
\begin{align*}
    \E \tr[\pa{\im G_0} \tLambda \pa{\im G_1} \tLambda]&\geq\E \sum_{j\in \cI} \frac{\eta}{(\lambda_j-\gamma_k)^2+\eta^2}   \left( A\bw_j-\Delta_{\txt{ev}} \bu_j \right)^*  \im \cG_1(z_0) \left( A\bw_j-\Delta_{\txt{ev}} \bu_j \right)\\
&\gtrsim \eta^{-1} \E\qa{ \left( A\bw_k-\Delta_{\txt{ev}} \bu_k \right)^*  \im \cG_1(z_0) \left( A\bw_k-\Delta_{\txt{ev}} \bu_k \right)} ,
\end{align*}
where in the last step, we applied the rigidity of $\lambda_k$ as given by \eqref{eq:rigidity}. On the other hand, using the spectral decomposition of $\cG_1\pa{z_0}$, we find that on the event $\mathscr E_1$, the following estimate holds with high probability:
\begin{align*}
&~\eta^2\left\| {\cal G}_1(\lambda_k-\Delta_{\txt{ev}}) \left( A\bw_k-\Delta_{\txt{ev}} \bu_k \right)\right\|^2 = \sum_{j} \frac{\eta^2\big|(\bu_j^{(1)})^* \left( A\bw_k-\Delta_{\txt{ev}} \bu_k \right)\big|^2}{(\lambda_j^{(1)}-\lambda_k+\Delta_{\txt{ev}})^2}\\
\lesssim&~ N^{2(\e+\delta)}\sum_{j} \frac{\eta^2\big|(\bu_j^{(1)})^* \left( A\bw_k-\Delta_{\txt{ev}} \bu_k \right)\big|^2}{(\lambda_j^{(1)}-\lambda_k+\Delta_{\txt{ev}})^2+\eta^2}
\lesssim N^{2(\e+\delta)}\sum_{j} \frac{\eta^2\big|(\bu_j^{(1)})^* \left( A\bw_k-\Delta_{\txt{ev}} \bu_k \right)\big|^2}{(\lambda_j^{(1)}-\gamma_k+\Delta_{\txt{ev}})^2+\eta^2}\\
=&~ N^{2(\e+\delta)} \cdot \eta \left( A\bw_k-\Delta_{\txt{ev}} \bu_k \right)^*  \im \cG_1(z_0) \left( A\bw_k-\Delta_{\txt{ev}} \bu_k \right) ,
\end{align*}
where $\{\bu_j^{(1)}:j\in \qq{N}\}$ denote the eigenvectors of $H_1$, and we used the definition of $\mathscr E_1$ in the second step and the rigidity of $\lambda_k$ in the third step. Combining the above two estimates establishes \eqref{eq:boundE1}. 

Given any constant $c_0\in\pa{0,\varepsilon_A/2}$, we choose $\delta$ and $\varepsilon$ to be sufficiently small, depending on $c_0$, such that $c+(1+C/2)\e< c_0/2$, where $C>0$ is the constant in \eqref{estimate_localized_eigenvector_local_law}. Then, applying Markov's inequality, we can derive from \eqref{eq:bwbu}, \eqref{estimate_localized_eigenvector_local_law} and \eqref{eq:boundE1} that
\begin{equation}\label{eq:E1}
    \begin{aligned}
        \P\pa{\norm{\bu_k}\geq N^{-1/3+c_0}k^{1/3}\norma{A}_{\HS};\mathscr E_1} \leq N^{-c_0/2}.
    \end{aligned}
\end{equation}
By symmetry, a similar bound holds for $\norm{\bw_k}$ on $\mathscr{E}_2$. Together with \eqref{eq:probab12}, this concludes \Cref{theorem_localized_eigenvector_localization} for case $D= 2$.

For the general case with $D>2$, given a small constant $\delta>0$, we define 
$$\mathscr E_{(a)}\equiv \mathscr E_{(a)}\pa{\delta}:=\left\{\dist\pa{\lambda_k-\Delta_{\txt{ev}}, \cup_{b\in \qq{D}\setminus\{a\}}\spec(H_b)}\ge N^{-2/3-\delta}k^{-1/3}\right\},\quad a\in \qq D.$$
A similar argument to that used for \eqref{eq:probab12} shows that $\P\pa{\cup_{a=1}^D\scr E_{(a)}}\geq 1-N^{-\delta_0}$ for some $\delta_0=\delta_0\pa{\delta}$. Moreover, we can prove that for any $a\in \qq D$, 
\begin{equation}\label{localized_bound_eigenvector_general_D}
    \begin{aligned}
         \E \left(\|E_{a} \bv_k\|^2;\mathscr E_{a} \right)\lesssim N^{2(\e+\delta)} \E \tr\left[(\im G_0) \Lambda (\im G_1)\Lambda \right]. 
    \end{aligned}
\end{equation}
To see this, without loss of generality, suppose $\mathscr E_{(1)}$ holds. We partition the $j$-th eigenvector as \smash{$\bv_j=\begin{pmatrix}
    \bu_j^\top,\bw_j^\top
\end{pmatrix}^\top$} with \smash{$\bu_j\in \C^N$ and $\bw_j\in\C^{\pa{D-1}N}$}, and partition the first row of blocks of $H_\Lambda$ as $\p{H_1,\cal A}$ for a matrix \smash{$\cal A\in \C^{N\times\pa{D-1}N}$}. Then, we have the equation
\[H_1\bu_k+\cal A\bw_k=\lambda_k\bu_k\implies \bu_k=-\cG_1\pa{\lambda_k-\Delta_\txt{ev}}\p{\wt A\bw_k - \Delta_\txt{ev}\bu_k}.\] 
Using a similar argument to that for the $D=2$ case, we can derive \eqref{localized_bound_eigenvector_general_D} for $a=1$. Finally, by applying Markov's inequality along with the estimates \eqref{estimate_localized_eigenvector_local_law} and \eqref{localized_bound_eigenvector_general_D}, we conclude the proof of \Cref{theorem_localized_eigenvector_localization} for general $D>2$. 
\end{proof}

\subsection{Localized regime: eigenvalues}

For the proof of \Cref{theorem_localized_eigenvalue}, we introduce another shift:
\begin{equation}\label{definition_shift_eigenvalue}
    \begin{aligned}      \Delta_{\txt{e}}\equiv\Delta_{\txt{e}}\pa{k,\eta}:=\int_{0}^{1}\Delta(t) \r d t,\quad \text{where}\quad \Delta\pa{t}:=\frac{\avga{M_t(z_t)\Lambda M_t^{*}(z_t)}}{\avga{M_t(z_t)M_t^*(z_t)}}.
    \end{aligned}
\end{equation}
Here, $M_t$ is obtained by replacing $\Lambda$ with $t\Lambda$ in the definition of $M$, and $m_t$ and $\gamma_k(t)$ are defined in the sense of \Cref{def_moving_Lambda}. We set 
\be\label{eq:z1t}
z_t\equiv z_t(k,\e)=\gamma_k(t)+\r i \eta,\quad \text{with}\quad \eta = N^{-2/3+\varepsilon}k^{-1/3},
\ee
for a sufficiently small constant $\e>0$. 
It is important to emphasize that, although the notation $M_t$ here may coincide with some notations used in \Cref{sec:delocalized_case_eigenvector,sec:delocalized_case_eigenvalue}, all instances of $M_t,m_t,\gamma_k\pa{t}$, and $z_t$ in this section refer exclusively to the quantities defined above.
First, we claim the following bounds that correspond to the estimates \eqref{shift_bound_localized_eigenvector} and \eqref{regular_estimate_localized_eigenvector}.
\begin{lemma}\label{lem_regular_estimate_localized_eigenvalue}
    Under the assumptions of \Cref{theorem_localized_eigenvalue}, the following bounds hold uniformly in $t\in[0,1]$:
    \begin{align}\label{shift_bound_localized_eigenvalue}
    \Delta(t)&=\OO\pa{\avga{\Lambda^2}},\\       \label{regular_estimate_localized_eigenvalue}
\avg{\sM_0\hLambda_t\sM_1E_a}&=\OO\pa{\im m_t(z_t)\cdot \avga{\Lambda^2}}
        \end{align}
for any $a\in \qq{D}$ and $\sM_0,\sM_1\in\ha{M_t\pa{z_t},M_t^*\pa{z_t}}$, where $\hLambda_t$ is defined by $\hLambda_t=\Lambda-\Delta\pa{t}$.
\end{lemma}
\begin{proof}
The first bound \eqref{shift_bound_localized_eigenvalue} follows directly from the estimate \eqref{add_one_more_Lambda} in the appendix. For the bound \eqref{regular_estimate_localized_eigenvalue}, we consider the case $\sM_0 = \sM_1 = M_t$ as an illustrative example; the remaining cases can be shown with a similar argument. For simplicity of notation, we denote $M\equiv M_t$ and $m\equiv m_t$. 
Using the block translation invariance of $M$ and \smash{$\hLambda_t$}, we can write that
    \begin{equation}\nonumber
        \begin{aligned}
            \avg{M\hLambda_t M E_a}= {D}^{-1}\avg{M \hLambda_t M},\quad \forall a\in \qq{D}.
        \end{aligned}
    \end{equation}
Moreover, we have that
    \begin{equation}\nonumber
        \begin{aligned}
            \avg{M \hLambda_t M}&=\frac{1}{\avga{MM^*}}\pa{\avga{M\Lambda M}\avga{MM^*}-\avga{M\Lambda M^*}\avga{M M}}\\
            &= \frac{1}{\avga{MM^*}}\pa{\avga{M\Lambda M}\avga{MM^*}-\avga{M\Lambda M}\avga{MM}+\avga{M\Lambda M}\avga{MM}-\avga{M\Lambda M^*}\avga{M M}}\\
            &= \OO\pa{\im m(z_t)\cdot \avga{\Lambda^2}},
        \end{aligned}
    \end{equation}
    where we used $\im M=\pa{\eta+\im m}MM^*$ and \eqref{add_one_more_Lambda} in the last step. This concludes \eqref{regular_estimate_localized_eigenvalue}.
\end{proof}


\Cref{theorem_localized_eigenvalue} follows from the following two-resolvent estimate, whose proof is also deferred to \Cref{subsection:proof_of_lemma_localized_eigenvector_local_law_and_lemma_localized_eigenvalue_local_law}.
 
\begin{lemma}\label{lemma_localized_eigenvalue_local_law}
In the setting of \Cref{theorem_localized_eigenvector_localization}, the following estimate holds uniformly in $t\in \qa{0,1}$ for some constant $C>0$ that does not depend on $\varepsilon$:
\begin{equation}\label{estimate_localized_eigenvalue_local_law}
        \begin{aligned}
            \E\avg{\pa{\im G_t}\hLambda_t\pa{\im G_t}\hLambda_t}\prec N^{C\varepsilon}N^{-5/3}k^{2/3}\norm{A}_{\HS}^2\leq N^{-1-2\varepsilon_A+C\varepsilon}
        \end{aligned}
    \end{equation}
    where $G_t$ is defined by $G_t :=\pa{H+t\Lambda-z_t}^{-1}$ with $z_t$ defined in \eqref{eq:z1t}.
\end{lemma}

\begin{proof}[\bf Proof of \Cref{theorem_localized_eigenvalue}]
    Denote $H_{\Lambda}\pa{t}=H+t\Lambda$ and let the eigenvalues and corresponding eigenvectors of $H_{\Lambda}(t)$ be denoted by $\lambda_i(t)$ and ${\bf v}_i(t)$, $i\in \cI$.
    Then, for any $k\in \cI$, we have that 
    \begin{equation} \label{ygfyz}
        \lambda_k(1)-\lambda_k(0)-\Delta_{\txt{e}}
        =\int_0^1 \frac{\dd}{\dd t}\lambda_{k}(t) \dd t-\int_0^1 \Delta\pa{t} \dd t
        =\int_0^1  {\bf v}_{k}(t)^* \hLambda_t {\bf v}_{k}(t) \dd\theta.
    \end{equation}
    Applying the Cauchy-Schwarz inequality, we obtain that
    \begin{equation} \label{ygfyz2}
        \mathbb E\left|\lambda_k(1)-\lambda_k(0)-\Delta_{\txt{e}}\right|^2
        \le 
         \mathbb E\int_0^1 \abs{{\bf v}_{k}(t)^* \hLambda_t {\bf v}_{k}(t) }^2 \dd t
        = \int_0^1  \mathbb E \abs{ {\bf v}_{k}(t)^* \hLambda_t {\bf v}_{k}(t) }^2  \dd t.
    \end{equation}
    Using the spectral decomposition of $\im G_t$, we can derive that
    \begin{equation}\label{bdwgy}
        \begin{aligned}
    |{\bf v}_{k}^*\pa{t} \hLambda_{t}{\bf v}_{k}\pa{t}|^2  
        &\le  \frac{\left[(\lambda_k\pa{t}-\gamma_k\pa{t})^2+\eta^2\right]^2}{\eta^2} \tr\left[  (\operatorname{Im} G_t) \hLambda_t (\operatorname{Im} G_t) \hLambda_t \right]\\ 
        &\prec \eta^2 \tr\left[  (\operatorname{Im} G_t) \hLambda_t (\operatorname{Im} G_t) \hLambda_t \right] ,
        \end{aligned}
    \end{equation}
    where we used the rigidity of $\lambda_k\pa{t}$ as given by \eqref{eq:rigidity}. Together with \Cref{lemma_localized_eigenvalue_local_law}. this implies that
    \begin{equation}\nonumber
        \begin{aligned}
            \E|{\bf v}_{k}^*\pa{t} \hLambda_{t}{\bf v}_{k}\pa{t}|^2\prec \eta^2 \E  \tr\left[  (\operatorname{Im} G_t) \hLambda_t (\operatorname{Im} G_t) \hLambda_t \right] \prec N^{-2+\pa{C+2}\varepsilon}  \norma{A}_{\HS}^2 .
        \end{aligned}
    \end{equation}
    Since $\varepsilon$ is arbitrary, we conclude that
    \begin{equation}\label{perturbation_bound_1}
        \begin{aligned}
            \E|{\bf v}_{k}^*\pa{t} \hLambda_{t}{\bf v}_{k}\pa{t}|^2\prec {\norma{A}_{\HS}^2}/{N^2}.
        \end{aligned}
    \end{equation}
  Applying \eqref{perturbation_bound_1} to \eqref{ygfyz2}, we obtain that for any constant $\varepsilon\in (0,\varepsilon_A)$,
    \begin{equation}\label{eq:pfeigenvalue}
        \begin{aligned}
            \qa{\mathbb E\left|\lambda_k(1)-\lambda_k(0)-\Delta_{\txt{e}} \right|^2}^{1/2}\prec {\norma{A}_{\HS}}/{N}.
        \end{aligned}
    \end{equation}
On the other hand, by \eqref{modification_shift_eigenvalue} in the appendix, we know that $\gamma_k -\Delta_{\txt{ev}}$ is approximately equal to $\gamma_{k}^{\txt{sc}}$ up to a negligible error $\OO\big(\norma{A}_{\HS}^4/N^2+N^{-4/3+\varepsilon/2}k^{1/3}\norma{A}_{\HS}^2\big)\ll  {\norma{A}_{\HS}}/{N}$. Together with \eqref{eq:pfeigenvalue}, it implies that  
     \begin{equation}\nonumber
        \begin{aligned}
            \qa{\mathbb E\left|\pa{\lambda_k(1)-\gamma_k}-\pa{\lambda_k(0)-\gamma_{k}^{\txt{sc}}}\right|^2}^{1/2}\prec 
            {\norma{A}_{\HS}}/{N}.
        \end{aligned}
    \end{equation}
    Finally, applying Markov's inequality concludes the proof of \Cref{theorem_localized_eigenvalue}.
\end{proof}

\subsection{Proof of \Cref{lemma_localized_eigenvector_local_law,lemma_localized_eigenvalue_local_law}}\label{subsection:proof_of_lemma_localized_eigenvector_local_law_and_lemma_localized_eigenvalue_local_law}

We will focus on the proof of \Cref{lemma_localized_eigenvector_local_law}, while the proof for \Cref{lemma_localized_eigenvalue_local_law} is nearly identical, with only minor changes in notation. Before presenting the formal proof, we provide an overview of the proof strategy. For notational simplicity, we denote $M_0\equiv M_\tsc(z_0)$, $M_1\equiv M(z_1)$ and $m_0\equiv \avga{M_0}$, $m_1\equiv \avga{M_1}$.

Our basic idea is to iteratively expand the LHS of \eqref{estimate_localized_eigenvector_local_law} according to a carefully designed rule, so that each step yields terms that either satisfy a better bound or become more ``deterministic". 
Specifically, we will consider expansions of \smash{\(\E\avg{\sG_0\tLambda\sG_1\tLambda}\)}, 
where $\sG_0\in\ha{G_0,G_0^*}$ and $\sG_1\in\ha{G_1,G_1^*}$, into a sum of $\OO(1)$ many terms. These terms will either be smaller by a factor of $N^{-c}$ for some constant $c>0$ or will contain fewer resolvent entries, along with some error terms.
After this, we will utilize the polarization identity
\begin{equation}\label{expansion_of_im}
    \begin{aligned}
        -4\avg{\im G_0\cdot\tLambda\cdot\im G_1\cdot\tLambda}=\avg{G_0\tLambda G_1\tLambda}+\avg{G_0^*\tLambda G_1^*\tLambda}-\avg{G_0^*\tLambda G_1\tLambda}-\avg{G_0\tLambda G_1^*\tLambda}
    \end{aligned}
\end{equation}
to bound the terms from the expansions and establish \Cref{lemma_localized_eigenvector_local_law}. 
In the proof, we will refer to the normalized trace $\avg{\cdot}$ of an expression as a ``\emph{loop}". 
As an example, we demonstrate the expansions of the loop 
\begin{equation}\label{starting_point_of_expansion}
    \begin{aligned}
        \avg{G_0\tLambda G_1\tLambda}.
    \end{aligned}
\end{equation}
For clarity of presentation, we label these two $\tLambda$ as $\tLambda_1$ and $\tLambda_2$, respectively. We then select one of these matrices, say \smash{$\tLambda_1$}, and find the very first $G$ factor to its left. Using the identities in \eqref{eq:id_G0G1}, we can decompose the expression into two parts: the part with $M_0$ is more deterministic than \eqref{starting_point_of_expansion}, while the part with $-G_0\pa{H+m_{0}}M_0$ exposes an $H$ entry. For the latter part, we can apply the cumulant expansion formula \eqref{5.16} to get that: 
    \begin{align}
        &~-\E\avg{G_0\pa{H+m_{0}}M_0\tLambda_1 G_1\tLambda_2}=-m_{0}\E\avg{M_0\tLambda_1 G_1\tLambda_2 G_0}-\frac{1}{ND}\sum_{a=1}^D\sum_{\alpha,\beta\in \cI_a}\p{M_0\tLambda_1 G_1\tLambda_2 G_0}_{\alpha\beta}H_{\iba}\nonumber\\
        =&~-m_{0}\E\avg{M_0\tLambda_1 G_1\tLambda_2 G_0}-\frac{1}{ND}\sum_{a=1}^D\sum_{\alpha,\beta\in \cI_a}\sum_{1\leq p+q\leq l}\frac{1}{p!q!}\cC_{\iab}^{p,q+1}\E\qa{\partial_{\iab}^p\partial_{\iba}^{q}\p{M_0\tLambda_1 G_1\tLambda_2 G_0}_{\alpha\beta}}+R_{l+1}\nonumber\\
        =&~D\sum_{a=1}^D\E\avg{ E_a M_0\tLambda_1 G_1\tLambda_2 G_0}\avg{\pa{G_0-M_0}E_a}+D\sum_{a=1}^D\E\avg{M_0\tl_1 G_1E_a}\avg{E_a G_1\tl_2 G_0}\nonumber\\
        &~-\frac{1}{ND}\sum_{2\leq p+q\leq l}\sum_{a=1}^D\sum_{\alpha,\beta\in \cI_a}\frac{1}{p!q!}\cC_{\iab}^{p,q+1}\E\qa{\partial_{\iab}^p\partial_{\iba}^{q}\p{M_0\tLambda_1 G_1\tLambda_2 G_0}_{\alpha\beta}}+R_{l+1},\label{example_expansion_1}
    \end{align}
where $\partial_{\iab}$ denotes the derivative $\partial_{H_{\iab}}$, and we have used the following identity for $\sG\in \{G_0,G_1\}$:
\[\partial_{\iab}\sG=-\sG\Delta_{\iab}\sG,\quad  \text{with}\quad \pa{\Delta_{\iab}}_{ij}=\delta_{i\alpha}\delta_{j\beta}.\] 
Moreover, $R_{l+1}$ comes from the remainder term in \eqref{5.16} and can be bounded by $\opr{N^{-C}}$ for arbitrarily large constant $C>0$, provided $l$ is chosen sufficiently large. In the first summation on the RHS of \eqref{example_expansion_1}, the structure of the first loop, \smash{$\avg{ E_a M_0\tLambda_1 G_1\tLambda_2 G_0}$}, closely resembles that of \eqref{starting_point_of_expansion}, thus satisfying a similar bound. The second loop \smash{$\avg{\pa{G_0-M_0}E_a}$} is bounded by $\opr{1/\pa{N\eta}}=\OO( N^{-\varepsilon})$ by the averaged local law \eqref{eq:aver_local}. Consequently, the first summation achieves a better bound than \eqref{starting_point_of_expansion}.
In the second summation on the RHS of \eqref{example_expansion_1}, the number of $G$ factors associated with \smash{$\tl_1$} decreases, rendering this factor ``more deterministic" than \eqref{starting_point_of_expansion}\footnote{One may notice that the total number of $G$ factors in the two loops associated with $\tl_1$ and $\tl_2$ increases, but this will not affect our strategy.}. 
We remark that a key point in reducing the number of $G$ factors associated with \smash{$\tl_1$} is to keep $M_0$ adjacent to the chosen \smash{$\tl_1$}; specifically, we need to use the identity $G_0=M_0-G_0\pa{H+m_{0}}M_0$ rather than $G_0=M_0-M_0\pa{H+m_{0}}G_0$. 
Finally, for the last term on the RHS of \eqref{example_expansion_1}, the terms with $p+q\geq 3$ can be directly bounded. However, for those with $p+q=2$, further expansions are required, necessitating a more intricate expansion strategy, which we will describe in \Cref{subsection:proof_of_claim_estimate_remainder_terms}.

Inspired by the above discussion, we design the expansion strategy as follows: we first ignore all terms from the $p+q\geq 2$ cases in the cumulant expansions and iteratively expand the expressions using Gaussian integration by parts. This process continues until the terms become either small enough or ``deterministic enough" to be bounded directly using the cancellation in \eqref{regular_estimate_localized_eigenvector} or the polarization identity \eqref{expansion_of_im}. 
After this, we are left with the terms generated from the $p+q\geq 2$ cases in the cumulant expansions. 
Most of these terms can be bounded directly, while the remaining ``troublesome terms" require further expansions. Following one cumulant expansion of a troublesome term, all resulting expressions associated with the $p+q\geq 2$ cases can be bounded directly, while the remaining terms with $p+q=1$ (i.e., those obtained from Gaussian integration by parts) can be handled using a similar expansion strategy as described above.

\subsubsection{Proof of \Cref{lemma_localized_eigenvector_local_law}}

Now, we are ready to present the proof of \Cref{lemma_localized_eigenvector_local_law} following the above strategy. We start with expressions of the form
    \begin{equation}\label{eq;expT0}
        \begin{aligned}
            \avg{\sG_0\tl_1\sG_1\tl_2},\quad \forall \sG_i\in \ha{G_i,G_i^*},\ i\in \{0,1\}.
        \end{aligned}
    \end{equation}
Denote the deterministic limit of $\sG_i$ by $\sM_i$ and let $\mathsf{m}_i:=\avga{\sM_i}$. Consider a class of expressions of the form:
\begin{equation}\label{definition_class_of_expression_loop}
        \begin{aligned}
            \cT:\ c_{\cT}\cdot\cW^{\pa{u}}\cdot \Gamma_{n}^{\pa{\mm}},
        \end{aligned}
    \end{equation}
    where $c_{\cT}$ is a deterministic coefficient, $\cW^{\pa{u}}$ is a product of \emph{light weights} of the form $\avga{B\pa{\sG_{i}-\sM_{i}}}$ for a deterministic matrix $B$:
    \begin{equation}\nonumber
        \prod_{l=1}^u \avga{B_{l}\pa{\sG_{i_l}-\sM_{i_l}}},\quad \text{where}\quad i_l\in\ha{0,1},  
    \end{equation}
    and $\Gamma_n^{\pa{\mm}}$ is a product of loops taking one of the following two forms:
    \begin{align}\label{Gamma_Type_1}
            \textbf{Type \rmnum{1}: }\qquad &\avg{\cG^{\pa{k_1}}\tl_1\cG^{\pa{k_2}}\tl_2}\prod_{l=1}^{n-1}\scG_l\, ;\\
            \label{Gamma_Type_2}
             \textbf{Type \rmnum{2}: } \qquad    &\avg{\cG^{\pa{k_1}}\tl_1}\avg{\cG^{\pa{k_2}}\tl_2}\prod_{l=1}^{n-2}\scG_l.
    \end{align}     
    Here, each $\scG_l$ is a loop of the form
    \begin{equation}\label{eq:loop}
    \avga{\prod_{s=1}^{r_l}\pa{\sG_{i_s}B_{s}}}, \quad \text{where}\ \  i_s \in \{0,1\}, \ r_l\geq 2,
    \end{equation}
    and each $\cG^{\pa{k_i}}$ represents a product of resolvents of the form
    \begin{equation}\nonumber
B_{0}\prod_{s=1}^{k_i}\pa{\sG_{i_s}B_{s}},\quad  \text{where}\ \ i_s\in\ha{0,1}, \ k_i\geq 0 . 
    \end{equation}
In this context, every $B_s$ is a deterministic matrix consisting of a finite product of matrices $E_a$ and $\sM_i$. Moreover, $\mm$ denotes the total number of resolvents in \smash{$\Gamma_n^{\pa{\mm}}$}, i.e.,
    \begin{equation}\label{number_of_G_type_1}
        \begin{aligned}
         \text{Type \rmnum{1} expression}: \   k_1+k_2+\sum_{l=1}^{n-1} r_l=\mm,\quad \text{Type \rmnum{2} expression}: \ k_1+k_2+\sum_{l=1}^{n-2} r_l=\mm .
        \end{aligned}
    \end{equation}
    We denote the set of expressions of the form \eqref{definition_class_of_expression_loop} by $\mathscr{T}$. As we will see, following our expansion strategy, given an element of $\mathscr{T}$, the $p+q=1$ case in its cumulant expansion will always produce elements that are also in $\mathscr{T}$. 

    Now, we describe our expansion procedure. 
    Given any expression $\cT\in \scT$ of the form \eqref{definition_class_of_expression_loop}, if $k_1\geq 1$, we identify the loop containing \smash{$\tl_1$} and find the first $\sG$ factor to the left of \smash{$\tl_1$} in this loop. For example, for the loop \smash{$\avg{\sG_0\tl_1\sG_1\tl_2}$}, we choose $\sG_0$, and for the loop \smash{$\avg{\sM_0\tl_1\sG_1\tl_2}$}, we choose $\sG_1$. Then, we express $\cT$ as 
    \begin{equation}\label{explicit_cT}
        \cT=c_{\cT} \cdot \avg{ \sG B_1 \tLambda_1 \Pi_{k_{\#}-1}} W_1 \cdots W_u f^{(1)} \cdots f^{(n-1)}.
    \end{equation}
    Here, $\Pi_{k_{\#}-1}$ consists of the product of $(k_{\#}-1)$ factors of $\sG_i$, some factors of $E_{a}$ and $\sM_i$, and at most one $\tLambda$; $B$ consists of the product of finitely many factors of $E_{a}$ and $\sM_i$; $k_{\#}=k_1+k_2$ if $\cT$ is of Type \rmnum{1}, and $k_{\#}=k_1$ if $\cT$ is of Type \rmnum{2}; $W_1, \ldots, W_u$ represent light weights, and \smash{$f^{(1)}, \ldots, f^{(n-1)}$} denote other loops of the form \eqref{eq:loop} or \smash{$\avg{\cG^{\pa{k_2}}\tl_2}$}. For simplicity of notation, we denote 
    \begin{equation}\label{explain_product_entries}
        F=\sM B_1 \tLambda_1 \Pi_{k_\#-1}=: F_0F_1 \cdots F_t, \quad f^{(j)}=\avg{ f_0^{(j)} f_1^{(j)} \cdots f_{n_j}^{(j)}}, \quad  W_j=\left\langle\left(\sG_{w_j}-\sM_{w_j}\right) E_{x_j}\right\rangle.
    \end{equation}
    Here, in the first equation, we treat the $\sG_i$ factors as separating points, and write $F$ as \[B\sG B\sG\cdots B\sG B=:F_0F_1\cdots F_t,\] where $B$ and $\sG$ represent general deterministic matrices (specifically, $F_0$ contains \smash{$\tLambda_1$}) and some $G_i$ factors, respectively. We denote the $\sG_i$ factors in $F$ by \smash{$F_{i\pa{1}}, \ldots, F_{i\pa{k_{\#}-1}}$}. 
    In the second equation of \eqref{explain_product_entries}, we express the product in the loop \smash{$f^{\pa{j}}$} in the form \[B\sG B\sG \cdots B\sG=:f_0^{(j)} f_1^{(j)} \cdots f_{n_j}^{(j)},\] and denote the $\sG_i$ factors in it by $f_{i_j\pa{1}}^{(j)}, \ldots f_{i_j\pa{s_j}}^{(j)}$. 
    Now, using \eqref{eq:id_G0G1}, we expand $\sG$ as $\sG=\sM-\sG\pa{H+\smm}\sM$, and apply the cumulant expansion in \Cref{lem:complex_cumu} with respect to the entries of $H$ to obtain that 
        \begin{align}
            \cT\eq &~
            c_{\cT}\cdot \avg{ \sM B_1 \tLambda_1 \Pi_{k_{\#}-1} } W_1 \cdots W_u f^{(1)} \cdots f^{(n-1)}\nonumber\\
            +&~ c_{\cT} \cdot \left[D \sum_{x=1}^D \sum_{j=1}^{k_{\#}-1}\avg{ F_0 F_1 \cdots F_{i(j)} E_x}\left\langle E_x F_{i(j)} F_{i(j)+1} \cdots F_t \sG\right\rangle W_1 \cdots W_u f^{(1)} \cdots f^{(n-1)}\right.\nonumber \\ & +D \sum_{x=1}^D\left\langle F \sG E_x\right\rangle\left\langle\left(\sG-\sM\right) E_x\right\rangle W_1 \cdots W_u f^{(1)} \cdots f^{(n-1)} \label{expansion_for_integral_by_part_term}\\ & +\frac{1}{D N^2} \sum_{x=1}^D \sum_{j=1}^u\left\langle F \sG E_x \sG_{w_j} E_{x_j} \sG_{w_j} E_x\right\rangle f^{(1)} \cdots f^{(n-1)} \prod_{i \neq j} W_i\nonumber \\ & \left.+\frac{1}{D N^2} \sum_{x=1}^D \sum_{j=1}^{n-1} \sum_{r=1}^{s_j}\avg{ F \sG E_x f_{i_j\pa{r}}^{(j)} f_{i_j\pa{r}+1}^{(j)} \cdots f_{n_j}^{(j)} f_0^{(j)} f_1^{(j)} \cdots f_{i_j\pa{r}}^{(j)} E_x} W_1 \cdots W_u \prod_{i \neq j} f^{(i)}\right]+\cR_{\cT}.\nonumber
        \end{align}
    Here and below, we will use ``$\eq$'' to mean ``equal in expectation". The term $\cR_\cT$ is defined by
    \begin{equation}\label{reminder_terms_expansion}
        \begin{aligned}
            \cR_{\cT}=\frac{-c_{\cT}}{ND}\sum_{2\leq p+q\leq l}\sum_{a=1}^D\sum_{\alpha,\beta\in \cI_a}\frac{\cC_{\iab}^{p,q+1}}{p!q!}\partial_{\iab}^p\partial_{\iba}^q\qa{\p{\sM B_1 \tLambda_1 \Pi_{k_{\#}-1}\sG}_{\iab} W_1 \cdots W_u f^{(1)} \cdots f^{(n-1)}}+R_{l+1},
        \end{aligned}
    \end{equation}
    where $R_{l+1}$ comes from the remainder term in \eqref{5.16} for a large enough $l\in \N$. 
    Temporarily ignoring the term $\cR_\cT$, we see that the RHS of \eqref{expansion_for_integral_by_part_term} is a sum of terms in $\scT$. Inspired by the above expansion, we introduce the following five operations on $\scT$:
\medskip
    
    \begin{enumerate}
        \item[$\sreplace$:] This operation corresponds to replacing a resolvent $\sG_i$ by its deterministic limit $\sM_i$, i.e., 
    \begin{equation}\label{sreplace_def}
        \begin{aligned}
            \cT \to c_{\cT}\cdot \avg{ \sM B_1 \tLambda_1 \Pi_{k_{\#}-1} } W_1 \cdots W_u f^{(1)} \cdots f^{(n-1)}.
        \end{aligned}
    \end{equation}
        \item[$\scut{1}$:] This operation represents cutting at the first $\sG$ factor in the loop $\avg{ \sG B_1 \tLambda_1 \Pi_{k_{\#}-1}}$:
        \begin{equation}\label{scut1_def}
            \begin{aligned}
                \cT \to c_{\cT} \cdot D \sum_{x=1}^D\left\langle F \sG E_x\right\rangle\left\langle\left(\sG-\sM\right) E_x\right\rangle W_1 \cdots W_u f^{(1)} \cdots f^{(n-1)}.
            \end{aligned}
        \end{equation}
        \item[$\scut{2}$:] This operation represents cutting at a middle $\sG_i$ factor of the loop $\avg{ \sG B_1 \tLambda_1 \Pi_{k_{\#}-1}}$: 
        \begin{equation}\label{scut_2}
            \begin{aligned}
             \qquad    \cT \to c_{\cT} \cdot D \sum_{x=1}^D \sum_{j=1}^{k_{\#}-1}\left\langle F_0 F_1 \cdots F_{i(j)} E_x\right\rangle\left\langle E_x F_{i(j)} F_{i(j)+1} \cdots F_t \sG\right\rangle W_1 \cdots W_u f^{(1)} \cdots f^{(n-1)}.
            \end{aligned}
        \end{equation}
        
        \item[$\splug{1}$:] This operation involves cutting a light weight into a chain and plugging it into the loop $\avg{\sM  B_1 \tLambda_1 \Pi_{k_{\#}-1}\sG}$: 
        \begin{equation}\label{splug1_def}
            \begin{aligned}
                \cT \to c_{\cT} \cdot \frac{1}{D N^2} \sum_{x=1}^D \sum_{j=1}^u\left\langle F \sG E_x \sG_{w_j} E_{x_j} \sG_{w_j} E_x\right\rangle f^{(1)} \cdots f^{(n-1)} \prod_{i \neq j} W_i.
            \end{aligned}
        \end{equation}
        \item[$\splug{2}$:] This operation involves cutting a $\mathscr{G}_l$ loop into a chain and plugging it into the loop $\avg{\sM  B_1 \tLambda_1 \Pi_{k_{\#}-1}\sG}$:
        \begin{equation}\label{splug2_def}
            \begin{aligned}
          \qquad  \cT \to c_{\cT} \cdot \frac{1}{D N^2} \sum_{x=1}^D \sum_{j=1}^{n-1} \sum_{r=1}^{s_j}\avg{ F \sG E_x f_{i_j\pa{r}}^{(j)} f_{i_j\pa{r}+1}^{(j)} \cdots f_{n_j}^{(j)} f_0^{(j)} f_1^{(j)} \cdots f_{i_j\pa{r}}^{(j)} E_x} W_1 \cdots W_u \prod_{i \neq j} f^{(i)}.
            \end{aligned}
        \end{equation}
    \end{enumerate}

We have defined the expansion strategy when $k_1\ge 1$. When $k_1=0$ and $k_2\geq 1$, we find the loop containing \smash{$\tl_2$} and the first $\sG$ factor to the left of \smash{$\tl_2$} in this loop, then perform a similar expansion. This induces similar operations on $\scT$, and we refer to these operations by the same names. Finally, if $k_1=k_2=0$, we will not expand $\cT$.

Next, we define our stopping criteria for the expansion procedure to ensure it terminates in a finite number of steps. For $\cT=c_\cT\cdot\cW^{\pa{u}}\Gamma_n^{\pa{\mm}}$, we define its ``size'' as a pair:
    \begin{equation}\label{eq:sizeT}
        \begin{aligned}
            \isize\pa{\cT}:=\pa{S+u,\mm-n+u},
        \end{aligned}
    \end{equation}
    where $S$ denotes the number of $N^{-1}$ factors in $c_{\cT}$. Let $\isize\p{\cT}_2$ and $\isize\p{\cT}_1$ denote the first and second components of $\isize\p{\cT}$, respectively. Then, using the local laws from \Cref{lem_loc,lemma_necessary_estimates}, we get that
    \begin{equation}\label{eq:Tsmallthansize}
        \begin{aligned}
            \cT\prec N^{-\isize\pa{\cT}_1}\eta^{-\isize\pa{\cT}_2-1_{k_1=0}-1_{k_2=0}}\norm{A}^2.
        \end{aligned}
    \end{equation}
In addition, from the definitions of the above five operations, we see that
    \begin{equation}\nonumber
        \begin{aligned}
            &\isize\qa{\sreplace\pa{\cT}}= \isize\pa{\cT}+\pa{0,-1},\quad \isize\qa{\scut{1}\pa{\cT}}=\isize\pa{\cT}+\pa{1,1},\quad  \isize\qa{\scut{2}\pa{\cT}}=\isize\pa{\cT},\\  &\isize\qa{\splug{1}\pa{\cT}}=\isize\pa{\cT}+\pa{1,1},\quad \isize\qa{\splug{2}\pa{\cT}}=\isize\pa{\cT}+\pa{2,2}.
        \end{aligned}
    \end{equation}
    We now define the following stopping criteria, under which our expansion procedure will terminate after $\OO\pa{1}$ iterations. We will stop expanding an expression if it satisfies one of the following conditions:
    \begin{enumerate}
        \item\label{stpping_condition_1} The $\isize$ of the expression satisfies $N^{-\isize\p{\cT}_1}\eta^{-\isize\p{\cT}_2-2}\leq N^{-2}$;
        \item\label{stpping_condition_2} $k_1\pa{\cT}=k_2\pa{\cT}=0$.
    \end{enumerate}
    To show that our expansions will stop after $\OO\pa{1}$ iterations, we start with an expression $\cT_0$ of the form \eqref{eq;expT0},
    and consider a sequence of operations on it:
    \begin{equation}\label{operation_sequence}
        \begin{aligned}
            \cO_1,\cO_2,\ldots,\cO_T,\quad \text{with}\ \ \cO_i\in \ha{\sreplace,\scut{1},\scut{2},\splug{1},\splug{2}}.
        \end{aligned}
    \end{equation}
    Note that $N^{-\isize\pa{\cT}_1}\eta^{-\isize\pa{\cT}_2-1_{k_1=0}-1_{k_2=0}}$ is non-increasing during expansions by any of our five operations. Moreover, it is reduced by at least a factor of $(N\eta)^{-1}\lesssim N^{-\varepsilon}$ when the operations $\scut{1}$, $\splug{1}$, and $\splug{2}$ are applied. 
    Thus, ignoring the remainder terms $\cR_\cT$ from our expansions, the procedure will have terminated before completing these $T$ operations if there are more than $C_0/\varepsilon$ of them belonging to $\ha{\scut{1},\splug{1},\splug{2}}$ for a sufficiently large constant $C_0>0$. 
    Denote by $\cO_{i_1},\ldots,\cO_{i_s}$ all the operations in $\cO_1,\cO_2,\ldots,\cO_T$ that belong to $\ha{\scut{1},\splug{1},\splug{2}}$. For $1\leq l\leq s$, with the convention that $i_0=1$, we note that
    \begin{equation}\nonumber
        \begin{aligned}
            i_l-i_{l-1}-1\leq \mm\pa{\cO_{i_{l-1}}\circ\cdots\circ\cO_1\pa{\cT_0}},
        \end{aligned}
    \end{equation}
    because each operation $\sreplace$ or $\scut{2}$ reduces the number of $\sG$ factors in the (one or two) loops containing \smash{$\tl_1$ and $\tl_2$} by at least $1$. Combining the above observations, there exists a constant $T_0>0$ depending on $\varepsilon$ such that the sequence \[\cT_0,\ \cO_1\pa{\cT_0},\ \ldots, \ \cO_T\circ\cdots\circ\cO_1\pa{\cT_0}\] must terminate after at most $T\leq T_0$ steps. In other words, our procedure will stop in $\OO\pa{1}$ many steps.

    The above procedure produces a sum of expressions satisfying the stopping criteria, along with some remainder terms. We first state the following lemma, which asserts that all remainder terms generated during our procedure (which were ignored in the argument above) can be properly bounded. For any sequence of operations $\cO_1,\ldots,\cO_T$, we say this sequence is admissible if it acts on $\cT_0$ successively without stopping before time $T$. We postpone the proof of \Cref{claim_estimate_remainder_terms} to \Cref{subsection:proof_of_claim_estimate_remainder_terms}.
    
    \begin{lemma}\label{claim_estimate_remainder_terms}
        For any admissible sequence of operations $\cO_1,\ldots,\cO_T$, there exists a constant $C>0$ that does not depend on $\varepsilon$ such that
        \begin{equation}\nonumber
            \begin{aligned}
            \E\cR_{\cO_T\circ\cdots\circ\cO_1\pa{\cT_0}}\prec  N^{C\varepsilon}N^{-5/3}k^{2/3}\norma{A}_{\HS}^2, 
            \end{aligned}
        \end{equation}
        where $\cR_{\cO_T\circ\cdots\circ\cO_1\pa{\cT_0}}$ is defined as in \eqref{reminder_terms_expansion}. 
    \end{lemma}

    \begin{remark}
We remark that, if the elements of the matrix $H$ are Gaussian, then \Cref{claim_estimate_remainder_terms} is trivial. This is because, for Gaussian random variables, all cumulants of order $\ge 3$ vanish, which implies that $\cR_\cT=0$ for any expression $\cT$. Moreover, for $H$ with symmetrically distributed elements, the proof of \Cref{claim_estimate_remainder_terms} can be significantly shortened. In this case, the third-order cumulants are zero, so the $p+q=2$ terms in \eqref{reminder_terms_expansion} will vanish. Consequently, the proof will primarily contain the \emph{direct estimation} part and will not require the \emph{further expansions} part, on which we will spend most of our efforts. 
    \end{remark}
    
    Now, we turn our attention to analyzing the expressions that satisfy the stopping criteria. Clearly, if a sequence of operations $\cO_1,\ldots,\cO_T$ stops due to the stopping criterion (i), then by \eqref{eq:Tsmallthansize}, the expression $\cO_T\circ\cdots\circ\cO_1\pa{\cT_0}$ will be bounded by $\opr{N^{-2}\norm{A}^2}=\opr{N^{-5/3}k^{2/3}\norm{A}_{\HS}^2}$. To analyze the terms generated by operation sequences that stop due to the stopping criterion (ii), we present the following table, which illustrates the effects of our five types of operations on the relevant characteristics of our expressions:
    \begin{table}[htbp]
        \caption{Effects of operations}
        \label{table_effect_of_operations}
        \begin{tabular}{|c|c|c|c|c|}
        \hline
        \diagbox{Operation}{Character} & $\mm$ & $n$ & $u$ & $S$ \\ \hline
        $\sreplace$ & $-1$ & $+0$ & $+0$ & $+0$ \\ \hline
        $\scut{1}$ & $+0$ & $+0$ & $+1$ & $+0$ \\ \hline
        $\scut{2}$ & $+1$ & $+1$ & $+0$ & $+0$ \\ \hline
        $\splug{1}$ & $+2$ & $+0$ & $-1$ & $+2$ \\ \hline
        $\splug{2}$ & $+1$ & $-1$ & $+0$ & $+2$ \\ \hline
        \end{tabular}
    \end{table}

With \Cref{table_effect_of_operations}, suppose $\cT=\cO_T\circ\cdots\circ\cO_1\pa{\cT_0}$ is a term generated by a sequence of operations that terminates due to criterion (ii). For such a term $\cT$, its characters satisfy $k_1=k_2=0$, and
      \begin{equation}\label{counting_character_of_term}
          \begin{aligned}
              & \mm=-\sR+\sC_2+2\sP_1+\sP_2+2,\quad n=\sC_2-\sP_2+1,\quad u=\sC_1-\sP_1,\quad S=2\sP_1+2\sP_2,
          \end{aligned}
      \end{equation}
      where $\sR,\sC_1,\sC_2,\sP_1,\sP_2$ respectively denote the number of operations $\sreplace,\scut{1},\scut{2},\splug{1},\splug{2}$ in the sequence $\cO_1,\ldots,\cO_T$. Moreover, by tracking the $\sG$ factors within the loops containing $\tl_1$ and $\tl_2$, we must have $\sR\geq 2$ when $k_1=k_2=0$. Thus, if $\cT$ is a Type \rmnum{1} expression, we have
      \begin{equation}\label{estimate_loop_type_1}
          \begin{aligned}
              \absa{\cT}\prec&~ N^{-S}\avga{\Lambda^2}\frac{\pa{\im m}^{n-1}}{\eta^{\mm-n+1}}\pa{\frac{1}{N\eta}}^{u}= N^{2-\sR} \avga{\Lambda^2} \pa{\im m}^{n-1} \pa{\frac{1}{N\eta}}^{\mm-n+u+1}\\
              \lesssim&~N^{1-\sR} \norma{A}_{\HS}^2\pa{k/N}^{\frac 1 3 (\mm+u)}.
          \end{aligned}
      \end{equation}
   In the first step of the estimate, we use \eqref{sim_m_sc_and_m} along with \Cref{lem_loc,lemma_necessary_estimates}. In the second step, we applied \eqref{counting_character_of_term}, and in the third step, we used
      \begin{equation}\label{light_weight_cancel_imaginary_part}
          \begin{aligned}
              \pa{\im m}^{n-1}\pa{\frac{1}{N\eta}}^{\mm-n+1}\lesssim&~ \pa{\frac{\sqrt{\kappa+\eta}}{N\eta}}^{n-1}\pa{\frac{1}{N\eta}}^{\mm-2n+2}\\
              \lesssim&~ \pa{{k}/{N}}^{\frac{2}{3}(n-1)}\pa{{k}/{N}}^{\frac{1}{3}(\mm-2n+2)}=\pa{{k}/{N}}^{\frac 1 3 \mm},
          \end{aligned}
      \end{equation}
      where we used \eqref{square_root_density} in the first step, and in the second step, we used \eqref{eq:kappa_gamma_k} together with the fact that $\mm\geq 2\pa{n-1}\geq 0$. Consequently, if $\mm+u\geq 2$ or $\sR\geq 3$, then from \eqref{estimate_loop_type_1} and \eqref{eq:condA2}, we conclude that
      \be\label{eq:easydirect}\cT\prec N^{-5/3}k^{2/3}\norma{A}_{\HS}^2\lesssim N^{-1-2\varepsilon_A} .\ee 
      In the remaining case, we must have $\sR=2$ and $\mm+u\leq 1$. These conditions imply $\sP_1=0$ and $\sC_1+\sC_2+\sP_2\leq 1$. By direct enumeration of the possible operations consistent with our procedure, we identify the only terms generated under these constraints:
      \begin{enumerate}
          \item $\sR=2$ and $\sC_1=\sC_2=\sP_1=\sP_2=0$, in which case we have:
          \begin{equation}\label{expansion_result_R_2_C_0_P_0}
              \begin{aligned}
                  \avg{\sM_0\tLambda\sM_1\tLambda};
              \end{aligned}
          \end{equation}
          \item $\sR=2$, $\sP_1=0$, and $\sC_1+\sC_2+\sP_2=1$, in which case we have:
          \begin{equation}\label{expansion_result_R_2_C_1_P_0}
              \begin{aligned}
                  D\dsum{a}\qa{\avg{\sM_0\tLambda\sM_1\tLambda\sM_0E_a}\avga{E_a\pa{\sG_0-\sM_0}}+\avg{\sM_1\tLambda\sM_0\tLambda\sM_1 E_a}\avga{E_a\pa{\sG_1-\sM_1}}}.
              \end{aligned}
          \end{equation}
      \end{enumerate}
      Plugging these contributions back into the polarization identity \eqref{expansion_of_im}, we analyze the resulting terms. The four terms of the form \eqref{expansion_result_R_2_C_0_P_0} contribute a factor 
      \begin{equation}\label{eq:polar_cancel}
          \begin{aligned}
              -4\avg{\pa{\im M_0}\tl\pa{\im M_1}\tl}\prec {\pa{\im m}^2 \avga{\Lambda^2}}\lesssim {N^{-5/3+\varepsilon}k^{2/3}\norma{A}_{\HS}^2}\lesssim N^{-1-2\varepsilon_A+\varepsilon},
          \end{aligned}
      \end{equation}
      using the identity $\im M_i=\pa{\im m_i+\eta}M_iM_i^*$, together with \eqref{sim_m_sc_and_m} and \eqref{two_loop_estimate} in the first step, and \eqref{square_root_density} in the second step.   Similarly, the four terms of the form \eqref{expansion_result_R_2_C_1_P_0} contribute the followin factor:  
      \begin{equation}\label{cancellation_im_example}
          \begin{aligned}
              D\dsum{a}\Big[&\pa{\avg{M_0\tLambda M_1\tLambda M_0E_a}\avga{E_a\pa{G_0-M_0}}+\avg{M_1\tLambda M_0\tLambda M_1 E_a}\avga{E_a\pa{G_1-M_1}}}\\
              +&\pa{\avg{M_0^*\tLambda M_1^*\tLambda M_0^*E_a}\avga{E_a\pa{G_0^*-M_0^*}}+\avg{M_1^*\tLambda M_0^*\tLambda M_1^* E_a}\avga{E_a\pa{G_1^*-M_1^*}}}\\
              -& \pa{\avg{M_0^*\tLambda M_1\tLambda M_0^*E_a}\avga{E_a\pa{G_0^*-M_0^*}}+\avg{M_1\tLambda M_0^*\tLambda M_1 E_a}\avga{E_a\pa{G_1-M_1}}}\\
              -&\pa{\avg{M_0\tLambda M_1^*\tLambda M_0 E_a}\avga{E_a\pa{G_0-M_0}}+\avg{M_1^*\tLambda M_0\tLambda M_1^* E_a}\avga{E_a\pa{G_1^*-M_1^*}}}\Big]\\
              =&~ \opr{\avga{\Lambda^2}\frac{\im m}{N\eta}}=\opr{\frac{k^{2/3}}{N^{5/3}}\norm{A}_{\HS}^2}\prec N^{-1-2\varepsilon_A}.
          \end{aligned}
      \end{equation}
  To bound these eight terms, we group them into four pairs and estimate them as in \eqref{eq:polar_cancel}:
      \begin{equation}\label{eq:polar2}
          \begin{aligned}
              &~\avg{M_0\tLambda M_1\tLambda M_0E_a}\avga{E_a\pa{G_0-M_0}}-\avg{M_0\tLambda M_1^*\tLambda M_0 E_a}\avga{E_a\pa{G_0-M_0}}\\
              =&~\avg{M_0\tLambda \pa{\im M_1}\tLambda M_0 E_a}\avga{E_a\pa{G_0-M_0}}=\opr{\avga{\Lambda^2}\frac{\im m}{N\eta}}=\opr{\frac{k^{2/3}}{N^{5/3}}\norm{A}_{\HS}^2},
          \end{aligned}
      \end{equation}
      where we again used $\im M_i=\pa{\im m_i+\eta}M_iM_i^*$, \eqref{sim_m_sc_and_m}, and \eqref{two_loop_estimate} in the second step.\footnote{Here, we do not exploit the fact that $M_0$ is a scalar matrix to simplify the estimates, since this simplification does not hold in the context of the proof of \Cref{lemma_localized_eigenvalue_local_law}.}
      	

      If $\cT$ is of Type \rmnum{2}, then by an argument analogous to that used in \eqref{estimate_loop_type_1} and \eqref{light_weight_cancel_imaginary_part}---together with the observation that $\mm\geq 2\pa{n-2}\geq 0$---we obtain the corresponding estimate
      \begin{equation}\nonumber
          \begin{aligned}
              \absa{\cT}\prec&~ N^{-S}\avga{\Lambda^2}^2\frac{\pa{\im m}^{n-2}}{\eta^{\mm-n+2}}\pa{\frac{1}{N\eta}}^{u}
              = N^{3-\sR} \avga{\Lambda^2}^2 \pa{\im m}^{n-2} \pa{\frac{1}{N\eta}}^{\mm-n+u+2}\\
              \lesssim&~ N^{2-\sR}\norm{A}_{\HS}^2N^{-1/3-2\varepsilon_A}k^{-2/3}\pa{k/N}^{\frac 1 3 (\mm+u)}.
          \end{aligned}
      \end{equation}
Now, if any of the following conditions hold: (i) $\mm+u\geq 4$, or (ii) $\sR=3, \mm+u\geq 1$, or (iii) $\sR\geq 4$, then we immediately deduce \eqref{eq:easydirect}.
It remains to analyze the two exceptional cases: (a) $\sR=3$ and $\mm+u=0$, or (b) $\sR= 2$ and  $1\leq \mm+u\leq 3$. Note that in order for a Type \rmnum{2} expression to be generated, it is necessary that $\sC_2\geq 1$. Furthermore, in the case  $\sR=2$, we must have $\sC_2\geq 2$. 
By direct enumeration of the possible operations consistent with our procedure, we identify the only terms generated under these constraints:
      \begin{enumerate}
          \item $\sR=3$, $\sC_1=\sP_1=\sP_2=0$, and $\sC_2=1$, in which case we have:
          \begin{equation}\label{sR_3_sC_2_2_sC_1_sP_1_sP_2_0}
              \begin{aligned}
                  D\dsum{a}\avg{\sM_0 \tLambda \sM_1 E_a}\avg{E_a \sM_1 \tLambda \sM_0};
              \end{aligned}
          \end{equation}
          \item $\sR=2$, $\sC_2=2$, and $\sC_1=\sP_1=\sP_2=0$, in which case we have:
          \begin{equation}\label{sR_2_sC_2_2_sC_1_sP_1_sP_2_0}
              \begin{aligned}
                  D^2\dsum{a,b}\avg{\sM_0\tLambda\sM_1 E_a}\avga{\sG_0 E_a\sG_1 E_b}\avg{\sM_1\tLambda\sM_0 E_b};
              \end{aligned}
          \end{equation}
          \item $\sR=2$ and $\sC_1+\sC_2+\sP_1+\sP_2=3$, in which case we have:
          \begin{equation}\label{sR_2_sC_2_2_sC_1_1_sP_1_sP_2_0}
              \begin{aligned}
                  D^3\dsum{a,b,c}\Big[&\avg{\sM_0\tLambda\sM_1 E_a}\avg{\sM_1 E_b \sM_1 \tLambda \sM_0 E_c}\avga{E_c\sG_0 E_a \sG_1}\avga{E_b\pa{\sG_1-\sM_1}}\\
                  +&\avg{\sM_0\tLambda\sM_1 E_a}\avg{\sM_0 E_b \sM_1\tLambda\sM_0 E_c}\avga{E_b \sG_0 E_a \sG_1}\avga{E_c\pa{\sG_0-\sM_0}}\\
                  +&\avg{\sM_1 E_a \sM_0 \tLambda \sM_1 E_b}\avg{\sM_1 \tLambda \sM_0 E_c}\avga{E_c \sG_0 E_a \sG_1}\avga{E_b\pa{\sG_1-\sM_1}}\\
                  +& \avg{\sM_0 E_a \sM_0 \tLambda \sM_1 E_b}\avga{E_a\pa{\sG_0-\sM_0}}\avg{\sM_1\tLambda\sM_0 E_c}\avga{E_c \sG_0 E_b \sG_1}\Big].
              \end{aligned}
          \end{equation}
      \end{enumerate}
  To bound these terms, we apply the key cancellation estimate \eqref{regular_estimate_localized_eigenvector}, along with \eqref{eq:aver_local}, \eqref{square_root_density}, \eqref{add_one_more_Lambda}, \eqref{entprodG}, and \eqref{sim_m_sc_and_m} to derive that 
      \begin{align}\label{eq:polar3}
              \eqref{sR_3_sC_2_2_sC_1_sP_1_sP_2_0}&\lesssim \pa{\im m}^2\avga{\Lambda^2}^2\lesssim N^{-2-2\varepsilon_A+\varepsilon}k^{2/3}\norma{A}_{\HS}^2\lesssim N^{-4/3-4\varepsilon_A+\varepsilon},\\
 \label{eq:polar4}
 \eqref{sR_2_sC_2_2_sC_1_sP_1_sP_2_0}&\prec {\pa{\im m}^2\avga{\Lambda^2}^2\frac{ \im m}{\eta}}\prec N^{-5/3-2\varepsilon_A+\varepsilon}k^{2/3}\norma{A}_{\HS}^2\leq N^{-1-4\varepsilon_A+\varepsilon},       \\   
\label{eq:polar5}
	\eqref{sR_2_sC_2_2_sC_1_1_sP_1_sP_2_0}&\prec{(\im m) \avga{\Lambda^2}^2 \frac{\im m}{\eta}\frac{1}{N\eta}}\prec N^{-5/3-2\varepsilon_A}k^{2/3}\norma{A}_{\HS}^2\leq N^{-1-4\varepsilon_A},
          \end{align}

Finally, by combining the estimates \eqref{eq:easydirect}, \eqref{eq:polar_cancel}, \eqref{cancellation_im_example}, and \eqref{eq:polar3}--\eqref{eq:polar5} with \Cref{claim_estimate_remainder_terms}, we conclude the proof of \Cref{lemma_localized_eigenvector_local_law}. The proof of \Cref{lemma_localized_eigenvalue_local_law} follows from the same argument.


\subsubsection{Proof of \Cref{claim_estimate_remainder_terms}}
\label{subsection:proof_of_claim_estimate_remainder_terms}

In this subsection, we present the proof of \Cref{claim_estimate_remainder_terms}, which follows a similar strategy to that of \Cref{lemma_localized_eigenvector_local_law}, albeit with more complex expansions and operations. 
We will consider an admissible sequence of operations $\cO_1,\ldots,\cO_T$
and estimate the term $\cR_{\cT}$ in \eqref{reminder_terms_expansion},  decomposed as
\begin{equation}\nonumber
    \begin{aligned}
        \cR_{\cT}=\sum_{2\leq p+q \leq l}\cR_{\cT}\pa{p,q}+R_{l+1},
    \end{aligned}
\end{equation}
where the terms $ \cR_{\cT}\pa{p,q}$ are defined as 
\begin{equation}\label{explicit_reminder_terms_before_expansions}
    \begin{aligned}
        \cR_{\cT}\pa{p,q}=-\frac{c_{\cT}}{ND}\sum_{a=1}^D\sum_{\alpha,\beta\in \cI_a}\frac{1}{p!q!}\cC_{\iab}^{p,q+1}\partial_{\iab}^p\partial_{\iba}^q\qa{\p{\sM B \tLambda_o \Pi_{k_{\#}-1}\sG}_{\iab} W_1 \cdots W_u f^{(1)} \cdots f^{(n-1)}}
    \end{aligned}
\end{equation}
and $o=1$ or $2$ depending on the structure of $\cT$. By choosing $l$ sufficiently large, we can ensure that the remainder term \(R_{l+1}\) is bounded by $\OO_{\prec}\p{N^{-2}\norm{A}_{\HS}^2}$. Therefore, in this context and in all subsequent cumulant expansions, we will omit the arguments used to control the remainder terms denoted by $R_{l+1}$.

%


The terms $\cR_{\cT}\pa{p,q}$ can be divided into two cases. Some of them can be bounded directly, while the remaining terms require a further expansion using a similar, but more refined, procedure. We begin by considering the first set of terms, which are easier to estimate.


\begin{proof}[\bf Proof of \Cref{claim_estimate_remainder_terms}: Direct Estimation]
We first consider all cases of the $\cR_{\cT}(p, q)$ terms that can be estimated directly. 
\medskip
    
\noindent(\rmnum{1}) Suppose that $\cT$ is of Type \rmnum{1}, with $k_1 \geq 1$, $k_2 \geq 1$, and at least one of the following conditions holds: $p+q\geq 3$, or $\sR\geq 1$. In this case, we have $o=1$, $k_{\#}=k_1+k_2$, and
        \begin{equation}\label{R_cT_typt_1_k_1_geq_1_k_2-geq_1_explicit_expression}
            \begin{aligned}
                \cR_{\cT}\p{p,q}=& -\frac{c_{\cT}}{ND}\sum_{\p{\txt{i}}}\sum_{a=1}^D\sum_{\alpha,\beta\in \cI_a}\frac{1}{p!q!}\cC_{\iab}^{p,q+1}\p{\sM B_1 \tl_1\Pi_{a_1}}_{\star\star}\p{\Pi_{a_2}}_{\star\star}\cdots\p{\Pi_{a_{s-2}}}_{\star\star}\p{\Pi_{a_{s-1}}\tl_2\Pi_{a_s}}_{\star\star}\\
                & \quad\times\prod_{l=1}^{u}\pa{\partial_{\iab}^{s_W\pa{l}}\partial_{\iba}^{t_W\pa{l}}W_l}\prod_{l=1}^{n-1}\pa{\partial_{\iab}^{s_f\pa{l}}\partial_{\iba}^{t_f\pa{l}}f^{\pa{l}}}\\
                &- \frac{c_{\cT}}{ND}\sum_{\pa{\txt{ii}}}\sum_{a=1}^D\sum_{\alpha,\beta\in \cI_a}\frac{1}{p!q!}\cC_{\iab}^{p,q+1}\p{\sM B_1 \tl_1\Pi_{a_1}\tl_2\Pi_{a_2}}_{\star\star}\p{\Pi_{a_3}}_{\star\star}\cdots\p{\Pi_{a_{s}}}_{\star\star}\\
                &\quad\times\prod_{l=1}^{u}\pa{\partial_{\iab}^{s_W\pa{l}}\partial_{\iba}^{t_W\pa{l}}W_l}\prod_{l=1}^{n-1}\pa{\partial_{\iab}^{s_f\pa{l}}\partial_{\iba}^{t_f\pa{l}}f^{\pa{l}}}.
            \end{aligned}
        \end{equation}
    Here, each $\star$ denotes either an $\alpha$ or a $\beta$; the quantities $s_W(l)$, $t_W(l)$, $s_f(l)$, and $t_f(l)$ represent certain non-negative integers; and $\Pi_{a_1}, \ldots, \Pi_{a_s}$ denote terms generated from the derivatives of \smash{$\p{\sM B_1 \tl_1 \Pi_{k_\#-1}\sG}_{\iab}$}, where $a_i$ indicates the number of $\sG$ factors in each term. 
    Note that we have 
    \be\label{eq:sumai}a_1+\cdots+a_s= k_1+k_2+s-2.\ee 
    The summations \smash{$\sum_{\pa{\txt{i}}}$ and $\sum_{\pa{\txt{ii}}}$} range over all possible structures generated by \smash{$\partial_{\iab}^p \partial_{\iba}^q$}. For simplicity of presentation, we also include the deterministic coefficients (of order $\OO(1)$) into the summations \smash{$\sum_{(\text{i})}$ and $\sum_{(\text{ii})}$}.
        
Using \Cref{lem_loc,lemma_necessary_estimates}, we get the bounds  
            \begin{align}
                & \absa{\cC_{\iab}^{p,q+1}}\lesssim N^{-\pa{p+q+1}/2},\quad \absa{\partial_{\iab}^{s_W\pa{l}}\partial_{\iba}^{t_W\pa{l}}W_l}\prec\frac{1}{N\eta},\quad \absa{\partial_{\iab}^{s_f\pa{l}}\partial_{\iba}^{t_f\pa{l}}f^{\pa{l}}}\prec \frac{\im m}{\eta^{r_l-1}},\nonumber\\
                & \abs{\p{\Pi_{a_{s-1}}\tl_2\Pi_{a_s}}_{\star\star}}\leq \norm{\bre_\star^{\top}\Pi_{a_{s-1}}\tl_2}\cdot\norma{\Pi_{a_s}\bre_\star}\prec \norm{\bre_\star^{\top}\Pi_{a_{s-1}}\tl_2}\cdot\sqrt{\frac{\im m}{\eta^{2a_s-1}}}, \label{bound_factors}\\
                & \abs{\p{\sM B \tl_1\Pi_{a_1}}_{\star\star}}\prec \norm{\bre_\star^\top\sM B \tl_1}\cdot\frac{1}{\eta^{a_1-1}},\quad \absa{\pa{\Pi_{a_l}}_{\star\star}}\prec \frac{1}{\eta^{a_l-1}} \ \ \txt{ for } \  \ 2\leq l\leq s-2,\nonumber
            \end{align}
            where recall that $r_l$ denotes the number of $\sG$ factors in $f^{\pa{l}}$. With these bounds, we can bound part (i) by
        \begin{equation}\label{concrete_bound_type_1_1_1}
            \begin{aligned}
                &~N^{-\pa{\mm-n+\sR-1}-1-\pa{p+q+1}/2}\cdot N\frac{\im m}{\eta^{k_1+k_2-1}}\norma{A}_{\HS}^2\cdot\pa{\frac{1}{N\eta}}^{u}\cdot\frac{\pa{\im m}^{n-1}}{\eta^{\mm-k_1-k_2-n+1}}\\
                \lesssim &~N^{1-\sR-\pa{p+q+1}/2}\norma{A}_{\HS}^2 \pa{k/N}^{\frac 1 3(\mm+u)}\leq N^{-5/3}k^{2/3}\norma{A}_{\HS}^2\leq N^{-1-2\varepsilon_A}.
            \end{aligned}
        \end{equation}
    To obtain the bound in the first line of \eqref{concrete_bound_type_1_1_1}, we also used the identity $S = \mm - n + \sR - 1$ from \eqref{counting_character_of_term}, the facts in \eqref{number_of_G_type_1} and \eqref{eq:sumai}, as well as the following bounds derived from the Cauchy–Schwarz inequality and \eqref{two_loop_estimate}:
        \begin{align}\label{Cauchy_example_1}
                \sum_{\star}\norm{\bre_\star^{\top}\Pi_{a_{s-1}}\tl_2}^2&=\tr\pa{\Pi_{a_{s-1}}\tl_2^2\Pi_{a_{s-1}}^*} \prec \norma{A}_{\HS}^2{\frac{\im m}{\eta^{2a_{s-1}-1}}},\\
        \label{Cauchy_example_2}
                \sum_{\star}\norm{\bre_\star^\top\sM B \tl_1}^2&=\tr\pa{ \sM B \tl_1^2B^* \sM^*}\lesssim \norma{A}_{\HS}^2.
            \end{align}
        In the first inequality of \eqref{concrete_bound_type_1_1_1}, we used \eqref{square_root_density}, \eqref{eq:kappa_gamma_k}, and arguments similar to those in \eqref{estimate_loop_type_1} and  \eqref{light_weight_cancel_imaginary_part}, together with the condition  $\mm-k_1-k_2\geq 2\pa{n-1}$. In the second step of \eqref{concrete_bound_type_1_1_1}, we used the assumption $p + q \geq 3$ or $\sR \geq 1$, along with the fact that $\mm + u \geq k_1 + k_2 \geq 2$. 
        For part (ii), we bound the corresponding factor as 
        \begin{equation}\label{R_cT_typt_1_k_1_geq_1_k_2-geq_1_operator_norm_bound}
            \begin{aligned}
                \abs{\p{\sM B_1 \tl_1\Pi_{a_1}\tl_2\Pi_{a_2}}_{\star\star}}\leq \norm{\bre_\star^\top\sM B \tl_1\Pi_{a_1}}\cdot\norm{\tl_2\Pi_{a_2}\bre_\star},
            \end{aligned}
        \end{equation}
       and estimate the remaining factors similarly to \eqref{bound_factors}. Then, we see that part (ii) is bounded in essentially the same way as \eqref{concrete_bound_type_1_1_1}:
        \begin{equation}\label{concrete_bound_type_1_1_2}
            \begin{aligned}
                &~N^{-\pa{\mm-n+\sR-1}-1-\pa{p+q+1}/2}\cdot N\frac{\im m}{\eta^{k_1+k_2-1}}\norma{\Lambda}_{\HS}^2\cdot\pa{\frac{1}{N\eta}}^{u}\cdot\frac{\pa{\im m}^{n-1}}{\eta^{\mm-k_1-k_2-n+1}} \lesssim  N^{-5/3}k^{2/3}\norma{A}_{\HS}^2 . 
            \end{aligned}
        \end{equation}

        \noindent(\rmnum{2}) Suppose that $\cT$ is of Type \rmnum{1} with $k_1=0$ and $ k_2\geq 1$. In this case, we have $o=2,\ k_{\#}=k_2$, $\sR\geq 1$, $\mm+u\geq k_2\geq 1$. The term $\cR_{\cT}(p, q)$ can then be written as
        \begin{equation}\nonumber
            \begin{aligned}
                \cR_{\cT}\pa{p,q}=&- \frac{c_{\cT}}{ND}\sum_{\pa{\txt{i}}}\sum_{a=1}^D\sum_{\alpha,\beta\in \cI_a}\frac{1}{p!q!}\cC_{\iab}^{p,q+1}\p{\sM B_1 \tl_2 B_2\tl_1\Pi_{a_1}}_{\star\star}\p{\Pi_{a_2}}_{\star\star}\cdots\p{\Pi_{a_{s}}}_{\star\star}\\
                &\quad \times\prod_{l=1}^{u}\pa{\partial_{\iab}^{s_W\pa{l}}\partial_{\iba}^{t_W\pa{l}}W_l}\prod_{l=1}^{n-1}\pa{\partial_{\iab}^{s_f\pa{l}}\partial_{\iba}^{t_f\pa{l}}f^{\pa{l}}},
            \end{aligned}
        \end{equation}
    using notation analogous to that in \eqref{R_cT_typt_1_k_1_geq_1_k_2-geq_1_explicit_expression}, where \smash{$\sum_{\text{(i)}}$} again denotes the summation over all possible structures generated by the derivatives \smash{$\partial_{\iab}^p\partial_{\iba}^q$}. 
        The factor \smash{$\p{\sM B_1 \tl_2 B_2\tl_1\Pi_{a_1}}_{\star\star}$} satisfies a bound similar to \eqref{R_cT_typt_1_k_1_geq_1_k_2-geq_1_operator_norm_bound}, and the remaining factors satisfy bounds analogous to those in \eqref{bound_factors}.
        Using these bounds and applying the Cauchy–Schwarz inequality as in \eqref{concrete_bound_type_1_1_1}, we can estimate $\cR_{\cT}(p, q)$ as
         \begin{equation}\label{concrete_bound_type_1_2_3}
             \begin{aligned}
                 \absa{\cR_{\cT}\pa{p,q}}\prec N^{-\pa{\mm-n+\sR-1}-1-\pa{p+q+1}/2}\cdot \pa{N\norma{A}_{\HS}^2\frac{1}{\eta^{k_1+k_2-1}}\sqrt{\frac{\im m}{\eta}}}\cdot \pa{\frac{1}{N\eta}}^{u}\cdot\frac{\pa{\im m}^{n-1}}{\eta^{\mm-k_1-k_2-n+1}}.
             \end{aligned}
         \end{equation}
         If at least one of the following conditions does {\bf not} hold: $\sR=1$, $p+q=2$, or $\mm+u=1$, then, following a similar argument to those in \eqref{estimate_loop_type_1} and \eqref{light_weight_cancel_imaginary_part}, we obtain that 
         \begin{equation}\nonumber
             \begin{aligned}
                 N^{1-\sR-\pa{p+q+1}/2}N^{1/2}\norma{A}_{\HS}^2\pa{k/N }^{\frac 1 3(\mm+u)}\leq N^{-5/3}k^{2/3}\norma{A}_{\HS}^2\leq N^{-1-2\varepsilon_A}.
             \end{aligned}
         \end{equation}
     On the other hand, if all three conditions hold---namely, $\sR = 1$, $p + q = 2$, and $\mm + u = 1$---then from \eqref{counting_character_of_term}, it follows that $\sC_1 = \sC_2 = \sP_1 = \sP_2 = 0$. Hence, $\cT$ must take the form $\cT = \avg{\sM_0 \tl_1 \sG_1 \tl_2}$, and the term $\cR_{\cT}(p, q)$ can be written as 
         \begin{equation}\label{direct_estimate_type_1_k_1_0_k_2_geq_1_sR_1_p_plus_q_2_m_plus_u_1}
             \begin{aligned}
                 \cR_{\cT}(p,q)=-\frac{c_\cT}{ND}\sum_{\pa{\txt{i}}}\bsum{a}{\alpha}{\beta}\frac{1}{p!q!}\cC_{\iab}^{p,q+1}\p{\sM_1\tl_2\sM_0\tl_1\sG_1}_{\star\star}\pa{\sG_1}_{\star\star}\pa{\sG_1}_{\star\star}.
             \end{aligned}
         \end{equation}
         Noting that there is only one $\sM_0$ factor, we apply the polarization identity from \eqref{expansion_of_im} to obtain a cancellation.  Specifically, by summing the contributions from the four terms on the right-hand side of \eqref{expansion_of_im}, we recover expressions similar to those in \eqref{direct_estimate_type_1_k_1_0_k_2_geq_1_sR_1_p_plus_q_2_m_plus_u_1}, where the $\sM_0$ factor is replaced by $\im \sM_0$, and $\sM_1$ either remains unchanged or is replaced by $\sM_1^*$. In summary, the total contribution from these terms can be bounded by 
         \begin{equation}\nonumber
             \begin{aligned}
                 N^{-5/2}\cdot N\norma{\Lambda}_{\HS}^2\sqrt{\frac{\im m}{\eta}}\cdot \im m\lesssim N^{-5/3+\varepsilon/2}k^{2/3}\norma{A}_{\HS}^2\lesssim N^{-1-2\varepsilon_A+\varepsilon/2}.
             \end{aligned}
         \end{equation}

         \noindent(\rmnum{3}) Suppose that $\cT$ is of Type \rmnum{2}, with $k_1\geq 1$ and $k_2\geq 1$. Moreover, at least one of the following conditions holds: $p+q\geq 3$ or $\sR\geq 1$. In this case, we have $o=1$, $k_{\#}=k_1$, and $k_2\geq 2$. The condition \(k_2 \geq 2\) arises because, when the second loop containing \smash{$\tl_2$} is generated, it must include at least two \(\sG\) factors. Furthermore, in subsequent expansions, no $\sreplace$ operation is applied to this loop, so the number of \(\sG\) factors within it does not decrease.  Using notation similar to that in \eqref{R_cT_typt_1_k_1_geq_1_k_2-geq_1_explicit_expression}, we can express $\cR_{\cT}\pa{p,q}$ as 
         \begin{align*}
                 \cR_{\cT}\pa{p,q}=& -\frac{c_{\cT}}{\pa{ND}^2}\sum_{\pa{\txt{i}}}\sum_{a=1}^D\sum_{\alpha,\beta\in \cI_a}\frac{1}{p!q!}\cC_{\iab}^{p,q+1}\p{\sM B_1 \tl_1\Pi_{a_1}}_{\star\star}\p{\Pi_{a_2}}_{\star\star}\cdots\p{\Pi_{a_{s-2}}}_{\star\star}\p{\Pi_{a_{s-1}}\tl_2\Pi_{a_s}}_{\star\star}\\
                &\quad\times\prod_{l=1}^{u}\pa{\partial_{\iab}^{s_W\pa{l}}\partial_{\iba}^{t_W\pa{l}}W_l}\prod_{l=1}^{n-2}\pa{\partial_{\iab}^{s_f\pa{l}}\partial_{\iba}^{t_f\pa{l}}f^{\pa{l}}}\\
                &-\frac{c_{\cT}}{ND}\sum_{\pa{\txt{ii}}}\sum_{a=1}^D\sum_{\alpha,\beta\in \cI_a}\frac{1}{p!q!}\cC_{\iab}^{p,q+1}\p{\sM B_1 \tl_1\Pi_{a_1}}_{\star\star}\p{\Pi_{a_2}}_{\star\star}\cdots\p{\Pi_{a_{s-1}}}_{\star\star}\avg{\tl_2\Pi_{a_s}}\\
                & \quad\times\prod_{l=1}^{u}\pa{\partial_{\iab}^{s_W\pa{l}}\partial_{\iba}^{t_W\pa{l}}W_l}\prod_{l=1}^{n-2}\pa{\partial_{\iab}^{s_f\pa{l}}\partial_{\iba}^{t_f\pa{l}}f^{\pa{l}}}.
             \end{align*}
         As in \eqref{concrete_bound_type_1_1_1}, under the assumption $p+q\geq 3$ or $\sR\geq 1$, and given that $\mm+u\geq k_1+k_2\geq 3$, we can bound part (i) as
         \begin{equation}\label{concrete_bound_type_2_1_1}
             \begin{aligned}
                 &~N^{-\pa{\mm-n+\sR-1} -2 -\pa{p+q+1}/2}\cdot N\frac{\im m}{\eta^{k_1+k_2-1}}\norma{A}_{\HS}^2\cdot \pa{\frac{1}{N\eta}}^{u}\cdot \frac{\pa{\im m}^{n-2}}{\eta^{\mm-k_1-k_2-n+2}}\\
                 \lesssim  &~ N^{1-\sR-\pa{p+q+1}/2}\norma{A}_{\HS}^2\pa{k/N}^{\frac 1 3(\mm+u)}\leq N^{-5/3}k^{2/3}\norma{A}_{\HS}^2\leq N^{-1-2\varepsilon_A},
             \end{aligned}
         \end{equation}
         and bound part (ii) as
         \begin{equation}\label{concrete_bound_type_2_1_2}
             \begin{aligned}
                 &~N^{-\pa{\mm-n+\sR-1} -1 -\pa{p+q+1}/2}\cdot N\frac{\im m}{\eta^{k_1+k_2-2}}\norma{A}_{\HS}^2\cdot \pa{\frac{1}{N\eta}}^{u}\cdot \frac{\pa{\im m}^{n-2}}{\eta^{\mm-k_1-k_2-n+2}}\\
                 \lesssim &~N^{1-\sR-\pa{p+q+1}/2}\norma{A}_{\HS}^2\pa{k/N }^{\frac 1 3\p{\mm+u-1}}\leq N^{-5/3}k^{2/3}\norma{A}_{\HS}^2\leq N^{-1-2\varepsilon_A}.
             \end{aligned}
         \end{equation}

\smallskip
         \noindent(\rmnum{4}) 
         Suppose that $\cT$ is of Type \rmnum{2}, with $k_1=0$ and $k_2\geq 1$. In this case, we have $o=2,\ k_{\#}=k_2$, and $\sR\geq 1$. Using notation similar to that in \eqref{R_cT_typt_1_k_1_geq_1_k_2-geq_1_explicit_expression}, we can express $\cR_{\cT}\pa{p,q}$ as 
         \begin{equation}\nonumber
             \begin{aligned}
                 \cR_{\cT}\pa{p,q}=&- \frac{c_{\cT}}{ND}\sum_{\pa{\txt{i}}}\sum_{a=1}^D\sum_{\alpha,\beta\in \cI_a}\frac{1}{p!q!}\cC_{\iab}^{p,q+1}\p{\sM B_1 \tl_2\Pi_{a_1}}_{\star\star}\p{\Pi_{a_2}}_{\star\star}\cdots\p{\Pi_{a_{s}}}_{\star\star}\avg{\tl_1 B_2}\\
                 &\times\prod_{l=1}^{u}\pa{\partial_{\iab}^{s_W\pa{l}}\partial_{\iba}^{t_W\pa{l}}W_l}\prod_{l=1}^{n-2}\pa{\partial_{\iab}^{s_f\pa{l}}\partial_{\iba}^{t_f\pa{l}}f^{\pa{l}}}.
             \end{aligned}
         \end{equation}
          If $k_2\geq 2$ and at least one of the following conditions does {\bf not} hold: $\sR=1,\ p+q=2$, or $ \mm+u=2$, then, similar to \eqref{concrete_bound_type_1_2_3}, we can bound $\cR_{\cT}\pa{p,q}$ by 
         \begin{equation}\label{concrete_bound_type_2_2_1}
             \begin{aligned}
                 &~ N^{-\pa{\mm-n+\sR-1}-1-\pa{p+q+1}/2}\cdot N^{3/2}\frac{\im m}{\eta^{k_1+k_2-1}}\norma{A}_{\HS}\avga{\Lambda^2}\cdot \pa{\frac{1}{N\eta}}^{u}\cdot \frac{\pa{\im m}^{n-2}}{\eta^{\mm-k_1-k_2-n+2}}\\
                 \lesssim &~ N^{3/2-\sR-\pa{p+q+1}/2}\cdot{N^{1/3-\varepsilon_A}k^{-1/3}}\norma{A}_{\HS}^2\pa{k/N}^{\frac 13 (\mm+u)}\leq N^{-5/3-\varepsilon_A}k^{2/3}\norma{A}_{\HS}^2\leq N^{-1-3\varepsilon_A}.
             \end{aligned}
         \end{equation}
         Conversely, if $k_2\ge 2$, $\sR=1$, $p+q=2$, and $\mm+u=2$, then, by \eqref{counting_character_of_term}, we have $\sC_1+\sC_2+\sP_1+\sP_2=1$. To obtain a Type \rmnum{2} expression in this case, we must have  $\sC_2=1$ and $\sC_1=\sP_1=\sP_2=0$. Consequently,  $\cT$ must take the following form:
         \begin{equation}\nonumber
             \begin{aligned}
                 \cT=D\dsum{a}\avg{\sM_0\tl_1\sM_1E_a}\avg{E_a\sG_1\tl_2\sG_0}.
             \end{aligned}
         \end{equation}
        Then, we can apply the estimate \eqref{regular_estimate_localized_eigenvector} to refine our bound as: 
         \begin{equation}\nonumber
             \begin{aligned}
                 \absa{\cR_{\cT}\pa{p,q}}\prec N^{-1 -\pa{p+q+1}/2}\cdot N^{3/2}\frac{\im m}{\eta}\norma{A}_{\HS}\cdot (\im m)\avga{\Lambda^2}\lesssim N^{-5/3-\varepsilon_A}k^{2/3}\norm{A}_{\HS}^2\leq N^{-1-3\varepsilon_A}.
             \end{aligned}
         \end{equation}
         When $k_2=1$, $\cR_\cT\pa{p,q}$ can be bounded as follows: 
         \begin{equation}\nonumber
             \begin{aligned}
                 & ~N^{-\pa{\mm-n+\sR-1}-1-\pa{p+q+1}/2}\cdot N^{3/2}\norma{\Lambda}_{\HS}\avga{\Lambda^2}\cdot \pa{\frac{1}{N\eta}}^{u}\cdot \frac{\pa{\im m}^{n-2}}{\eta^{\mm-n+1}}\\
                 \lesssim &~N^{3/2-\sR-\pa{p+q+1}/2}\cdot N^{1/3-\varepsilon_A}k^{-1/3}\norma{A}_{\HS}^2\pa{k/N}^{\frac 1 3 (\mm+u-1)}\leq N^{-5/3-\varepsilon_A}k^{2/3}\norma{A}_{\HS}^2\leq N^{-1-3\varepsilon_A},
             \end{aligned}
         \end{equation}
         unless one of the following scenarios occurs: (i) $\sR=1$, $p+q=3$, and $\mm+u\leq 2$; or (ii) $\sR=1$, $p+q=2$, and $\mm+u\leq 3$. A direct enumeration shows that, in scenario (i), the conditions $\sR=1$ and $\mm+u\leq 2$ imply $\sC_2=1$ and $\sC_1=\sP_1=\sP_2=0$, which contradicts $k_2=1$. Therefore, only scenario (ii) can occur, where $\cT$ must satisfy $\sC_1+\sC_2+\sP_1=2$ and $\sC_2\geq 1$. 
         Moreover, by reasoning similar to that used for the condition $k_2 \geq 2$ in case (\rmnum{3}), we again find that $k_2 \geq 2$ whenever $\sC_2 = 1$. This contradicts the assumption that $k_2 = 1$. Consequently, we must have $\sC_2 = 2$ and $\sC_1 = \sP_1 = \sP_2 = 0$, in which case $\cT$ must take the following form:
         \begin{equation}\nonumber
             \begin{aligned}
                 \cT=D^2\dsum{a,b}\avg{\sM_0\tl_1\sM_1 E_a}\avg{\sM_1\tl_2\sG_0 E_b}\avg{E_b\sG_0 E_a \sG_1}.
             \end{aligned}
         \end{equation}
         Then, we again apply the bound \eqref{regular_estimate_localized_eigenvector} to improve the estimate of $\cR_{\cT}\pa{p,q}$ as:
         \begin{equation}\nonumber
             \begin{aligned}
                 \absa{\cR_{\cT}\pa{p,q}}\prec N^{-1-\pa{p+q+1}/2}\cdot N^{3/2}\norma{\Lambda}_{\HS}\cdot (\im m)\avga{\Lambda^2}\cdot \frac{\im m}{\eta}\lesssim N^{-5/3-\varepsilon_A}k^{2/3}\norma{A}_{\HS}^2\leq N^{-1-3\varepsilon_A}.
             \end{aligned}
         \end{equation}

    Combining all the above Cases (\rmnum{1})-(\rmnum{4}) completes the first part of the proof of \Cref{claim_estimate_remainder_terms}.
\end{proof}

Based on the previous discussion, it remains to consider cases satisfying one of the following conditions: 
    \begin{enumerate}
        \item $\cT$ is of Type \rmnum{1}, with $\sR=0$, $k_1\geq 1$, $k_2\geq 1$, and $p+q=2$.
        \item $\cT$ is of Type \rmnum{2}, with $\sR=0$, $k_1\geq 1$, $k_2\geq 1$, and $p+q=2$.
    \end{enumerate}
To complete the proof of \Cref{claim_estimate_remainder_terms}, we need to perform further expansions on the terms satisfying these conditions.

\begin{proof}[\bf Proof of \Cref{claim_estimate_remainder_terms}: Further Expansions]\label{proof_further_expansion}
We begin by describing the expansion strategy for the two types of remainder terms that satisfy conditions (i) or (ii). We introduce a class of expressions that will appear from our expansions:
    \begin{equation}\nonumber
        \begin{aligned}
            \cR:\ c_{\cR}\cdot \cW^{\pa{u}}\cdot \Upsilon_n^{\pa{\mm}},
        \end{aligned}
    \end{equation}
    where $\cW^{\pa{u}}$ is defined as in \eqref{definition_class_of_expression_loop}, while $\Upsilon_n^{\pa{\mm}}$ has a structure given by one of the following forms:
        \begin{align}\label{further_expansion_expression_Type_1}
         \textbf{Type \rmnum{1}:}\quad       &-\frac{1}{ND}\bsum{a}{\alpha}{\beta}\frac{1}{p_0!q_0!}\cC_{\iab}^{p_0,q_0+1}\p{\sM B\tl_1 \Pi_{a_1}}_{\star\star}\p{\Pi_{a_2}\tl_2\Pi_{a_3}}_{\star\star}\p{\Pi_{a_4}}_{\star\star}\prod_{i=1}^{n-3}f^{\pa{i}};\\
         \label{further_expansion_expression_Type_2}
        \textbf{Type \rmnum{2}:}\quad        &-\frac{1}{ND}\bsum{a}{\alpha}{\beta}\frac{1}{p_0!q_0!}\cC_{\iab}^{p_0,q_0+1}\p{\sM B\tl_1 \Pi_{a_1}}_{\star\star}\p{\Pi_{a_2}}_{\star\star}\p{\Pi_{a_3}}_{\star\star}\avg{\tl_2\Pi_{a_4}}\prod_{i=1}^{n-4}f^{\pa{i}};    \\
       \label{further_expansion_expression_Type_3}
         \textbf{Type \rmnum{3}:}\quad       &-\frac{1}{ND}\bsum{a}{\alpha}{\beta}\frac{1}{p_0!q_0!}\cC_{\iab}^{p_0,q_0+1}\p{\sM B\tl_1 \Pi_{a_1}\tl_2\Pi_{a_2}}_{\star\star}\p{\Pi_{a_3}}_{\star\star}\p{\Pi_{a_4}}_{\star\star}\prod_{i=1}^{n-3}f^{\pa{i}}.
            \end{align}
Here, each $f^{(i)}$ denotes a loop, defined in the same manner as $f^{(j)}$ in \eqref{explain_product_entries}. The terms $\Pi_{a_i}$ are defined similarly to those in \eqref{R_cT_typt_1_k_1_geq_1_k_2-geq_1_explicit_expression}, where $a_i$ indicates the number of $\sG$ factors within $\Pi_{a_i}$; each $a_i$ is nonzero unless it appears in a factor containing \smash{$\tl$}. Every expression in \eqref{further_expansion_expression_Type_1}--\eqref{further_expansion_expression_Type_3} contains six $\star$ placeholders, representing three $\alpha$ indices and three $\beta$ indices. 
Then, $n$ denotes the number of factors in \smash{$\Upsilon_{n}^{\pa{\mm}}$}, and $\mm$ is the total number of $\sG$ entries in \smash{$\Upsilon_{n}^{\pa{\mm}}$}. 
If $\cR$ is of Type \rmnum{1} or Type \rmnum{2}, we denote by $k_1$ and $k_2$ the number of $\sG$ entries within the factors containing \smash{$\tl_1$} and \smash{$\tl_2$}, respectively.  If $\cR$ is of Type \rmnum{3}, then $k_1$ denotes the number of $\sG$ entries between \smash{$\tl_1$} and \smash{$\tl_2$}, and $k_2$ denotes the number of entries to the right of \smash{$\tl_2$}.
For simplicity of presentation, we refer to all factors of the form $(\cdot)_{\star\star}$ as \emph{heavy packages}, and denote the class of expressions of the forms \eqref{further_expansion_expression_Type_1}–\eqref{further_expansion_expression_Type_3} by $\scR$.
 

    We are now ready to describe the expansion procedure. Clearly, for any $p_0 + q_0 = 2$ and $\cT \in \scT$, we have $\cR_0 := \cR_{\cT}(p_0, q_0) \in \scR$. Given any $\cR \in \scR$, we choose a $\sG$ entry to expand according to the following rules:
    \begin{enumerate}
        \item (Right of \smash{$\tl_2$} in a heavy package) If \smash{$\tl_2$} is contained in a heavy package and there is a $\sG$ factor to the right of \smash{$\tl_2$} within this package, we choose the first such $\sG$.
        
        
        \item (Left of \smash{$\tl_2$} in a loop) If condition (i) does not apply, and \smash{$\tl_2$} is contained in a loop that includes at least one $\sG$ factor, we choose the first $\sG$ to the left of \smash{$\tl_2$} in that loop.
                
        
        \item (Right of \smash{$\tl_1$} in a heavy package) If neither (i) nor (ii) applies, and there is a $\sG$ factor to the right of \smash{$\tl_1$} within the heavy package containing \smash{$\tl_1$} (note that \smash{$\tl_1$} must be contained in a heavy package, and there is no $\sG$ to its left), we choose the first such $\sG$.


        \item (Left of \smash{$\tl_2$} in a heavy package) If none of (i)--(iii) applies, and there is a $\sG$ factor to the left of \smash{$\tl_2$} within the heavy package containing \smash{$\tl_2$} (note that if (ii) fails, then \smash{$\tl_2$} must be in a heavy package), we choose the first such $\sG$.
        
        \item If none of the above conditions (i)--(iv) holds, we stop expanding $\cR$.
        
    \end{enumerate}
    Next, we apply the expansion $\sG=\sM-\sM\pa{H+\smm}\sG$ if the selected $\sG$ is to the right of the considered \smash{$\tl_o$} for $o \in \{1, 2\}$,  and $\sG=\sM-\sG\pa{H+\smm}\sM$ if the selected $\sG$ is to the left of the corresponding \smash{$\tl_o$}. After this, we apply the cumulant expansion from \Cref{lem:complex_cumu} with respect to the entries of $H$. 
    

First, suppose that the considered $\tl_o$ is contained in a heavy package, and that $\cR$ is of Type \rmnum{1} or \rmnum{3}. As a representative example, assume $\cR$ is of Type \rmnum{1}, and there exists a $\sG$ factor to the right of \smash{$\tl_2$}. In this case, we can write $\cR$ as follows:
    \begin{equation}\label{explicit_reminder_term_1}
        \begin{aligned}
            \cR=-\frac{c_{\cR}}{ND}\bsum{a}{\alpha}{\beta}\frac{1}{p_0!q_0!}\cC_{\iab}^{p_0,q_0+1}\p{\Pi_1\tl_2B_1\sG\Pi_2}_{\star_1\star_2}\prod_{i=1}^{2}g^{\pa{i}}\cdot \prod_{i=1}^{n-3}f^{\pa{i}}\cdot \prod_{i=1}^{u}W_i,
        \end{aligned}
    \end{equation}
where $B_1$ denotes the product of deterministic matrices between \smash{$\tl_2$} and the first $\sG$ to its right; $\Pi_1$ and $\Pi_2$ denote the products of matrices to the left and right of \smash{$\tl_2 B_1 \sG$}, respectively; and $g^{(i)}$, for $i \in \{1,2\}$, denote other heavy packages within $\cR$. We then expand the $\sG$ factor to the right of \smash{$\tl_2$} and apply the cumulant expansion, yielding Gaussian integration by parts terms as well as a remainder term \smash{$\cE_\cR^{(2)}$} involving higher-order cumulants:
        \begin{align}\label{concrete_expansion_tl_2_heavy_package_right}
             \cR\eq  -&\frac{c_{\cR}}{ND}\bsum{a}{\alpha}{\beta}\frac{1}{p_0!q_0!}\cC_{\iab}^{p_0,q_0+1}\p{\Pi_1\tl_2B_1\sM\Pi_2}_{\star_1\star_2}\prod_{i=1}^{2}g^{\pa{i}}\cdot\prod_{i=1}^{n-3}f^{\pa{i}}\cdot\prod_{i=1}^{u}W_i\notag\\
            -&   \frac{c_{\cR}}{N} \sum_{x=1}^D \sum_{j=1}^{n_{F}}\bsum{a}{\alpha}{\beta}\frac{1}{p_0!q_0!}\cC_{\iab}^{p_0,q_0+1}\p{F_0\cdots F_{i\pa{j}}E_x\sG\Pi_2}_{\star_1\star_2}\notag\\
            &\qquad\times \avga{E_xF_{i\pa{j}}F_{i\pa{j}+1}\cdots F_{s}}\prod_{i=1}^{2}g^{\pa{i}}\cdot\prod_{i=1}^{n-3}f^{\pa{i}}\cdot\prod_{i=1}^{u}W_i\notag \\
            -& \frac{c_\cR}{N} \sum_{x=1}^D \sum_{j=n_F+1}^{n_F+m_F}\bsum{a}{\alpha}{\beta}\frac{1}{p_0!q_0!}\cC_{\iab}^{p_0,q_0+1}\p{F_0\cdots F_{t} E_x F_{i\pa{j}}F_{i\pa{j}+1}\cdots F_{s+t}}_{\star_1\star_2}\notag\\
            &\qquad\times \avga{E_x\sG F_{s+1}\cdots F_{i\pa{j}}} \prod_{i=1}^{2}g^{\pa{i}}\cdot\prod_{i=1}^{n-3}f^{\pa{i}}\cdot\prod_{i=1}^{u}W_i\notag\\
            -& \frac{c_\cR}{N} \sum_{x=1}^D\bsum{a}{\alpha}{\beta}\frac{1}{p_0!q_0!}\cC_{\iab}^{p_0,q_0+1}\p{\Pi_1\tl_2B_1\sM E_x \sG\Pi_2}_{\star_1\star_2}\avga{E_x\pa{\sG-\sM}}\prod_{i=1}^{2}g^{\pa{i}}\cdot\prod_{i=1}^{n-3}f^{\pa{i}}\cdot\prod_{i=1}^{u}W_i \notag\\
            -& \frac{c_\cR}{D^2 N^3} \sum_{x=1}^D \sum_{j=1}^u\bsum{a}{\alpha}{\beta}\frac{1}{p_0!q_0!}\cC_{\iab}^{p_0,q_0+1}\p{\Pi_1\tl_2B_1\sM E_x\sG_{w_j}E_{x_j}\sG_{w_j}E_x \sG\Pi_2}_{\star_1\star_2}\prod_{i=1}^{2}g^{\pa{i}}\cdot\prod_{i=1}^{n-3}f^{\pa{i}}\cdot \prod_{i \neq j} W_i \notag\\
            -& \frac{c_\cR}{D N^2} \sum_{x=1}^D \sum_{j=1}^2 \sum_{r=1}^{t_j}\bsum{a}{\alpha}{\beta}\frac{1}{p_0!q_0!}\cC_{\iab}^{p_0,q_0+1}\p{\Pi_1\tl_2B_1\sM E_x g_{i_{g,j}\pa{r}}^{\pa{j}} g_{i_{g,j}\pa{r}+1}^{\pa{j}}\cdots g_{n_{g,j}}^{\pa{j}}}_{\star_1\star_4}\notag\\
            &\qquad\times \p{ g_{0}^{\pa{j}} g_{1}^{\pa{j}}\cdots g_{i_{g,j}\pa{r}}^{\pa{j}} E_x \sG\Pi_2}_{\star_3\star_2} \prod_{i\neq j}g^{\pa{i}}\cdot\prod_{i=1}^{n-3}f^{\pa{i}}\cdot\prod_{i=1}^{u}W_i\notag\\
            -& \frac{c_\cR}{D^2 N^3} \sum_{x=1}^D \sum_{j=1}^{n-3} \sum_{r=1}^{s_j}\bsum{a}{\alpha}{\beta}\frac{1}{p_0!q_0!}\cC_{\iab}^{p_0,q_0+1}(\Pi_1\tl_2B_1\sM E_x f_{i_{f,j}\pa{r}}^{\pa{j}} f_{i_{f,j}\pa{r}+1}^{\pa{j}}\cdots f_{n_{f,j}}^{\pa{j}}\notag\\
            &\qquad\times f_{0}^{\pa{j}} f_{1}^{\pa{j}}\cdots f_{i_{f,j}\pa{r}}^{\pa{j}} E_x \sG\Pi_2)_{\star_1\star_2} \prod_{i=1}^{2}g^{\pa{i}}\cdot\prod_{i\neq j}f^{\pa{i}}\cdot\prod_{i=1}^{u}W_i+\cE_{\cR}^{\pa{2}},
        \end{align}
    where the corresponding factors are denoted using notation similar to that in \eqref{explain_product_entries}: 
    \begin{equation}\nonumber
        \begin{aligned}
        &\Pi_{1}\tLambda_2 B_1\sM =: F_0 \cdots F_s,\quad \Pi_2=F_{s+1}\cdots F_{s+t},\quad W_j=\left\langle\left(\sG_{w_j}-\sM_{w_j}\right) E_{x_j}\right\rangle, \\
        &f^{(j)}=\avg{ f_0^{(j)} f_1^{(j)} \cdots f_{n_{f,j}}^{(j)}},\quad g^{(j)}=( g_0^{(j)} g_1^{(j)} \cdots g_{n_{g,j}}^{(j)})_{\star_3\star_4},
        \end{aligned}
    \end{equation}
Moreover, suppose $F_{i(1)}, \ldots, F_{i(n_F)}$ and $F_{i(n_F+1)}, \ldots, F_{i(n_F + m_F)}$ denote the $\sG$ factors in the products $F_0 \cdots F_s$ and $F_{s+1} \cdots F_{s+t}$, respectively. Similarly, let \smash{$f_{i_{f,j}\pa{1}}^{(j)}, \ldots f_{i_{f,j}\pa{s_j}}^{(j)}$} and \smash{$g_{i_{g,j}\pa{1}}^{(j)}, \ldots g_{i_{g,j}\pa{t_j}}^{(j)}$} denote the $\sG$ factors in $f^{(j)}$ and $g^{(j)}$, respectively. All remaining factors represent deterministic matrices, which are products involving $\sM$, $E_a$, and \smash{$\tl_i$}. The remainder term in \eqref{concrete_expansion_tl_2_heavy_package_right} is given by
        \begin{align*}
            \cE_{\cR}^{\pa{2}}=&\frac{c_{\cR}}{ND}\bsum{a}{\alpha}{\beta}\frac{1}{p_0!q_0!}\cC_{\iab}^{p_0,q_0+1}\sum_{2\leq p+q\leq l}\sum_{a=1}^D\sum_{i,j\in \cI_a}\frac{1}{p!q!}\cC_{ij}^{p,q+1}\partial_{ij}^p\partial_{ji}^q\left[\p{\Pi_1\tl_2B_1\sM}_{\star_1j}\p{\sG\Pi_2}_{i\star_2}\right.\\
            &\left.\times \prod_{r=1,2}g^{\pa{r}}\cdot \prod_{r=1}^{n-3}f^{\pa{r}}\cdot \prod_{r=1}^{u}W_r\right]+R_{l+1}^{\pa{2}},
        \end{align*}
where \smash{$R_{l+1}^{(2)}$} is a sufficiently small error term when $l$ is chosen to be sufficiently large, as previously discussed. In general, the expansion of heavy packages always yields an expression of the same structure as in \eqref{concrete_expansion_tl_2_heavy_package_right}.
    
    
 Second, suppose $\cR$ is of Type \rmnum{2} and a $\sG$ entry in the loop containing $\tl_2$ is chosen. We can express $\cR$ as
    \begin{equation}\nonumber
        \begin{aligned}
            \cR=-\frac{c_{\cR}}{ND}\bsum{a}{\alpha}{\beta}\frac{1}{p_0!q_0!}\cC_{\iab}^{p_0,q_0+1}\avg{\sG B_2\tl_2\Pi_1}\prod_{i=1}^{3}g^{\pa{i}}\cdot \prod_{i=1}^{n-4}f^{\pa{i}}\cdot \prod_{i=1}^{u}W_i,
        \end{aligned}
    \end{equation}
with the notations defined similarly to those in \eqref{explicit_reminder_term_1}. By expanding the $\sG$ factor to the left of \smash{$\tl_2$} and applying the cumulant expansion, we obtain an expression analogous to \eqref{concrete_expansion_tl_2_heavy_package_right}: 
        \begin{align}
            \cR\eq -&\frac{c_{\cR}}{ND}\bsum{a}{\alpha}{\beta}\frac{1}{p_0!q_0!}\cC_{\iab}^{p_0,q_0+1}\avg{\sM B_2\tl_2\Pi_1}\prod_{i=1}^{3}g^{\pa{i}}\cdot\prod_{i=1}^{n-4}f^{\pa{i}}\cdot \prod_{i=1}^{u}W_i\notag\\
            -&   \frac{c_{\cR}}{N} \sum_{x=1}^D \sum_{j=1}^{n_{F}}\bsum{a}{\alpha}{\beta}\frac{1}{p_0!q_0!}\cC_{\iab}^{p_0,q_0+1}\left\langle F_0 F_1 \cdots F_{i_j} E_x\right\rangle\left\langle E_x F_{i_j} F_{i_j+1} \cdots F_t \sG\right\rangle\prod_{i=1}^{3}g^{\pa{i}}\cdot\prod_{i=1}^{n-4}f^{\pa{i}}\cdot\prod_{i=1}^{u}W_i\notag \\
            -& \frac{c_{\cR}}{N} \sum_{x=1}^D\bsum{a}{\alpha}{\beta}\frac{1}{p_0!q_0!}\cC_{\iab}^{p_0,q_0+1}\avg{\sM B_2 \tl_2 \Pi_1 \sG E_x}\avga{E_x\pa{\sG-\sM}}\prod_{i=1}^{3}g^{\pa{i}}\cdot\prod_{i=1}^{n-4}f^{\pa{i}}\cdot\prod_{i=1}^{u}W_i\notag \\
            -& \frac{c_{\cR}}{D^2 N^3} \sum_{x=1}^D \sum_{j=1}^u\bsum{a}{\alpha}{\beta}\frac{1}{p_0!q_0!}\cC_{\iab}^{p_0,q_0+1}\avg{\sM B_2\tl_2\Pi_1\sG E_x\sG_{w_j}E_{x_j}\sG_{w_j}E_x}\prod_{i=1}^{3}g^{\pa{i}}\prod_{i=1}^{n-4}f^{\pa{i}} \prod_{i \neq j} W_i\notag \\
            -& \frac{c_{\cR}}{D^2 N^3} \sum_{x=1}^D \sum_{j=1}^3 \sum_{r=1}^{t_j}\bsum{a}{\alpha}{\beta}\frac{1}{p_0!q_0!}\cC_{\iab}^{p_0,q_0+1}\p{g_{0}^{\pa{j}} g_{1}^{\pa{j}}\cdots g_{i_{g,j}\pa{r}}^{\pa{j}} E_x\sM B_2\tl_2 \Pi_1 \sG E_x g_{i_{g,j}\pa{r}}^{\pa{j}} g_{i_{g,j}\pa{r}+1}^{\pa{j}}\cdots g_{n_{g,j}}^{\pa{j}}}_{\star_1\star_2}\notag\\
            &\qquad\times \prod_{i\neq j}g^{\pa{i}}\cdot\prod_{i=1}^{n-4}f^{\pa{i}}\cdot\prod_{i=1}^{u}W_i\notag\\
            -& \frac{c_{\cR}}{D^2 N^3} \sum_{x=1}^D \sum_{j=1}^{n-4} \sum_{r=1}^{s_j}\bsum{a}{\alpha}{\beta}\frac{1}{p_0!q_0!}\cC_{\iab}^{p_0,q_0+1}\avg{\sM B_2\tl_2 \Pi_1 \sG E_x f_{i_{f,j}\pa{r}}^{\pa{j}} f_{i_{f,j}\pa{r}+1}^{\pa{j}}\cdots f_{n_{f,j}}^{\pa{j}}f_{0}^{\pa{j}} f_{1}^{\pa{j}}\cdots f_{i_{f,j}\pa{r}}^{\pa{j}} E_x}\notag\\
            &\qquad\times\prod_{i=1}^{3}g^{\pa{i}}\cdot\prod_{i\neq j}f^{\pa{i}}\cdot\prod_{i=1}^{u}W_i+\cE_{\cR}^{\pa{2}}.\label{concrete_expansion_tl_2_loop}
        \end{align}
    This includes the Gaussian integration by parts terms, as well as a remainder term \smash{$\cE_\cR^{(2)}$} involving higher-order cumulants, where the relevant factors are denoted using notations similar to that in \eqref{explain_product_entries}: 
    \begin{equation}\nonumber
        \begin{aligned}
        &\sM B_2\tl_2 \Pi_{1} =: F_0 \cdots F_t, \quad W_j=\left\langle\left(\sG_{w_j}-\sM_{w_j}\right) E_{x_j}\right\rangle,\\
        & f^{(j)}=\avg{ f_0^{(j)} f_1^{(j)} \cdots f_{n_{f,j}}^{(j)}},\quad g^{(j)}=( g_0^{(j)} g_1^{(j)} \cdots g_{n_{g,j}}^{(j)})_{\star_1\star_2}.
        \end{aligned}
    \end{equation}
    Moreover, \smash{$F_{i\pa{1}}, \ldots, F_{i\pa{n_F}}$} represent the $\sG$ factors in $F_0\cdots F_{t}$, and  \smash{$f_{i_{f,j}\pa{1}}^{(j)}, \ldots, f_{i_{f,j}\pa{s_j}}^{(j)}$} and \smash{$g_{i_{g,j}\pa{1}}^{(j)}, \ldots, g_{i_{g,j}\pa{t_j}}^{(j)}$} denote the $\sG$ factors in $f^{(j)}$ and $g^{\pa{j}}$, respectively. The remainder term is expressed as
\begin{align*}
            \cE_{\cR}^{\pa{2}}=\frac{c_{\cR}}{N^2D^2}\bsum{a}{\alpha}{\beta}&\frac{1}{p_0!q_0!}\cC_{\iab}^{p_0,q_0+1} \sum_{2\leq p+q\leq l}\sum_{b=1}^D\sum_{i,j\in \cI_b}\frac{1}{p!q!}\cC_{ij}^{p,q+1}\\
            \times&\partial_{ij}^p\partial_{ji}^q\left[\p{\sM B_2\tl_2 \Pi_{1} \sG}_{ij} \prod_{r=1}^{3}g^{\pa{r}}\cdot \prod_{r=1}^{n-4}f^{\pa{r}}\cdot \prod_{r=1}^{u}W_r\right] + R_{l+1}^{\pa{2}},
        \end{align*}
    where \smash{$R_{l+1}^{(2)}$} again indicates a sufficiently small error term for sufficiently large $l$. 
    
    To proceed with the proof, we define the following operations derived from the expressions \eqref{concrete_expansion_tl_2_heavy_package_right}--\eqref{concrete_expansion_tl_2_loop}:
    \begin{enumerate}\centering
        \item[$\freplace$:] the first term in \eqref{concrete_expansion_tl_2_heavy_package_right}, and the first term in \eqref{concrete_expansion_tl_2_loop};
        \item[$\fcut{1}$:] the third term in \eqref{concrete_expansion_tl_2_loop};
        \quad $\fcut{2}$: the second term in \eqref{concrete_expansion_tl_2_loop};
        \item[$\fplug{1}$:] the fourth  term in \eqref{concrete_expansion_tl_2_loop};
        \quad$\fplug{2}$: the sixth term in \eqref{concrete_expansion_tl_2_loop};
        \item[$\fmerge$:] the fifth term in \eqref{concrete_expansion_tl_2_loop};
        \item[$\fslash{1}$:] the fourth term in \eqref{concrete_expansion_tl_2_heavy_package_right};
        \quad$\fslash{2}$: the second and third terms in \eqref{concrete_expansion_tl_2_heavy_package_right};
        \item[$\finsert{1}$:] the fifth term in \eqref{concrete_expansion_tl_2_heavy_package_right};
        \quad$\finsert{2}$: the seventh term in \eqref{concrete_expansion_tl_2_heavy_package_right};
        \item[$\fexchange$:] the sixth term in \eqref{concrete_expansion_tl_2_heavy_package_right}.
    \end{enumerate}
We then summarize how these operations affect the characters of our expressions in the following \Cref{table_effect_of_operations_2}.

    \begin{table}[htbp]
        \caption{Effects of operations}
        \label{table_effect_of_operations_2}
        \begin{tabular}{|c|c|c|c|c|}
        \hline
        \diagbox{Operation}{Character} & $\mm$ & $n$ & $u$ & $S$ \\ \hline
        $\freplace$ & $-1$ & $+0$ & $+0$ & $+0$ \\ \hline
        $\fcut{1}$ & $+0$ & $+0$ & $+1$ & $+0$ \\ \hline
        $\fcut{2}$ & $+1$ & $+1$ & $+0$ & $+0$ \\ \hline
        $\fplug{1}$ & $+2$ & $+0$ & $-1$ & $+2$ \\ \hline
        $\fplug{2}$ & $+1$ & $-1$ & $+0$ & $+2$ \\ \hline
        $\fmerge$ & $+1$ & $-1$ & $+0$ & $+2$ \\ \hline
        $\fslash{1}$ & $+0$ & $+0$ & $+1$ & $+0$ \\ \hline
        $\fslash{2}$ & $+1$ & $+1$ & $+0$ & $+0$ \\ \hline
        $\finsert{1}$ & $+2$ & $+0$ & $-1$ & $+2$ \\ \hline
        $\finsert{2}$ & $+1$ & $-1$ & $+0$ & $+2$ \\ \hline
        $\fexchange$ & $+1$ & $+0$ & $+0$ & $+1$ \\ \hline
        \end{tabular}
    \end{table}

    Recall that $\cT$ is generated by $\cT=\cO_{T}\circ\cdots\circ\cO_1\pa{\cT_0}$, where $\cT_0$ is an expression of the form \eqref{eq;expT0}, and $\cO_1, \ldots, \cO_T$ is an admissible sequence of operations defined as in \eqref{sreplace_def}--\eqref{splug2_def}.  
    Our goal is to estimate $\cR_0=\cR_\cT\pa{p_0,q_0}$, as defined in  \eqref{explicit_reminder_terms_before_expansions}, under the assumptions $p_0+q_0=2$ and $\sR=0$. 
    Depending on which $\sG$ factor the derivative $\partial_{\iab}$ or $\partial_{\iba}$ acts upon, we obtain different relations between the characters of $\cT$---denoted by $\mm_\cT, n_\cT, u_\cT,$ and $S_\cT$---and those of the expression \smash{$\cR_0=c_{\cR_0}\cdot \cW^{\pa{u_0}}\Upsilon_{n_0}^{\pa{\mm_0}}$}, whose characters are denoted by $\mm_0,\ n_0, \ u_0,$ and $S_0$. These relations are summarized in \Cref{table_Classification_of_Initial_Values}. 
    We emphasize that $S_0$ includes only the $N^{-1}$ factors appearing in the coefficient $c_{\cR_0}$, and not those arising from the expressions \eqref{further_expansion_expression_Type_1}--\eqref{further_expansion_expression_Type_3}.

    \begin{table}[htbp]
        \caption{Classification of initial values for the characters of $\cR_0$}
        \label{table_Classification_of_Initial_Values}
        \begin{tabular}{|c|c|c|c|c|}
        \hline
        \diagbox{Action positions of $\partial_{\iab},\partial_{\iba}$}{Differences in characters} & $\mm_0-\mm_\cT$ & $n_0-n_\cT$ & $u_0-u_\cT$ & $S_0-S_\cT$ \\ \hline
        Both on heavy packages & $+2$ & $+2$ & $+0$ & $+0$ \\ \hline
        One on heavy packages, one on light weights & $+3$ & $+2$ & $-1$ & $+1$ \\ \hline
        One on heavy packages, one on loops & $+2$ & $+1$ & $+0$ & $+1$ \\ \hline
        One on light weights, one on loops & $+3$ & $+1$ & $-1$ & $+2$ \\ \hline
        Two on different light weights & $+4$ & $+2$ & $-2$ & $+2$ \\ \hline
        Both on the same light weight & $+3$ & $+2$ & $-1$ & $+1$ \\ \hline
        Two on different loops & $+2$ & $+0$ & $+0$ & $+2$ \\ \hline
        Both on the same loop & $+2$ & $+1$ & $+0$ & $+1$ \\ \hline
        \end{tabular}
    \end{table}

    Next, suppose we have an expression $\cR=\fO_{T'}\circ\cdots\circ \fO_{1}\pa{\cR_0}$, generated by applying a sequence of operations $\fO_{1},\ldots,\fO_{T'}$. Denote by $\fR,\ \fC_1,\ \fC_2,\ \fP_1,\ \fP_2,\ \fM,\ \fS_1,\ \fS_2,\ \fI_1,\ \fI_2,\ \fE$ the number of operations of type $\freplace,\ \fcut{1},\ \fcut{2},\ \fplug{1},\ \fplug{2},\ \fmerge,\ \fslash{1},\ \fslash{2},\ \finsert{1},\ \finsert{2},\ \fexchange$, respectively, in this sequence. We write \smash{$\cR=c_\cR\cdot \cW^{\pa{u}}\Upsilon_{n}^{\pa{\mm}}$}, with characters denoted by $\mm, n, u$, and $S$. 
    From \Cref{table_effect_of_operations_2}, we obtain the corresponding character relations:
    \begin{equation}\label{counting_characters_2}
        \begin{aligned}
            &\mm=-\fR+\fC_2+2\fP_1+\fP_2+\fM+\fS_2+2\fI_1+\fI_2+\fE+\mm_0,\\
            &n=\fC_2-\fP_2-\fM+\fS_2-\fI_2+n_0,\\
            &u=\fC_1-\fP_1+\fS_1-\fI_1+u_0,\\
            &S=2\fP_1+2\fP_2+2\fM+2\fI_1+2\fI_2+\fE+S_0.
        \end{aligned}
    \end{equation} 
On the other hand, recall that the characters $\mm_\cT,\ n_\cT,\ u_\cT,$ and $S_\cT$  of the expression $\cT$ satisfy \eqref{counting_character_of_term}. Together with \eqref{counting_characters_2} and the initial values in \Cref{table_Classification_of_Initial_Values}, we obtain the following character identities:
    \begin{align}\label{character_relation_further_expansion}
                S-\mm+n&=S_0-\mm_0+n_0+\fR=S_\cT-\mm_\cT+n_\cT+\fR=\sR+\fR-1=\fR-1,\\
                \mm+u&=\mm_0+u_0+\fC_1+\fC_2+\fP_1+\fP_2+\fM+\fS_1+\fS_2+\fI_1+\fI_2+\fE-\fR \nonumber\\
                &=2+\mm_\cT+u_\cT+\fC_1+\fC_2+\fP_1+\fP_2+\fM+\fS_1+\fS_2+\fI_1+\fI_2-\fR \nonumber\\
                & =2+\fC_1+\fC_2+\fP_1+\fP_2+\fM+\fS_1+\fS_2+\fI_1+\fI_2+\fE-\fR+\sC_1+\sC_2+\sP_1+\sP_2+2-\sR \nonumber\\
                & =\fC_1+\fC_2+\fP_1+\fP_2+\fM+\fS_1+\fS_2+\fI_1+\fI_2+\fE+\sC_1+\sC_2+\sP_1+\sP_2+4-\fR. \label{character_relation_further_expansion2}
\end{align}

Now, we show that our expansion procedure terminates after $\OO\pa{1}$ steps. Similar to \eqref{eq:sizeT}, we define the ``size'' of \smash{$\cR=c_\cR\cdot \cW^{\pa{u}}\Upsilon_n^{\pa{\mm}}$} as a pair:
    \begin{equation}\nonumber
        \begin{aligned}
            \isize'\pa{\cR}:=\pa{S+u,\mm-n+u},
        \end{aligned}
    \end{equation}
and denote its first and second components by $\isize'(\cR)_1$ and $\isize'(\cR)_2$, respectively. By applying the local laws from \Cref{lem_loc,lemma_necessary_estimates}, we obtain a bound similar to that in \eqref{eq:Tsmallthansize}: 
    \begin{equation}\nonumber
        \begin{aligned}
            \cR\prec N^{-\isize'\pa{\cR}_1}\eta^{-\isize'\pa{\cR}_2-1_{k_1=0}-1_{k_2=0}}\norm{A}^2.
        \end{aligned}
    \end{equation}
Using the same stopping criteria as those defined above \eqref{operation_sequence}, it follows that our expansion procedure will terminate after $\OO\pa{1}$ steps, following the same argument as below \eqref{operation_sequence}. 
    Then, similar to the proof in \Cref{subsection:proof_of_lemma_localized_eigenvector_local_law_and_lemma_localized_eigenvalue_local_law}, we need to estimate the terms resulting from expansions that stop under criterion (ii), i.e., when $k_1(\cR)=k_2(\cR)=0$. Note that these expressions have $\fR\geq 2$ and $\sR=0$. We classify them into the following three cases (i)--(iii). For ease of presentation, we will adopt the notations in \eqref{further_expansion_expression_Type_1}--\eqref{further_expansion_expression_Type_3} throughout the discussion below.
    
    \medskip
        \noindent(i) Suppose that $\cR$ is a Type \rmnum{1} expression as in \eqref{further_expansion_expression_Type_1}: 
        \begin{equation}\label{further_expansion_Type_1}
        	\begin{aligned}
        		\cR=-\frac{1}{ND}\bsum{a}{\alpha}{\beta}\frac{1}{p_0!q_0!}\cC_{\iab}^{p_0,q_0+1}\p{\sM B\tl_1 \Pi_{a_1}}_{\star_1\star_2}\p{\Pi_{a_2}\tl_2\Pi_{a_3}}_{\star_3\star_4}\p{\Pi_{a_4}}_{\star_5\star_6}\prod_{i=1}^{n-3}f^{\pa{i}},
        	\end{aligned}
        \end{equation}
    where $\mm\geq 1$, and the six $\star$ placeholders represent three $\alpha$-indices and three $\beta$-indices. 
     Considering a heavy package of the form \smash{$\p{B_1\tl B_2}_{\star_1\star_2}$} with $\star_1,\star_2\in\cI_a$ for some $a\in \qq D$, where $B_1$ and $B_2$ are deterministic matrices representing products of matrices $E_a$ and $\sM_i$. Here, each $\sM_i$ is either a scalar matrix or can be expanded as         
        \begin{equation}\label{further_expansion_first_order_expansion}
            \begin{aligned}
                \sM_i (z) =-\frac{1}{\smm_i (z) +z}-\Lambda\wt \sM_i  (z) ,\quad \text{where}\quad  \wt \sM_i (z) =\sum_{l=0}^{\infty}\pa{\smm_i (z) +z}^{-l-2}\Lambda^l.
            \end{aligned}
        \end{equation}
        By applying expansion \eqref{further_expansion_first_order_expansion} to all non-scalar $\sM_i$ factors in $B_1$ and $B_2$, we obtain that 
        \begin{equation}\label{eq:addoneLambda}
            \begin{aligned}
                \p{B_1\tl B_2}_{\star_1\star_2}=\pa{B_1\Lambda B_2}_{\star_1\star_2}-\Delta_{\txt{ev}}\pa{B_1 B_2}_{\star_1\star_2}\lesssim \norm{\Lambda B_1'\bre_{\star_1}}\norm{\Lambda B_2'\bre_{\star_2}}+\avga{\Lambda^2}.
            \end{aligned}
        \end{equation}
        In the above derivation, we also used \eqref{shift_bound_localized_eigenvector} and the fact that \smash{$\pa{E_{a_0}\Lambda E_{a_1}}_{\star_1\star_2}=0$} for any $\star_1,\star_2\in \cI_a$ and $a_0,a_1\in\qq D$. The matrices $B_1'$ and $B_2'$ denote deterministic matrices satisfying $\norma{B_1'}+\norm{B_2'}=\OO\pa{1}$. With the above estimate \eqref{eq:addoneLambda}, we can bound the two heavy packages involving \smash{$\tl_1$ and $\tl_2$} in \eqref{further_expansion_Type_1} by
        \begin{equation}\nonumber
            \begin{aligned}
                \pa{\norm{\Lambda B_1 \bre_{\star_1}}\norm{\Lambda B_2 \bre_{\star_2}}+\avga{\Lambda^2}}\pa{\norm{\Lambda B_3 \bre_{\star_3}}\norm{\Lambda B_4 \bre_{\star_4}}+\avga{\Lambda^2}},
            \end{aligned}
        \end{equation}
        where $B_j$, for $j\in\{1,2,3,4\}$, are deterministic matrices with $\norm{B_j}=\OO\pa{1}$. Since the six $\star$ indices consist of exactly three $\alpha$'s and three $\beta$'s, there must exist two indices, say $\star_{j_1}=\star_{j_2}$, that are identical, while at least one of the remaining indices, denoted $\star_{j_3}$, differs from both $\star_{j_1}$ and $\star_{j_2}$. The last remaining index is denoted as $\star_{j_4}$. 
Using the bound $\norm{\Lambda B_{j_4} \bre_{\star_{j_4}}}\lesssim \norm{\Lambda}$ and applying the Cauchy-Schwarz inequality with respect to $\star_{j_1},\star_{j_2},\star_{j_3}$,  we obtain that
        \begin{equation}\label{expanding_results_type_1_further_expansion}
            \begin{aligned}
                \absa{\cR}\prec N^{-1-S}\cdot N^{-3/2}\cdot N^{1/2}\norma{A}_{\HS}^3\norma{A}\cdot\frac{\pa{\im m}^{n-3}}{\eta^{\mm-n+2}}\cdot\pa{\frac{1}{N\eta}}^{u}.
            \end{aligned}
        \end{equation}
    The other factors in \eqref{further_expansion_Type_1} are estimated similarly, following the approach in \eqref{bound_factors}. 
Then, applying an argument analogous to \eqref{light_weight_cancel_imaginary_part}, and using $\mm\geq 2\pa{n-3}+1\geq 1$, together with the identities \eqref{character_relation_further_expansion} and \eqref{character_relation_further_expansion2}, we can bound \eqref{expanding_results_type_1_further_expansion} by
        \begin{equation}\nonumber
            \begin{aligned}
                &~N^{-\pa{S-\mm+n}}\cdot N^{1/3-\varepsilon_A}k^{-1/3}\cdot\norm{A}_{\HS}^2\cdot\pa{k/N }^{\frac 1 3(\mm+u-1)}= N^{1-\fR}\cdot N^{1/3-\varepsilon_A}k^{-1/3}\cdot\norm{A}_{\HS}^2\cdot\pa{k/N}^{\frac13(\mm+u-1)}.
            \end{aligned}
        \end{equation}
        As a consequence, if $\fR\geq 3$, then $\absa{\cR}\prec N^{-5/3-\varepsilon_A}k^{-1/3}\norm{A}_{\HS}^2$. If $\fR=2$, $\mm+u\geq 4$, then $\absa{\cR}\prec N^{-5/3}k^{2/3}\norm{A}_{\HS}^2$. Finally, if $\fR=2$ and $\mm+u\leq 3$, we see from \eqref{character_relation_further_expansion} and \eqref{character_relation_further_expansion2} that
        \begin{equation}\nonumber
            \begin{aligned}
                \fC_1+\fC_2+\fP_1+\fP_2+\fM+\fS_1+\fS_2+\fI_1+\fI_2+\fE+\sC_1+\sC_2+\sP_1+\sP_2\leq 1.
            \end{aligned}
        \end{equation}
    We can directly verify that no such $\cR$ exists.
        
        \medskip
        \noindent(ii) Suppose $\cR$ is of Type \rmnum{2}, as in \eqref{further_expansion_expression_Type_2}, where we must have $\mm\geq 2$. Using \eqref{add_one_more_Lambda} below, together with a similar bound as in \eqref{eq:addoneLambda}, we can derive that 
         \begin{align}
                \absa{\cR}&\prec N^{-1-S}\cdot N^{-3/2}\cdot N\norma{A}_{\HS}^2\cdot \avga{\Lambda^2}\cdot \frac{\pa{\im m}^{n-4}}{\eta^{\mm-n+2}}\cdot \pa{\frac{1}{N\eta}}^{u} \nonumber\\
                &\lesssim N^{1/2-\fR}\cdot N^{2/3-2\varepsilon_A}k^{-2/3}\norm{A}_{\HS}^2\cdot \pa{k/N}^{\frac13(\mm+u-2)}, \label{eq:weak_further_expansion2}
            \end{align}
        where in the first step, we estimate the remaining factors in \eqref{further_expansion_expression_Type_2} similarly to the approach used in \eqref{bound_factors}. In the second step, we apply the identities \eqref{character_relation_further_expansion} and \eqref{character_relation_further_expansion2}, along with an argument analogous to \eqref{light_weight_cancel_imaginary_part}, noting that $\mm\geq 2\pa{n-4}+2$.
        If $\fR\geq 3$, we get $\absa{\cR}\prec N^{-11/6-2\varepsilon_A}\norm{A}_{\HS}^2\le {N^{-5/3}k^{2/3}\norm{A}_{\HS}^2}$. 
        If $\fR=2$ and $\mm+u\geq 5$, we get $\absa{\cR}\prec N^{-11/6-2\varepsilon_A}k^{1/3}\norm{A}_{\HS}^2\le {N^{-5/3}k^{2/3}\norm{A}_{\HS}^2}$. If $\fR=2$ and $\mm+u\leq 4$, we have
        \begin{equation}\nonumber
            \begin{aligned}
                \fC_1+\fC_2+\fP_1+\fP_2+\fM+\fS_1+\fS_2+\fI_1+\fI_2+\fE+\sC_1+\sC_2+\sP_1+\sP_2\leq 2.
            \end{aligned}
        \end{equation}
        Moreover, to generate an additional loop containing only $\tl_2$ (without $\tl_1$), we must have $\sC_2+\fC_2+\fS_2\geq 1$, which implies $\mm+u\geq 3$. If $\mm+u=3$, we have
        \begin{equation}\nonumber
            \begin{aligned}
                \fC_1+\fC_2+\fP_1+\fP_2+\fM+\fS_1+\fS_2+\fI_1+\fI_2+\fE+\sC_1+\sC_2+\sP_1+\sP_2=1.
            \end{aligned}
        \end{equation}
        Since $\fR=2$, and all $\sG$ factors in the heavy packages or loops containing $\tl_1$ or $\tl_2$ must be removed, the loop containing \smash{$\tl_2$} must have been generated by a $\fslash{2}$ operation. This further implies $a_2\vee a_3\geq 2$, noting that the ``slashed" heavy package must contain at least two $\sG$ factors. Moreover, the loop containing \smash{$\tl_2$} must take the form \smash{$\avg{\sM_0\tl_2\sM_1 E_x}$}, since otherwise it would contain at least three $\sM_i$ factors, which contradicts the assumptions $\sR=0$ and $\fR=2$.     
        Together with \eqref{regular_estimate_localized_eigenvector}, this allows us to improve the estimate \eqref{eq:weak_further_expansion2} to 
        \begin{equation}\nonumber
            \begin{aligned}
                \absa{\cR}&\prec N^{-1-S}\cdot N^{-3/2}\cdot N\norma{A}_{\HS}^2\cdot \p{\im m} \avga{\Lambda^2}\cdot \frac{\pa{\im m}^{n-3}}{\eta^{\mm-n+2}}\cdot \pa{\frac{1}{N\eta}}^{u}\\
                &\lesssim N^{1/2-\fR}\cdot N^{2/3-2\varepsilon_A}k^{-2/3}\norm{A}_{\HS}^2\cdot\pa{k/N}^{\frac13\p{\mm+u}}\leq N^{-11/6-2\varepsilon_A}k^{1/3}\norm{A}_{\HS}^2.
            \end{aligned}
        \end{equation}
        If $\mm+u=4$, we have the relation
        \begin{equation}\nonumber
            \begin{aligned}
                \fC_1+\fC_2+\fP_1+\fP_2+\fM+\fS_1+\fS_2+\fI_1+\fI_2+\fE+\sC_1+\sC_2+\sP_1+\sP_2=2.
            \end{aligned}
        \end{equation}
        By a similar argument as above, the loop containing \smash{$\tl_2$} must again take the form \smash{$\avg{\sM_0\tl_2\sM_1 E_x}$}. 
        Hence, the estimate \eqref{eq:weak_further_expansion2} can improved as
        \begin{equation}\nonumber
            \begin{aligned}
                \absa{\cR}&\prec N^{-1-S}\cdot N^{-3/2}\cdot N\norma{A}_{\HS}^2\cdot \p{\im m} \avga{\Lambda^2}\cdot \frac{\pa{\im m}^{n-4}}{\eta^{\mm-n+2}}\cdot \pa{\frac{1}{N\eta}}^{u}\\
                &\lesssim N^{1/2-\fR}N^{2/3-2\varepsilon_A}k^{-2/3}\norm{A}_{\HS}^2\pa{k/N}^{\frac13(\mm+u-1)}\leq N^{-11/6-2\varepsilon_A}k^{1/3}\norm{A}_{\HS}^2. 
            \end{aligned}
        \end{equation}
        
        \medskip
        \noindent(iii)  Suppose $\cR$ is of Type \rmnum{3}, as in \eqref{further_expansion_expression_Type_3}, where we must have $\mm\geq 2$. Then, we can obtain that
        \begin{align}
                \absa{\cR}\prec& N^{-1-S}\cdot N^{-3/2}\cdot N\norma{A}_{\HS}^2\cdot \frac{\pa{\im m}^{n-3}}{\eta^{\mm-n+1}}\cdot \pa{\frac{1}{N\eta}}^u\lesssim N^{1/2-\fR}\norm{A}_{\HS}^2\pa{k /N }^{\frac 13(\mm+u-2)},
            \end{align}
       by using the estimates in \eqref{bound_factors} and \eqref{R_cT_typt_1_k_1_geq_1_k_2-geq_1_operator_norm_bound}. If $\fR\geq 3$, we have $\absa{\cR}\prec N^{-5/2}\norm{A}_{\HS}^2\le N^{-5/3}k^{2/3}\norm{A}_{\HS}^2$. If $\fR=2$ and $\mm+u\geq 3$, we have $\absa{\cR}\prec N^{-11/6}k^{1/3}\norm{A}_{\HS}^2\le {N^{-5/3}k^{2/3}\norm{A}_{\HS}^2}$. If $\fR=2$ and $\mm+u\leq 2$, then we must have $\mm+u=2$, and
        \begin{equation}\nonumber
            \begin{aligned}
                \fC_1+\fC_2+\fP_1+\fP_2+\fM+\fS_1+\fS_2+\fI_1+\fI_2+\fE+\sC_1+\sC_2+\sP_1+\sP_2=0.
            \end{aligned}
        \end{equation}
       Then, by a straightforward enumeration, we observe that $\cR$ must take the following specific form:
        \begin{equation}\label{expansion_result_type_3_fr_2_m_plus_u_2}
            \begin{aligned}
                -\frac{1}{ND}\bsum{a}{\alpha}{\beta}\cC_{\iab}^{p_0,q_0+1}\p{\sM_0\tl_1\sM_1\tl_2\sM_0}_{\star\star}\p{\sG_0}_{\star\star}\p{\sG_0}_{\star\star},
            \end{aligned}
        \end{equation}
        which comes from the expression
        \begin{equation}\nonumber
            \begin{aligned}
                -\frac{1}{ND}\sum_{a=1}^D\sum_{\alpha,\beta\in \cI_a}\frac{1}{p_0!q_0!}\cC_{\iab}^{p_0,q_0+1}\partial_{\iab}^{p_0}\partial_{\iba}^{q_0}\pa{\sM_0\tl_1\sG_1\tl_2\sG_0}_{\iab}.
            \end{aligned}
        \end{equation}
        Since there is only one $\sM_1$ factor in \eqref{expansion_result_type_3_fr_2_m_plus_u_2}, we can exploit a cancellation by applying the polarization identity \eqref{expansion_of_im} (see \eqref{direct_estimate_type_1_k_1_0_k_2_geq_1_sR_1_p_plus_q_2_m_plus_u_1} for a similar argument).  Specifically, subtracting $\cR$ in \eqref{expansion_result_type_3_fr_2_m_plus_u_2} from the corresponding expression, where $\sM_1$ is replaced by $\sM_1^*$, yields an additional $\im m$ factor. This improvement allows us to strengthen the estimate as follows:
        \begin{equation}\nonumber
            \begin{aligned}
                N^{-5/2}\cdot (\im m)\cdot N\norma{A}_{\HS}^2\lesssim N^{-11/6}k^{1/3}\norm{A}_{\HS}^2\leq N^{-5/3}k^{2/3}\norm{A}_{\HS}^2.
            \end{aligned}
        \end{equation}

\medspace

 Now, to complete the proof of \Cref{claim_estimate_remainder_terms}, it remains to bound the remainder terms arising from the expansions of $\cR$, namely the terms \smash{$\cE_{\cR}^{\pa{2}}$} in \eqref{concrete_expansion_tl_2_heavy_package_right} and \eqref{concrete_expansion_tl_2_loop}. 
 Our argument below again relies on the inequalities previously used in the first part of the proof of \Cref{claim_estimate_remainder_terms} (specifically, the argument following \eqref{R_cT_typt_1_k_1_geq_1_k_2-geq_1_explicit_expression}). 
 The main difference here is the presence of factors of the form $\pa{\cdot}_{\alpha j}$ or $\pa{\cdot}_{i \beta}$. To handle these, we apply the Cauchy-Schwarz inequality, Ward's identity, and the simple bound 
    \begin{equation}\label{light_weights_to_im_m}
        \begin{aligned}
            \sqrt{{\im m}/{\eta}}\lesssim N^{1/2}\im m
        \end{aligned}
    \end{equation}
to extract additional $\im m$ factors. We first present a fully detailed example in  \Cref{example_estimate_reminder_reminder_terms}, which illustrates the estimation process for the remainder terms. For the remaining cases, we provide only the final estimates, omitting the full derivations. The estimation in each case follows a similar case-by-case analysis as shown in \Cref{example_estimate_reminder_reminder_terms}.


    \begin{example}\label{example_estimate_reminder_reminder_terms}
As an example, we take $ \cT=\avg{\sG_0\tl_1\sG_1\tl_2}$ and consider the expression
        \begin{equation}\nonumber
            \begin{aligned}
                \cR_0=-\frac{1}{ND}\bsum{a}{\alpha}{\beta}\cC_{\iab}^{p_0,q_0+1}\p{\sM_0\tl_1\sG_1}_{\star_1\star_2}\p{\sG_1\tl_2\sG_0}_{\star_3\star_4}\p{\sG_0}_{\star_5\star_6}.
            \end{aligned}
        \end{equation}
In this setting, we have $p_0+q_0=2$, and the six $\star$'s in $\cR_0$ represent exactly three $\alpha$ indices and three $\beta$ indices. Following our expansion strategy, we select the package containing \smash{$\tl_2$} and expand the $\sG_0$ factor within it. Then, the resulting reminder term is given by
        \begin{equation}\nonumber
            \begin{aligned}
                \cE_{\cR_0}^{\pa{2}}:=\sum_{2\leq p+q\leq l}\cE_{\cR_0}^{\pa{2}}\pa{p,q}+R_{l+1}^{\pa{2}},
            \end{aligned}
        \end{equation}
        where \smash{$\cE_{\cR_0}^{\pa{2}}\pa{p,q}$} is defined as 
        \begin{equation}\nonumber
            \begin{aligned}
                \cE_{\cR_0}^{\pa{2}}\pa{p,q}
                =-\frac{1}{ND}\bsum{a}{\alpha}{\beta}&\frac{1}{p_0!q_0!}\cC_{\iab}^{p_0,q_0+1}\sum_{a=1}^D\sum_{i,j\in \cI_a}\frac{1}{p!q!}\cC_{ij}^{p,q+1}\\
                &\times\partial_{ij}^p\partial_{ji}^q\left[\p{\sG_{1}\tl_2\sM_0}_{\star_3j}\p{\sG_0}_{i\star_4}\p{\sM_0\tl_1\sG_1}_{\star_1\star_2}\p{\sG_0}_{\star_5\star_6}\right].
            \end{aligned}
        \end{equation}
        We then expand the derivatives $\partial_{ij}^p\partial_{ji}^q$ and estimate the resulting terms on a case-by-case basis.

        \medskip
             \noindent (i) If none of the derivatives act on the factor $\p{\sG_{1}\tl_2\sM_0}_{\star_3 j}$, then we obtain
            \begin{equation}\nonumber
                \begin{aligned}
               {\cE_{\cR_0}^{\pa{2}}\pa{p,q}}\prec N^{-5/2-\pa{p+q+1}/2}\sum_{\alpha,\beta}\sum_{i,j}\abs{\p{\sG_{1}\tl_2\sM_0}_{\star_3j}}\cdot\abs{\p{\sG_0}_{\#_1 \star_4}}\cdot\norm{\bre_{\star_1}^\top\sM_0\tl_1},
                \end{aligned}
            \end{equation}
            where each $\#$ denotes either an $i$ or a $j$ index. Applying the Cauchy-Schwarz inequality with respect to $j$ and $\#_1$, we derive that 
            \begin{align}
                {\cE_{\cR_0}^{\pa{2}}\pa{p,q}} &\prec N^{-5/2-\pa{p+q+1}/2}\sum_{\alpha,\beta} N \norm{\bre_{\star_3}^\top\sG_1\tl_2\sM_0}\cdot\norm{\sG_0\bre_{\star_4}}\cdot\norm{\bre_{\star_1}^\top\sM_0\tl_1} \nonumber\\
                  &\prec N^{-1-\pa{p+q+1}/2}\p{\im m}\sum_{\alpha,\beta}\norm{\bre_{\star_3}^\top\sG_1\tl_2\sM_0}\cdot\norm{\bre_{\star_1}^\top\sM_0\tl_1},\label{eq:further_expad_remain1}
                \end{align}
            where we also used \eqref{entprodG} below and the following bound from Ward's identity combined with \eqref{light_weights_to_im_m}: 
            \begin{equation}\nonumber
                \begin{aligned}
                    \norm{\sG_0 \bre_{\star_4}}=\pa{\bre_{\star_4}^\top \sG_0^*\sG_0\bre_{\star_4}}^{1/2}\prec \sqrt{{\im m}/{\eta}}\lesssim N^{1/2}\im m.
                \end{aligned}
            \end{equation}
            A further application of the Cauchy-Schwarz inequality to \eqref{eq:further_expad_remain1}, this time with respect to $\star_1$ and $\star_3$, yields:
          \begin{align*}
                    {\cE_{\cR_0}^{\pa{2}}\pa{p,q}}&\prec N^{-1-\pa{p+q+1}/2}(\im m)\cdot N\norm{\sG_1\tl_2\sM_0}_{\txt{HS}}\norm{\tl_2}_{\txt{HS}}\\
                    &\prec N^{-\pa{p+q+1}/2}(\im m)\cdot N^{1/2}\p{\im m}\norma{A}_{\txt{HS}}^2\lesssim N^{-5/3+\varepsilon}k^{2/3}\norm{A}_{\HS}^2\leq N^{-1-2\varepsilon_A+\varepsilon}.
                \end{align*}
            In the second step, we again utilize \eqref{two_loop_estimate} and \eqref{light_weights_to_im_m}.
            
            
            \medspace
            
            \noindent (ii) If some of the derivatives act on the factor $\p{\sG_{1}\tl_2\sM_0}_{\star_3 j}$, then we obtain that
    \begin{align*}
                {\cE_{\cR_0}^{\pa{2}}\pa{p,q}}\prec N^{-5/2-\pa{p+q+1}/2}\sum_{\alpha,\beta}\sum_{i,j}\abs{\pa{\sG_1}_{\star_3\#_1}}\cdot\abs{\p{\sG_{1}\tl_2\sM_0}_{\#_2j}}\cdot\abs{\p{\sG_0}_{\#_3 \star_4}}\cdot\norm{\bre_{\star_1}^\top\sM_0\tl_1}.
                \end{align*}
            If the indices $\star_1$, $\star_3$, and $\star_4$ are not all identical, there are three possible cases to consider. In the first case where $\star_1=\star_3\neq \star_4$, applying the Cauchy–Schwarz inequality, along with the bounds \eqref{entprodG} and \eqref{light_weights_to_im_m}, we obtain that
            \begin{align*}
                     {\cE_{\cR_0}^{\pa{2}}\pa{p,q}} &\prec N^{-5/2-\pa{p+q+1}/2}\sum_{\star_4}\sum_{i,j}N^{1/2}\p{\im m}\norma{A}_{\HS}\norm{\tl_2\sM_0\bre_j}\cdot\abs{\p{\sG_0}_{\#_3 \star_4}}\\
                    &\prec N^{-5/2-\pa{p+q+1}/2}\cdot N^{1/2}\p{\im m}\norma{A}_{\HS}\cdot N^{5/2}\p{\im m}\norma{A}_{\HS}\lesssim N^{-5/3+\varepsilon}k^{2/3}\norm{A}_{\HS}^2\leq N^{-1-2\varepsilon_A+\varepsilon}.
                \end{align*}
            The $\star_1=\star_4\neq \star_3$ case can be treated similarly. Finally, for the $\star_3=\star_4\neq \star_1$ case, we again apply the Cauchy-Schwarz inequality, the bound \eqref{entprodG} below, and \eqref{light_weights_to_im_m} to derive that 
            \begin{align*}
              {\cE_{\cR_0}^{\pa{2}}\pa{p,q}}&\prec N^{-5/2-\pa{p+q+1}/2}\sum_{\star_1}\sum_{j}N^{2}\pa{\im m}^2\norm{\tl_2\sM_0\bre_j}\cdot \norm{\bre_{\star_1}^\top\sM_0\tl_1}\\
           & \prec N^{-5/2-\pa{p+q+1}/2}\cdot N^{2}\pa{\im m}^2\cdot N\norma{A}_{\HS}^2\lesssim N^{-5/3+\varepsilon}k^{2/3}\norm{A}_{\HS}^2\leq N^{-1-2\varepsilon_A+\varepsilon}.
                \end{align*}
            
 It remains to consider the $\star_1=\star_3=\star_4$ case, where we must have $\star_1\neq \star_2$. In this case, if none of the derivatives acts on the factor $\p{\sM_0\tl_1\sG_1}_{\star_1\star_2}$, then we obtain that 
            \begin{equation}\nonumber
                \begin{aligned}
                 {\cE_{\cR_0}^{\pa{2}}\pa{p,q}}&\prec N^{-5/2-\pa{p+q+1}/2}\sum_{\alpha,\beta}\sum_{i,j}\norm{\tl_2\sM_0\bre_j}\cdot\abs{\pa{\sG_0}_{\#_3 \star_4}}\cdot\abs{\p{\sM_0\tl_1\sG_1}_{\star_1\star_2}}\\
                    &\prec N^{-5/2-\pa{p+q+1}/2}\sum_{\alpha,\beta}N^{3/2}\norma{A}_{\HS}(\im m)\cdot  \abs{\p{\sM_0\tl_1\sG_1}_{\star_1\star_2}}\\
                    &\prec N^{-5/2-\pa{p+q+1}/2}\cdot N^{3/2}\norma{A}_{\HS}(\im m)\cdot N\norm{\sM_0\tl_1\sG_1}_{\HS}\\
                    &\prec N^{-\pa{p+q}/2}\pa{\im m}^2\norma{A}_{\HS}^2\lesssim N^{-5/3+\varepsilon}k^{2/3}\norm{A}_{\HS}^2\leq N^{-1-2\varepsilon_A+\varepsilon},
                \end{aligned}
            \end{equation}
            by using the Cauchy-Schwarz inequality, \Cref{lemma_necessary_estimates}, and \eqref{light_weights_to_im_m} again. If instead some derivatives act on this factor, then we obtain that
                \begin{align}
                {\cE_{\cR_0}^{\pa{2}}\pa{p,q}}&\prec N^{-5/2-\pa{p+q+1}/2}\sum_{\alpha,\beta}\sum_{i,j}\norm{\tl_2\sM_0\bre_j}\cdot\abs{\p{\sM_0\tl_1\sG_1}_{\star_1\#_1}}\cdot \abs{\p{\sG_1}_{\#_2 \star_2}}\nonumber\\
                    &\prec N^{-5/2-\pa{p+q+1}/2}\sum_{i,j}N^{3/2}(\im m)\cdot \norm{\tl_2\sM_0\bre_j}\cdot \norm{\sM_0\tl_1\sG_1\bre_{\#_1}}\nonumber\\
                   & \prec N^{-5/2-\pa{p+q+1}/2}\cdot N^{3/2}(\im m)\cdot N\norma{A}_{\HS}\norm{\sM_0\tl_1 \sG_1}_{\HS}\prec N^{-5/3+\varepsilon}k^{2/3}\norm{A}_{\HS}^2\leq N^{-1-2\varepsilon_A+\varepsilon}\nonumber,
                \end{align}
            through an analogous argument. This concludes \Cref{example_estimate_reminder_reminder_terms}. 
    \end{example}
    
   Now, adopting the notations in \eqref{further_expansion_expression_Type_1}--\eqref{further_expansion_expression_Type_3}, and following similar arguments as in \Cref{example_estimate_reminder_reminder_terms}, we can estimate all possible cases one by one as follows.
    
    
    \medskip
        \noindent(1) If $\cR$ is of Type \rmnum{1} and $a_1,a_2,a_3\geq 1$, then $\mm\geq 4$ and we choose the first $\sG$ factor to the right of $\tl_2$. In this case, the remainder term takes the following form:
        \begin{equation}\label{explicit_reminder_reminder_terms_example_1}
            \begin{aligned}
                \cE_{\cR}^{\pa{2}}=&\frac{c_{\cR}}{ND}\bsum{a}{\alpha}{\beta}\frac{1}{p_0!q_0!}\cC_{\iab}^{p_0,q_0+1}\sum_{2\leq p+q\leq l}\sum_{a=1}^D\sum_{i,j\in \cI_a}\frac{1}{p!q!}\cC_{ij}^{p,q+1}\partial_{ij}^p\partial_{ji}^q\Bigg[\p{\Pi_{a_2}\tl_2B_1\sM}_{\star j}\p{\sG\wt\Pi_{a_3}}_{i\star}\\
            &\left.\times \p{\wt\sM B\tl_1\Pi_{a_1}}_{\star\star}\p{\Pi_{a_4}}_{\star\star}\cdot \prod_{r=1}^{n-3}f^{\pa{r}}\cdot \prod_{r=1}^{u}W_r\right]+R_{l+1}^{\pa{2}}=:\sum_{2\leq p+q\leq l}\cE_{\cR}^{\pa{2}}\pa{p,q}+R_{l+1}^{\pa{2}},
            \end{aligned}
        \end{equation}
        where $\wt\sM$ denotes the $\sM$ factor appearing in \eqref{further_expansion_expression_Type_1}, the expression \smash{$\tl_2\Pi_{a_3}$} is factorized as \smash{$\tl_2\Pi_{a_3}=:\tl_2B_1\sG \wt\Pi_{a_3}$}, and $B_1$ is the deterministic matrix between \smash{$\tl_2$} and $\sG$. Noting that these remainder terms are structurally similar to those in \Cref{example_estimate_reminder_reminder_terms}, we can derive the following estimate using the same argument: 
        \begin{equation}\label{concrete_estimate_remainder_remainder_terms_example_1}
            \begin{aligned}
               {\cE_{\cR}^{\pa{2}}\pa{p,q}}&\prec N^{-1-S-3/2-\pa{p+q+1}/2}\cdot N^{3}\pa{\im m}^2\norma{A}_{\HS}^2\cdot \frac{\pa{\im m}^{n-3}}{\eta^{\mm-n-1}}\cdot \pa{\frac{1}{N\eta}}^{u}\\
                &\lesssim  N^{-1-\fR}\norm{A}_{\HS}^2\cdot N^{\varepsilon}\pa{k/N }^{\frac13\p{\mm+u-2}}\leq N^{-5/3+\varepsilon}k^{2/3}\norm{A}_{\HS}^2.
            \end{aligned}
        \end{equation}

        \noindent(2) If $\cR$ is of Type \rmnum{1} and $a_1\geq 1,\ a_2\geq 1,\ a_3=0$, then $\mm\geq 3$, $\fR\geq 1$, and we select the first $\sG$ factor to the right of $\tl_1$. In this case, the remainder term takes the form
        \begin{equation}\nonumber
            \begin{aligned}
                \cE_{\cR}^{\pa{2}}=&\frac{c_{\cR}}{ND}\bsum{a}{\alpha}{\beta}\frac{1}{p_0!q_0!}\cC_{\iab}^{p_0,q_0+1}\sum_{2\leq p+q\leq l}\sum_{a=1}^D\sum_{i,j\in \cI_a}\frac{1}{p!q!}\cC_{ij}^{p,q+1}\partial_{ij}^p\partial_{ji}^q\Bigg[\p{\wt\sM B\tl_1B_1\sM}_{\star j}\p{\sG\wt\Pi_{a_1}}_{i\star }\\
            &\left.\times \p{\Pi_{a_2}\tl_2\Pi_{a_3}}_{\star \star }\p{\Pi_{a_4}}_{\star \star }\cdot \prod_{r=1}^{n-3}f^{\pa{r}}\cdot \prod_{r=1}^{u}W_r\right]+R_{l+1}^{\pa{2}}=:\sum_{2\leq p+q\leq l}\cE_{\cR}^{\pa{2}}\pa{p,q}+R_{l+1}^{\pa{2}},
            \end{aligned}
        \end{equation}
    with analogous notation as in \eqref{explicit_reminder_reminder_terms_example_1}.
    Then, applying the Cauchy-Schwarz inequality to a product of the following form
        \begin{equation}\nonumber
            \begin{aligned}
                \sum_{\alpha,\beta,i,j}\abs{\p{\wt\sM B\tl_1B_1\sM}_{\star j}}\cdot \abs{\pa{\Pi_0}_{\# \star}}\cdot \norm{\bre_{\star}^\top\Pi_{a_2}\tl_2},
            \end{aligned}
        \end{equation}
        where $\Pi_0$ arises from the derivatives of $\p{\sG\wt\Pi_{a_1}}_{i\star}$ and contains at least one $\sG$ factor, we obtain that
        \begin{equation}\nonumber
            \begin{aligned}
          {\cE_{\cR}^{\pa{2}}\pa{p,q}}&\prec N^{-5/2-S-\pa{p+q+1}/2}\cdot N^{2}\p{\im m}\norma{A}_{\HS}^2\cdot \frac{\pa{\im m}^{n-3}}{\eta^{\mm-n+1}}\cdot \pa{\frac{1}{N\eta}}^{u}\\
                &\lesssim N^{-\fR}\norm{A}_{\HS}^2\pa{k /N }^{\frac13\p{\mm+u-1}}\leq N^{-5/3+\varepsilon}k^{2/3}\norm{A}_{\HS}^2.
            \end{aligned}
        \end{equation}

        \noindent(3) If $\cR$ is of Type \rmnum{1} and $a_1=0,\ a_2\geq 1,\ a_3=0$, then $\mm\geq 2$, $\fR\geq 2$, and we choose the first $\sG$ factor on the left of $\tl_2$. In this case, the remainder term takes the form
        \begin{equation}\nonumber
            \begin{aligned}
                \cE_{\cR}^{\pa{2}}=&\frac{c_{\cR}}{ND}\bsum{a}{\alpha}{\beta}\frac{1}{p_0!q_0!}\cC_{\iab}^{p_0,q_0+1}\sum_{2\leq p+q\leq l}\sum_{a=1}^D\sum_{i,j\in \cI_a}\frac{1}{p!q!}\cC_{ij}^{p,q+1}\partial_{ij}^p\partial_{ji}^q\Bigg[\p{\wt\Pi_{a_2}\sG}_{\star j}\p{\sM B_1\tl_2 \Pi_{a_3}}_{i\star }\\
            &\left.\times \p{\wt\sM B\tl_1\Pi_{a_1}}_{\star \star }\p{\Pi_{a_4}}_{\star \star }\cdot \prod_{r=1}^{n-3}f^{\pa{r}}\cdot \prod_{r=1}^{u}W_r\right]+R_{l+1}^{\pa{2}}=:\sum_{2\leq p+q\leq l}\cE_{\cR}^{\pa{2}}\pa{p,q}+R_{l+1}^{\pa{2}},
            \end{aligned}
        \end{equation}
         with analogous notation as in \eqref{explicit_reminder_reminder_terms_example_1}. Applying the Cauchy-Schwarz inequality to a product of the form
        \begin{equation}\nonumber
            \begin{aligned}
                \sum_{\alpha,\beta,i,j}\abs{\p{\Pi_0}_{\star \#}}\cdot \abs{\p{\sM B_1\tl_2 \Pi_{a_3}}_{i\star }}\cdot \norm{\tl_1\Pi_{a_1}\bre_\star },
            \end{aligned}
        \end{equation}
        where $\Pi_0$ is generated from the derivatives of $\p{\wt\Pi_{a_2}\sG}_{\star j}$ and contains at least one $\sG$ factor, we obtain that
        \begin{equation}\nonumber
            \begin{aligned}
          {\cE_{\cR}^{\pa{2}}\pa{p,q}}&\prec N^{-1-S-3/2-\pa{p+q+1}/2}\cdot N^{5/2}\p{\im m}\norma{A}_{\HS}^2\cdot \frac{\pa{\im m}^{n-3}}{\eta^{\mm-n+1}}\cdot \pa{\frac{1}{N\eta}}^{u}\\
               & \lesssim N^{1/2-\fR}\norm{A}_{\HS}^2\pa{k/N }^{\frac13\p{\mm+u-1}}\leq N^{-11/6}k^{1/3}\norm{A}_{\HS}^2\leq N^{-5/3}k^{2/3}\norm{A}_{\HS}^2.
            \end{aligned}
        \end{equation}

\smallskip

        \noindent(4) All remaining cases in which $\cR$ is of Type \rmnum{1} are not possible. 

\smallskip

        \noindent(5) If $\cR$ is of Type \rmnum{2} with $a_1\geq 1$ and $a_4\geq 1$, then $\mm\geq 4$, and we choose the first $\sG$ factor to the left of \smash{$\tl_2$}. Moreover, to generated a loop containing \smash{$\tl_2$}, we must have 
        \begin{equation}\label{sum_geq_1_because_of_a_loop_containing_tl_2}
            \begin{aligned}
                \fC_1+\fC_2+\fP_1+\fP_2+\fM+\fS_1+\fS_2+\fI_1+\fI_2+\fE+\sC_1+\sC_2+\sP_1+\sP_2\geq 1,
            \end{aligned}
        \end{equation}
    which, combined with \eqref{character_relation_further_expansion} and \eqref{character_relation_further_expansion2}, implies that $\mm + u \geq 5 - \fR$. The remainder term in this case takes a form
        \begin{equation}\nonumber
            \begin{aligned}
                \cE_{\cR}^{\pa{2}}=&\frac{c_{\cR}}{N^2D^2}\bsum{a}{\alpha}{\beta}\frac{1}{p_0!q_0!}\cC_{\iab}^{p_0,q_0+1}\sum_{2\leq p+q\leq l}\sum_{b=1}^D\sum_{i,j\in \cI_b}\frac{1}{p!q!}\cC_{ij}^{p,q+1}\partial_{ij}^p\partial_{ji}^q\Bigg[\p{\sM B_1\tl_2 \wt\Pi_{a_4} \sG}_{ij}\\
                &\left.\times\p{\wt\sM B\tl_1 \Pi_{a_1}}_{\star\star}\p{\Pi_{a_2}}_{\star\star} \p{\Pi_{a_3}}_{\star\star}\cdot \prod_{r=1}^{n-4}f^{\pa{r}}\cdot\prod_{r=1}^{u}W_r\right] + R_{l+1}^{\pa{2}}=:\sum_{2\leq p+q\leq l}\cE_{\cR}^{\pa{2}}\pa{p,q}+R_{l+1}^{\pa{2}},
            \end{aligned}
        \end{equation}
    with notation understood similarly as in \eqref{explicit_reminder_reminder_terms_example_1}. Then, applying the Cauch-Schwarz inequality to a product of the form
        \begin{equation}\nonumber
            \begin{aligned}
                \sum_{\alpha,\beta,i,j}\norm{\bre_i^\top B_1 \tl_2}\cdot\norm{\bre_{\star}^\top \wt \sM B \tl_1},
            \end{aligned}
        \end{equation}
        we obtain the following rough bound if one of the following conditions holds: (i) $\fR\geq 1$, or (ii) $\mm+u\geq 6$: 
        \begin{equation}\nonumber
            \begin{aligned}
        {\cE_{\cR}^{\pa{2}}\pa{p,q}}&\prec N^{-2-S-3/2-\pa{p+q+1}/2}\cdot N^{3}\norma{A}_{\HS}^2\cdot \frac{\pa{\im m}^{n-4}}{\eta^{\mm-n}}\cdot \pa{\frac{1}{N\eta}}^{u}\\
                &\lesssim N^{-1-\fR}\norm{A}_{\HS}^2\pa{k/N}^{\frac13\p{\mm+u-4}}\leq N^{-5/3}k^{2/3}\norm{A}_{\HS}^2.
            \end{aligned}
        \end{equation}
        It remains to consider the case $\fR=0$ and $\mm+u=5$, in which case we must have
        \begin{equation}\nonumber
            \begin{aligned}
                \fC_1+\fC_2+\fP_1+\fP_2+\fM+\fS_1+\fS_2+\fI_1+\fI_2+\fE+\sC_1+\sC_2+\sP_1+\sP_2= 1.
            \end{aligned}
        \end{equation}
    Here, it is straightforward to verify that $a_i \geq 2$ for at least one index $i$, since when the loop containing $\tl_2$ is generated, either the loop itself or the portion produced by the operations $\scut{}$, $\fcut{}$, or $\fslash{}$ must contain at least two $\sG$ factors. This allows us to gain an additional $\im m$ factor via \eqref{entprodG}, which improves the above rough estimate to: 
        \begin{equation}\nonumber
            \begin{aligned}
         {\cE_{\cR}^{\pa{2}}\pa{p,q}}&\prec N^{-2-S-3/2-\pa{p+q+1}/2}\cdot N^{3}\norma{A}_{\HS}^2\cdot \frac{\pa{\im m}^{n-3}}{\eta^{\mm-n}}\cdot \pa{\frac{1}{N\eta}}^{u}\\
       & \lesssim N^{-1-\fR+\varepsilon/2}\norm{A}_{\HS}^2\pa{k/N }^{\frac 13 \p{\mm+u-3}}\leq N^{-5/3+\varepsilon/2}k^{2/3}\norm{A}_{\HS}^2.
            \end{aligned}
        \end{equation}

\smallskip

        \noindent(6) If $\cR$ is of Type \rmnum{2} with $a_1\geq 1$ and $a_4=0$, then $\mm\geq 3$, $\fR\geq 1$. In this case, we select the first $\sG$ factor to the right of $\tl_1$. The corresponding remainder term takes a form
        \begin{equation}\nonumber
            \begin{aligned}
                \cE_{\cR}^{\pa{2}}=&\frac{c_{\cR}}{ND}\bsum{a}{\alpha}{\beta}\frac{1}{p_0!q_0!}\cC_{\iab}^{p_0,q_0+1}\sum_{2\leq p+q\leq l}\sum_{a=1}^D\sum_{i,j\in \cI_a}\frac{1}{p!q!}\cC_{ij}^{p,q+1}\partial_{ij}^p\partial_{ji}^q\Bigg[\p{\wt \sM B \tl_1 B_1 \sM}_{\star j}\p{\sG \wt \Pi_{a_1}}_{i\star }\\
            &\left.\times \p{\Pi_{a_2}}_{\star \star }\p{\Pi_{a_3}}_{\star \star }\avg{\tl_2\Pi_{a_4}}\prod_{r=1}^{n-3}f^{\pa{r}}\cdot \prod_{r=1}^{u}W_r\right]+R_{l+1}^{\pa{2}}=:\sum_{2\leq p+q\leq l}\cE_{\cR}^{\pa{2}}\pa{p,q}+R_{l+1}^{\pa{2}},
            \end{aligned}
        \end{equation}
        with notation understood similarly as in \eqref{explicit_reminder_reminder_terms_example_1}. Then, applying the Cauchy-Schwarz inequality to a product of the form
        \begin{equation}\nonumber
            \begin{aligned}
                \sum_{\alpha,\beta,i,j}\abs{\p{\wt \sM B \tl_1 B_1 \sM}_{\star j}}\cdot \abs{\p{\Pi_0}_{\# \star }},
            \end{aligned}
        \end{equation}
        where $\Pi_0$ arises from the derivatives of $\p{\sG \wt \Pi_{a_1}}_{i\star }$ and contains at least one $\sG$ factor, we obtain the following rough bound  if one of the following conditions holds: (i) $\fR\geq 2$, or (ii) $\mm+u\geq 5$:
        \begin{equation}\nonumber
            \begin{aligned}
         {\cE_{\cR}^{\pa{2}}\pa{p,q}}&\prec N^{-1-S-3/2-\pa{p+q+1}/2}\cdot N^{2}\p{\im m}\norma{A}_{\HS}^3\cdot \frac{\pa{\im m}^{n-4}}{\eta^{\mm-n+1}}\cdot \pa{\frac{1}{N\eta}}^{u}\\
& \lesssim N^{-\fR+\varepsilon}\norm{A}_{\HS}^3\pa{k/N}^{\frac 13 \p{\mm+u-2}}\leq N^{-5/3+\varepsilon}k^{2/3}\norm{A}_{\HS}^2.
            \end{aligned}
        \end{equation}
        It remains to consider the exceptional case $\fR=1$ and $\mm+u\leq 4$. However, by an argument similar to that used in \eqref{sum_geq_1_because_of_a_loop_containing_tl_2}, we have $\mm+u\geq 5-\fR=4$, and
        \begin{equation}\nonumber
            \begin{aligned}
                \fC_1+\fC_2+\fP_1+\fP_2+\fM+\fS_1+\fS_2+\fI_1+\fI_2+\fE+\sC_1+\sC_2+\sP_1+\sP_2= 1.
            \end{aligned}
        \end{equation}
    A direct enumeration under this constraint shows that no such term exists.

\smallskip

        \noindent(7) All other cases in which $\cR$ is of Type \rmnum{2} are impossible. 

\smallskip

        \noindent(8) If $\cR$ is of Type \rmnum{3} with $a_1\geq 1$ and $a_2\geq 1$, then $\mm\geq 4$, and we choose the first $\sG$ factor to the right of $\tl_2$. The corresponding remainder term takes a form
            \begin{align}
                \cE_{\cR}^{\pa{2}}=&\frac{c_{\cR}}{ND}\bsum{a}{\alpha}{\beta}\frac{1}{p_0!q_0!}\cC_{\iab}^{p_0,q_0+1}\sum_{2\leq p+q\leq l}\sum_{a=1}^D\sum_{i,j\in \cI_a}\frac{1}{p!q!}\cC_{ij}^{p,q+1}\partial_{ij}^p\partial_{ji}^q\Bigg[\p{\wt \sM B \tl_1 \Pi_{a_1}\tl_2 B_1 \sM}_{\star j}\p{\sG \wt \Pi_{a_2}}_{i\star }\nonumber\\
            &\left.\times \p{\Pi_{a_3}}_{\star \star }\p{\Pi_{a_4}}_{\star \star }\cdot \prod_{r=1}^{n-3}f^{\pa{r}}\cdot \prod_{r=1}^{u}W_r\right]+R_{l+1}^{\pa{2}}=:\sum_{2\leq p+q\leq l}\cE_{\cR}^{\pa{2}}\pa{p,q}+R_{l+1}^{\pa{2}}\nonumber,
            \end{align}
           with notation understood similarly as in \eqref{explicit_reminder_reminder_terms_example_1}. If at least one derivative acts on $\Pi_{a_1}$, we apply the Cauchy-Schwarz inequality to a product of the form
        \begin{equation}\nonumber
            \begin{aligned}
                \sum_{\alpha,\beta,i,j}\abs{\p{\wt \sM B \tl_1 \Pi_0}_{\star\#}}\cdot \norm{\tl_2 B_1 \sM\bre_j}\cdot \abs{\p{\Pi_1}_{\# \star}},
            \end{aligned}
        \end{equation}
        where $\Pi_0$ and $ \Pi_1$ are generated from the derivatives of $\Pi_{a_1}$ and $\sG\wt\Pi_{a_2}$, respectively, and each contains at least one $\sG$ factor. This gives that 
        \begin{equation}\nonumber
            \begin{aligned}
         {\cE_{\cR}^{\pa{2}}\pa{p,q}} &\prec N^{-1-S-3/2-\pa{p+q+1}/2}\cdot N^{3}\pa{\im m}^2\norma{A}_{\HS}^2\cdot \frac{\pa{\im m}^{n-3}}{\eta^{\mm-n-1}}\cdot \pa{\frac{1}{N\eta}}^{u}\\
     & \lesssim N^{-1-\fR+\varepsilon}\norm{A}_{\HS}^2\pa{k/N}^{\frac13\p{\mm+u-2}}\leq N^{-5/3+\varepsilon}k^{2/3}\norm{A}_{\HS}^2.
            \end{aligned}
        \end{equation}
        If none of the derivatives acts on $\Pi_{a_1}$, we instead apply the Cauchy-Schwarz inequality to a product of the form
        \begin{equation}\nonumber
            \begin{aligned}
                \sum_{\alpha,\beta,i,j}\norm{{\bre_{\star}^\top\wt \sM B \tl_1}}\cdot \norm{\tl_2 B_1 \sM\bre_j}\cdot \abs{\p{\Pi_0}_{\# \star}},
            \end{aligned}
        \end{equation}
        where $\Pi_0$ is generated from the derivatives of $\sG\Pi_{a_2}$ and contains at least one $\sG$ factor. This yields that
        \begin{equation}\nonumber
            \begin{aligned}
   {\cE_{\cR}^{\pa{2}}\pa{p,q}}&\prec N^{-1-S-3/2-\pa{p+q+1}/2}\cdot N^{3}\p{\im m}\norma{A}_{\HS}^2\cdot \frac{\pa{\im m}^{n-3}}{\eta^{\mm-n-1}}\cdot \pa{\frac{1}{N\eta}}^{u}\\
& \lesssim N^{-1-\fR+\varepsilon}\norm{A}_{\HS}^2\pa{k/N}^{\frac13\p{\mm+u-3}}\leq N^{-5/3+\varepsilon}k^{2/3}\norm{A}_{\HS}^2,
            \end{aligned}
        \end{equation}
       provided that at least one of the following conditions holds: (i) $\fR\geq 1$, or (ii) $\mm+u\geq 5$. It remains to consider the exceptional case $\fR=0$ and $\mm+u\leq 4$. From \eqref{character_relation_further_expansion} and \eqref{character_relation_further_expansion2}, this implies the identity 
        \begin{equation}\nonumber
            \begin{aligned}
                \fC_1+\fC_2+\fP_1+\fP_2+\fM+\fS_1+\fS_2+\fI_1+\fI_2+\fE+\sC_1+\sC_2+\sP_1+\sP_2= 0.
            \end{aligned}
        \end{equation}
        In this case, $\cR$ can only take the form
        \begin{equation}\nonumber
            \begin{aligned}
                -\frac{1}{ND}\bsum{a}{\alpha}{\beta}\cC_{\iab}^{p_0,q_0+1}\p{\sM_0\tl_1\sG_1\tl_2\sG_0}_{\star\star}\p{\sG_0}_{\star\star}\p{\sG_0}_{\star\star}.
            \end{aligned}
        \end{equation}
    Under the assumption that no derivatives act on the $\sG_1$ factor, we can extract an additional $\im m$ factor by applying the polarization identity \eqref{expansion_of_im} and using the resulting cancellation. This improves the previous estimate to the sharper bound
        \begin{equation}\nonumber
            \begin{aligned}
    {\cE_{\cR}^{\pa{2}}\pa{p,q}}&\prec N^{-1-S-3/2-\pa{p+q+1}/2}\cdot N^{3}\pa{\im m}^2\norma{A}_{\HS}^2\cdot \frac{\pa{\im m}^{n-3}}{\eta^{\mm-n-1}}\cdot \pa{\frac{1}{N\eta}}^{u}\\
               & \lesssim N^{-1-\fR+\varepsilon}\norm{A}_{\HS}^2\pa{k/N }^{\frac13\p{\mm+u-2}}\leq N^{-5/3+\varepsilon}k^{2/3}\norm{A}_{\HS}^2.
            \end{aligned}
        \end{equation}

        \noindent(9) If $\cR$ is of Type \rmnum{3} with $a_1\geq 1$ and $a_2=0$, then $\mm\geq 3$, $\fR\geq 1$, and we choose the first $\sG$ factor to the right of \smash{$\tl_1$}. In this case, the remainder term takes the form
            \begin{align}
                \cE_{\cR}^{\pa{2}}=&\frac{c_{\cR}}{ND}\bsum{a}{\alpha}{\beta}\frac{1}{p_0!q_0!}\cC_{\iab}^{p_0,q_0+1}\sum_{2\leq p+q\leq l}\sum_{a=1}^D\sum_{i,j\in \cI_a}\frac{1}{p!q!}\cC_{ij}^{p,q+1}\partial_{ij}^p\partial_{ji}^q\Bigg[\p{\wt \sM B \tl_1 B_1 \sM}_{\star j}\p{\sG \wt \Pi_{a_1}\tl_2\Pi_{a_2}}_{i\star}\nonumber\\
            &\left.\times \p{\Pi_{a_3}}_{\star\star}\p{\Pi_{a_4}}_{\star\star}\cdot \prod_{r=1}^{n-3}f^{\pa{r}}\cdot \prod_{r=1}^{u}W_r\right]+R_{l+1}^{\pa{2}}=:\sum_{2\leq p+q\leq l}\cE_{\cR}^{\pa{2}}\pa{p,q}+R_{l+1}^{\pa{2}},\nonumber
            \end{align}
        with notation understood similarly as in \eqref{explicit_reminder_reminder_terms_example_1}. Then, applying the Cauchy-Schwarz inequality to a product of the form
        \begin{equation}\nonumber
            \begin{aligned}
                \sum_{\alpha,\beta,i,j}\abs{\p{\wt \sM B \tl_1 B_1 \sM}_{\star j}}\cdot \abs{\p{\Pi_0\tl_2 \Pi_{a_2}}_{\# \star}},
            \end{aligned}
        \end{equation}
        where $\Pi_0$ is generated from the derivatives of $\p{\sG \wt \Pi_{a_1}\tl_2\Pi_{a_2}}_{i*}$ and contains at least one $\sG$ factor, we obtain 
        \begin{equation}\nonumber
            \begin{aligned}
        {\cE_{\cR}^{\pa{2}}\pa{p,q}}&\prec N^{-1-S-3/2-\pa{p+q+1}/2}\cdot N^{5/2}\p{\im m}\norma{A}_{\HS}^2\cdot \frac{\pa{\im m}^{n-3}}{\eta^{\mm-n}}\cdot \pa{\frac{1}{N\eta}}^{u}\\
               & \lesssim N^{-1/2-\fR+\varepsilon}\norm{A}_{\HS}^2\pa{k /N }^{\frac13\p{\mm+u-2}}\leq N^{-11/6+\varepsilon}k^{1/3}\norm{A}_{\HS}^2\leq N^{-5/3+\varepsilon}k^{2/3}\norm{A}_{\HS}^2.
            \end{aligned}
        \end{equation}

\smallskip
        \noindent(10) All remaining cases in which $\cR$ is of Type \rmnum{3} are impossible. 

\medskip
Combining all the above cases completes the proof of \Cref{claim_estimate_remainder_terms}, and thereby concludes the proof of \Cref{lemma_localized_eigenvector_local_law}. 
\end{proof}

\appendix

\section{Auxiliary estimates}\label{appendix}

In this appendix, we collect some auxiliary estimates that have been extensively used in our main proofs. 

\subsection{Some deterministic estimates}
First, we provide some basic estimates for the matrices $M$ and \smash{$\wh M$} defined in \Cref{defn_Mm} and \Cref{def_wtM}, respectively.

\begin{lemma}\label{lemma_usual_properties_of_m_and_mu}
 Let $A$ be an arbitrary $N\times N$ deterministic matrix with $\|A\|=\OO(N^{-\delta_A})$. Recall that $\qa{E^-,E^+}$ is the support of $\rho_N$. For any constant $\tau>0$, the following estimates hold uniformly for all $z=E+\ii\eta $ with $|z|\le \tau^{-1}$ and $\eta > 0$. 
    
    \begin{enumerate}
        
        \item For $x\in \qa{E^-,E^+}$ and $z=E+\ii \eta$, we have
        \begin{equation}\label{square_root_density}
            \begin{aligned}
                \rho_N(x)\sim \sqrt{\pa{E^+-x}\pa{x-E_-}},\quad\im m(z)\sim
                \begin{cases}
                    \sqrt{\kappa+\eta} & \txt{ for } E\in\qa{E^-,E^+}\\
                    \frac{\eta}{\sqrt{\kappa+\eta}} & \txt{ for } E\notin \qa{E^-,E^+},
                \end{cases}
            \end{aligned}
        \end{equation}
        where recall that $\kappa$ denotes $\kappa:=\abs{E-E^-}\wedge\abs{E-E^+}$. Moreover, we have that        \begin{equation}\label{edge_close_to_edge_sc}
            \begin{aligned}
                \absa{2-E^+}+\absa{2+E^-}=\OO\pa{\|A\|}.
            \end{aligned}
        \end{equation}

        \item  We have the following identity for any $z=E+\ii \eta\in \C_+$:
        \begin{equation}\label{equation_z_im_m}
            \begin{aligned}
                \avga{M(z)M^* (z) }=\frac{\im m (z) }{\im m (z) +\eta}.
            \end{aligned}
        \end{equation}
        In particular, for $E\in\qa{E^-,E^+}$, it gives that
    \begin{equation}\label{equation_E_in_spectrum}
            \begin{aligned}
                \avga{M(E)M^*(E)}=1 . 
            \end{aligned}
        \end{equation}

        \item We have that
        \begin{equation}\label{eq:msc}
            \begin{aligned}
                \absa{m (z) -m_{\txt{sc}} (z) }\lesssim \norma{A}^{1/2}, \quad \norma{M (z) -m_{\txt{sc}} (z)I }\lesssim \norma{A}^{1/2}.
            \end{aligned}
        \end{equation}

        \item Given any polynomial $P$ with maximum degree and coefficients of order $\OO\pa{1}$, we have that
    \begin{equation}\label{estimate_replace_M_by_m}
            \begin{aligned}
                \avga{P\pa{M (z) ,M^* (z) }}-P\pa{m (z) ,\Bar{m} (z) }=\OO\pa{\avga{\Lambda^2}} .
            \end{aligned}
        \end{equation}

        \item For any fixed $k\in \N$, $\pa{s_1,\ldots,s_{k-1}}\in \ha{\emptyset,*}^{k-1}$, and $\pa{a_1,\ldots,a_{k}}\in \qq{D}^{k}$, we have that
        \begin{equation}\label{add_one_more_Lambda}
            \begin{aligned}
                \avga{\pa{\prod_{i=1}^{k-1}M^{s_i}(z)E_{a_i}}\Lambda E_{a_{k}}}=\OO\pa{\avga{\Lambda^2}},
            \end{aligned}
        \end{equation}
        where we adopt the convention that $M^{\emptyset}(z)\equiv M(z)$.

        \item $\wh M$ defined in \eqref{def_ML} is translationally invariant, which means that $\wh M_{ab}(z_1,z_2)=\wh M_{a'b'}(z_1,z_2)$ whenever $a-b = a'-b' \mod D$.

        \item  For any $z_1 = \bar z_2\in \{z,\bar z\}$, we have that
        \begin{align}
            \big\|[1-\wh M(z_1,z_2)]^{-1}\big\|& = \frac{\im m(z)+\eta} {\eta }\lesssim \frac{\im m (z) }{\eta}.\label{1-M}
        \end{align}

        \item For any $z_1 = z_2\in \{z,\bar z\}$ with $\eta/\im m (z) \sim N^{-\varepsilon_g}$ for a constant $\varepsilon_g\in (0, \delta_A/4)$, we have that
        \begin{align}
            \big\|[1-\wh M(z_1,z_2)]^{-1}\big\|&\lesssim \pa{\im m (z) }^{-1}\wedge N^{\varepsilon_g},\label{1-M-2}\\
            \label{1-M-3}
                \left|1-\avga{M(z_1)M(z_2)}\right|^{-1}&\lesssim \pa{\im m (z) }^{-1}\wedge N^{\varepsilon_g}.
            \end{align}

        \item 
        For any $z_1, z_2\in \{z,\bar z\}$ with $\eta=\oo(1)$, we have that
        \be\label{flat_M}
        \max_{a,b,a',b'\in \qq D}\left|\left[\big(1-\wh M_{(1,2)}\big)^{-1}\wh M_{(1,2)} \right]_{ab}-\left[\big(1-\wh M_{(1,2)}\big)^{-1}\wh M_{(1,2)} \right]_{a'b'}\right|\lesssim \frac{N }{\|A\|_{\HS}^2},
        \ee
        where $\wh M_{(1,2)}$ denotes $\wh M_{(1,2)}\equiv \wh M(z_1,z_2)$.

        \item For $z=E+\ii \eta$ with $E\in \qa{E^-,E^+}$, we have that
        \begin{equation}\label{de_ir_0}
            \begin{aligned}
    \qquad \quad \im m (z) \lesssim \absa{1-\avga{M^2 (z) }}\lesssim \im m (z) +\avga{\Lambda^2},\quad \im m (z) \lesssim\absa{1-m^{2} (z) }\lesssim \im m (z) +\avga{\Lambda^2}.
            \end{aligned}
        \end{equation}
        In particular, when $E=\gamma_k$ and $\norma{A}_{\HS}$ satisfies \eqref{eq:condA2}, the estimates  \eqref{de_ir_0} and \eqref{eq:kappa_gamma_k} give that
        \begin{equation}\label{de_ir_1}
            \begin{aligned}
                \absa{1-\avga{M^2 (z) }}\sim \absa{1-m^2 (z) }\sim \sqrt{\kappa+\eta} \sim \fr\pa{k}^{1/3}/N^{1/3} +\sqrt{\eta}.
            \end{aligned}
        \end{equation}

        \item For any $z_1 = \bar z_2\in \{z,\bar z\}$, the leading eigenvalue of $\wh M\pa{z_1,z_2}$ is given by
        \be\label{sumwtM}
        d_1:=\sum_{b=1}^D  \wh M(z,\bar{z})_{1b}=\frac{\im m(z) }{\im m(z)+\eta}, 
        \ee
        which is the Perron–Frobenius eigenvalue of $\wh M(z_1,z_2)$ with $(1,\ldots, 1)^\top$ as the corresponding eigenvector. The other eigenvalues of $\wh M\pa{z_1,z_2}$ satisfy 
        \be\label{eq:otherM}
        d_l = d_1 - a_l - \ii b_l, \quad l=2,3,\ldots,D,
        \ee
        where $a_l,b_l\in \R$ satisfy that  
        \be\label{eq:akbk}
        a_l\ge 0,\quad a_l+|b_l|=\oo(1).
        \ee

        \item For any $z_1 = z_2\in \{z,\bar z\}$ with $\kappa+\eta=\oo\pa{1}$, we can arrange the eigenvalues of $\wh M\pa{z_1,z_2}$ as $\wh d_1,\ldots,\wh d_D$, such that
        \begin{equation}\nonumber
            \begin{aligned}
                \wh{d}_1=\avga{M^2 (z) },\quad \text{and}\quad \wh{d}_l=\wh{d}_1+\oo\pa{1}, \ \ l=2,3,\ldots, D.
            \end{aligned}
        \end{equation}
        Furthermore, we have 
        \be\label{eq:otherM2}
        \wh d_l = d_1 - \wh a_l - \ii \wh b_l, \quad l=1,2,\ldots,D,
        \ee
        where $\wh a_k,\wh b_k\in \R$ satisfy that 
        \be\label{eq:akbk2}
        \wh a_k\ge 0,\quad \wh a_k+|\wh b_k|=\oo(1).
        \ee
        
    \end{enumerate}
\end{lemma}

\begin{proof}
Note that $\rho_N$ is the free convolution of the empirical spectral measure of $\Lambda$ and the semicircle law, whose properties have been extensively studied in the literature. 
First, since $\norm{\Lambda}\lesssim N^{-\delta_A}\ll 1$, we obtain the estimates in \eqref{square_root_density} by applying \cite[Lemma 4.3]{lee2015edge}. 
Second, the estimate  \eqref{edge_close_to_edge_sc} follows directly from equation \eqref{derivative_of_edge} below and the estimate \eqref{add_one_more_Lambda}. 
Third, the identity \eqref{equation_z_im_m} can be derived by taking the imaginary part of both sides of \eqref{self_m}. Then,  \eqref{equation_E_in_spectrum} is an immediate consequence if $E\in \pa{E^-,E^+}$, and this extends to the boundary $E\in \{E^-,E^+\}$ by continuity. Fourth, the first estimate in \eqref{eq:msc} follows from the stability of the self-consistent equation for $m_{\txt{sc}}(z)$, while the second estimate can be easily derived from the estimate 
\[-(m(z)+z)^{-1}-m_{sc}(z)=-(m_{\txt{sc}}(z)+z)^{-1}-m_{\txt{sc}}(z)+\OO(|m(z)-m_{sc}(z)|)=\OO(|m(z)-m_{sc}(z)|),\]
and the Taylor expansion
    \begin{equation}\label{taylor_expansion_M}
        M (z) =\pa{\Lambda-z-m (z) }^{-1}=-\sum_{l=0}^{\infty} \pa{m (z) +z}^{-l-1}\Lambda^l.
    \end{equation}

To show \eqref{estimate_replace_M_by_m}, we first note that by taking the trace of \eqref{taylor_expansion_M} and using $\avga{\Lambda}=0$, we obtain 
\be\label{eq:tracem} m (z) =\avga{M (z) } = -(m(z)+z)^{-1}+\OO\pa{\avga{\Lambda^2}}.\ee
Then, we substitute \eqref{taylor_expansion_M} into $P\pa{M (z) ,M^* (z) }$ and observe that the constant terms cancel, resulting in an error of order $\OO(\langle \Lambda^2\rangle)$ due to \eqref{eq:tracem}. In addition, the contributions from the first-order terms in $\Lambda$ also vanish due to $\avga{\Lambda}=0$, which leads to \eqref{estimate_replace_M_by_m}. 
The estimate \eqref{add_one_more_Lambda} can also be proved by substituting \eqref{taylor_expansion_M} into the LHS and noting that $\avga{\Lambda E_a}=0$ for any $a\in \qq D$.
The translation invariance of \smash{$\wh M$} in part (vi) follows easily from the block translation symmetry of $M$. For part (vii), note that \smash{$\wh M$} is a matrix with non-negative entries when $z_1=\overline z_2$. Thus, with the Perron-Frobenius theorem and the following identity by \eqref{equation_z_im_m}:
    \begin{equation}\nonumber
        \begin{aligned}
            \sum_{b=1}^D\wh M_{ab}\pa{z_1,z_2}=D\avga{M(z_1)E_aM(z_2)}=\avga{M (z) M (z) ^*}=\frac{\im m (z) }{\im m (z) +\eta},
        \end{aligned}
    \end{equation}
    we can conclude that the largest eigenvalue of $\wh M\pa{z_1,z_2}$ is $\im m (z) /\pa{\im m (z) +\eta}$, which implies \eqref{1-M}.

To show the estimate \eqref{1-M-2}, we can assume without loss of generality that $z_1=z_2=z$ and abbreviate $M=M(z)$, $m=m (z)$, \smash{$\wh M(z_1,z_2)=\wh M$}. Using \eqref{taylor_expansion_M}, we find that
    \begin{align}\label{wh_M_approx}
            \wh M_{ab}-\pa{m+z}^{-2}\delta_{ab}&=\OO\pa{\norma{A}},\\
            \label{im_wh_M_approx}
            \im \wh M_{ab}-\im [\pa{m+z}^{-2}]\delta_{ab}&=\OO\pa{\im m\cdot \norma{A}}.
        \end{align}
Then, we decompose $1-\wh M$ as
    \begin{equation}\label{1-wh_M_approx}
        \begin{aligned}
            1-\wh M=\big(1-\pa{m+z}^{-2}\big)-\big(\wh M-\pa{m+z}^{-2}\big).
        \end{aligned}
    \end{equation}
    When $\absa{\re \pa{m+z}}\geq 1/10$, we have
        \(|\im \q{\pa{m+z}^{-2}}|\gtrsim \im\pa{m+z}\geq \im m\),
    while $\im  \wh M - \im[\pa{m+z}^{-2}]$ gives an error by \eqref{im_wh_M_approx}. 
    Hence, for any $\wh \lambda\in \mathrm{Spec}\p{\wh M}$, we have $\im \wh \lambda\gtrsim \im m$. Then, using \eqref{1-wh_M_approx}, we obtain 
    \begin{equation}\nonumber
        \begin{aligned}
            \big\|\big(1-\wh M\big)^{-1}\big\|\lesssim \pa{\im m}^{-1}.
        \end{aligned}
    \end{equation}
    On the other hand, if $\abs{\re \pa{m+z}}\leq 1/10$, then by \eqref{edge_close_to_edge_sc} and \eqref{eq:msc}, we have $E\notin\qa{-2-\kappa_0,-2+\kappa_0}\cup\qa{2-\kappa_0,2+\kappa_0}$ for some small constant $\kappa_0>0$. This gives that
    \begin{equation}\nonumber
        \begin{aligned}
            \abs{1-\pa{m+z}^{-2}}\geq \abs{1-\pa{m_{\txt{sc}} (z) +z}^{-2}}-\oo\pa{1}\gtrsim 1,
        \end{aligned}
    \end{equation}
    which, together with \eqref{1-wh_M_approx} and \eqref{wh_M_approx}, implies 
    \begin{equation}\nonumber
        \begin{aligned}
            \norm{\p{1-\wh M}^{-1}}\lesssim 1\lesssim \pa{\im m}^{-1}.
        \end{aligned}
    \end{equation}
    It remains to show that $\norm{\p{1-\wh M}^{-1}}\lesssim N^{\varepsilon_g}$. By \eqref{eq:msc}, we have  
     \begin{equation}\label{first_appro_1_wh_M}
         \begin{aligned}
             \p{1-\wh M}_{ab}=\pa{1-m^2 (z) }\delta_{ab}+\OO\p{N^{-\delta_A/2}},\quad \forall a,b\in \qq{D}.
         \end{aligned}
     \end{equation}
     Then, using \eqref{equation_z_im_m}, \eqref{eq:msc}, and the condition $\eta/\im m (z) \sim N^{-\varepsilon_g}$ for a constant $\varepsilon_g\in (0, \delta_A/4)$, we obtain 
     \begin{equation}\nonumber
         \begin{aligned}
       \big\|1-\wh M\big\|&\gtrsim 
      1-\absa{m (z)}^2+\OO\p{N^{-\delta_A/2}}\geq 1-\avga{M (z) M^* (z) }+\OO\p{N^{-\delta_A/2}}\\
        & =\frac{\eta}{\im m (z) +\eta}+\OO\p{N^{-\delta_A/2}}\gtrsim N^{-\varepsilon_g},
         \end{aligned}
     \end{equation}
     which implies $\norm{\p{1-\wh M}^{-1}}\lesssim N^{\varepsilon_g}$. 
     This concludes the proof of \eqref{1-M-2}. The estimate \eqref{1-M-3} can be proved using exactly the same argument.

For the estimate \eqref{flat_M}, since $\widehat{M}\pa{z_1,z_2}$ is translationally invariant, its eigenvectors are given by $\mathbf{u}_l$ with $u_l(a)=D^{-1 / 2} \exp (2 \pi \ii (l-1)(a-1) / D)$ for $a,l\in\qq D$. The corresponding eigenvalue can be expressed as
     \begin{equation}\label{eq_lambdak}
     \wh \lambda_l=\sum_{b=1}^D \widehat{M}_{1 b}\left(z_1, z_2\right) e^{2 \pi \ii (l-1)(b-1) / D}.
     \end{equation}
With the spectral decomposition of \smash{$\wh M_{(1,2)}$}, we obtain that
     \begin{equation}\nonumber
         \left(K_{(1,2)}\right)_{a b}\equiv  \left[\big(1-\wh M_{(1,2)}\big)^{-1}\wh M_{(1,2)} \right]_{ab} =\frac{1}{D}\frac{\widehat{\lambda}_1}{1-\widehat{\lambda}_1}+\frac{1}{D} \sum_{l=2}^D \frac{\widehat{\lambda}_l}{1-\widehat{\lambda}_l} e^{2 \pi \ii (l-1)(a-b) / D},
    \end{equation}
    from which we derive that
    \begin{equation}\nonumber
        \bigg|\left(K_{(1,2)}\right)_{a b}-\frac{1}{D}\frac{\widehat{\lambda}_1}{1-\widehat{\lambda}_1}\bigg| \lesssim \max _{2 \leq l \leq D}\big|1-\widehat{\lambda}_l\big|^{-1}.
    \end{equation}
    Hence, it suffices to estimate $1-\wh \lambda_l$ for $l\neq 1$. In fact, the estimate \eqref{flat_M} has been proved in \cite[Lemma A.1]{stone2024randommatrixmodelquantum} when $\kappa\gtrsim 1$. It remains to consider the case where $\kappa\le c$ for a sufficiently small constant $c>0$. Without loss of generality, suppose $E$ is sufficiently close to $E^+$ and $z_1=z_2=z$; the other cases can be shown similarly. In this scenario, we have 
    \be\label{eq:rem+z}
    \operatorname{Re}[(m+z)^{-4}]>0,\quad \text{and}\quad \operatorname{Re}[(m+z)^{-4}] \sim 1. 
    \ee
 Using the expansion \eqref{taylor_expansion_M}, we can write 
    \begin{equation}\label{wh_M_expansion}
        \begin{aligned}
            \widehat{M}(z, z)_{1 b}= & \left(\frac{1}{(m+z)^2}+\frac{2(1+\mathbf 1_{D>2})}{(m+z)^4} \cdot \frac{\|A\|_{\HS}^2}{N}\right) \delta_{1 b} \\
            &+\frac{1}{(m+z)^{4}}\cdot \frac{\|A\|_{\HS}^2} {N}\left(\delta_{2 b}+  \delta_{D b}\mathbf 1_{D>2}\right)+\oo\left(\frac{\|A\|_{\HS}^2} {N}\right).
        \end{aligned}
    \end{equation}
Moreover, applying the identity \eqref{equation_z_im_m} and the Cauchy-Schwarz inequality, we obtain that 
    \begin{equation}\label{sum_wh_M_bound}
        \begin{aligned}
|\wh\lambda_l|\le \sum_{b=1}^D\abs{\wh M_{1b}}\le \frac{1}{N}\sum_{i\in \cI_1}\sum_{j}\absa{M_{ij}}^2=\frac{\im m}{\im m+\eta}<1,\quad \forall   l \in \qq{D}.
        \end{aligned}
    \end{equation}    
Then, when $D>2$, applying \eqref{eq_lambdak}, \eqref{eq:rem+z},  \eqref{wh_M_expansion}, and \eqref{sum_wh_M_bound}, we find that for any $l\ge 2$,
    \begin{equation}\label{eq:whlambdal}
        \begin{aligned}
            \abs{1-\wh \lambda_l}&\geq 1-\abs{\re \wh \lambda_l}\geq 
            \sum_{b\in\{2,D\}}|\operatorname{Re} \widehat{M}_{1 b}|\left[1-\left|\cos (2 \pi(l-1)(b-1) / D)\right|\right] 
            \gtrsim \norma{A}_{\HS}^2/N.
        \end{aligned}
    \end{equation}
    For the case $D=2$, from \eqref{eq_lambdak}, we see that $\widehat{\lambda}_2=\widehat{\lambda}_1-2 \widehat{M}_{12}$. By \eqref{eq:rem+z}, \eqref{wh_M_expansion}, and \eqref{sum_wh_M_bound}, we have that $|\wh\lambda_1|< 1$, $\operatorname{Re} \widehat{M}_{12} > 0$, and \smash{$\operatorname{Re} \widehat{M}_{12} \gtrsim  \norma{A}_{\HS}^2/N.$} Thus, we get that 
\begin{equation}\label{eq:whlambdal2}
\begin{aligned}
&~|1-\widehat{\lambda}_2| =\qa{(1-\operatorname{Re} \widehat{\lambda}_1+2 \operatorname{Re} \widehat{M}_{12})^2+(\operatorname{Im} \widehat{\lambda}_2)^2}^{1/2} \geq 2\operatorname{Re} \widehat{M}_{12}  \gtrsim \|A\|_{H S}^2/N \, .
\end{aligned}
\end{equation}
Combining \eqref{eq:whlambdal} and \eqref{eq:whlambdal2}, we conclude \eqref{flat_M}.

For the estimate \eqref{de_ir_0}, we can assume $E\geq 0$ without loss of generality. We have
     \begin{equation}\nonumber
         \begin{aligned}
             \absa{1-m^2 (z) }\sim \absa{1+m (z) }\sim\absa{1+\re m (z) }+\im m (z) \sim  \big|1-\pa{\re m (z) }^2\big| +\im m (z) ,
         \end{aligned}
     \end{equation}
     where, in the first and third steps, we used that $\absa{1-m (z) }=\absa{1-m_{\txt{sc}} (z) }+\oo\pa{1}\sim 1$ and $\absa{1-\re m (z) }=\absa{1-\re m_{\txt{sc}} (z) }+\oo\pa{1}\sim 1$ for $z=E+\ii \eta$ with $E\geq 0$. Then, we obtain that 
     \begin{equation}\nonumber
         \begin{aligned}
             &\big|1-\pa{\re m (z) }^2\big|+\im m (z) \leq \big|1-\pa{\re m (z) }^2-\pa{\im m (z) }^2\big|+\im m (z) +\pa{\im m (z) }^2\\
             \sim& \big|1-\absa{m (z) }^2\big|+\im m (z) \lesssim \absa{1-\avga{M (z) M^* (z) }}+\im m (z) +\avga{\Lambda^2}\lesssim \im m (z) +\avga{\Lambda^2},
         \end{aligned}
     \end{equation}
     where we applied \eqref{estimate_replace_M_by_m} in the third step and used \eqref{equation_z_im_m} and \eqref{square_root_density} in the last step. From the two estimates above, we derive:
     \begin{equation}\nonumber
         \begin{aligned}
             \im m (z) \lesssim\absa{1-m^2 (z) }\lesssim \im m (z) +\avga{\Lambda^2}.
         \end{aligned}
     \end{equation}
Together with \eqref{estimate_replace_M_by_m}, this gives us $\absa{1-\avga{M^2 (z) }}=\absa{1-m^2 (z) }+ \OO\pa{\avga{\Lambda^2}}\lesssim \im m (z) +\avga{\Lambda^2}$.
    On the other hand, using a similar approach as in the proof of \eqref{1-M-3}, we can show that \[\absa{1-\avga{M^2 (z) }}\gtrsim \im m.\] 
    This concludes the proof of \eqref{de_ir_0}. Next, applying \eqref{eq:kappa_gamma_k} and \eqref{square_root_density}, we obtain that $\im m (z)\sim \sqrt{\kappa+\eta} \gg \avga{\Lambda^2}$, which implies \eqref{de_ir_1}.

 For part (xi), we assume without loss of generality that $z_1=\bar{z}_2=z$. Using \eqref{taylor_expansion_M}, we can derive 
\begin{equation}\label{wh_M_z_bar_z_expansion}
    \begin{aligned}
        \widehat{M}(z, \bar{z})_{1 b}= & \bigg(\frac{1}{|m+z|^2}+\frac{2\left(1+\mathbf{1}_{D>2}\right)\re [\pa{m+z}^{-2}]}{|m+z|^2} \cdot \frac{\|A\|_{H S}^2}{N}\bigg) \delta_{1 b} \\ & +\frac{1}{|m+z|^{4}}\cdot \frac{\|A\|_{H S}^2}{N}\left(\delta_{2 b}+\delta_{D b}\mathbf{1}_{D>2} \right)+\oo\left(\frac{\|A\|_{H S}^2}{N}\right).
    \end{aligned}
\end{equation}
Since both $\widehat{M}$ and $\widehat{d}$ are real, the estimate \eqref{eq:akbk} follows easily by taking the real part of \eqref{eq_lambdak} and applying \eqref{wh_M_z_bar_z_expansion}.
Finally, for part (xii), we assume $z_1=z_2=z$ without loss of generality. By using the translation invariance, we obtain that 
$$
\sum_{b=1}^D \widehat{M}(z, z)_{1 b}=\frac{1}{D} \sum_{a, b=1}^D \widehat{M}(z, z)_{a b}=\frac{1}{D N} \sum_{i, j} M_{i j}(z) M_{j i}(z)=\left\langle M^2(z)\right\rangle,
$$
which gives that $\widehat{d}_1=\left\langle M^2 (z) \right\rangle$. From \eqref{wh_M_approx}, we have $|\widehat{d}_k-\widehat{d}_1|=\oo(1)$ for $2\le k\le D$. Additionally, we observe that \smash{$\widehat{d}_1=1+\oo(1)$} when $\kappa+\eta=\oo(1)$ by \eqref{de_ir_0}.  Moreover, from \eqref{sum_wh_M_bound}, we have\smash{$\operatorname{Re} \widehat{d}_k \leq d_1$}. These results conclude the proofs of \eqref{eq:otherM2} and \eqref{eq:akbk2}. 
\end{proof}

In the above proof, we have used the following differential equation for $E_t^\pm$.

\begin{lemma}
    In the setting of \Cref{def_moving_Lambda}, suppose $\Lambda_t=f\pa{t}\Lambda$ for some differentiable function $f\in C^1(\qa{a,b})$. Then, we have    \begin{equation}\label{derivative_of_edge}
        \begin{aligned}
            \partial_t E_t^{\pm}=f'\pa{t}\avga{\Lambda M_t^2(E_t^{\pm})},\quad \forall t\in\qa{a,b}.
        \end{aligned}
    \end{equation}
\end{lemma}

\begin{proof}
    Without loss of generality, we only prove the differential equation for $E_t^+$. Taking the derivative of both sides of
    \begin{equation}\nonumber
        \begin{aligned}
     m_t\p{E_t^+}=\big\langle\pa{\Lambda_t-E_t^+-m_t\p{E_t^+}}^{-1}\big\rangle,
        \end{aligned}
    \end{equation}
    we obtain that     \begin{equation}\label{derivative_of_eq_of_m_E_t}
        \begin{aligned}
            \partial_t m_t\p{E_t^+}=\avga{\pa{\partial_t E_t^+ +\partial_t m_t\p{E_t^+}-f'\p{t}\Lambda} M_t^2\p{E_t^+}}.
        \end{aligned}
    \end{equation}
    By equation \eqref{equation_E_in_spectrum}, we have
    \begin{equation}\nonumber
        \begin{aligned}
            \avga{M_t^2\p{E_t^+}}=\avga{ M_t\p{E_t^{+}} M_t^*\p{E_t^{+}}}=1.
        \end{aligned}
    \end{equation}
    Applying it to \eqref{derivative_of_eq_of_m_E_t}, we get \eqref{derivative_of_edge}.
\end{proof}

\subsection{Estimates on deterministic shifts}
Second, we show that the two shifts $\Delta_{\txt{ev}}$ (defined in \eqref{definition_shift_eigenvector}) and $\Delta_{\txt{e}}$ (defined in \eqref{definition_shift_eigenvalue}) indeed represent the shift of the quantiles up to some negligible error. 

\begin{lemma}
\label{lemma_modification_of_shift}
For any $1\le k\leq DN/2$, suppose that $\norma{A}_{\txt{HS}}$ satisfies the condition \eqref{eq:condA2}. Defined $z_1$ and $\Delta_{\txt{ev}}$ as in \eqref{eq:z1} and \eqref{definition_shift_eigenvector} for a constant $\varepsilon\in\pa{0,1}$.
Then, we have 
\begin{equation}\label{modification_shift_eigenvector}
            \begin{aligned}
                \Delta_{\txt{ev}}(z_1)=\gamma_k-\gamma_{k}^{\txt{sc}}+\OO\pa{\avga{\Lambda^2}^2+\avga{\Lambda^2}\sqrt{\kappa(\gamma_k)+\eta}}=\gamma_k-\gamma_{k}^{\txt{sc}}+\OO\bigg(\frac{\norma{A}_{\HS}^4}{N^{2}}+\frac{k^{ 1/ 3}\norma{A}_{\HS}^2}{N^{4/ 3-{\varepsilon}/{2}}}\bigg).
            \end{aligned}
        \end{equation}
where $\kappa(\gamma_k)=|E^+-\gamma_k|\wedge|\gamma_k-E^-|\sim {k^{2/3}}/{N^{2/3}}$ by \eqref{eq:kappa_gamma_k}.
The shift $\Delta_{\txt{e}}$ in \eqref{definition_shift_eigenvalue} satisfies a similar bound:         \begin{equation}\label{modification_shift_eigenvalue}
\begin{aligned}
\Delta_{\txt{e}}(k,\eta)=\gamma_k-\gamma_{k}^{\txt{sc}}+\OO\pa{\avga{\Lambda^2}^2+\avga{\Lambda^2}\sqrt{\kappa(\gamma_k)+\eta}}=\gamma_k-\gamma_{k}^{\txt{sc}}+\OO\bigg(\frac{\norma{A}_{\HS}^4}{N^{2}}+\frac{k^{ 1/ 3}\norma{A}_{\HS}^2}{N^{4/ 3-{\varepsilon}/{2}}}\bigg). 
\end{aligned}
\end{equation}
        In particular, if we take $\varepsilon<\varepsilon_A$, the errors in these two estimates are bounded by $N^{-2/3-\e_A}k^{-1/3}$. 
        The corresponding results also hold for $DN/2<k\le DN$.
\end{lemma}

\begin{proof}
We always assume $k\leq DN/2$ throughout the following proof. We begin with the proof of \eqref{modification_shift_eigenvalue}. First, we can replace $z_t=\gamma_k(t)+\ii \eta$ in the definition of $\Delta(t)$ with its real part $\gamma_k\pa{t}$ and establish that:    \begin{equation}\label{estimate_change_definition_of_shift_eigenvalue}    
        \begin{aligned}        
            \absa{\frac{\avga{M_t(z_t)\Lambda M_t^{*}(z_t)}}{\avga{M_t(z_t)M_t^{*}(z_t)}}-\avga{M_t(\gamma_k(t))\Lambda M_t^{*}(\gamma_k(t))}}\lesssim \avga{\Lambda^2} \frac{\eta}{\sqrt{\kappa_t+\eta}},
        \end{aligned}
    \end{equation}
    where $\kappa_t=\absa{\gamma_k(t)-E_t^+}\wedge\absa{\gamma_k(t)-E_t^-}$ (recall \Cref{def_moving_Lambda}). 
    Without loss of generality, we assume $t=1$ for the proof, and we abbreviate $M_1(z_1)\equiv M$, $m_1(z_1)\equiv m$, and $\gamma_k=\gamma_k(1)$. 
    Since $\absa{1-\avga{MM^{*}}}=\eta/\pa{\im m+\eta}\lesssim \eta/\sqrt{\kappa_1+\eta}$ by \eqref{equation_z_im_m} and \eqref{square_root_density}, and $\avga{M\Lambda M^{*}}=\OO(\avg{\Lambda^2})$ by \eqref{add_one_more_Lambda}, we have
    \begin{equation}\label{first_modification_shift_eigenvalue}
        \begin{aligned}
            \absa{\frac{\avga{M(z_1)\Lambda M^{*}(z_1)}}{\avga{M(z_1)M^{*}(z_1)}}-\avga{M(z_1)\Lambda M^{*}(z_1)}}\lesssim \avga{\Lambda^2} \frac{\eta}{\sqrt{\kappa+\eta}}.
        \end{aligned}
    \end{equation}
Next, by \eqref{taylor_expansion_M}, we have the decomposition for any $z\in \C$:
\begin{equation}\label{first_order_decomposition_M}
        \begin{aligned}
            M(z)=-\pa{m (z) +z}^{-1}-\Lambda \tilde{M} (z) ,
        \end{aligned}
    \end{equation}
    where \smash{$\tilde{M} (z)$} is defined as
    \begin{equation}\nonumber
        \begin{aligned}
            \tilde{M} (z) :=\sum_{l=0}^{\infty}\pa{m (z) +z}^{-l-2}\Lambda^{l}.
        \end{aligned}
    \end{equation}
    Furthermore, we have that   \begin{equation}\label{eliminate_imaginary_part_m}
        \begin{aligned}
            \absa{m(z_1)-m\pa{\gamma_k}}&=\absa{\int_{0}^{\eta} m'\pa{\gamma_k+\r i s} \rd s}= \absa{\int_{0}^{\eta} \frac{\avga{M^2\pa{\gamma_k+\r i s} }}{1-\avga{M^2\pa{\gamma_k+\r i s}}} \rd s} \\
            &\lesssim  \int_{0}^{\eta} \frac{1}{\sqrt{\kappa+s}}\r d s\lesssim \frac{\eta}{\sqrt{\kappa+\eta}},
        \end{aligned}
    \end{equation}
    where in the second step, we used the equation in \eqref{derivative_m_t_z_with_t_and_z} below, and in the third step, we applied \eqref{de_ir_1}. From \eqref{eliminate_imaginary_part_m}, we derive that
\begin{equation}\label{eliminate_imaginary_part_tilde_M}
        \begin{aligned}
            \norm{\tilde{M}(z_1)-\tilde{M}\p{\gamma_k}}&=\bigg\|{\sum_{l=0}^{\infty}\pa{\pa{m(z_1)+z_1}^{-l-2}-\pa{m(\gamma_k)+\gamma_k}^{-l-2}}\Lambda^{l}}\bigg\|\\
            &\lesssim \pa{\absa{m(z_1)-m(\gamma_k)}+\absa{z_1-\gamma_k}}\sum_{l=0}^{\infty}\pa{C\norma{\Lambda}}^l\lesssim \frac{\eta}{\sqrt{\kappa+\eta}}.
        \end{aligned}
    \end{equation}
Additionally, with \eqref{first_order_decomposition_M}, we can express that
    \begin{equation}\nonumber
        \begin{aligned}
            \avga{M (z) \Lambda M^{*} (z) }=\avg{\tilde{M} (z) \Lambda^3\tilde{M}^{*} (z) }+\frac{1}{m (z) +z}\avg{\Lambda^2 \tilde{M}^{*} (z) }+\frac{1}{\Bar{m} (z) +\Bar{z}}\avg{\Lambda^2 \tilde{M} (z) },
        \end{aligned}
    \end{equation}
    which, together with \eqref{eliminate_imaginary_part_tilde_M}, implies that
    \begin{equation}\label{second_modification_shift_eigenvalue}
        \begin{aligned}
            \absa{\avga{M(z_1)\Lambda M^{*}(z_1)}-\avga{M\p{\gamma_k}\Lambda M^{*}\p{\gamma_k}}}\lesssim \avga{\Lambda^2}\frac{\eta}{\sqrt{\kappa+\eta}}.
        \end{aligned}
    \end{equation}
Finally, combining \eqref{first_modification_shift_eigenvalue} and \eqref{second_modification_shift_eigenvalue}, we conclude \eqref{estimate_change_definition_of_shift_eigenvalue}.
    
The estimate \eqref{modification_shift_eigenvalue} follows directly from the equation:
\begin{equation}\label{derivative_estimate_modification_shift_eigenvalue}
        \begin{aligned}
            \frac{\rd}{\rd t} \gamma_k\p{t}-\avga{M_t\p{\gamma_k(t)}\Lambda M_t^{*}\p{\gamma_k(t)}}=\OO\pa{\avga{\Lambda^2}\sqrt{\kappa_t}+\avga{\Lambda^2}^2}.
        \end{aligned}
    \end{equation}
In fact, noting that $\gamma_k\pa{0}=\gamma_k^{\txt{sc}}$, we can integrate equation \eqref{derivative_estimate_modification_shift_eigenvalue} and apply \eqref{estimate_change_definition_of_shift_eigenvalue} to complete the proof of \eqref{modification_shift_eigenvalue}. 
During this process, we also utilize the estimates \smash{$\kappa_t\sim \kappa(\gamma_k)\sim k^{2/3}/N^{2/3}$, $\eta=N^{-2/3+\varepsilon}k^{-1/3}$}, and \smash{$\avga{\Lambda^2}\lesssim N^{-1/3-2\varepsilon_A}k^{-2/3}$}.

For the proof of \eqref{derivative_estimate_modification_shift_eigenvalue}, we take the derivative of both sides of 
    \begin{equation}\nonumber
        \begin{aligned}
            m_t (z) =\big\langle{\pa{t\Lambda-m_t (z) -z}^{-1}}\big\rangle
        \end{aligned}
    \end{equation}
    with respect to $t$ or $z$, yielding: \begin{equation}\label{derivative_m_t_z_with_t_and_z}
        \begin{aligned}
            \partial_t m_t (z) =-\frac{\avga{\Lambda M_t^2 (z) }}{1-\avga{M_t^2 (z) }},\quad
            \partial_z m_t (z) =\frac{\avga{M_t^2 (z) }}{1-\avga{M_t^2 (z) }}.
        \end{aligned}
    \end{equation}
   This gives the identities 
    \begin{equation}\label{replace_derivative}
        \begin{aligned}
            \partial_t m_t (z) =-\partial_z m_t (z) \frac{\avga{\Lambda M_t^2 (z) }}{\avga{ M_t^2\pa
            z}},\quad \text{and}\quad \partial_z M_t (z) =\partial_z \pa{t\Lambda-m_t (z) -z}^{-1}=\frac{M_t^2 (z) }{1-\avga{M_t^2 (z) }}.
        \end{aligned}
    \end{equation}
    By the definition of $\gamma_k\pa{t}$ (recall \eqref{eq:gammak}), we have 
    \begin{equation}\label{def_gamma_k_t}
        \begin{aligned}
            \frac{1}{\pi}\int_{\gamma_k\p{t}}^{E_t^+} \im m_t\pa{x} \rd x=\frac{k-1/2}{DN}.
        \end{aligned}
    \end{equation}
    Taking the derivative of both sides of this equation with respect to $t$ and using $\im m_t\p{E_t^+}=0$, we obtain:   
        \begin{align}
            &~\gamma_k'\pa{t}\im m_t\p{\gamma_k\p{t}}=\im \int_{\gamma_k\p{t}}^{E_t^+}\partial_t m_t\pa{x}\rd x=-\im \int_{\gamma_k\p{t}}^{E_t^+}\partial_x m_t\p{x}\frac{\avga{\Lambda M_t^2\p{x}}}{\avga{ M_t^2\p
            x}}\rd x\label{equation_derivative_gamma_k_t}\\
            =&~\im \bigg(m_t\p{\gamma_k\pa{t}}\frac{\avga{\Lambda M_t^2\p{\gamma_k\p{t}}}}{\avga{ M_t^2\p{
            \gamma_k\p{t}}}}-m_t\p{E_t^+}\frac{\avga{\Lambda M_t^2\p{E_t^+} } }{\avga{M_t^2\p{E_t^+}}} \bigg)+ \im \int_{\gamma_k\p{t}}^{E_t^+} m_t\p{x}\partial_x\bigg(\frac{\avga{\Lambda M_t^2\p{x}}}{\avga{ M_t^2\p{x}}}\bigg)\rd x\nonumber,
        \end{align}
    where we used \eqref{replace_derivative} in the second step and applied integration by parts in the third step.  
    Using \eqref{replace_derivative}, \eqref{de_ir_0}, \eqref{add_one_more_Lambda}, and \eqref{square_root_density}, we can get the estimate:    \begin{equation}\label{estimate_derivative_1_modification_shift}
        \begin{aligned}
            \partial_x\bigg({\frac{\avga{\Lambda M_t^2\pa{x}}}{\avga{ M_t^2\pa{x}}}}\bigg)=\OO\bigg(\frac{\avga{\Lambda^2}}{\sqrt{E_t^+-x}}\bigg).
        \end{aligned}
    \end{equation}
    Also, combining \eqref{de_ir_0} with the fact that $\absa{1-m_t\p{x}}=\absa{1-m_{\txt{sc}} (x) }+\oo\pa{1}\sim 1$ for $x\in[\gamma_k\p{t},E_t^+]$, we get 
    \begin{equation}\label{bound_1_plus_m_t_x}
        \begin{aligned}
            \absa{1+m_t\p{x}}\lesssim \sqrt{E_t^+-x}+\avga{\Lambda^2}.
        \end{aligned}
    \end{equation}
    Applying  \eqref{estimate_derivative_1_modification_shift} and \eqref{bound_1_plus_m_t_x} to \eqref{equation_derivative_gamma_k_t}, we obtain that
        \begin{align}
            &\gamma_k'\p{t}\im m_t\p{\gamma_k\p{t}}\nonumber\\
            =&\im \pa{m_t\p{\gamma_k\p{t}}\frac{\avga{\Lambda M_t^2\p{\gamma_k\p{t}}}}{\avga{ M_t^2\p{\gamma_k\p{t}}}}-m_t\p{E_t^+}\frac{\avga{\Lambda M_t^2\p{E_t^+} } }{\avga{M_t^2\p{E_t^+}}} -  \int_{\gamma_k\p{t}}^{E_t^+}\partial_x\pa{\frac{\avga{\Lambda M_t^2\p{x}}}{\avga{ M_t^2\p{x}}}}\rd x} +\OO\pa{\avga{\Lambda^2}^2\sqrt{\kappa_t}+\avga{\Lambda^2}\kappa_t}\nonumber\\
            =& \im \pa{\qa{1+m_t\p{\gamma_k\p{t}}}\frac{\avga{\Lambda M_t^2\p{\gamma_k\pa{t}}}}{\avga{ M_t^2\p{\gamma_k\pa{t}}}}-\qa{1+m_t\p{E_t^+}}\frac{\avga{\Lambda M_t^2\p{E_t^+} } }{\avga{M_t^2\p{E_t^+}}}}+\OO\pa{\avga{\Lambda^2}^2\sqrt{\kappa_t}+\avga{\Lambda^2}\kappa_t}\nonumber\\
            =& \re\qa{1+m_t\p{\gamma_k\p{t}}}\im \pa{\frac{\avga{\Lambda M_t^2\p{\gamma_k\p{t}}}}{\avga{ M_t^2\p{\gamma_k\p{t}}}}} +\re\pa{\frac{\avga{\Lambda M_t^2\p{\gamma_k\p{t}}}}{\avga{ M_t^2\p{\gamma_k\p{t}}}}}\im m_t\p{\gamma_k\p{t}} +\OO\pa{\avga{\Lambda^2}^2\sqrt{\kappa_t}+\avga{\Lambda^2}\kappa_t}\nonumber\\
            =& \avga{M_t\p{\gamma_k\p{t}}\Lambda M_t^*\p{\gamma_k\p{t}}}\im m_t\p{\gamma_k\p{t}}+\OO\pa{\avga{\Lambda^2}^2\sqrt{\kappa_t}+\avga{\Lambda^2}\kappa_t},  \label{eq:derivegammak}
        \end{align}
    where in the third step, we used that $m_t\p{E_t^+}$ is real and $M_t\p{E_t^+}$ is a Hermitian matrix, and in the fourth step, we used \eqref{bound_1_plus_m_t_x} and that
    \begin{equation}\label{bring_in_im_m}
        \begin{aligned}
            &\frac{\avga{\Lambda M_t^2\p{\gamma_k\p{t}}}}{\avga{ M_t^2\p{\gamma_k\p{t}}}}-\avga{M_t\p{\gamma_k\p{t}}\Lambda M_t^*\p{\gamma_k\p{t}}}\\
            =& \frac{\avga{\Lambda M_t^2\p{\gamma_k\p{t}}}}{\avga{ M_t^2\p{\gamma_k\p{t}}}}-\frac{\avga{M_t\p{\gamma_k\p{t}}\Lambda M_t^*\p{\gamma_k\p{t}}}}{\avga{M_t\p{\gamma_k\p{t}}M_t^*\p{\gamma_k\p{t}}}}=\OO\pa{\avga{\Lambda^2}^2+\avga{\Lambda^2}\sqrt{\kappa_t}}.
        \end{aligned}
    \end{equation}
   In the derivation, we have used \eqref{equation_E_in_spectrum}, \eqref{add_one_more_Lambda}, and noted that $M_t(\gamma_k(t))-M_t^*(\gamma_k(t))=2\ii \im m_t(\gamma_k(t))\cdot  M_t(\gamma_k(t))M_t^*(\gamma_k(t))$, where $\im m_t(\gamma_k(t))\sim \sqrt{\kappa_t}$ by \eqref{square_root_density}. From \eqref{eq:derivegammak} and using \eqref{square_root_density}, we obtain \eqref{derivative_estimate_modification_shift_eigenvalue}.

For the proof of \eqref{modification_shift_eigenvector}, we again consider the flow in \Cref{def_moving_Lambda} with $\Lambda_t=t\Lambda$, $t\in \qa{0,1}$, and denote
    \begin{equation}\label{def_shift_f_t}
        \begin{aligned}
            f\p{t}:=\re \pa{z_t+m_t\pa{z_t}+\frac{1}{m_t\pa{z_t}}}.
        \end{aligned}
    \end{equation}
Note that $\Delta_{\txt{ev}}=f\p{1}$ and $f\p{0}=0$. Thus, it suffices to prove the following estimate that is similar to \eqref{derivative_estimate_modification_shift_eigenvalue}:
    \begin{equation}\nonumber
        \begin{aligned}
            f{'}\p{t}-\gamma_k'\p{t}=\OO\pa{\avga{\Lambda^2}^2+\avga{\Lambda^2}\sqrt{\kappa_t+\eta}},\quad \forall t\in [0,1].
        \end{aligned}
    \end{equation}
    First, taking the derivative of $f\p{t}-\gamma_k(t)=f\p{t}-\re z_t$ gives us    \begin{equation}\label{derivative_f_t_minus_gamma_k_t}
        \begin{aligned}
            f{'}\p{t}-\gamma_k'\p{t}=\re \pa{\frac{\rd m_t\p{z_t}}{\rd t}\pa{1-\frac{1}{m_t^2\p{z_t}}}}.
        \end{aligned}
    \end{equation}
    Next, taking the derivative of both sides of 
            \(m_t\p{z_t}=\avg{\p{t\Lambda-m_t\p{z_t}-z_t}^{-1}}\)
    with respect to $t$, and using \smash{$\partial_t z_t=\gamma_k'\pa{t}$}, we get that 
    \begin{equation}\nonumber
        \begin{aligned}
            \frac{\rd m_t\p{z_t}}{\rd t}=\frac{\gamma_k'\p{t}\avg{M_t^2\p{z_t}}-\avg{\Lambda M_t^2\p{z_t}}}{1-\avg{M_t^2\p{z_t}}}.
        \end{aligned}
    \end{equation}
    Plugging this equation into \eqref{derivative_f_t_minus_gamma_k_t} and using \eqref{de_ir_1}, we deduce:
    \begin{equation}\nonumber
        \begin{aligned}
            \absa{f'\p{t}-\gamma_k'\p{t}}\lesssim \absa{\gamma_k'\p{t}\avg{M_t^2\p{z_t}}-\avga{\Lambda M_t^2\p{z_t}}}.
        \end{aligned}
    \end{equation}
With a similar argument as in \eqref{bring_in_im_m} above, we obtain that 
    \begin{equation}\nonumber
        \begin{aligned}
            \frac{\avga{M_t\p{z_t}\Lambda M_t^{*}\p{z_t}}}{\avga{M_t\p{z_t}M_t^{*}\p{z_t}}}-\frac{\avga{M_t\p{z_t}\Lambda M_t\p{z_t}}}{\avga{M_t\p{z_t}M_t\p{z_t}}}=\OO\pa{\avg{\Lambda^2}\sqrt{\kappa_t+\eta}}.
        \end{aligned}
    \end{equation}
    Combining this with \eqref{estimate_change_definition_of_shift_eigenvalue} and \eqref{derivative_estimate_modification_shift_eigenvalue}, we get
    \begin{equation}\nonumber
        \begin{aligned}
            \gamma_k'\p{t}\avga{M_t^2\p{z_t}}-\avga{\Lambda M_t^2\p{z_t}}=\OO\pa{\avga{\Lambda^2}^2+\avga{\Lambda^2}\sqrt{\kappa_t+\eta}},
        \end{aligned}
    \end{equation}
    which completes the proof of \eqref{modification_shift_eigenvector}.
\end{proof}

\subsection{Some multi-resolvent estimates}
Finally, we provide some multi-resolvent estimates that follow from the anisotropic local law \eqref{eq:aniso_local}.

\begin{lemma}[Estimates on resolvents]\label{lemma_necessary_estimates}
For any fixed integer $p\in \N$, let $(\Lambda_i)_{1\leq i \leq p}$ be an arbitrary sequence of $D\times D$ block matrices similar in form to $\Lambda$, consisting of $N\times N $ deterministic blocks $A_i$ and $A_i^*$ with $\|A_i\|=\oo(1)$. Let $(B_i)_{1\leq i \leq p}$ be an arbitrary sequence of deterministic matrices satisfying $\|B_i\|\le 1$.  
Furthermore, consider a sequence of spectral parameters $z_i=E_i+\ii \eta_i\in \C_+$ for $i\in\qq{p}$, satisfying $\absa{z_i}\leq \tau^{-1}$ for a small constant $\tau>0$.  
Suppose the anisotropic local law \eqref{eq:aniso_local} holds for all $G_i-M_i$, where $G_i\equiv  G(z_i,H,\Lambda_i)$ and $M_i\equiv M(z_i,\Lambda_i)$ is the deterministic limit of $G_i$ as defined in \Cref{defn_Mm} with parameter $z_i$. 
Moreover, denote $m_i\equiv \avga{M_i}$ and assume that $ N\eta_i \im m_i (z_i) \gtrsim 1$ for all $i\in \qq{p}$. Then, for any deterministic unit vectors $\bu, \bv\in\mathbb{C}^{DN}$ and $s_i\in \{\emptyset,*\}^{p}$, the following estimates hold:
    \begin{equation}\label{entprodG}
        \begin{aligned}
                \bu^*\pa{\prod_{i=1}^p G_i^{s_i} B_i}\bv\prec \frac{\pa{\max_{1\leq i\leq p} \im m_i}^{1_{p\geq 2}}}{\eta^{p-1}},\qquad \avga{\prod_{i=1}^p G_i^{s_i} B_i}\prec \frac{\pa{\max_{1\leq i\leq p} \im m_i}^{1_{p\geq 2}}}{\eta^{p-1}},
        \end{aligned}
    \end{equation}
    where we denote $\eta:=\min_{i}\eta_i$ and adopt the convention that $G_i^\emptyset=G_i$. 
    
    We denote by $\Pi_l$ a product consisting of $l$ elements in $\ha{G_i^{s}:i\in \qq{p},s\in \{\emptyset,*\}}$, along with some elements from $\ha{M_i}$ and \smash{$\ha{E_a}_{a=1}^D$}. Moreover, suppose all $\Lambda_i$ have the form $\Lambda_i=c_i\Lambda$ for some deterministic coefficients $c_i$ of order $\OO(1)$. Then, we have the following estimates.
    \begin{enumerate}
        \item A loop containing one factor of $\Lambda$ satisfies that
    \begin{equation}\label{one_loop_estimate}
    \left\langle\Pi_l \Lambda\right\rangle \prec \begin{cases} N^{-1}\norma{\Lambda}_{\txt{HS}}^2=D\avga{\Lambda^2}, & \text { if } l=0, \\ N^{-1 / 2}\norm{\Lambda}_{\txt{HS}}\cdot \pa{\max_{1\leq i\leq p} \im m_i }^{\mathbf 1_{l\geq 2}}\cdot \eta^{-(l-1)}, & \text { if } l \geq 1.\end{cases}
    \end{equation}
    \item A loop containing two factors of $\Lambda$ satisfies that
    \begin{equation}\label{two_loop_estimate}
        \left\langle\Pi_{l_1} \Lambda \Pi_{l_2} \Lambda\right\rangle \prec \begin{cases}N^{-1}\norm{\Lambda}_{\txt{HS}}^2=D\avga{\Lambda^2}, & \text { if } l_1+l_2=0, \\ N^{-1}\norm{\Lambda}_{\txt{HS}}^2\cdot \pa{\max_{1\leq i\leq p} \im m_i}^{\mathbf 1_{l_1+l_2\geq 2}} \cdot \eta^{-(l_1+l_2-1)},  & \text { if } l_1+l_2 \geq 1.\end{cases}
    \end{equation}
    \end{enumerate}
    The same estimates also hold if some $\Lambda$ factors on the LHS of \eqref{one_loop_estimate} and \eqref{two_loop_estimate} are replaced by $\tLambda$ (defined in \Cref{lem_regular_estimate_localized_eigenvector}) or $\hLambda_t$ (defined in \Cref{lem_regular_estimate_localized_eigenvalue}) for $t\in\qa{0,1}$.
\end{lemma}

\begin{proof}
    When $p=1$, the estimate \eqref{entprodG} is an immediate consequence of the anisotropic local law \eqref{eq:aniso_local}. For $p\geq 2$, using the trivial bound $\|G_i\|\le \eta_i^{-1}\le \eta$, we find that for any deterministic unit vectors $\bu,\bv\in \C^N$, 
    \begin{equation}\label{eq:multilocal}
        \begin{aligned}
            \bu^*\pa{\prod_{i=1}^p G_i^{s_i} B_i}\bv\lesssim \norma{\bu^*G_1^{s_1}}\cdot\norma{G_p^{s_p}B_p\bv}\cdot \eta^{-\pa{p-2}}.
        \end{aligned}
    \end{equation}
On the other hand, for any deterministic unit vector $\bv\in \C^N$, we have
    \begin{equation}\label{eq:multilocal2}
        \begin{aligned}            \norma{G_i\bv}=\sqrt{\bv^*G_i^*G_i\bv}=\sqrt{\frac{\im (\bv^* G_i\bv)}{\eta}}\prec \sqrt{\frac{\im m_i}{\eta}},
        \end{aligned}
    \end{equation}
    where in the second step, we used Ward's identity \eqref{eq_Ward}, and in the third step, we applied the anisotropic local law \eqref{eq:aniso_local} along with the condition $N\eta_i\im m_i\gtrsim 1$. Plugging \eqref{eq:multilocal2} into \eqref{eq:multilocal} yields the first estimate in \eqref{entprodG}. The second estimate in \eqref{entprodG} follows immediately from the first.
    
    
    When $l=0$, \eqref{one_loop_estimate} follows directly from the expansion of $M$ in \eqref{taylor_expansion_M} and the fact that $\avga{\Lambda E_{a}}=0$ for any $a\in \qq{D}$. For the case $l\ge 1$, we can prove it by applying the eigendecomposition of $\Lambda$ and utilizing \eqref{entprodG}. 
    For \eqref{two_loop_estimate}, the case $l_1+l_2=0$ case is trivial, so we only need to consider the case $l_1+l_2\geq 1$. If $l_1,l_2\geq 1$, using the Cauchy-Schwarz inequality, we get that 
    \begin{equation}\nonumber
        \begin{aligned}
            \absa{\left\langle\Pi_{l_1} \Lambda \Pi_{l_2} \Lambda\right\rangle}\leq \left\langle\Pi_{l_1} \Lambda^2 \Pi_{l_1}^*\right\rangle^{1/2}\left\langle\Pi_{l_2} \Lambda^2 \Pi_{l_2}^*\right\rangle^{1/2}.
        \end{aligned}
    \end{equation}
    Then, applying the eigendecomposition of $\Lambda^2$ and using \eqref{entprodG}, we obtain \eqref{two_loop_estimate}. 
    Next, suppose $l_1=0$ or $l_2=0$. Assume $l_2=0$ without loss of generality. Then, $\Pi_{l_2}$ is a product of some elements from $\ha{M_i}$ and \smash{$\ha{E_a}_{a=1}^D$}. We apply the decomposition \eqref{first_order_decomposition_M} to $M_i$'s in $\Pi_{l_2}$, use the singular value decompositions of $A^2$ and $AA^*$, and apply the estimate \eqref{entprodG} to conclude the proof of \eqref{two_loop_estimate} (for more details, readers can refer to \cite[equation (8.25)-(8.31)]{stone2024randommatrixmodelquantum}).
   Finally, when some $\Lambda$ factors are replaced by \smash{$\tLambda$ or $\hLambda_t$}, 
   we only need to use \eqref{shift_bound_localized_eigenvector}, \eqref{shift_bound_localized_eigenvalue}, and \eqref{entprodG} to bound the additional terms generated by the shifts $\Delta_{\txt{ev}}$ or $\Delta\pa{t}$.
\end{proof}


\end{document}